\documentclass[paper=a4,twoside=false,DIV=calc,BCOR=0.0cm,draft=false]{scrbook}

\usepackage[a4paper,showframe=false]{geometry}

\usepackage[T1]{fontenc}
\usepackage{newpxtext}
\usepackage{newpxmath}

\usepackage{pdfpages}

\usepackage{amsthm,amsmath,amsfonts,amssymb}

\usepackage{mathrsfs}

\usepackage[disable,textsize=footnotesize]{todonotes}

\usepackage{comment}

\usepackage{etoc, blindtext}	
\newcommand{\chaptertoc}[1][Chapter contents]{%
  \etocsettocstyle{\addsec*{#1}}{}%
  \setcounter{tocdepth}{2}
  \localtableofcontents%
}

\usepackage[intoc]{nomencl} 
\makenomenclature
\makeatletter
\def\thenomenclature{%
  \section*{\nomname}
  \if@intoc\addcontentsline{toc}{section}{\nomname}\fi%
  \nompreamble
  \list{}{%
    \labelwidth\nom@tempdim
    \leftmargin\labelwidth
    \advance\leftmargin\labelsep
    \itemsep\nomitemsep
    \let\makelabel\nomlabel}}
\makeatother

\usepackage[colorlinks=true]{hyperref}	
\hypersetup{
  bookmarksopen, 
  bookmarksopenlevel=1, 
  pdftitle={The Hitchin-cscK system - an infinite-dimensional hyperk\"ahler reduction},
  pdfauthor={Carlo Scarpa},
  pdfsubject={PhD thesis, complex differential geometry},
  pdfkeywords={differential geometry, constant scalar curvature, Hitchin system, compact Kahler manifolds}}

\usepackage{tikz-cd}

\usepackage{mathtools}	
\usepackage{bbm}	

\usepackage{tensor}	

\newcommand{\m}[1]{\mathcal{#1}}
\newcommand{\bb}[1]{\mathbb{#1}}
\newcommand{\bm}[1]{\boldsymbol{#1}}	

\newcommand{\mrm}[1]{\mathrm{#1}}
\newcommand{\mscr}[1]{\mathscr{#1}}
\newcommand{\f}[1]{\mathfrak{#1}}

\newcommand{\sslash}{/\!\!/}			

\newcommand{\diff}{\partial}
\newcommand{\bdiff}{\bar{\partial}} 

\newcommand{\SpecRad}{r} 

\newcommand{\I}{\mathrm{i}\mkern1mu}	
\newcommand{\transpose}{\intercal}		

\DeclarePairedDelimiter\abs{\lvert}{\rvert}	
\DeclarePairedDelimiter\norm{\lVert}{\rVert}	

\DeclarePairedDelimiter{\set}{\{}{\}}	
\newcommand{\tc}{\mathrel{}\mathclose{}\middle|\mathopen{}\mathrel{}}	

\newcommand{\ext}[1]{\bigwedge\nolimits^{#1}}	

\newcommand{\cotJ}{T^*\!\!\!\mscr{J}}	

\newtheorem{thm}{Theorem}[section]
\newtheorem{thm_intro}{Theorem}

\newtheorem{prop}[thm]{Proposition}
\newtheorem{lemma}[thm]{Lemma}
\newtheorem{cor}[thm]{Corollary}

\theoremstyle{definition}
\newtheorem{definition}[thm]{Definition}

\theoremstyle{remark}

\newtheorem{rmk}[thm]{Remark}

\graphicspath{{./figure/}}

\title{The Hitchin-cscK system}
\subject{
}
\subtitle{an infinite-dimensional hyperk\"ahler reduction}
\date{
}
\author{Carlo Scarpa}
\titlehead{\textit{A thesis submitted 
for the degree of Doctor of Philosophy in Geometry and Mathematical Physics of the Scuola Internazionale Superiore di Studi Avanzati - SISSA}
}
\publishers{Advisor: prof. Jacopo Stoppa}
\lowertitleback{The results in this thesis are the author's own, except when explicitly indicated}

\begin{document}

\frontmatter


\maketitle

\chapter*{}

\begin{center}
{\Large \textsc{abstract}}
\end{center}

\noindent We present an infinite-dimensional hyperk\"ahler reduction that extends the classical moment map picture of Fujiki and Donaldson for the scalar curvature of K\"ahler metrics. We base our approach on an explicit construction of hyperk\"ahler metrics due to Biquard and Gauduchon. The construction is motivated by how one can derive Hitchin’s equations for harmonic bundles from the Hermitian Yang-Mills equation, and yields a system of moment map equations which modifies the constant scalar curvature K\"ahler (cscK) condition by adding a ``Higgs field'' to the cscK equation. In the special case of complex curves, we recover previous results of Donaldson, while for higher-dimensional manifolds the system of equations has not yet been studied. We study the existence of solutions to the system in some special cases. On a Riemann surface, we extend an existence result for Donaldson's equation to our system. We then study the existence of solutions to the moment map equations on a class of ruled surfaces which do not admit cscK metrics, showing that adding a suitable Higgs term to the cscK equation can stabilize the manifold. Lastly, we study the system of equations on abelian and toric surfaces, taking advantage of a description of the system in symplectic coordinates analogous to Abreu's formula for the scalar curvature.

\setcounter{tocdepth}{1}
\tableofcontents

\listoftodos

\chapter{Introduction}

This thesis studies the problem of finding metrics with special curvature properties on a compact complex manifold. It builds on a parallel between two important problems in complex differential geometry, the existence of constant scalar curvature K\"ahler metrics and the theory of Hermitian Yang-Mills connections, to study a new set of equations whose solutions should exhibit interesting geometric properties.

We briefly outline some features of the two problems. These will serve as a motivation for the main question we address in this thesis: \emph{is there a natural way to add a ``Higgs field'' to the constant scalar curvature equation and to obtain a system of equations analogous to Hitchin's equations for Higgs bundles?} 

The existence of K\"ahler metrics with constant scalar curvature is one of the most studied problems in complex differential geometry and many interesting results have been obtained in the last twenty years. The general setting is as follows: consider a compact complex manifold~$M$ together with a K\"ahler class~$\Omega\in H^2(M,\bb{R})$ and look for K\"ahler forms~$\omega\in\Omega$ such that
\begin{equation*}
S(\omega)=\hat{S}
\end{equation*}
where~$\hat{S}$ is a constant determined by the class~$\Omega$ and the first Chern class of~$M$. By parametrizing the K\"ahler forms in~$\Omega$ as~$\omega_0+\I\diff\bdiff\varphi$ for a fixed reference form~$\omega_0$ and potentials~$\varphi\in\m{C}^\infty_0(M,\bb{R})$ we can see the constant scalar curvature equation for K\"ahler metrics (the cscK equation, from now on) as a fourth order fully nonlinear elliptic PDE.

It is now well-known that the cscK equation can be interpreted as the zero-locus equation for the moment map of an infinite-dimensional Hamiltonian action on the space~$\mscr{J}$ of complex structures on~$M$ that are compatible with a reference symplectic form. This was first noted by Quillen, Donaldson~\cite{Donaldson_scalar} and Fujiki~\cite{Fujiki_moduli} in the~$90$'s; framing the cscK problem in the context of Hamiltonian actions has led to many important results (e.g.~\cite{Donaldson_scalarcurvature_embeddingsI}). Together with previous work of Tian on the Futaki invariant~\cite{Tian_positive_scalar_curvature}, this feature of the cscK problem prompted the formulation of the ``Yau-Tian-Donaldson'' conjecture for the existence of cscK metrics. A major success of this theory is a verification of the YTD conjecture in the case when~$\Omega$ is proportional to~$\mrm{c}_1(M)$, establishing the equivalence between existence of K\"ahler-Einstein metrics and~$K$-stability. We refer to~\cite{Szekelyhidi_libro} for an overview of these results.

Broadly speaking, the YTD conjecture states that a polarized manifold~$(M,L)$ should admit a metric of constant scalar curvature in the class~$\mrm{c}_1(L)$ if and only if~$(M,L)$ satisfies some kind of algebraic stability condition. A motivation for this conjecture comes from the relation between symplectic reductions and GIT quotients in a finite-dimensional setting: the Kempf-Ness Theorem implies that the existence of zeroes of a moment map can be characterized in terms of a stability notion for the orbits for the action, and this stability condition in turn can be tested using a numerical condition, Mumford's criterion.

When studying the cscK equation we have to consider a K\"ahler reduction in \emph{infinite} dimensions, so the Kempf-Ness Theorem can not be directly applied to obtain an algebraic characterization of the existence of cscK metrics. However, there is at least one case in which this finite-dimensional picture has been successfully applied to a similar infinite-dimensional problem, namely the Hermitian Yang-Mills (HYM) equation. Let~$E$ be a holomorphic vector bundle over a complex curve with a fixed K\"ahler form~$\omega$: the HYM equation for a Hermitian metric~$h$ on~$E$ is
\begin{equation*}
F(h) = \lambda\,\mathbbm{1}\otimes\omega
\end{equation*}
where~$\lambda$ is a topological constant, essentially the slope of~$E$. Atiyah and Bott~\cite{AtiyahBott_YangMills} showed that this can be interpreted as the zero moment map equation for the Hamiltonian action of unitary gauge transformations on the space~$\mscr{A}$ of compatible~$\bar{\diff}$-operators of~$E$. This space is an affine infinite-dimensional K\"ahler manifold, and the action of the gauge group preserves this structure. The HYM equation arises when looking for zeroes of the moment map along the orbits of the complexified action, and for a fixed holomorphic structure on~$E$ there is a Hermitian metric solving the HYM equation if and only if~$E$ is \emph{slope-stable}. This is known as the Kobayashi-Hitchin correspondence; we refer to~\cite{Lubke_KobayashiHitchin} for an in-depth discussion of these results. Slope stability is a numerical condition for the bundle~$E$, hence we see an example that a ``Kempf-Ness criterion'' can hold also in an infinite-dimensional setting.

The Hamiltonian interpretation of both problems suggests that, to some extent, we might hope to find new features of the cscK equations drawing inspiration from the well-developed theory of the HYM equation. One of the best examples of the fruitfulness of this parallel between the two problems are the lower bounds for the Calabi functional found by Donaldson~\cite{Donaldson_lower_bounds}, that were directly inspired from the bounds on the Yang-Mills functional in~\cite{AtiyahBott_YangMills}.

In this thesis we will focus on a feature of the moment map description of the Hermitian Yang Mills problem noticed by Hitchin~\cite{Hitchin_self_duality}. Let~$\mscr{A}$ be the space of~$\bdiff$-operators of a complex vector bundle~$E$; the holomorphic cotangent space of~$\mscr{A}$ carries a hyperk\"ahler structure, and the Hamiltonian action of the gauge group on~$\mscr{A}$ lifts naturally to an action on~$T^*\!\mscr{A}$ that is Hamiltonian with respect to all the symplectic forms in the hyperk\"ahler family. Along the orbits of a complexification of the action this gives \emph{Hitchin's Higgs bundle equations}
\begin{equation*}
\begin{cases}
\bar{\diff}\phi = 0\\
\nonumber F(h) + [\phi, \phi^{*_h}] = \lambda\,\mathbbm{1}\otimes\omega
\end{cases}
\end{equation*}
the system of moment map equations coming from the infinite-di\-men\-sio\-nal hyperk\"ahler reduction. The term~$\phi\in\m{A}^{1,0}(\mrm{End}(E))$
is called a \emph{Higgs field}, and is an element of the cotangent space of~$\mscr{A}$. The Higgs bundle equations have been extensively studied and lead to an extraordinarily rich theory. For example, if we fix a Higgs field~$\phi$ satisfying the equation~$\bdiff\phi=0$ (i.e.~$\phi$ is holomorphic) then the existence of solutions to the second equation is characterized by an algebraic stability condition on the pair~$(E,\phi)$, generalizing slope-stability. The hyperk\"ahler geometry underlying the system of equations descends to the moduli space of solutions, that has a very rich geometric structure.

It is natural to ask if a similar construction can be carried out for the cscK equation, i.e. if we can upgrade the moment map picture for the cscK equation to a hyperk\"ahler reduction. This problem in the case of a complex curve has already been studied by Donaldson~\cite{Donaldson_hyperkahler}, who found a hyperk\"ahler metric in a neighbourhood of the zero section of the holomorphic cotangent bundle~$\cotJ$ and studied a Hamiltonian action on this space. In fact we can expect a hyperk\"ahler structure to exist on~$\cotJ$ for a base manifold~$M$ of arbitrary dimension, since~$\mscr{J}$ is a K\"ahler manifold. From general results of Feix~\cite{Feix_hyperkahler} and Kaledin~\cite{Kaledin_hyperkahler} the cotangent bundle of a K\"ahler manifold always carries a hyperk\"ahler metric, at least in a neighbourhood of the zero section. One of the issues that arise when trying to generalize Donaldson's approach to higher dimensional manifolds is that one would like to obtain a fairly explicit expression for the hyperk\"ahler structure, but it is not clear how to do this in general from the theory developed by Feix and Kaledin.

In this work instead we consider an explicit description due to Biquard and Gauduchon~\cite{Biquard_Gauduchon} of the hyperk\"ahler metric on the cotangent bundle of a K\"ahler symmetric space. This allows us to find a hyperk\"ahler metric on~$\cotJ$ for a base manifold~$M$ of any dimension, as the infinite-dimensional manifold~$\mscr{J}$ is formally a symmetric space. More precisely, the hyperk\"ahler metric is defined on an open neighbourhood of the zero section~$\mscr{J}\hookrightarrow\cotJ$, and the restriction of this metric to~$\mscr{J}$ coincides with the usual K\"ahler metric of~$\mscr{J}$. We then show how the Hamiltonian action on~$\mscr{J}$ lifts to an action on~$\cotJ$ that is Hamiltonian with respect to all the symplectic forms in the hyperk\"ahler family, giving a set of moment maps on~$\cotJ$ that restricted to the zero section coincide with the scalar curvature.

We call the resulting system of moment map equations the \emph{Hitchin-constant scalar curvature} system (HcscK, in short). It is a system of equations for a complex structure~$J$ and a first-order deformation~$\alpha$ of~$J$; there is a clear similarity with Hitchin's system, although the equations have a less clear geometric interpretation. The HcscK system is
\begin{equation}\label{eq:HcscK_system}
\begin{cases}
\mrm{div}\left(\diff^*\alpha\right)=0\\
2\,S(\omega,J)-2\,\hat{S}+\mrm{div}\,X(J,\alpha)=0
\end{cases}
\end{equation}
where all the metric quantities are computed with respect to the K\"ahler metric defined from the reference symplectic form and the complex structure~$J$, and~$X(J,\alpha)$ is a vector field defined from~$\alpha$ and~$J$. The general expression of~$X$ is slightly complicated, and can be computed from the potential of the Biquard-Gauduchon hyperk\"ahler metric. To write~$X(J,\alpha)$, notice first that~$\alpha\bar{\alpha}$ is similar to a Hermitian matrix (see the proof of Proposition~\ref{prop:BiquardGauduchonAC+}), so that it has~$n$ real non-negative eigenvalues. Introduce the spectral function
\begin{equation}\label{eq:funzioni_accessorie_HcscK}
\psi(x)=\frac{1}{2}\left(1+\sqrt{1-x}\right)^{-1}
\end{equation}
and apply it to~$\alpha\bar{\alpha}$, obtaining a matrix~$\hat{\alpha}:=\psi(\alpha\bar{\alpha})$. Then, we can finally write~$X(J,\alpha)$ as 
\begin{equation*}
X(J,\alpha)=2\,\mrm{Re}\left(g(\nabla^a\alpha,\bar{\alpha}\hat{\alpha})\diff_{z^a}-g(\nabla^{\bar{b}}\alpha,\bar{\alpha}\hat{\alpha})\diff_{\bar{z}^b}-2\nabla^*(\alpha\bar{\alpha}\hat{\alpha})\right).
\end{equation*}
As in the case of the HYM equation and the cscK equation, also this system of equations should be \emph{complexified}, by fixing a complex structure~$J$ on~$M$ and considering the HcscK system as a system of equations for a K\"ahler form~$\omega$ in a fixed class and a ``Higgs term''~$\alpha$. It is slightly more difficult to describe this complexification for the HcscK system than for the classical cscK equation, and this is one of the reasons why we have not yet been able to formulate a stability condition analogous to~$K$-stability for the HcscK system.

\section{Summary of results}

Most of the results in this thesis have already appeared in articles of the same author,~\cite{ScarpaStoppa_hyperk_reduction},~\cite{ScarpaStoppa_HcscK_curve} and~\cite{ScarpaStoppa_HcscK_abelian}. In particular the general theory of the HcscK system developed in Chapter~\ref{chap:HyperkahlerReduction} is essentially taken from~\cite{ScarpaStoppa_hyperk_reduction}. The main improvements over~\cite{ScarpaStoppa_hyperk_reduction} are a formulation of the moment map equations in arbitrary dimensions and a discussion of an infinitesimal complexification of the equations.

Chapter~\ref{chap:classicalScalCurv} provides some background material for the rest of the thesis. We will mainly focus on describing~$\mscr{J}$ as the space of smooth sections of a bundle over~$M$ with fibres isomorphic to Siegel's upper half space~$\mrm{Sp}(2n)/U(n)$, after~\cite{Donaldson_scalar}. This allows us to describe a natural K\"ahler structure on~$\mscr{J}$, and we will briefly study the Riemannian geometry of~$\mscr{J}$ in Section~\ref{sec:KahlerStrJ}. In particular we will prove that~$\mscr{J}$ is a formally symmetric space, i.e. its curvature tensor is covariantly constant. This is analogous to a result of Donaldson on the space of K\"ahler potentials, and in fact the two spaces are closely related, as we describe in Section~\ref{sec:complexification_cscK}. This result has also been found with a different approach by Berndtsson (private communication), and gives another motivation for expecting that the Biquard-Gauduchon construction of an explicit hyperk\"ahler metric on~$\cotJ$ will also work in our situation. A large part of Chapter~\ref{chap:classicalScalCurv} is devoted to explaining Donaldson's interpretation of the constant scalar curvature as a moment map in~\cite{Donaldson_scalar}, and in particular we show how to formally complexify the resulting moment map equation, following the discussion in~\cite{Donaldson_SymmKahlerHam}.

In Chapter~\ref{chap:HyperkahlerReduction} we present the construction of a hyperk\"ahler structure on~$\cotJ$ using the Hermitian symmetric space approach of~\cite{Biquard_Gauduchon}.
\begin{thm_intro}\label{thm:HKThmIntro}
A neighbourhood of the zero section in the holomorphic cotangent bundle~$\cotJ$ is endowed with a natural hyperk\"ahler structure, whose restriction to the zero section coincides with the K\"ahler metric of~$\mscr{J}$. For each~$J\!\in\!\!\mscr{J}$, the first order deformations that are in this neighbourhood are those~$\alpha\in\m{A}^{1,0}(T^{0,1}M)$ such that the eigenvalues of the endomorphism~$\alpha\bar{\alpha}$ are strictly less than~$1$.
\end{thm_intro}
\begin{rmk}
The hyperk\"ahler metric we will describe in Chapter~\ref{chap:HyperkahlerReduction} is invariant under the~$U(1)$-action on the fibres of~$\cotJ\to\mscr{J}$, as also happens for the hyperk\"ahler metric on~$T^*\!\mscr{A}$ described by Hitchin~\cite{Hitchin_self_duality}. From the finite-dimensional picture of~\cite{Feix_hyperkahler, Kaledin_hyperkahler} we expect that in fact any~$U(1)$-invariant hyperk\"ahler metric on (an open subset of)~$\cotJ$ that restricts to the usual K\"ahler metric on~$\mscr{J}$ essentially coincides with the metric of Theorem~\ref{thm:HKThmIntro}.
\end{rmk}
A review of the results of Biquard and Gauduchon can be found in Section~\ref{sec:BiquardGauduchon}. The construction of the hyperk\"ahler metric and the proof of Theorem~\ref{thm:HKThmIntro} are given in Section~\ref{sec:cotJ_hyperk}. After this we focus on the extension of the Hamiltonian action on~$\mscr{J}$, obtaining the second main result of the Chapter.
\begin{thm_intro}\label{thm:HamiltonianHyperkIntro}
The Hamiltonian action on~$\mscr{J}$ lifts to an action on~$\cotJ$ that is Hamiltonian with respect to the hyperk\"ahler family of symplectic forms. The moment map equations for this action form the HcscK system~\eqref{eq:HcscK_system}.
\end{thm_intro}
In order to write the moment map equations explicitly we need to first compute the Biquard-Gauduchon hyperk\"ahler metric on~$T^*\left(\mrm{Sp}(2n)/U(n)\right)$ in Section~\ref{sec:BiquardGauduchon}, see in particular Lemma~\ref{lemma:differenziale_BiquardGauduchon}. We remark again the importance of having an explicit expression of the hyperk\"ahler metric, such as the one provided in~\cite{Biquard_Gauduchon}.

In Section~\ref{sec:complexification} we study the complexification of the HcscK system in the case of an \emph{integrable} deformation of the complex structure. This means that we restrict to considering Higgs terms~$\alpha_0\in\m{A}^{1,0}(T^{0,1}M)$ satisfying~$\diff\alpha_0=0$; in this case~$\alpha_0$ determines a cohomology class~$[\bar{\alpha}_0]\in H^1(T^{1,0}M)$. The main result of Section~\ref{sec:complexification} is that we can complexify the HcscK system by considering~\eqref{eq:HcscK_system} for a fixed complex structure~$J$, K\"ahler class~$[\omega_0]$ and deformation class~$[\alpha_0]\in H^1(T^{0,1}M)$; the variables are instead~$\omega\in[\omega_0]$ and~$\alpha\in[\alpha_0]$ with the additional condition that~$\omega$ and~$\alpha$ should be compatible.

Chapter~\ref{chap:examples} is focused on obtaining some existence results for the HcscK system on curves and surfaces. The equations in dimension~$1$ and~$2$ seem more treatable than in the general case mainly because the eigenvalues of~$\alpha\bar{\alpha}$, and so the morphism~$\psi(\alpha\bar{\alpha})$, can be written in a particularly simple form. For example in complex dimension~$1$ there is only one eigenvalue~$\delta$ that can be expressed equivalently as
\begin{equation*}
\delta=\abs{\mrm{det}(\alpha)}^2=\norm{\alpha}_{g_J}^2
\end{equation*}
while in complex dimension~$2$ we can express the two eigenvalues as the roots of the characteristic polynomial of~$\alpha\bar{\alpha}$
\begin{equation*}
\delta^\pm=\frac{1}{2}\left(\mrm{Tr}\left(\alpha\bar{\alpha}\right)\pm\sqrt{\mrm{Tr}\left(\alpha\bar{\alpha}\right)^2-4\abs{\mrm{det}(\alpha)}^2}\right).
\end{equation*}

In Section~\ref{sec:curve} we consider the HcscK system on a curve. The system essentially becomes trivial for genus~$0$ or~$1$, while for a high genus curve we partially generalize a result of Donaldson~\cite{Donaldson_hyperkahler} and~\cite{Hodge_phd_thesis}, finding a general existence result.
\begin{thm_intro}\label{thm:curve_existence_intro}
Let~$\omega_0$ be the hyperbolic metric on a high-genus compact Riemann surface~$\Sigma$. The HcscK system on~$\Sigma$ admits a set of solutions parametrized by pairs~$(\tau,\beta)$, consisting of a holomorphic quadratic differential~$\tau$ and a holomorphic~$1$-form~$\beta$, with~$\m{C}^{0,\frac{1}{2}}(\omega_0)$-norm bounded by some Sobolev constants of~$\Sigma$.
\end{thm_intro}

The situation for complex surfaces is more complicated. One interesting result is that a surface on which there are no cscK metrics can nonetheless admit solutions to the HcscK system. In other words, an appropriate choice of a Higgs field~$\alpha$ can \emph{stabilize} a previously unstable manifold. We show an example of this phenomenon in Section~\ref{sec:ruled_surface}, by studying the equations on a ruled surface using the explicit description of K\"ahler metrics via momentum profiles of Hwang and Singer~\cite{Hwang_Singer}.
\begin{thm_intro}\label{thm:solution_mm_reale}
Fix a high-genus compact Riemann surface~$\Sigma$, endowed with the hyperbolic metric~$\omega_{\Sigma}$. Let~$M$ be the ruled surface~$M = \bb{P}(\m{O}\oplus T\Sigma)$, with projection~$\pi\!: M \to \Sigma$ and relative hyperplane bundle~$\m{O}(1)$, endowed with the K\"ahler class 
\begin{equation*}
[\omega_m] = [\pi^*\omega_{\Sigma}] + m\, c_1(\m{O}(1)),\, m > 0.
\end{equation*}
Then for all sufficiently small~$m$ the HcscK system~\eqref{eq:HcscK_system} can be solved on~$(M, [\omega_m])$. 
\end{thm_intro}
On the other hand it is well-known that, for all positive~$m$,~$(M,[\omega_m])$ does not admit a cscK metric (see~\cite[\S~$3.3$ and \S~$5.2$]{Szekelyhidi_phd}).

Finally, in Chapter~\ref{chap:symplectic_coords} we will show an interesting feature of the HcscK system in symplectic coordinates, that can be applied to study the equations on abelian or toric manifolds: if we restrict to considering~$\bb{T}^n$-invariant tensors in~\eqref{eq:HcscK_system} then, under a simple change of variables, the HcscK equations can be decoupled. The resulting system is similar to Abreu's equation (c.f.~\cite{Abreu_toric}) for the scalar curvature of K\"ahler metrics on toric manifolds:
\begin{equation}\label{eq:symplectic_HcscK_intro}
\begin{cases}
\left(\xi^{ab}\right)_{,ab}=0\\
\left(\left(1-\xi\,D^2u\,\bar{\xi}\,D^2u\right)^{\frac{1}{2}}D^2u^{-1}\right)^{ab}_{,ab}=-C, 
\end{cases}
\end{equation}
where~$u$ is a local potential of the complex structure (or equivalently of the K\"ahler metric) with positive-definite Hessian~$u_{,ab}=D^2u$,~$\xi$ is a symmetric matrix-valued function such that~$\alpha=\xi\,D^2u$ and~$C$ is a topological constant (essentially the average of the scalar curvature). The matrix-valued functions appearing in~\eqref{eq:symplectic_HcscK_intro} must also satisfy some boundary conditions: in the case of a complex torus, one just has to impose periodicity of~$\xi$ and~$D^2u$, while for a toric manifold the correct boundary conditions are more complicated, analogous to \emph{Guillemin's boundary conditions}. We refer to Section~\ref{sec:symp_boundary_abelian} and Section~\ref{sec:symp_boundary_toric} for the details.

On abelian varieties in particular we will use this expression to describe a variational characterization of the HcscK system, proving a uniqueness result for solutions to the system in dimension~$2$. We also describe a more explicit set of solutions to the HcscK system, under some simplifying assumptions.

On a toric surface instead we find that~\eqref{eq:symplectic_HcscK_intro} has some interesting properties: if the matrix~$\xi$ comes from an integrable first-order deformation~$\alpha$ of the complex structure, then it solves the complex moment map equation~$(\xi^{ij})_{,ij}=0$ if and only if it is identically~$0$, see~Corollary~\ref{cor:toric_integrability}. Also, we show in Corollary~\ref{cor:toric_Futaki} that the real moment map equation can be solved only if the toric surface is~$K$-stable, or equivalently (c.f. Corollary~$1$ in \cite{Donaldson_cscKmetrics_toricsurfaces}), only if the surface admits a cscK metric. This is a first obstruction result to the HcscK system, and might be a first step towards the development of a stability theory for the HcscK system. We will get back to this point in the next section.

\section{Future research and open questions}

This work is just a first study of the Hitchin-cscK system, and many aspects have not been treated completely. The main question that remains open is that of the formal complexification of the equations, that we only gave in an implicit form in Section~\ref{sec:complexification}. Once we have an explicit complexification of the equations there is a natural way (c.f.~\cite{Wang_Futaki}) to define a Futaki-like invariant for the HcscK system, obtain results analogous to Matsushima's criterion and define a generalization of the~$K$-energy for the HcscK system. In~\cite{GarciaFernandez_PHD} it is shown how to obtain these classic results for the cscK problem directly from the formal complexification of the moment map equation. These results would allow us to generalize some of the aspects of the theory of the cscK problem to our system: these would allow us to take a first step in defining a stability condition for triples~$(M,[\omega_0],[\alpha_0])$, generalizing~$K$-stability, to characterize the existence of solutions to the HcscK system.

The results in Chapter~\ref{chap:symplectic_coords} for abelian varieties and toric manifolds seem to go in this direction, even without a precise description of the complexified equations. Indeed we find a functional, generalizing Mabuchi's~$K$-energy, whose Euler-Lagrange equation is the real moment map equation of the HcscK system. The convexity properties of the~\emph{$\textit{HK}$-energy} are not clear, but nonetheless these results give a hint of the interesting features of the HcscK system. For a toric surface we proved that~$K$-stability is a necessary condition to having solutions to the HcscK system; this has been obtained by a simple integration by parts, following the approach of~\cite{Donaldson_stability_toric}. However it is not clear if~$K$-stability is also a sufficient condition for the solvability of the toric HcscK system when we consider a non-zero Higgs term, while it is well-known from \cite{Donaldson_cscKmetrics_toricsurfaces} that~$K$-stability implies the existence of a cscK metric in this context. Evidently, to find the appropriate stability condition for the HcscK system on a toric surface we should consider a different point of view. It is possible that a careful study of the toric~$\textit{HK}$-energy of Lemma~\ref{lemma:toricHKenergy} might point to a possible modification of~$K$-stability that takes into account the presence of a non-zero Higgs term.

Another aspect of the formal complexification of the HcscK system we should consider is that our (implicit) description of the complexification in Section~\ref{sec:complexification} bears on the assumption that the Higgs term~$\alpha$ is an integrable deformation of the complex structure. However, the results in Chapter~\ref{chap:symplectic_coords} seem to indicate that solutions to the complex moment map equation are, in general, \emph{not} integrable. For example, the integrability of the Higgs term is actually an obstruction to having zeroes of the complex moment map on a toric surface, see Corollary~\ref{cor:toric_integrability}. It is therefore important to generalize the infinitesimal complexification of the Hamiltonian action~$\mscr{G}\curvearrowright\cotJ$ to Higgs terms that are not necessarily integrable.

There are some other natural directions along which to study the HcscK system, following the parallel with Hitchin's Higgs bundle equations. In this work we mostly studied the HcscK system from the point of view of the cscK equation, but it would be really interesting to find an analogue of Hitchin's integrable system in our setting.

Recall that in some situations there is a direct connection between solutions to the HYM equation and the cscK equation. Consider a slope-stable vector bundle over a curve~$E\to\m{C}$ of genus~$g\geq 1$, so that by the Hitchin-Kobayashi correspondence there is a solution to the HYM equation. In~\cite{Burns_stability} it is shown that the manifold~$\bb{P}(E)$ then carries a cscK metric; the precise statement can be found in~\cite[Theorem~$1.6$]{Fine_phd}, see also~\cite[Theorem~$A$]{Hong_fibered} for a generalization to higher-dimensional base manifolds. Ross and Thomas~\cite{RossThomas_obstruction} gave a converse to this result. The natural question is to start instead from a slope-stable Higgs bundle~$(E,\phi)$ on a curve and see if one can find solutions to the HcscK system on~$\bb{P}(E)$; a first step would be to see if~$\phi$ induces in a natural way a ``Higgs term''~$\alpha$. We have a first result in this direction in Section~\ref{sec:def_complex_bundle}, but the deformations of the complex structure considered in that Section are not compatible with any K\"ahler form.

It would also be interesting to apply the ideas that we used to derive the HcscK system to other problems in differential geometry. This can probably be done in a variety of situations, since the techniques used to obtain the HcscK system can likely be adapted to many other geometric differential equations coming from K\"ahler reductions. The first problem to examine from this point of view is the existence of K\"ahler-Einstein metrics; we can take as a starting point the Hamiltonian description of the K\"ahler-Einstein equations in~\cite{Donaldson_Ding} and start the hyperk\"ahler construction anew. This is of interest also for studying the HcscK equations, since we expect to obtain a set of equations similar to the HcscK system but with PDEs of lower order. The cscK problem reduces to a Monge-Ampére equation when the K\"ahler class is a multiple of the canonical class, so there might be some algebraic conditions on the Higgs term that allow to obtain a similar result for the HcscK system.

\section*{Acknowledgements}

The research leading to these results has received funding from the European
Research Council under the European Union's Seventh Framework Programme (FP7/2007-2013) /
ERC Grant agreement no. 307119.

First and foremost I would like to thank my advisor, Jacopo Stoppa, for his support and encouragement throughout my journey as a PhD student in SISSA and for sharing his insight on the cscK problem and mathematics in general.

I wish to thank Joel Fine and Julius Ross for reading and reviewing this thesis. They provided valuable suggestions for improving this work and suggested interesting new lines of research for studying the HcscK system. Much of the work in Chapter~\ref{chap:symplectic_coords} was developed while visiting Julius Ross at the University of Illinois at Chicago. I sincerely thank Julius for the many important discussions and the Department of Mathematics at UIC for welcoming and supporting me via the grant NSF~RTG~1246844.

I wish to thank all the people who encouraged me to pursue this project and especially Ruadha\'i Dervan for useful advice and suggestions.

Finally, I would like to thank all the members of the complex geometry group at IGAP, Claudio Arezzo, Zakarias Sj\"ostr\"om Dyrefelt, Samreena and especially Enrico Schlitzer, Veronica Fantini and Annamaria Ortu for sharing the joys and woes of K\"ahler geometry.

\newpage

\renewcommand{\nompreamble}{All the manifolds appearing in this work are compact, orientable and without boundary, unless otherwise specified. The imaginary unit is ``$\I$''; if~$(M,J)$ is a complex manifold of complex dimension~$n$ we use~$i,j,k\dots$ as indices for tensors defined on the underlying real manifold, so~$i,j\dots\in\set{1,2,\ldots,2n}$. Instead we use~$a,b,c,\dots$ as indices ranging from~$1$ to~$n$. We always use the Einstein convention on repeated indices.

Our conventions for the Laplacian are the following:~$\Delta=\mrm{d}\,\mrm{d}^*+\mrm{d}^*\mrm{d}$, and in particular for a function~$\varphi$ we get~$\Delta(\varphi)=-\mrm{div}\,\mrm{grad}(\varphi)$. In complex coordinates we find, for a K\"ahler metric,~$\Delta(\varphi)=-2g^{a\bar{b}}\diff_a\diff_{\bar{b}}\varphi$. The ``complex Laplacian'' is~$\Delta_{\bar{\diff}}=\Delta_{\diff}=\frac{1}{2}\Delta$.

Here is a brief glossary of some of the notation that is used throughout the thesis.}
\thispagestyle{plain}
\renewcommand{\nomname}{Notation and conventions}
\printnomenclature[2cm]


\mainmatter

\chapter{Scalar curvature as a moment map}\label{chap:classicalScalCurv}

\chaptertoc{}

\bigskip

Our general setting is this: fix a compact complex manifold~$M$ of dimension~$n$, a K\"ahler form~$\omega$ on~$M$, and let~$\mscr{J}$ %
\nomenclature[J]{$\mscr{J}$}{space of all almost complex structures compatible with a reference symplectic form}%
be the set of all almost complex structures that are \emph{compatible} with~$\omega$, i.e.
\begin{equation*}
\mscr{J}=\set*{J\tc J^2=-\mathbbm{1},\ \omega(J-,J-)=\omega(-,-)\mbox{ and } \omega(-,J-)>0}.
\end{equation*}
For~$J\in\mscr{J}$ the bilinear form~$\omega(-,J-)$ is a Riemannian metric on~$M$, that will be denoted by~$g_J$. %
\nomenclature[gJ]{$g_J$}{metric defined by the reference symplectic form~$\omega$ and a compatible almost complex structure~$J$,~$g_J(v,w)=\omega(v,Jw)$}%
Notice that if~$J$ is an integrable complex structure,~$(g_J,J,\omega)$ is a K\"ahler triple.

Around any point of~$M$ we can find a coordinate system~$\bm{u}$ such that~$\omega(\bm{u})$ is the canonical~$2$-form~$\sum_i\mrm{d}u^i\wedge\mrm{d}u^{n+i}$ on~$\bb{R}^{2n}$ (in other words,~$\bm{u}$ is a local system of Darboux coordinates for~$\omega$); the matrix-valued function that represents a compatible almost complex structure~$J\in\mscr{J}$ in this coordinate system satisfies
\begin{equation*}
J(\bm{u})^\transpose\begin{pmatrix}0 &\mathbbm{1}\\-\mathbbm{1} &0\end{pmatrix}J(\bm{u})=\begin{pmatrix}0 &\mathbbm{1}\\-\mathbbm{1} &0\end{pmatrix}\ \mbox{ and }\ 
\begin{pmatrix}0 &\mathbbm{1}\\-\mathbbm{1} &0\end{pmatrix}J(\bm{u})>0.
\end{equation*}%
Our first object of study in Section~\ref{sec:matrici} is the set of all~$2n\times 2n$ real matrices that have this property, denoted by~$\m{AC}^+(2n)$. %
\nomenclature[almost complex]{$\m{AC}^+$}{space of all matrices that are compatible with the standard symplectic structure of~$\bb{R}^{2n}$; isomorphic to Siegel's upper half space as~$\mrm{Sp}(2n)$-spaces}%
This gives a ``point-wise description'' of~$\mscr{J}$, and in fact it is easy to see that~$\mscr{J}$ is the space of sections of a bundle over~$M$ with fibres~$\m{AC}^+(2n)$.

The space~$\m{AC}^+(2n)$ carries a K\"ahler structure, so we can equip~$\mscr{J}$ with the structure of an infinite-dimensional K\"ahler manifold, and in particular we will get a symplectic form on~$\mscr{J}$. We then study the action of an infinite-dimensional Lie group on~$\mscr{J}$ that preserves this symplectic form, with the aim of explaining the following classical result, originally proven by Fujiki~\cite{Fujiki_moduli}, and by Donaldson~\cite{Donaldson_scalar} in this degree of generality:
\begin{thm}\label{thm:Donaldson_scalcurv_mm}
For a compact K\"ahler manifold~$(M,\omega)$, the map that sends a compatible almost complex structure~$J$ to the Hermitian scalar curvature~$S(J)$ %
of the Riemannian metric~$g_J$ is, up to a scalar multiple, a moment map for an infinite-dimensional Hamiltonian action.
\end{thm}
We will give later a more detailed account of this statement, along with some material that is needed for the theory developed in Chapter~\ref{chap:HyperkahlerReduction}. For now we just remark that if~$J$ is an integrable complex structure then the Hermitian scalar curvature of~$g_J$ is just the usual scalar curvature. This is the context in which Theorem~\ref{thm:Donaldson_scalcurv_mm} was first proven in~\cite{Fujiki_moduli}, and the computation for this case can be found, for example, in~\cite[\S$4$]{Tian_libro} and~\cite{Szekelyhidi_phd}. It is however useful for us to retrace the original proof of the general case in~\cite{Donaldson_scalar} to clearly identify the constant mentioned in~\ref{thm:Donaldson_scalcurv_mm}, since we will later consider a different moment map whose precise expression depends on this computation.

Before studying the space~$\mscr{J}$ in Section~\ref{sec:J} we recall some definitions about actions of Lie groups on symplectic manifolds, and we give a brief account of the connection between Hamiltonian actions and stability conditions. We do not aim to give an introduction to GIT, since there are many excellent references in the literature. We refer to~\cite{Mumford_GIT},~\cite[chapter~$4$]{Moduli_VectorBundles} and~\cite{Schmitt_GIT} for a more in-depth discussion of GIT, and to~\cite{RThomas_GIT} and~\cite{Szekelyhidi_libro} for the relation of GIT with the cscK problem.

\section{Moment maps and stability}\label{sec:stability}

Here we give a very brief description of two different ways to take quotients under the action of a Lie group on a manifold: Marsden-Weinstein reductions, or symplectic reductions, and \textsc{GIT} quotients. Then we will explain the connection between the two points of view, giving some motivation for studying these moment map descriptions of well-known problems in complex differential geometry.

\subsection{Hamiltonian actions}\label{sec:hamiltonian}

\begin{definition}
Let~$( M ,\omega)$ be a symplectic manifold, and let~$h$ be a real smooth function on~$M$. The \emph{Hamiltonian vector field} defined by~$h$ and~$\omega$ is the unique vector field~$X_h$ %
 such that
\begin{equation*}
\mrm{d}{h}=-X_h\lrcorner\omega.
\end{equation*}
\end{definition}

Sometimes we will also say that~$h$ is a \emph{Hamiltonian potential} or a \emph{Hamiltonian function} for~$X_h$ with respect to~$\omega$. Notice that any Hamiltonian potential for~$X_h$ differs from~$h$ by a constant. From Cartan's formula we can readily see that Hamiltonian vector fields preserve the symplectic form.
  
If~$(M,J,\omega)$ is a K\"ahler manifold and~$g=g_J$ is the Riemannian metric defined by~$\omega$ and~$g$, then~$X_h=J\mrm{grad}(h)$; indeed, for every vector field~$Y$
\begin{equation*}
\mrm{d}h(Y)=-\omega(X_h,Y)=-g(JX_h,Y).
\end{equation*}

\paragraph{The group of Hamiltonian symplectomorphisms.} The flow~$\Phi_{X_h}^t$ %
\nomenclature[Flow]{$\Phi^t_X$}{flow of a vector field~$X$}%
of the Hamiltonian vector field of~$h$ is called called the \emph{Hamiltonian flow} of~$h$ on~$ M$.
\begin{lemma}
Fix~$h\in\m{C}^\infty_0$, and let~$\psi$ be a symplectomorphism of~$( M ,\omega)$. Then:
\begin{enumerate}
\item for all~$t$,~$\Phi^t_{X_h}$ is a symplectomorphism;
\item~$\psi^*X_h=X_{h\circ\psi}$.
\end{enumerate}
\end{lemma}
Any smooth path~$f:\bb{R}\to\mrm{Diff}( M )$ such that~$f_0=\mathbbm{1}_M$ defines a \emph{time-dependent} vector field~$X_t$ as follows:
\begin{equation*}
(X_t)(p)=\frac{\mrm{d}}{\mrm{d}s}\Big|_{s=t}f_s(f^{-1}_t(p)).
\end{equation*}
Notice that if~$f$ is a group homomorphism, i.e.~$f_s\circ f_t=f_{s+t}$, then~$X_t=X_0$ and~$f_t$ is the flow of~$X_0$.
\begin{definition}
Fix~$\varphi\in\mrm{Diff}_0( M )$, let~$f:\bb{R}\to\mrm{Diff}( M )$ be a smooth path between~$\mathbbm{1}_M$ and~$\varphi$ and let~$X_t$ be the time-dependent vector field defined by~$f$. We say that~$f$ is a \emph{Hamiltonian isotopy} if there is a smooth family~$\set*{h_t\in\m{C}^\infty_0\tc t\in\bb{R}}$ such that for all~$t$,~$X_t$ is the Hamiltonian vector field of~$h_t$. The \emph{group of Hamiltonian symplectomorphisms}~$\mrm{Ham}( M ,\omega)$ %
\nomenclature[Hamiltonian]{$\mrm{Ham}(M,\omega)$}{group of Hamiltonian symplectomorphisms of~$(M,\omega)$, also called the group of exact symplectomorphisms}%
of~$( M ,\omega)$ is the group of all~$\varphi\in\mrm{Diff}_0( M )$ that can be connected to~$\mathbbm{1}_M$ by a Hamiltonian isotopy.
\end{definition}
We list the properties of this group that we will need in the next sections; for proofs of these results we refer to~\cite[chapter~$10$]{McDuff_Salamon}. 
\begin{prop}\label{prop:Ham_properties}
Let~$( M ,J,\omega)$ be a K\"ahler manifold. Then
\begin{enumerate}
\item~$\mrm{Ham}( M ,\omega)$ is a Lie subgroup of~$\mrm{Symp}( M ,\omega)$;
\item the Lie algebra of~$\mrm{Ham}( M ,\omega)$ is the Lie algebra of Hamiltonian vector fields on~$ M$, which is isomorphic to~$\m{C}^\infty_0( M ,\bb{R})$ equipped with the Poisson bracket;
\item if~$H^{1}( M ,\bb{R})=0$ then~$\mrm{Ham}( M ,\omega)$ is just the connected component of the identity of~$\mrm{Symp}( M ,\omega)$.
\end{enumerate}
\end{prop}
We recall that the \emph{Poisson bracket} mentioned in Proposition~\ref{prop:Ham_properties} is defined, for two functions~$f,h\in\m{C}^\infty(M)$, as
\begin{equation*}
\{f,h\}=\omega(X_f,X_h).
\end{equation*}
\begin{rmk}
The results in Proposition~\ref{prop:Ham_properties} are not trivial; for example the fact that~$\mrm{Ham}( M ,\omega)$ is a Lie subgroup of~$\mrm{Symp}( M ,\omega)$ bears on our assumption that~$( M ,\omega)$ is K\"ahler; in general this is an open conjecture (if we give~$\mrm{Symp}$ the~$\m{C}^1$-topology), see for example~\cite[chapter~$14$, \S~$2$]{Polterovich_sympdiff} and~\cite[Conjecture~$10.21$]{McDuff_Salamon}.
\end{rmk}

\paragraph{Hamiltonian actions.} We are interested in some left actions~$\sigma:G\times M\to M$ of Lie groups~$G$ with additional geometric properties, such as \emph{symplectic actions}, i.e. actions that preserve a symplectic form. We will usually denote left actions by~$(g,x)\mapsto g.x$, when there is no risk of confusion. Given such an action we define for~$a\in\f{g}$ the \emph{fundamental vector field}~$\hat{a}$ %
\nomenclature[infinitesimal action]{$\hat{a}$}{vector field describing the infinitesimal action of~$a\in\f{g}$}%
on~$M$ as
\begin{equation*}
\hat{a}_x =\frac{\mrm{d}}{\mrm{d}t}\Bigr|_{t=0}\left(\mrm{exp}(-ta).x\right)\in T_xM.
\end{equation*}
The vector field~$\hat{a}$ is also called the \emph{infinitesimal action} of~$a$ on~$M$. The minus sign in the definition is due to the fact that, with this convention, the map
\begin{equation*}
\begin{split}
\f{g}&\to\Gamma(TM)\\
a&\mapsto\hat{a}
\end{split}
\end{equation*}
is a Lie algebra homomorphism (see~\cite[Proposition~$3.8$, Appendix~$5$]{LibermannMarle}).
\begin{definition}
Let~$G\curvearrowright(M,\omega)$ be a symplectic action. We say that the action is \emph{Hamiltonian} if there is a \emph{moment map}
\begin{equation*}
\mu:M\to\f{g}^*
\end{equation*}
that is equivariant with respect to~$G\curvearrowright M$ and the co-adjoint action of~$G$ on~$\f{g}^*$, and such that~$\langle\mu,a\rangle$ is a Hamiltonian function of the vector field~$\hat{a}$ on~$M$. In a more concise way:
\begin{align*}
& \forall g\in G,\,\forall x\in M,\,\forall a\in\f{g}\quad\quad \langle\mu_{g.x},a\rangle=\langle\mu_x,\mrm{Ad}_{g^{-1}}(a)\rangle;&\\
& \forall g\in G,\,\forall a\in\f{g}\quad\quad\quad\quad\quad\quad \mrm{d}\left(x\mapsto\langle\mu_x,a\rangle\right)=-\hat{a}\lrcorner\omega.&
\end{align*}
\end{definition}

An important feature of Hamiltonian actions is that in many cases they allow us to build a \emph{symplectic quotient} of the action. Assume for simplicity that~$K$ is a compact Lie group, and consider a Hamiltonian action~$K\curvearrowright M$ with moment map~$\mu$. The equivariance condition implies that the~$K$-action preserves the level sets of~$\mu$. From the definition of a moment map it is clear that for all~$x\in M$, the kernel of~$\mrm{d}\mu_x:T_x M \to\f{k}^*$ is the set
\begin{equation*}
 K (x)^\omega:=\set*{v\in T_x M \tc \omega(v,w)=0\mbox{ for all }w\in T_x( K (x))}.
\end{equation*}
Moreover~$\xi\in\f{k}^*$ is a regular value of~$\mu$ if and only if for all~$x\in\mu^{-1}(\xi)$ and all~$a\in\f{k}$,~$\hat{a}_x\not=0$; in other words,~$\xi$ is a regular value if and only if the stabilizer~$ K _x$ of~$x$ is discrete (hence finite). If the stabilizer~$ K _\xi$ of~$\xi\in\f{k}^*$ under the coadjoint action acts freely and properly on~$\mu^{-1}(\xi)$ then the level set~$\mu^{-1}(\xi)$ is a submanifold of~$ M$ and the quotient~$\mu^{-1}(\xi)/ K _{\xi}$ is a symplectic manifold, whose symplectic form is induced by~$\omega$.

This construction is known as the \emph{Marsden-Weinstein reduction}, or the \emph{symplectic quotient} of the action. For the details we refer to~\cite{Marsden_Weinstein} and to~\cite[\S~$5.4$]{McDuff_Salamon}; here we just note that in particular~$ K _0= K$, and the general idea to define the symplectic structure on~$\mu^{-1}(0)/ K$ is to identify~$T_x\left(\mu^{-1}(0)/ K \right)$ with a complement of~$T_x( K (x))$ in~$ K (x)^\omega$ in such a way that~$\omega$ restricted to this space is still a symplectic form.

We mention also that if~$(M,J,\omega)$ is a K\"ahler manifold and the symplectic action of~$ K$ is by biholomorphisms (i.e. the action also preserves~$J$) then the action preserves the metric~$g$ induced by~$J$ and~$\omega$, and the symplectic quotient also inherits the structure of a K\"ahler manifold.

\subsection{Elements of \textsc{GIT}}

Let~$ M \subseteq\bb{CP}^N$ be a projective variety, and let~$ G \subseteq\mrm{GL}(N+1,\bb{C})$ be a linear group whose action on~$\bb{CP}^N$ preserves~$ M$. Geometric Invariant Theory (\textsc{GIT}, from now on) describes a way to construct another algebraic variety that is, in some sense, the quotient of~$ M$ by the action of~$ G$. In this section we briefly explain some of the ideas coming from \textsc{GIT} that motivate the interest in formulating differential-geometric equations as zeroes of a moment map.

Hilbert's Theorem suggests that we should restrict attention to actions of \emph{reductive} groups, Lie groups that are the complexification of a real compact Lie group.
\begin{definition}
A \emph{good quotient} for the action~$G\curvearrowright M$ is a pair~$(Y,\pi)$ where~$Y$ is an algebraic variety and~$\pi$ is a~$G$-invariant affine morphism~$\pi:M \to Y$ such that
\begin{enumerate}
\item for every open set~$U\subseteq Y$, the induced morphism of rings~$\m{O}_Y(U)\to\m{O}_M(\pi^{-1}(U))$ is an isomorphism between~$\m{O}_Y(U)$ and the~$G$-invariant subring of~$\m{O}_M(\pi^{-1}(U))$;
\item if~$Z\subseteq M$ is~$G$-invariant and closed, also~$\pi(Z)$ is closed;
\item if~$Z_1,Z_2\subseteq V$ are~$G$-invariant, closed and disjoint, then also~$\pi(Z_1)$ and~$\pi(Z_2)$ are closed and disjoint.
\end{enumerate}
A good quotient is \emph{geometric} if each fibre~$M_y$ is a single~$G$-orbit in~$ M$.
\end{definition}
Given an action of a reductive group~$G$ on a projective variety~$M$, it is not clear how to find a geometric quotient, if any exist at all. It turns out that to construct such quotients we should first remove from~$M$ some points that are not ``well-behaved'' under the action of~$G$. For some insight into the motivation of the following definitions it may be useful to consult~\cite[Chapter~$5$]{Szekelyhidi_libro} and the Introduction in~\cite{Schmitt_GIT}.

\begin{definition}
Let~$M\subseteq\bb{CP}^N$ be a projective variety, and let~$G$ be a reductive subgroup of~$\mrm{GL}(N+1)$ whose natural action on~$\bb{CP}^N$ preserves~$M$. For~$p\in M$, denote by~$\hat{p}$ a preimage of~$p$ for the quotient map~$\bb{C}^{N+1}\setminus\set{0}\to\bb{CP}^N$. Then we say that~$p$ under the action~$ G \curvearrowright M$ is
\begin{enumerate}
\item \emph{semistable} if~$0$ does not belong to the closure of~$ G .\hat{p}$ in~$\bb{C}^{N+1}$;
\item \emph{polystable} if the orbit~$ G .\hat{p}$ is closed in~$\bb{C}^{N+1}$;
\item \emph{stable} if it is polystable and moreover the stabilizer of~$p$ in~$ G$ is zero-dimensional.
\end{enumerate}
We denote the sets of semistable, polystable and stable points of~$M$ respectively by~$M ^{ss}$, $M ^{ps}$ and~$M ^s$.
\end{definition}
\begin{rmk}
This definition differs slightly from the one in~\cite{Mumford_GIT}; what we call ``stable'' here would be, in the original terminology, ``properly stable'' (c.f. Definition~$1.7$ and Definition~$1.8$ in~\cite{Mumford_GIT}). However, it is by now quite standard to use this terminology.
\end{rmk}

Notice that the definition does not depend on the choice of the lift~$\hat{p}$ of~$p$, since all the orbits of these lifts differ just by a scaling factor. It is not hard to show that~$p$ is semistable if and only if there is a~$f\in\bb{C}[ M ]^{ G }$, homogeneous and non-constant, such that~$f(p)\not=0$; in other words, the semistable points are those that can actually be ``seen'' by~$ G$-invariant functions on~$ M$. This is, in fact, the more common definition of semistable points. See~\cite[Proposition~$2.2$]{Mumford_GIT},~\cite[Proposition~$2.1$]{Moduli_VectorBundles} or~\cite[Proposition~$5.14$]{Szekelyhidi_libro}.

The following is the main existence result for algebraic quotients, which also (a posteriori) sheds some light on why it is necessary to consider the semistable and stable loci of the action. For a proof see~\cite[Theorem~$1.4.3.8$]{Schmitt_GIT} and~\cite[Theorem~$1.7$]{Moduli_VectorBundles}.
\begin{thm}
Let~$ M$,~$ G$ be as above. Then there is a good quotient~$\left( M \sslash G ,\pi\right)$ of~$ M ^{ss}$ for the induced~$ G$-action, and this quotient is a projective variety. Moreover, there is an open subset~$U$ of~$ M \sslash G$ such that~$(U,\pi_{\restriction M ^s})$ is a geometric quotient of~$ M ^s$.
\end{thm}
The problem of finding a good quotient for an action~$ G \curvearrowright M$ is then reduced to finding a way to identify (semi-)stable points of the action. To check (semi-)stability one can use the \emph{Hilbert-Mumford criterion} (Theorem~$2.1$ in~\cite{Mumford_GIT}): this characterizes the stability of~$p\in M$ in terms of numerical properties of the orbits of~$p$ under~$1$-dimensional subgroups of~$G$.

\subsection{The Kempf-Ness Theorem}

Symplectic quotients given by the Marsden-Weinstein reduction and algebraic quotients obtained via Geometric Invariant Theory are closely related, and this relation is the main reason why we are interested in formulating the cscK problem in the framework of Hamiltonian actions. In this section we give an account of a result by Kempf and Ness that explains the relation between the quotients~$\mu^{-1}(0)/ K$ and~$ M \sslash G$, for a compact Lie group~$K$ and the corresponding reductive group~$G=K^c$.

Consider a complex reductive Lie group~$ G \leq\mrm{GL}(N+1,\bb{C})$ acting in the usual way on~$\bb{CP}^N$. If~$ M$ is a smooth projective subvariety of~$\bb{CP}^N$ that is~$G$-invariant, we can form the GIT quotient~$ M ^s\sslash G$ that was described above. On the other hand since~$ G$ is reductive we know that~$ G$ is the complexification of a maximal compact real Lie group~$ K \leq G$; if~$( M ,\omega)$ is symplectic and~$ K \curvearrowright M$ is Hamiltonian with moment map~$\mu$, we can form the symplectic quotient~$\mu^{-1}(0)/ K$. The Kempf-Ness theorem tells us that the two orbit spaces are essentially the same.
\begin{thm}[\cite{Kempf_Ness}]
With the previous notation, any~$p\in M$ is in the semistable locus of the action~$ G \curvearrowright M$ if and only if the closure of~$ G .p$ intersects~$\mu^{-1}(0)$. In particular~$\mu^{-1}(0)\subseteq M ^{ss}$, and this inclusion induces a homeomorphism between~$\mu^{-1}(0)/ K$ and~$ M ^s\sslash G$.
\end{thm}
This result inspires a suggestive principle, that proved to be very fruitful both in complex differential geometry and algebraic geometry: \emph{let~$G\curvearrowright X$ be a Hamiltonian action of the Lie group~$G$ on the symplectic complex manifold~$(X,\omega)$. Then there should be some notion of stability for the complexified action of~$G$ on~$X$, such that, for any point~$p\in X$, being a zero of the moment map is equivalent to belonging to the~$G^c$-orbit of a stable point.}

This general principle proved to be quite effective for studying famous problems in differential geometry, as we already mentioned in the Introduction. The example we are most interested in is the Hermitian Yang-Mills equation: the Hitchin-Kobayashi correspondence (see~\cite{Lubke_KobayashiHitchin} for a beautiful exposition of the problem) and the results of Atiyah and Bott~\cite{AtiyahBott_YangMills} show that an infinite-dimensional version of this approach can be successful. 
The formulation by Donaldson and Fujiki of the cscK problem as a Hamiltonian reduction together with the ``Kempf-Ness principle'', suggest that there should be a way to characterize the existence of cscK metrics in terms of some algebraic stability condition.

\section{Some matrix spaces}\label{sec:matrici}

We go back to the pointwise description of the space~$\mscr{J}$ of compatible almost complex structures. Consider the symplectic vector space~$(\bb{R}^{2n},\Omega_0)$, where~$\Omega_0$ %
\nomenclature[Omega]{$\Omega_0$}{standard symplectic form on~$\bb{R}^{2n}$}%
is the canonical symplectic form, i.e. the matrix 
\begin{equation*}
\Omega_0=\begin{pmatrix}0 & \mathbbm{1}_n \\ -\mathbbm{1}_n & 0\end{pmatrix}.
\end{equation*}
The \emph{symplectic group}~$\mrm{Sp}(2n)$ is defined as
\begin{equation*}
\mrm{Sp}(2n)=\set*{A\in\mrm{GL}(2n,\bb{R})\tc A^{\intercal}\Omega_0 A=\Omega_0}.
\end{equation*}
This is a connected real Lie group, and we are particularly interested in some actions of~$\mrm{Sp}(2n)$. Notice that~$\mrm{Sp}(2n)$ is a subgroup of~$\mrm{Sl}(2n,\bb{R})$, since any symplectic matrix preserves the standard volume form on~$\bb{R}^{2n}$.

By the usual identification of~$\bb{C}^n$ with~$\bb{R}^{2n}$ as real vector spaces, we consider~$\mrm{GL}(n,\bb{C})$ as the subgroup of~$\mrm{GL}(2n,\bb{R})$ consisting of all the real invertible~$2n\times 2n$ matrices that commute with the standard complex structure on~$\bb{R}^{2n}$ (given by~$-\Omega_0$). The groups~$\mrm{Sp}(2n)$,~$\mrm{SO}(2n)$ and~$\mrm{U}(n)$ are tied together by the well-known \emph{two out of three property}:
\begin{equation*}
\mrm{Sp}(2n)\cap \mrm{SO}(2n)=\mrm{Sp}(2n)\cap \mrm{U}(n)=\mrm{SO}(2n)\cap \mrm{U}(n)=\mrm{U}(n).
\end{equation*}
The coset space~$\mrm{Sp}(2n)/\mrm{U}(n)$ will play a fundamental role in what follows. It carries a natural K\"ahler metric, coming from its identification with \emph{Siegel's upper half space}~$\f{H}$, and at the same time it can be identified naturally with the space~$\m{AC}^+$ of linear complex structures compatible with a linear symplectic form.

\subsection{Siegel's upper half space}\label{sec:Siegel}

Siegel's upper half space is a generalization of the well-known hyperbolic plane, and these two spaces share many interesting geometric properties. For example they are both geodesically complete, and their groups of biholomorphisms can be easily described. Some references for these properties are the original article~\cite{Siegel_halfspace}, the memoir~\cite{DallaVolta_matrici} and the paper~\cite{Maass}. Here we just collect some of the properties we need for the next sections.
\begin{definition}
Siegel's upper half space~$\f{H}(n)$ %
\nomenclature[Siegel]{$\f{H}(n)$}{Siegel's upper half space, consisting of~$n\times n$ complex symmetric matrices with positive-definite imaginary part}%
is the set of all symmetric~$n\times n$ complex matrices with positive-definite imaginary part.
\end{definition}
We will often denote Siegel's upper half space by~$\f{H}$, omitting the dimension when it will not cause confusion. Notice that~$\f{H}$ inherits a complex structure from the inclusion in~$M_{n\times n}(\bb{C})$, given simply by multiplication by~$\I$. It will be more notationally convenient, however, to consider on~$\f{H}$ the \emph{conjugate} complex structure, i.e. we will define the complex structure on~$\f{H}$ to be the multiplication by~$-\I$. The reason for this choice will become clear when we will use it to define a complex structure on~$\m{AC}^+$, see Proposition~\ref{prop:metrica_AC+_Siegel}.

On~$\f{H}$ there is also a holomorphic action of~$\mrm{Sp}(2n)$ by an analogue of the well-known \emph{M\"obius transformations}. For~$P=\begin{pmatrix}A & B\\ C & D\end{pmatrix}\in\mrm{Sp}(2n)$ and~$Z\in\f{H}$, let
\begin{equation*}
P.Z:=(AZ+B)(CZ+D)^{-1}.
\end{equation*}
This action generalizes the usual~$\mrm{Sl}(2,\bb{R})$-action on Poincaré's upper half plane. To check that this is a well-defined left action on~$\f{H}$ it is useful to write~$P.Z$ as~$WU^{-1}$, for~$\begin{pmatrix} W\\ U\end{pmatrix}=\begin{pmatrix}A & B\\ C&D\end{pmatrix}\begin{pmatrix} Z\\\mathbbm{1}\end{pmatrix}$.
\begin{prop}[Theorem~$1$ in~\cite{Siegel_halfspace}]
The action of~$\mrm{Sp}(2n)$ on~$\f{H}$ is transitive. Moreover, every holomorphic bijection~$\f{H}\to\f{H}$ is a M\"obius transformation.
\end{prop}
On~$\f{H}$ there is also a~$\mrm{Sp}(2n)$-invariant K\"ahler structure; the metric tensor looks like
\begin{equation*}
\mrm{d}s^2_{X+\I Y}=\mrm{Tr}\left(Y^{-1}\mrm{d}Z\,Y^{-1}\mrm{d}\overline{Z}\right)
\end{equation*}
where~$\mrm{d}Z$ and~$\mrm{d}\overline{Z}$ are the matrix of differentials~$(\mrm{d}z_{ab})_{1\leq a,b\leq n}$ and its conjugate. We refer to~\cite{Siegel_halfspace} for the details. It is not difficult to see that this metric has a local potential, of the form
\begin{equation*}
\frac{\diff}{\diff z\indices{^a_b}}\frac{\diff}{\diff\bar{z}\indices{^c_d}}\log\mrm{det}(Y)
\end{equation*}
see for example~\cite[\S~$5$]{DallaVolta_matrici}. Then this is a K\"ahler metric, so~$\f{H}$ is a homogeneous K\"ahler manifold.

The matrix~$\begin{pmatrix}A & B\\ C & D\end{pmatrix}\in\mrm{Sp}(2n)$ stabilizes~$\I\mathbbm{1}$ if and only if~$\I A+B=\I D-C$, i.e.~$B+C=0$ and~$A=D$, so the stabilizer of~$\I\mathbbm{1}$ is
\begin{equation*}
\mrm{Sp}(2n)\cap\mrm{GL}(n,\bb{C})=\mrm{U}(n)
\end{equation*}
and we find that~$\mrm{Sp}(2n)/\mrm{U}(n)\cong\f{H}$. 

\subsubsection{Linear almost complex structures}
Let~$J\in\mrm{Sp}(2n)$ be a linear almost complex structure preserving~$\Omega_0$. Then the product~$\Omega_0J$ is a non-degenerate symmetric matrix, defining a bilinear form~$\beta_J$. We are interested in the set of all almost complex structures~$J\in\mrm{Sp}(2n)$ such that~$\beta_J$ is positive-definite,
\begin{equation*}
\m{AC}^+(2n):=\set*{J\in\mrm{Sp}(2n)\tc J^2=-\mathbbm{1},\ \beta_J>0}.
\end{equation*}
Notice that the matrix~$-\Omega_0$ is an element of~$\m{AC}^+(2n)$, and~$\beta_{-\Omega_0}$ is just the usual Euclidean product.
\begin{lemma}\label{lemma:stabilizzatore}
The action by conjugation of~$\mrm{Sp}(2n)$ on~$\m{AC}(2n)$ is transitive. Moreover, the stabilizer of any~$J\in\m{AC}(2n)$ is~$\mrm{Sp}(2n)\cap\mrm{SO}(\beta_J)$.
\end{lemma}
In particular the stabilizer of~$-\Omega_0$ is~$\mrm{Sp}(2n)\cap\mrm{O}(2n)=\mrm{U}(n)$. Moreover, for any~$J\in\m{AC}^+(2n)$ there is some~$P\in\mrm{Sp}(2n)$ which conjugates~$J$ to~$-\Omega_0$; a possible choice of~$P$ is given by
\begin{equation*}
P_J =(\Omega_0 J)^{1/2}.
\end{equation*}
Lemma~\ref{lemma:stabilizzatore} tells us that we can identify~$\m{AC}^+(2n)$ with the quotient
\begin{equation*}
\mrm{Sp}(2n)/\left(\mrm{SO}(2n)\cap\mrm{Sp}(2n)\right)\cong\mrm{Sp}(2n)/\mrm{U}(n).
\end{equation*}
Then in fact~$\m{AC}^+$ and~$\f{H}$ are isomorphic as~$\mrm{Sp}(2n)$-spaces. In this way we can induce on~$\m{AC}^+$ geometric structures coming from~$\f{H}$, in particular a complex structure and a Riemannian metric.

\subsection{K\"ahler geometry of~$\m{AC}^+$}\label{sec:mAC+}

Let~$\phi:\m{AC}^+\to\f{H}$ be the diffeomorphism given by composing the two identifications of~$\m{AC}^+$ and~$\f{H}$ with~$\mrm{Sp}(2n)/\mrm{U}(n)$. These identifications are defined by fixing the reference points~$-\Omega_0\in\m{AC}^+$ and~$\I\mathbbm{1}\in\f{H}$, so that~$\phi$ is the composition
\begin{center}
\begin{tabular}{c c c c c}
$\m{AC}^+$ &~$\xrightarrow{\sigma}$ &~$\mrm{Sp}(2n)/\mrm{U}(n)$ &~$\xrightarrow{\chi}$ &~$\f{H}$\\
$J$ &~$\mapsto$ &~$(\Omega_0J)^{-1/2}\mrm{U}(n)$ &~$\mapsto$ &~$(\Omega_0J)^{-1/2}.(\I\mathbbm{1})$
\end{tabular}
\end{center}
and~$\phi$ is a smooth isomorphism of~$\mrm{Sp}(2n)$-spaces, i.e. it is a diffeomorphism that commutes with the~$\mrm{Sp}(2n)$-actions. Using this identification of the two spaces we obtain a K\"ahler structure on~$\m{AC}^+$.

The differential of~$\phi$ at the point~$-\Omega_0\in\m{AC}^+$ is quite easy to compute directly, by composing the differentials of~$\chi$ and~$\sigma$:
\begin{equation*}
\mrm{d}\sigma_{-\Omega_0}(A)=-\frac{1}{2}\Omega_0A
\end{equation*}
\begin{equation*}
\mrm{d}\chi_{\sigma(-\Omega_0)}\left(\mrm{d}\sigma_{-\Omega_0}(A)\right)=
-\frac{1}{2}\left(\I (\Omega_0A)_{11}+(\Omega_0A)_{12}\right)+\frac{\I}{2}\left(\I (\Omega_0A)_{21}+(\Omega_0A)_{22}\right)
\end{equation*}
so that we find
\begin{equation*}
\mrm{d}\phi_{-\Omega_0}(A)=\frac{1}{2}\left(A_{11}-A_{22}-\I\left(A_{21}+A_{12}\right)\right).
\end{equation*}
Notice that
\begin{equation*}
\mrm{d}\phi_{-\Omega_0}(-\Omega_0A)=-\I\,\mrm{d}\phi_{-\Omega_0}(A)
\end{equation*}
so the pull-back of the complex structure of~$\f{H}$ to~$T_{-\Omega_0}\m{AC}^+$ is multiplication by~$-\Omega_0$. To see what the complex structure is on the rest of~$\m{AC}^+$ we can use the~$\mrm{Sp}(2n)$-action on~$\m{AC}^+$, since~$\phi$ is a~$\mrm{Sp}(2n)$-isomorphism.

Fix~$J\in\m{AC}^+$, and let~$f:I\mapsto(\Omega_0J)^{\frac{1}{2}}I(\Omega_0J)^{-\frac{1}{2}}$ be the conjugation by an element of~$\mrm{Sp}(2n)$ that sends~$J$ to~$-\Omega_0$. Then
\begin{equation*}
\mrm{d}f_J(A)=(\Omega_0J)^{\frac{1}{2}}A(\Omega_0J)^{-\frac{1}{2}}
\end{equation*}
and so we have, for the complex structure
\begin{equation*}
\begin{split}
-\Omega_0\mrm{d}f_J(A)=&-\Omega_0(\Omega_0J)^{\frac{1}{2}}A(\Omega_0J)^{-\frac{1}{2}}=\Omega_0(\Omega_0J)^{\frac{1}{2}}JJA(\Omega_0J)^{-\frac{1}{2}}=\\
=&-\Omega_0\Omega_0(\Omega_0J)^{\frac{1}{2}}JA(\Omega_0J)^{-\frac{1}{2}}=\mrm{d}f_J(JA)
\end{split}
\end{equation*}

Similarly, we can pull back the metric tensor at~$-\Omega_0\in\m{AC}^+$:
\begin{equation*}
\mrm{d}s^2_{\I\mathbbm{1}}(\mrm{d}\phi_{-\Omega_0}(A),\mrm{d}\phi_{-\Omega_0}(B))=\mrm{Tr}\left(A_{11}B_{11}+A_{12}B_{12}\right)=\frac{1}{2}\mrm{Tr}(AB).
\end{equation*}
Where we have used that any~$A\in T_{-\Omega_0}\m{AC}^+$ is written, in matrix-block notation, as~$A=\begin{pmatrix}A_{11} & A_{12}\\ A_{12} &-A_{22}\end{pmatrix}$. Now we can pull back this metric to the tangent space at any~$J\in\m{AC}^+$, obtaining
\begin{equation*}
\left(\phi^*\mrm{d}s^2\right)_{J}(A,B)=\frac{1}{2}\mrm{Tr}\left((\Omega_0J)^{\frac{1}{2}}AB(\Omega_0J)^{-\frac{1}{2}}\right)=\frac{1}{2}\mrm{Tr}(AB).
\end{equation*}
\begin{prop}\label{prop:metrica_AC+_Siegel}
Equip~$\m{AC}^+$ with the K\"ahler structure pulled back from Siegel's upper half space. Then the complex structure on~$T_J\m{AC}^+$ is given by
\begin{equation*}
A\mapsto JA
\end{equation*}
while the metric is given by
\begin{equation*}
\left\langle A,B\right\rangle=\frac{1}{2}\mrm{Tr}(AB).
\end{equation*}  
\end{prop}
In Chapter~\ref{chap:HyperkahlerReduction} we will also need an explicit expression for the curvature of its K\"ahler metric. Notice that it is enough to compute the curvature at the point~$-\Omega_0$, since the~$\mrm{Sp}(2n)$-action on~$\m{AC}^+$ is by isometries.
\begin{prop}\label{prop:curvatura_AC}
Consider~$\m{AC}^+(2n)$ with the K\"ahler metric induced by its identification with~$\f{H}(n)$. The curvature of this metric at the point~$-\Omega_0$ is given by
\begin{equation*}
R_{-\Omega_0}(A,B)(C)=-\frac{1}{4}\Big[[A,B],C\Big].
\end{equation*}
\end{prop}
Proposition~\ref{prop:curvatura_AC} is best explained by the theory of reductive spaces, using the symmetric space structure of~$\m{AC}^+\cong\mrm{Sp}(2n)/U(n)$. This is the subject of the next section.

\subsubsection{Curvature of~$\m{AC}^+$}

\begin{definition}
Let~$G$ be a Lie group,~$H\leq G$ a compact subgroup, and consider the space~$X=G/H$. We say that~$X$ is \emph{reductive} if there is a subspace~$\f{r}\subseteq\f{g}$ such that~$\f{g}=\f{h}\oplus\f{r}$, and~$\mrm{Ad}(H)(\f{r})\subseteq\f{r}$.
\end{definition}
For~$G$,~$H$ and~$X$ as in this definition we let~$o\in X$ be the coset~$eH$, where~$e\in G$ denotes the identity element. We can naturally identify~$T_oX$ with the vector space quotient~$\f{g}/\f{h}$, using the differential of the projection~$\pi:G\to G/H$; if moreover~$X$ is reductive, we obtain a natural Lie algebra structure on~$T_oX\cong\f{r}$ from that of~$\f{g}$.

\begin{definition}
With the previous notation, assume that~$X$ is reductive and that~$\f{g}=\f{h}\oplus\f{r}$ is the reductive decomposition of~$\f{g}$. We say that~$X$ is \emph{naturally reductive} if~$\f{g}$ has an~$\mrm{Ad}(H)$-invariant inner product~$\langle-,-\rangle$ such that
\begin{equation*}
\forall\, U,V,W\in\f{r}\quad\langle[U,V]_{\f{r}},W\rangle+\langle V,[U,W]_{\f{r}}\rangle=0.
\end{equation*}
Here~$[U,V]_{\f{r}}$ is the projection onto~$\f{r}$ induced by the decomposition~$\f{g}=\f{h}\oplus\f{r}$.
\end{definition}

If~$\f{g}$ has an~$\mrm{Ad}(H)$-invariant inner product, this induces an inner product on~$T_oX$; in turn then this gives us a Riemannian metric on~$X=G/H$, by left-translating with elements of~$G$. When~$\left(X,\langle-,-\rangle\right)$ is naturally reductive there is a simple expression for the curvature of the Riemannian metric induced by~$\langle-,-\rangle$ on~$X$.

\begin{thm}\label{thm:curvatura_nat_red}
Let~$\left(X=G/H,\langle-,-\rangle\right)$ be a naturally reductive space, let~$\f{g}=\f{h}\oplus\f{r}$ be the reductive decomposition, identify~$T_oX$ with~$\f{r}$ and consider the Riemannian metric induced by~$\langle-,-\rangle$ on~$X$. Then the curvature tensor~$R_o(U,V)W$ is, for~$U$,~$V$ and~$W$ in~$\f{r}$
\begin{equation*}
-\Big[[U,V]_{\f{h}},W\Big]-\frac{1}{2}\Big[{[U,V]}_{\f{r}},W\Big]_{\f{r}}-\frac{1}{4}{\Big[{[V,W]}_{\f{r}},U\Big]}_{\f{r}}-\frac{1}{4}{\Big[{[W,U]}_{\f{r}},V\Big]}_{\f{r}}.
\end{equation*}
\end{thm}
For a proof of Theorem~\ref{thm:curvatura_nat_red} we refer to~\cite[chapter~$10$, \S~$3$]{KobayashiNomizu2}; more precisely, the statement of the Theorem can be found in the proof of Proposition~$3.4$, ibid.

Our goal is to show that~$\mrm{Sp}(2n)/\mrm{U}(n)$ is a naturally reductive space, and that the naturally reductive metric on~$\mrm{Sp}(2n)/\mrm{U}(n)$ is the same K\"ahler metric induced by the identification
\begin{equation*}
\mrm{Sp}(2n)/\mrm{U}(n)\xrightarrow{\sim}\f{H}.
\end{equation*}
This will allow us to get an easy expression for the curvature of the Hermitian symmetric space~$\mrm{Sp}(2n)/\mrm{U}(n)$ (and for the curvature of~$\m{AC}^+$) from Theorem~\ref{thm:curvatura_nat_red}.

Recall that~$\mrm{U}(n)$ is a closed subgroup of~$\mrm{Sp}(2n)$ in the following way:
\begin{equation*}
\begin{split}
\mrm{U}(n)&\rightarrow\mrm{Sp}(2n)\\
X+\I Y&\mapsto\begin{pmatrix}X & Y\\-Y& X\end{pmatrix}
\end{split}
\end{equation*}
so we can identify~$\f{u}(n)$ with
\begin{equation*}
\begin{split}
\f{u}(n)=&\set*{\begin{pmatrix}X & Y\\ -Y& X\end{pmatrix}\tc X^\transpose=-X\mbox{ and }Y^\transpose=Y}\subseteq\f{sp}(2n).
\end{split}
\end{equation*}

\begin{lemma}
Consider the set
\begin{equation*}
\f{r}=\set*{\begin{pmatrix}A & B\\ B&-A\end{pmatrix}\tc A^\transpose=A\mbox{ and }B^\transpose=B}.
\end{equation*}
Then~$\f{sp}(2n)=\f{u}(n)\oplus\f{r}$, and this decomposition shows that~$\mrm{Sp}(2n)/\mrm{U}(n)$ is a reductive space.
\end{lemma}

\begin{proof}
It's clear that~$\f{u}(n)\cap\f{r}=\set{0}$. Consider~$\begin{pmatrix}A & B\\ C &-A^\transpose\end{pmatrix}\in\f{sp}(2n)$; then we can write it as a sum of elements in~$\f{u}(n)$ and~$\f{r}$ as follows:
\begin{equation*}
\begin{pmatrix}A & B\\ C &-A^\transpose\end{pmatrix}=\begin{pmatrix}\frac{A-A^\transpose}{2} & \frac{B-C}{2}\\
-\frac{B-C}{2} &\frac{A-A^\transpose}{2}\end{pmatrix}+\begin{pmatrix}\frac{A+A^\transpose}{2} & \frac{B+C}{2}\\
\frac{B+C}{2} &-\frac{A+A^\transpose}{2}\end{pmatrix}
\end{equation*}
so~$\f{sp}(2n)=\f{u}(n)\oplus\f{r}$. To check that~$\f{r}$ is~$\mrm{Ad}(\mrm{U}(n))$-invariant is a quick computation.
\end{proof}

\begin{rmk}\label{nota:comm_Sp}
For any two elements~$P,Q\in\f{r}$, the commutator~$[P,Q]$ is an element of~$\f{u}(n)$.
As a consequence,~$[P,Q]_{\f{r}}=0$ for all~$P,Q\in\f{r}$. This implies that \emph{any} product on~$\f{sp}(2n)$ that is~$\mrm{Ad}(\mrm{U}(n))$-invariant makes~$\mrm{Sp}(2n)/\mrm{U}(n)$ into a naturally reductive homogeneous space.
\end{rmk}

The product~$\langle U,V\rangle =2\,\mrm{Tr}(U\,V^\transpose)$ is a positive-definite pairing on~$\mrm{Sp}(2n)$ that is invariant under the adjoint action of~$\mrm{U}(n)$. Moreover, it defines on~$\mrm{Sp}(2n)/\mrm{U}(n)$ the same K\"ahler metric defined by its identification with~$\f{H}$. Indeed, by the definition of the action~$\chi$ on~$\f{H}$, for~$P=\begin{pmatrix}P_1 & P_2\\ P_2 &-P_1\end{pmatrix}\in\f{r}\cong T_o\left(\mrm{Sp}(2n)/\mrm{U}(n)\right)$
\begin{equation*}
\mrm{d}\chi_{o}(P)=2\,\I\,P_1+2\,P_2
\end{equation*}
so the metric induced on~$\mrm{Sp}(2n)/\mrm{U}(n)$ by~$\chi$ is
\begin{equation*}
\langle\mrm{d}\chi_o(P),\mrm{d}\chi_o(Q)\rangle=4\mrm{Tr}\left(P_1Q_1+P_2Q_2\right)=2\,\mrm{Tr}(P\,Q^\transpose).
\end{equation*}

\begin{rmk}
In fact the set~$\f{r}\subseteq\f{sp}(2n)$ is the orthogonal complement to~$\f{u}(n)$ under the pairing~$(U,V)\mapsto\mrm{Tr}(U\,V^\transpose)$. This product is~$\mrm{Ad}(\mrm{U}(n))$-invariant, so this gives an alternative way to show that~$\f{sp}(2n)=\f{u}(n)\oplus\f{r}$ is a reductive decomposition of~$\f{sp}(2n)$.
\end{rmk}

Putting together Theorem~\ref{thm:curvatura_nat_red} and Remark~\ref{nota:comm_Sp} we can compute the curvature of~$\mrm{Sp}(2n)/U(n)$.

\begin{prop}\label{prop:curvatura_Sp}
Consider the K\"ahler metric on~$\mrm{Sp}(2n)/\mrm{U}(n)$ induced by the isomorphism~$\mrm{Sp}(2n)/\mrm{U}(n)\xrightarrow{\sim}\f{H}(n)$. With the previous notation, its curvature tensor is
\begin{equation*}
\forall\,U,V,W\in\f{r},\quad R_o(U,V)(W)=-\Big[[U,V],W\Big].
\end{equation*}
\end{prop}

At this point computing the curvature of~$\m{AC}^+$ is straightforward:
\begin{proof}[Proof of Proposition~\ref{prop:curvatura_AC}]
The identification of~$\m{AC}^+(2n)$ with~$\mrm{Sp}(2n)/\mrm{U}(n)$ is given by the map~$\sigma(J)=\left(\Omega_0J\right)^{-\frac{1}{2}}\mrm{U}(n)$, and its differential at the point~$-\Omega_0$ is
\begin{equation*}
\mrm{d}\sigma_{-\Omega_0}(A)=-\frac{1}{2}\Omega_0A\in\f{r}.
\end{equation*}
Then, by Proposition~\ref{prop:curvatura_Sp} we have
\begin{equation*}
\begin{split}
R_{-\Omega_0}(A,B)(C)&=\sigma^*\!\left(-\Big[[\sigma_*A,\sigma_*B],\sigma_*C\Big]\right)=2\Omega_0\left(\frac{1}{8}\Big[[\Omega_0A,\Omega_0B],\Omega_0C\Big]\right)=\\
&=-\frac{1}{4}\Big[[A,B],C\Big]
\end{split}
\end{equation*}
using that~$A$,~$B$ and~$C$ anti-commute with~$\Omega_0$.
\end{proof}

\section{Compatible almost complex structures}\label{sec:J}

Consider again the set~$\mscr{J}$ of all almost complex structures compatible with a symplectic form~$\omega$. In a system of Darboux coordinates~$\bm{u}$ for~$\omega$ the matrix associated to~$J$ is an element of~$\m{AC}^+$, by the definition of~$\mscr{J}$. Notice that, for a different system of Darboux coordinates~$\bm{v}$, the matrix~$\frac{\diff\bm{v}}{\diff\bm{u}}$ that describes the change of coordinates is a~$\mrm{Sp}(2n)$-valued function. Considering the matrices associated to~$J$ in the two Darboux coordinate systems we have
\begin{equation*}
J(\bm{v})=\frac{\diff\bm{v}}{\diff\bm{u}}\,J(\bm{u})\,\left(\frac{\diff\bm{v}}{\diff\bm{u}}\right)^{-1}
\end{equation*}
so~$J(\bm{u})$ and~$J(\bm{v})$ are two elements of~$\m{AC}^+$ that differ by the action of an element of~$\mrm{Sp}(2n)$.

We can define a fibre bundle~$\m{E}$ on~$M$ with fibre~$\m{AC}^+$, that is trivial over Darboux coordinate systems and with~$\mrm{Sp}(2n)$-valued cocycles. Elements of~$\mscr{J}$ are then global sections of~$\m{E}$, i.e.~$\mscr{J}=\mrm{\Gamma}(M,\m{E})$.

This description of~$\mscr{J}$ as a space of sections is useful to define extra structures on~$\mscr{J}$: for example, it is an infinite-dimensional Fréchet manifold. Also, for any~$J\in\mscr{J}$ the tangent space at~$J$ is
\begin{equation*}
T_J\mscr{J}=T_J\Gamma(M,\m{E})=\Gamma(M,J^*(\mrm{Vert}\,\m{E}))
\end{equation*}
where~$\mrm{Vert}\,\m{E}$ is the vertical distribution of~$\m{E}$, the kernel of the projection on the base~$\pi:\m{E}\to M$. For any~$x\in M$,
\begin{equation*}
J^*(\mrm{Vert}\m{E})_x=\mrm{Vert}_{J(x)}\m{E}\cong T_{J(x)}\m{AC}^+
\end{equation*}
here the identification is done by fixing a Darboux coordinate system around~$x$, i.e. by locally trivializing~$\m{E}$. In other words, any~$A\in T_J\mscr{J}$ is itself a section of a fibre bundle on~$M$ that is trivial over any system of Darboux coordinates, and in any such trivialization~$A(x)\in T_{J(x)}\m{AC}^+$. This description of~$T_J\mscr{J}$ can be made more intrinsic by noticing that any first-order deformation~$A$ of an almost complex structure~$J$ is itself an endomorphism of~$TM$, and
\begin{equation*}
T_J\mscr{J}=\set*{A\tc AJ+JA=0\mbox{ and }\omega(A-,J-)+\omega(J-,A-)=0}.
\end{equation*}
We can rephrase the second condition,~$\omega(A-,J-)+\omega(J-,A-)=0$, in terms of the Riemannian metric~$g_J$ defined by~$\omega$ and~$J$:
\begin{equation*}
\omega(A-,J-)+\omega(J-,A-)=0\iff g_J(A-,-)-g_J(-,A-)=0
\end{equation*}
so if~$A\in T_J\mscr{J}$ the bilinear form defined by~$(v,w)\mapsto g_J(Av,w)$ is symmetric. Notice also that for~$A\in T_J\mscr{J}$, the condition~$AJ+JA=0$ implies that the~$\bb{C}$-linear extension of~$A$ to~$T_{\bb{C}} M$ maps~$T_J^{0,1}M$ to~$T_J^{1,0}M$ and vice versa. Moreover~$A$ is uniquely determined by its restriction to~$T^{0,1}_JM$, since it is a real endomorphism. In fact we could identify~$T_J\mscr{J}$ with a subset of~$\m{A}^{0,1}_J(T^{1,0}_J M )$, via the map
\begin{equation*}
A\mapsto A^{1,0}=A_{\restriction T^{0,1}_JM}:T^{0,1}_JM\to T^{1,0}_JM.
\end{equation*}
In a system of local coordinates for~$M$, the conditions for an endomorphism~$A$ to be in~$T_J\mscr{J}$ are equivalent to the identities:
\begin{equation}\label{eq:indentita_tangenti}
\begin{split}
J\indices{^i_j}A\indices{^j_k}&=-A\indices{^i_j}J\indices{^j_k}\\
g(J)_{ij}A\indices{^j_k}&=g(J)_{kj}A\indices{^j_i}.
\end{split}
\end{equation}

Since elements of~$\mscr{J}$ are locally maps for~$M$ to~$\m{AC}^+$, we can induce various geometric structures on~$\mscr{J}$ using the ones of~$\m{AC}^+$. We are interested in particular in the K\"ahler structure of~$\m{AC}^+$.

First of all, we define a complex structure~$\bb{J}:T\mscr{J}\to T\mscr{J}$ as follows: fix~$J\in\mscr{J}$ and~$A\in T_J\mscr{J}$; for any~$x\in M$ consider a trivialization of~$\m{E}$ around~$x$ (i.e. a system of Darboux coordinates for~$\omega$ around~$x$), giving~$A(x)\in T_{J(x)}\m{AC}^+$; on this vector space we have the complex structure described in Proposition~\ref{prop:metrica_AC+_Siegel}, given by~$A(x)\mapsto J(x)A(x)=(JA)(x)$. Notice that this does not depend on the choice of the trivialization, since the action of~$\mrm{Sp}(2n)$ on~$\m{AC}^+$ preserves the complex structure. Then
\begin{align*}
\bb{J}:T_J\mscr{J}&\to T_J\mscr{J}\\
A&\mapsto JA
\end{align*}
is an almost complex structure on~$\mscr{J}$. The same approach works to define a metric; for~$A,B\in T_J\mscr{J}$ and~$x\in M$ the number~$\frac{1}{2}\mrm{Tr}(A(x)B(x))$ depends just on~$x$, and not on the particular trivialization chosen to see~$A$ and~$B$ as elements of~$T_J\m{AC}^+$. We can then define a metric
\begin{align*}
\bb{G}:T_J\mscr{J}\times& T_J\mscr{J}\to\bb{R}\\
(A,&B)\mapsto\frac{1}{2}\int_{M}\mrm{Tr}(AB)\,\frac{\omega^n}{n!}
\end{align*}
and all the algebraic relations of~$\bb{J}$,~$\bb{G}$ carry over from those of the metric and the complex structure on~$\m{AC}^+$; in particular~$\bb{G}(\bb{J}-,\bb{J}-)=\bb{G}(-,-)$, and so we obtain a~$2$-form on~$\mscr{J}$,
\begin{equation}\label{eq:symp_form_J}
\Omega_J(A,B) =\bb{G}_J(\bb{J}A,B)=\frac{1}{2}\int_{M}\mrm{Tr}(JAB)\,\frac{\omega^n}{n!}.
\end{equation}
\begin{rmk}
If we denote by~$g_J$ the Hermitian metric on~$M$ defined by~$\omega$ and~$J\in\mscr{J}$, then~$g_J(A,B)=\mrm{Tr}(AB)$ for any~$A,B\in T_J\mscr{J}$. Indeed by working in a system of coordinates we have
\begin{equation*}
g_J(A,B)=g\indices{^j^k}g\indices{_i_l}A\indices{^i_j}B\indices{^l_k}=g\indices{^j^k}g\indices{_i_j}A\indices{^i_l}B\indices{^l_k}=A\indices{^i_l}B\indices{^l_i}
\end{equation*}
where we have used~\eqref{eq:indentita_tangenti} in the second equality. So we can rewrite the expression of~$\bb{G}$ in a way that makes the role of the point~$J$ more explicit, i.e.
\begin{equation*}
\begin{split}
\bb{G}_J(A,B)&=\frac{1}{2}\int_{M}g_J(A,B)\,\frac{\omega^n}{n!}=\frac{1}{2}\int_{M}\omega(A,JB)\,\frac{\omega^n}{n!}.
\end{split}
\end{equation*}
\end{rmk}

\begin{thm}\label{thm:kahler_J}
With the almost complex structure~$\bb{J}$ and the metric~$\bb{G}$,~$\mscr{J}$ is an infinite-dimensional (formally) K\"ahler manifold.
\end{thm}
This is actually a particular case of a more general result. Indeed, it holds \emph{for any} fibre bundle over a manifold with a fixed volume form whose fibres are K\"ahler manifolds, see~\cite[Theorem~$2.4$]{Koiso_complex_structure}. We will however sketch a proof of Theorem~\ref{thm:kahler_J} in Section~\ref{sec:KahlerStrJ}, based on the argument in~\cite{Koiso_complex_structure}, as this will also be useful in Chapter~\ref{chap:HyperkahlerReduction}.

We now come to explaining the statement of Theorem~\ref{thm:Donaldson_scalcurv_mm}. The first ingredient is the definition, for~$J\in\mscr{J}$, of the \emph{Hermitian scalar curvature} of~$g_J$.

Any almost complex structure~$J$ on~$ M$ induces a splitting of the complexified tangent bundle,~$T_{\bb{C}}{ M }=T^{1,0}_J M \oplus T^{0,1}_J M$. Moreover, this splitting induces a decomposition of the cotangent bundles,~$\ext{k}M =\bigoplus_{p+q=k}\ext{p,q}_J M$. This also gives a decomposition of the exterior differential
\begin{equation*}
\begin{split}
\mrm{d}:\m{A}^0 M &\to\m{A}^{1,0}_J M \oplus\m{A}^{0,1}_J M \\
\alpha &\mapsto \diff_J(\alpha)+\bar{\diff}_J(\alpha).
\end{split}
\end{equation*}
There is a similar decomposition for~$\mrm{d}:\m{A}^1 M \to\m{A}^{2,0}_J M \oplus\m{A}^{1,1}_J M \oplus\m{A}^{0,2}_J M$, involving the Nijenhuis tensor of~$J$:
\begin{equation*}
N_J(X,Y):=[X,Y]+J[JX,Y]+J[X,JY]-[JX,JY].
\end{equation*}
Notice that for~$Z,W\in\Gamma(T^{1,0}_J M )$ one has
\begin{equation*}
\begin{split}
N_J(Z,W)&=2[Z,W]+2\I J[Z,W]\in\Gamma(T^{0,1}_J M )\\
N_J(Z,\bar{W})&=0\\
N_J(\bar{Z},\bar{W})&=2[\bar{Z},\bar{W}]-2\I J[\bar{Z},\bar{W}]\in\Gamma(T^{1,0}_J M ).
\end{split}
\end{equation*}
\begin{lemma}\label{lemma:decomposizione_d}
For any almost complex structure~$J$, the~$\bb{C}$-linear extension of the exterior differential~$\mrm{d}:\m{A}^1(M,\bb{C})\to\m{A}^2(M,\bb{C})$ is decomposed as
\begin{equation*}
\mrm{d}(\varphi)=\diff_J(\varphi)+\bar{\diff}_J(\varphi)-\frac{1}{4}\varphi\circ N_J.
\end{equation*}
\end{lemma}
\begin{proof}
Let~$X,Y$ be real vector fields on~$ M$, and let~$\varphi\in\m{A}^{1,0}_J M$. Then by definition
\begin{equation*}
(\mrm{d}\varphi-\diff\varphi-\bar{\diff}\varphi)(X,Y)=\mrm{d}\varphi(X^{0,1},Y^{0,1})
\end{equation*}
and since~$\varphi$ is of type~$(1,0)$
\begin{equation*}
\mrm{d}\varphi(X^{0,1},Y^{0,1})=-\varphi([X^{0,1},Y^{0,1}]).
\end{equation*}
Decomposing this commutator in its~$(1,0)$ and~$(0,1)$ parts we have
\begin{equation*}
\begin{split}
[X^{0,1},Y^{0,1}]=&\frac{1}{2}\left([X^{0,1},Y^{0,1}]-\I J[X^{0,1},Y^{0,1}]\right)+\frac{1}{2}\left([X^{0,1},Y^{0,1}]+\I J[X^{0,1},Y^{0,1}]\right)=\\
=&\frac{1}{4}N_J(X^{0,1},Y^{0,1})+\frac{1}{2}\left([X^{0,1},Y^{0,1}]+\I J[X^{0,1},Y^{0,1}]\right)
\end{split}
\end{equation*}
so in the end
\begin{equation*}
\mrm{d}\varphi(X^{0,1},Y^{0,1})=-\frac{1}{4}\varphi(N_J(X^{0,1},Y^{0,1}))=-\frac{1}{4}\varphi(N_J(X,Y)).
\end{equation*}
A similar computation gives the same result also for a~$(0,1)$-form.
\end{proof}

Consider now the complex vector bundle~$\ext{1,0}_J M$ and the Hermitian metric~$h$ defined by~$g_J$ on the fibres of~$\ext{1,0}_JM$. A~$\bb{C}$-linear connection~$\nabla:\Gamma(\ext{1,0}_J M )\to\m{A}^1(\ext{1,0}_J M )$ defines two operators
\begin{equation*}
\begin{split}
\nabla^{(1,0)}&:\Gamma(\ext{1,0}_J M )\to\m{A}^{1,0}(\ext{1,0}_J M )\\
\nabla^{(0,1)}&:\Gamma(\ext{1,0}_J M )\to\m{A}^{0,1}(\ext{1,0}_J M ).
\end{split}
\end{equation*}
On the other hand the type decomposition of~$\mrm{d}:\Gamma(\ext{1,0}_J M )\to\m{A}^1(\ext{1,0}_J M )$ already gives an operator~$\bdiff_J:\Gamma(\ext{1,0}_J M )\to\m{A}^{0,1}(\ext{1,0}_J M )$. As in the integrable case,~$\bdiff_J$ and~$h$ together determine a connection on~$\ext{1,0}_J M$. When the complex structure is integrable this is usually called the \emph{Chern connection} of~$M$; however we can drop the integrability assumption and still obtain essentially  the same result.
\begin{lemma}\label{lemma:Chern_conn}
Let~$h$ be a Hermitian metric on the complex vector bundle~$E\xrightarrow{\pi}M$. Then any metric connection on~$E$ is uniquely determined by its~$(0,1)$-part.
\end{lemma}
\begin{proof}
Let~$D:\Gamma(E)\to\m{A}^{0,1}(E)$ be a~$(0,1)$-operator and let~$\tau\indices{^i_j}$ be the matrix of~$(0,1)$-forms associated to~$D$ in a frame~$\mathbf{s}=(\sigma_1,\dots,\sigma_n)$ for~$E$. Assume that~$\nabla$ is a metric connection that has~$D$ as~$(0,1)$-part, let~$D'$ be its~$(1,0)$-part and let~$\vartheta\indices{^i_j}$ be the matrix of~$D'$ in the frame~$\mathbf{s}$, so that~$\nabla\sigma_i=(\tau\indices{^j_i}+\vartheta\indices{^j_i})\sigma_j$. Since we are assuming~$\nabla$ to be compatible with the Hermitian fibre metric~$h$, we have
\begin{equation*}
\mrm{d}(h(\sigma_i,\sigma_j))=h(\nabla\sigma_i,\sigma_j)+h(\sigma_i,\nabla\sigma_j)=(\tau\indices{^k_i}+\vartheta\indices{^k_i})h_{kj}+(\overline{\tau\indices{^l_j}}+\overline{\vartheta\indices{^l_j}})h_{il}
\end{equation*}
but on the other hand
\begin{equation*}
\mrm{d}(h(\sigma_i,\sigma_j))=\diff h_{ij}+\bar{\diff}h_{ij}.
\end{equation*}
Comparing the~$(1,0)$-parts of these expressions we find
\begin{equation*}
\diff h_{ij}=\vartheta\indices{^k_i}h_{kj}+\overline{\tau\indices{^l_j}}h_{il}
\end{equation*}
and in the end we obtain
\begin{equation}\label{eq:Chern_quasi_complessa}
\vartheta\indices{^k_i}=h^{jk}\diff h_{ij}-h^{jk}\overline{\tau\indices{^l_j}}h_{il}.
\end{equation}
This shows the uniqueness claim. Tracing back the proof it is easy to check that we get a metric connection if we use as~$(1,0)$-part the operator defined by equation~\eqref{eq:Chern_quasi_complessa} and as~$(0,1)$-part~$D$.
\end{proof}

In particular there is a unique affine connection~$\nabla_J$ on~$\ext{1,0}_J M$ such that~$\nabla^{(0,1)}_J=\bar{\diff}_J$ and~$\nabla_J h=0$. This connection~$\nabla_J$ induces of course also a connection on~$\ext{0,1} M$ (by conjugation),~$T^{1,0} M$ (by duality) and~$T^{0,1} M$. Moreover, if we consider~$T M$ as a~$\bb{C}$-vector bundle on~$ M$ by~$\I X:=JX$, we also have a connection induced by~$\nabla_J$ on~$T M$ via the isomorphism
\begin{equation*}
\begin{split}
T M &\to T^{1,0}_J M \\
X&\mapsto X^{1,0}=\frac{1}{2}(X-\I JX).
\end{split}
\end{equation*}
This connection has been studied in the context of almost Hermitian manifolds. For a proof of the following result see~\cite[Theorem~$3.2$]{KobayashiNomizu2} and~\cite[Proposition~$2$ and the discussion at pages~$272-273$]{Gauduchon_HermitianConnections}.
\begin{lemma}\label{lemma:torsion}
With the previous notation, the torsion of~$\nabla_J$ on~$TM$ is
\begin{equation*}
T^{\nabla}(X,Y)=-\frac{1}{4}N_J(X,Y).
\end{equation*}
\end{lemma}
Notice that~$h$ induces a Hermitian metric on the line bundle~$K_J:=\ext{n,0}_J M$, and~$\nabla$ defines a metric connection on this bundle.
\begin{lemma}
The curvature of~$\nabla$ on~$K_J$ is a purely imaginary~$2$-form.
\end{lemma}
\begin{proof}
Let~$\Theta$ be the~$2$-form representing locally the curvature~$F(\nabla)$. From the proof of Lemma~\ref{lemma:Chern_conn} we can get a local expression for~$\Theta$ in terms of~$h$ and the local representative~$\vartheta$ of~$\nabla^{0,1}$:
\begin{equation*}
\Theta=\mrm{d}\left(\diff\,\mrm{log}\,h(s,s)+\vartheta-\bar{\vartheta}\right).
\end{equation*}
It is enough to prove that~$\mrm{d}\left(\diff\,\mrm{log}\,h(s,s)\right)$ is purely imaginary. To check this use Lemma~\ref{lemma:decomposizione_d} to get:
\begin{equation*}
\mrm{d}\left(\diff\,\mrm{log}\,h(s,s)\right)=\diff^2\mrm{log}\,h(s,s)+\bdiff\diff\,\mrm{log}\,h(s,s)-\frac{1}{4}\left(\diff\,\mrm{log}\,h(s,s)\right)\circ N_J.
\end{equation*}
Now we can use Lemma~\ref{lemma:decomposizione_d} again to see that this form is purely imaginary, since~$\mrm{d}^2=0$ implies
\begin{equation*}
\diff^2-\frac{1}{4}N_J^\transpose\circ\diff+\bdiff\diff=-\left(\bdiff^2-\frac{1}{4}N_J^\transpose\circ\bdiff+\diff\bdiff\right).\qedhere
\end{equation*}
\end{proof}

Then we can write~$F(\nabla)=\I\rho$ for a real~$2$-form~$\rho$. When~$J$ is integrable, and so~$(M,J,\omega)$ is a K\"ahler manifold,~$\rho$ is just the Ricci form of the manifold, and its contraction with the metric tensor gives the usual scalar curvature. In our case we can still take the contraction of~$\rho$ with~$\omega$ to find a function, the \emph{Hermitian scalar curvature} of~$( M ,J,\omega)$. More precisely we define~$S(J)$ as the unique (real) function on~$ M$ such that
\begin{equation*}
S(J)\omega^n=n\rho\wedge\omega^{n-1}.
\end{equation*}

Notice that the average~$\hat{S}$ of~$S(J)$ is a topological quantity, since
\begin{equation*}
\int_{ M }S(J)\frac{\omega^n}{n!}=2\pi c_1( M ,\omega,J)\cup[\omega]^{n-1}.
\end{equation*}
Indeed, the first Chern class of a symplectic manifold does not depend on the choice of the compatible almost complex structure that we choose on it: the space of compatible almost complex structures is connected, and the first Chern class is left invariant under an infinitesimal deformation of the complex structure. So we have a continuous function~$J\to c_1( M ,\omega,J)\in H^2( M ,\bb{Z})$. But this target space is discrete, hence~$c_1( M ,\omega)$ does not depend on the choice of~$J\in\mscr{J}$.

The group~$\mscr{G}:=\mrm{Ham}( M ,\omega)$ %
\nomenclature[G]{$\mscr{G}$}{shorthand for~$\mrm{Ham}(M,\omega)$}%
of Hamiltonian diffeomorphisms of~$ M$ acts naturally on~$\mscr{J}$ by pull-backs, i.e. by letting
\begin{equation*}
\varphi.J:=\left(\varphi^{-1}\right)^*J=\mrm{d}\varphi\circ J\circ \mrm{d}\varphi^{-1}
\end{equation*}
for any~$J\in\mscr{J}$ and~$\varphi\in\mscr{G}$. Recall that the Lie algebra of~$\mscr{G}$ is the algebra of Hamiltonian vector fields, and we identify this with~$\m{C}^\infty_0( M ,\bb{R})$, as in Section~\ref{sec:hamiltonian}; the dual of the Lie algebra~$\m{C}^\infty_0( M )$ is identified with~$\m{C}^\infty_0( M )$ via the usual~$L^2$ pairing of functions on~$ M$.

We can now state Theorem~\ref{thm:Donaldson_scalcurv_mm} more precisely.

\begin{thm}\label{thm:moment_map}
The map~$J\mapsto 2\big(S(J)-\hat{S}\big)$ is a moment map for the action of~$\mscr{G}$ on~$\mscr{J}$.
\end{thm}

The infinitesimal action on~$\mscr{J}$ of a function~$f\in\m{C}^\infty_0(M,\bb{R})$ is, from the definition
\begin{equation*}
\hat{f}_J=\frac{\mrm{d}}{\mrm{d}t}\Bigr|_{t=0}\mrm{exp}(-tf).J=\frac{\mrm{d}}{\mrm{d}t}\Bigr|_{t=0}\left(\Phi_{X_f}^t\right)^*J=\m{L}_{X_f}J
\end{equation*}
where~$\Phi_{X_f}^t$ is the flow of the Hamiltonian vector field~$X_f$. Then we can formulate Theorem~\ref{thm:moment_map} as follows: for every~$A\in T_J\mscr{J}$ and every~$f\in\m{C}^\infty_0(M)$
\begin{equation}\label{eq:integral_form_mm}
2\int DS_J(A) f\frac{\omega^n}{n!}=-\frac{1}{2}\int\mrm{Tr}\left(J(\m{L}_{X_f}J)\,A\right)\frac{\omega^n}{n!}.
\end{equation}
It is not straightforward to check that this equation holds for an almost complex structure~$J$ that is not necessarily integrable. We will give an account of Donaldson's proof in Section~\ref{sec:thm:moment_map_proof}; for the integrable case instead the computation is quite easier and can be found for example in~\cite{Szekelyhidi_phd} or~\cite{Tian_libro}.

Theorem~\ref{thm:moment_map} also implicitly states that the map~$J\to S(J)$ is \emph{equivariant} with respect to the action~$\mscr{G}\curvearrowright\mscr{J}$ and the co-adjoint action of~$\mscr{G}$ on~$T_e^*\mscr{G}$. To spell out this condition more explicitly, notice that for any symplectomorphism~$\varphi$ and any function~$f$, for the Hamiltonian vector field~$X^\omega_f$ we have
\begin{equation*}
\varphi_*\left(X^\omega_f\right)=X^\omega_{f\circ\varphi^{-1}}.
\end{equation*}
Then under the identification (Proposition~\ref{prop:Ham_properties}) of the Lie algebra of~$\mscr{G}$ with~$\m{C}^\infty_0$ we have~$\mrm{Ad}_{\varphi^{-1}}(f)=\varphi^*f$. This means that the equivariance of the moment map can be restated in our case as
\begin{equation}\label{eq:equivariance_scalar}
\int_MS(\varphi.J)f\frac{\omega^n}{n!}=\int_MS(J)\,\varphi^*f\frac{\omega^n}{n!}
\end{equation}
for every~$\varphi\in\mscr{G}$ and~$f\in\m{C}^\infty_0$. To prove that~\eqref{eq:equivariance_scalar} holds, first notice that if we write~$g(\omega,J)$ for the metric defined by~$\omega$ and~$J\in\mscr{J}$ then
\begin{equation*}
g(\omega,(\varphi^{-1})^*J)=(\varphi^{-1})^*g(\omega,J)
\end{equation*}
since~$\varphi$ is a symplectomorphism. Moreover
\begin{equation*}
\ext{1,0}_{(\varphi^{-1})^*J}M=(\varphi^{-1})^*\left(\ext{1,0}_JM\right)
\end{equation*}
so the metric defined by~$\omega$ on~$\ext{1,0}_{\varphi.J}M$ is the pull-back of the metric defined by~$\omega$ on~$\ext{1,0}_JM$. From our definition of~$S(J)$ then we have~$S(\varphi.J)=(\varphi^{-1})^*S(J)$ and equation~\eqref{eq:equivariance_scalar} easily follows, as~$\varphi$ preserves the volume form~$\omega^n$.

\begin{rmk}\label{rmk:integrable_submanifoldJ}
The action of~$\mscr{G}$ on~$\mscr{J}$ preserves the locus of \emph{integrable} almost complex structures, that we denote by~$\mscr{J}_{\mrm{int}}$. This is essentially a consequence of the naturality of the Lie bracket: for any almost complex structure~$J$ and a diffeomeorphism~$\phi$, it is easy to check that
\begin{equation*}
N_{\phi^*J}(X,Y)=\phi^{-1}_*\left(N_J(\phi_*X,\phi_*Y)\right)
\end{equation*}
so that if~$J$ is integrable, also~$\phi^*J$ is. Since the action~$\mscr{G}\curvearrowright\mscr{J}$ is by pull-backs, the action preserves~$\mscr{J}_{\mrm{int}}$. Moreover,~$\mscr{J}_{\mrm{int}}$ is a K\"ahler submanifold of~$\mscr{J}$; this is a consequence of the next Lemma.
\end{rmk}
\begin{definition}\label{def:integrable}
Let~$A\in\Gamma(\mrm{End}(TM))$ be a first-order deformation of the complex structure~$J$, i.e.~$AJ+JA=0$. Then we say that~$A$ is an \emph{integrable} first-order deformation of~$J$ if
\begin{equation*}
N(J+\varepsilon\,A)=O(\varepsilon^2).
\end{equation*}
\end{definition}
This condition is equivalent to the more well-known framing of integrability in terms of the Maurer-Cartan equation. Indeed if we have a family of complex structures~$J+t\,A+O(t^2)$ then the first-order Maurer-Cartan equation, i.e.
\begin{equation*}
\bdiff_J A^{1,0}=0
\end{equation*}
is equivalent to~$A$ being an integrable first-order deformation of~$J$ in the sense of Definition~\ref{def:integrable}. An immediate consequence is that~$\mscr{J}_{\mrm{int}}$ is a complex submanifold of~$\mscr{J}$.
\begin{lemma}\label{lemma:integrable_submanifoldJ}
Let~$J$ be a complex structure, and let~$A$ be an integrable first-order deformation of~$J$. Then~$JA$ is also an integrable first-order deformation of~$J$.
\end{lemma}

\subsection{K\"ahler structure of~$\mscr{J}$}\label{sec:KahlerStrJ}

The main goal of this section is to prove Theorem~\ref{thm:kahler_J}. We will also digress to study the connection and curvature of~$\mscr{J}$, showing that the curvature of~$\mscr{J}$ is essentially the same as that of~$\m{AC}^+$ and proving that~$\mscr{J}$ is a formally symmetric space. Compare this result with the curvature of the space of K\"ahler potentials in~\cite{Donaldson_SymmKahlerHam}. The expression we will find for the curvature of~$\mscr{J}$ is well-know, see~\cite[Example~$4.3.6$]{McDuff_Salamon} and also~\cite{Smolentsev_curvature}.

Notice that Theorem~\ref{thm:kahler_J} consists of two statements:~$\bb{J}$ is formally integrable and the symplectic form~$\Omega$ of~\eqref{eq:symp_form_J} is closed. We will actually prove a more general result, that will also be needed in a slightly different form in Chapter~\ref{chap:HyperkahlerReduction}. This is essentially based on the discussion of a more general problem in~\cite{Koiso_complex_structure}.
\begin{thm}\label{thm:differenziazione}
Let~$k$ be a~$r$-form on~$\m{AC}^+$ invariant under the~$\mrm{Sp}(2n)$-action, and let~$K$ be a~$r$-form on~$\mscr{J}$ such that for all~$J\in\mscr{J}$ and~$v_1,\dots,v_r\in T_{J}\mscr{J}$
\begin{equation*}
K_{J}(v_1,\dots,v_r)=\int_{x\in M} k_{J(x)}(v_1(x),\dots,v_r(x))\frac{\omega^n}{n!}
\end{equation*}
where the second expression is computed by taking a local trivialization of~$\m{E}$ around each~$x\in M$. Then
\begin{equation*}
\mrm{d}K_{J}(v_0,v_1,\dots,v_r)=\int_{x\in M}\mrm{d}k_{J(x)}(v_0,v_1,\dots,v_r)\frac{\omega^n}{n!}.
\end{equation*}
\end{thm}
Using Theorem~\ref{thm:differenziazione} we can prove that~$\Omega$ is closed just by noting that~$\Omega$ is defined by integration from the K\"ahler form of~$\m{AC}^+$.
\begin{rmk}
In fact we will need Theorem~\ref{thm:differenziazione} just for~$r=0,1,2$. For~$r=0$ the result is quite easy to prove: for~$J\in\mscr{J}$ and~$v\in T_J\mscr{J}$, let~$J_t$ be a path in~$\mscr{J}$ such that~$v=\frac{\mrm{d}}{\mrm{d}t}\Bigr|_{t=0}J_t$. Then
\begin{equation*}
\begin{split}
\mrm{d}K_J(v)&=\frac{\mrm{d}}{\mrm{d}t}\Bigr|_{t=0}\int_{x\in M}k(J_t(x))\frac{\omega^n}{n!}=\int_{x\in M}\frac{\mrm{d}}{\mrm{d}t}\Bigr|_{t=0}k(J_t(x))\frac{\omega^n}{n!}=\\
&=\int_{x\in M}\mrm{d}k_{J(x)}(v(x))\frac{\omega^n}{n!}.
\end{split}
\end{equation*}
Here the last equality holds since the matrix~$v(x)$ associated to~$v$ in a Darboux coordinate system around~$x\in M$ is given by~$\frac{\mrm{d}}{\mrm{d}t}\Bigr|_{t=0}J_t(x)$.
\end{rmk}

\begin{proof}[Proof of Theorem~\ref{thm:differenziazione}.]
We sketch the proof for ~$r=1$; the other cases are very similar. It will be convenient to introduce some additional notation: for~$x\in M$ and a system of Darboux coordinates~$\bm{u}$ around~$x$, let~$\Phi^x_{\bm{u}}$ be the map
\begin{equation*}
\begin{split}
\Phi^x_{\bm{u}}:\mscr{J}&\to \m{AC}^+\\
J&\mapsto J(x)
\end{split}
\end{equation*}
given by locally trivializing the fibre bundle over the coordinate system~$\bm{u}$.

For a tangent vector~$v\in T_J\mscr{J}$, we can extend it to a vector field~$V$ on an open neighbourhood of~$J\in\mscr{J}$ in such a way that~$V$ is \emph{constant} in a system of local coordinates for~$\mscr{J}$. For the details about how to find local coordinates for~$\mscr{J}$, see~\cite[proof of Theorem~$1.2$]{Koiso_complex_structure}. Moreover, this extension~$V$ is such that~$\mrm{d}\Phi^x_{\bm{u}}(V)$ is a vector field on~$\m{AC}^+(2n)$, itself constant in a system of coordinates for~$\m{AC}^+(2n)$. This essentially reduces the proof to the case~$r=0$ that we already discussed.

Now fix~$J\in\mscr{J}$ and~$v,w\in T_J\mscr{J}$. If we extend~$v, w$ to constant vectors~$V,W$ as in the previous paragraph, we can compute
\begin{equation*}
\mrm{d}K_J(v,w)=v_J(K(W))-w_J(K(V))-K([V,W])
\end{equation*}
however,~$[V,W]=0$ since the vector fields are constant; for the other two terms we have, if~$v=\diff_t\Bigr|_{t=0}J_t$:
\begin{equation*}
\begin{split}
v_J(K(W))&=\frac{\mrm{d}}{\mrm{d}t}\Bigr|_{t=0}\int_{x\in M} k_{\Phi^x_{\bm{u}}(J_t)}\left(\left(\mrm{d}\Phi^x_{\bm{u}}\right)_{J_t}(W)\right)\,\frac{\omega^n}{n!}=\\
&=\int_{x\in M} \left(\mrm{d}\Phi^x_{\bm{u}}\right)_J(v)\left(k(\mrm{d}\Phi^x_{\bm{u}}(W))\right)\,\frac{\omega^n}{n!}
\end{split}
\end{equation*}
so we find
\begin{equation*}
\begin{split}
\mrm{d}K_J(v,w)=&\int_{x\in M} \Big[\left(\mrm{d}\Phi^x_{\bm{u}}\right)_J(v)\left(k(\mrm{d}\Phi^x_{\bm{u}}(W))\right)-\left(\mrm{d}\Phi^x_{\bm{u}}\right)_J(w)\left(k(\mrm{d}\Phi^x_{\bm{u}}(V))\right)-\\
&\quad\quad\quad-k_{\Phi^x_{\bm{u}}(J)}\left(\left[\mrm{d}\Phi^x_{\bm{u}}(V),\mrm{d}\Phi^x_{\bm{u}}(W)\right]\right)\Big]\frac{\omega^n}{n!}=\\
=&\int_{x\in M} \mrm{d}k_{\Phi^x_{\bm{u}}(J)}\left(\mrm{d}\Phi^x_{\bm{u}}(v),\mrm{d}\Phi^x_{\bm{u}}(w)\right)\,\frac{\omega^n}{n!}.\qedhere
\end{split}
\end{equation*}
\end{proof}
Using Theorem~\ref{thm:differenziazione} we can also find an expression for the Levi-Civita connection of the metric~$\bb{G}$. Notice that the usual Koszul formula for the Levi-Civita connection, i.e.
\begin{equation}\label{eq:six_terms_formula}
\begin{split}
2\langle Z,\nabla_XY\rangle=&X\langle Z,Y\rangle+Y\langle Z,X\rangle-Z\langle X,Y\rangle+\\
&+\langle Y,[Z,X]\rangle+\langle X,[Z,Y]\rangle+\langle Z,[X,Y]\rangle
\end{split}
\end{equation}
gives the uniqueness of a torsion-free metric connection also in an infinite-dimensional setting, but does not guarantee its existence. However, a different characterization of the covariant derivative shows that it does exist in our case, see~\cite[Example~$4.3.6$]{McDuff_Salamon}. We can show that the connection is essentially the same of~$\m{AC}^+$.
\begin{lemma}\label{lemma:connection_scrJ}
Let~$\nabla$ be the Levi-Civita connection of~$\mscr{J}$, and let~$X$ and~$Y$ be vector fields on~$\mscr{J}$. Then for every point~$x$ of~$M$ and any system of Darboux coordinates~$\bm{u}$ around~$x$, the vector~$\nabla_XY\in\Gamma(\mrm{End}(TM))$ satisfies
\begin{equation}\label{eq:connection_scrJ}
{\Phi^x_{\bm{u}}}_*\left(\nabla_XY\right)=\nabla_{{\Phi^x_{\bm{u}}}_*X}{\Phi^x_{\bm{u}}}_*Y
\end{equation}
where~$\Phi^x_{\bm{u}}$ is a local trivialization of~$\m{E}\to M$, as in the Proof of Theorem~\ref{thm:differenziazione}.
\end{lemma}
In other words, the Levi-Civita connection of~$\mscr{J}$ is, pointwise, the pull-back connection obtained from~$\m{AC}^+$ by a local trivialization.
\begin{proof}
First notice that~\eqref{eq:connection_scrJ} does define a connection on~$\mscr{J}$, since to define a vector on~$\mscr{J}$ it is sufficient to give its expression over any system of Darboux coordinates for~$\omega_0$, as long as this expression is compatible with the transition functions between two different Darboux coordinate systems.

We will show that any connection satisfying equation~\eqref{eq:connection_scrJ} coincides with the Levi-Civita connection of~$\bb{G}$, by checking that it satisfies the Koszul formula~\eqref{eq:six_terms_formula}. For notational convenience we let~$\Phi=\Phi^x_{\bm{u}}$ for the course of this proof. Let~$Z$ be a third vector field, and denote by~$\langle-,-\rangle$ the metric of~$\m{AC}^+$; then, if~$\nabla$ is a connection that satisfies~\eqref{eq:connection_scrJ}
\begin{equation*}
2\bb{G}_J(Z,\nabla_XY)=\int_M 2\,\left\langle\Phi_*Z,\nabla_{\Phi_*X}\Phi_*Y\right\rangle\frac{\omega^n}{n!}
\end{equation*}
then we can use the Koszul formula for the connection of~$\m{AC}^+$ and the naturality of the Lie derivative to find
\begin{equation*}
\begin{split}
&2\bb{G}_J(Z,\nabla_XY)=\\
&=\int\!\!\Big[\Phi_*X\left(\langle\Phi_*Z,\Phi_*Y\rangle\right)+\Phi_*Y\left(\langle\Phi_*Z,\Phi_*X\rangle\right)-\Phi_*Z\left(\langle\Phi_*X,\Phi_*Y\rangle\right)\!\Big]\frac{\omega^n}{n!}+\\
&\quad\quad\quad+\bb{G}(Y,[Z,X])+\bb{G}(X,[Z,Y])+\bb{G}(Z,[X,Y])
\end{split}
\end{equation*}
Then Theorem~\ref{thm:differenziazione} shows, for~$r=0$, that~$\nabla$ satisfies the Koszul formula.
\end{proof}
The same ideas can be used to see that the curvature tensor of~$\mscr{J}$ is also pointwise equal to that of~$\m{AC}^+$, c.f. Proposition~\ref{prop:curvatura_AC}.
\begin{cor}\label{cor:curvature_mscrJ}
For~$J\in\mscr{J}$ and~$A,B,C\in T_J\mscr{J}$, the curvature at~$J$ is
\begin{equation*}
R_J(A,B)C=-\frac{1}{4}\Big[[A,B],C\Big].
\end{equation*}
\end{cor}
The expression for the covariant derivative in~\ref{lemma:connection_scrJ} also allows us to reduce metric properties of~$\mscr{J}$ to those of~$\m{AC}^+$.
\begin{cor}
The complex structure and the curvature tensor of~$\mscr{J}$ are covariantly constant. In particular,~$\mscr{J}$ is a formally symmetric infinite-dimensional K\"ahler manifold.
\end{cor}
\begin{proof}
Fix a point~$J\in\mscr{J}$, and consider~$A\in T_J\mscr{J}$. To prove that~$\left(\nabla_A\bb{J}\right)_J=0$ it is enough to show that for all~$B,C\in T_J\mscr{J}$
\begin{equation*}
\bb{G}_J\left(C,(\nabla_A\bb{J})B\right)=0.
\end{equation*}
Using Lemma~\ref{lemma:connection_scrJ} and the definitions of~$\bb{G}$ and~$\bb{J}$, we can reduce this to a condition on~$\m{AC}^+$: extend~$A$,~$B$ and~$C$ to vector fields near~$J$. Then
\begin{equation*}
\begin{split}
\bb{G}_J\left(C,(\nabla_A\bb{J})B\right)=&\bb{G}_J\left(C,\nabla_A(\bb{J}B)-\bb{J}\nabla_AB\right)=\\
=&A\left(\bb{G}(C,\bb{J}B)\right)-\bb{G}_J\left(\nabla_AC,\bb{J}B)\right)-\bb{G}\left(C,\bb{J}\nabla_AB\right).
\end{split}
\end{equation*}
Using Theorem~\ref{thm:differenziazione} and Lemma~\ref{lemma:connection_scrJ} we find, if we let~$\langle-,-\rangle$ and~$I$ be respectively the metric and the complex structure on~$\m{AC}^+$
\begin{equation*}
\begin{split}
\bb{G}_J\left(C,(\nabla_A\bb{J})B\right)&=\int_M\Big[
{\Phi_{\bm{u}}^x}_*A\left(\langle{\Phi_{\bm{u}}^x}_*C,I{\Phi_{\bm{u}}^x}_*B\rangle\right)
-\\
-\langle\nabla_{{\Phi_{\bm{u}}^x}_*A}&{\Phi_{\bm{u}}^x}_*C,I{\Phi_{\bm{u}}^x}_*B\rangle-\langle{\Phi_{\bm{u}}^x}_*C,I\nabla_{{\Phi_{\bm{u}}^x}_*A}{\Phi_{\bm{u}}^x}_*B\rangle\Big]\frac{\omega^n}{n!}=\\
&=\int_M \langle{\Phi_{\bm{u}}^x}_*C,(\nabla_{{\Phi_{\bm{u}}^x}_*A}I){\Phi_{\bm{u}}^x}_*B\rangle\frac{\omega^n}{n!}=0.
\end{split}
\end{equation*}
Of course this computation relies on the integrability of the complex structure of~$\m{AC}^+$. Similarly, since~$\m{AC}^+$ is a symmetric space we can compute that for any five tangent vectors~$A$,~$B$,~$C$,~$D$,~$E$ at~$J$ we have
\begin{equation*}
\bb{G}_J\left(A,(\nabla_BR)(C,D)E\right)=0.
\end{equation*}
The computation is completely analogous to the one for~$\bb{J}$, but slightly more cumbersome since it involves more terms. Of course one should use the expression for the curvature of Corollary~\ref{cor:curvature_mscrJ}, and more precisely the formula
\begin{equation*}
{\Phi_{\bm{u}}^x}_*\left(R(A,B)C\right)=R_{\m{AC}^+}\left({\Phi_{\bm{u}}^x}_*A,{\Phi_{\bm{u}}^x}_*B\right){\Phi_{\bm{u}}^x}_*C.\qedhere
\end{equation*}
\end{proof}
Given the expression for the curvature of Corollary~\ref{cor:curvature_mscrJ}, it is quite natural to conjecture that~$\mscr{J}$ is in fact a K\"ahler-Einstein manifold, since~$\m{AC}^+$ is and the curvatures of the two spaces have very similar expressions. The main difficulty to prove this is to actually \emph{define} in a sensible way the Ricci curvature of an infinite-dimensional manifold.

\section{Proof of Theorem~\ref{thm:Donaldson_scalcurv_mm}}\label{sec:thm:moment_map_proof}

In this Section we give a proof of Theorem~\ref{thm:Donaldson_scalcurv_mm}, after~\cite{Donaldson_scalar}. We expand the details of the original proof, both for the reader's convenience and to fix the constant mentioned in Theorem~\ref{thm:Donaldson_scalcurv_mm}.

We have to show that~\eqref{eq:integral_form_mm} holds; in terms of the~$L^2$ product of functions,~\eqref{eq:integral_form_mm} can be restated as
\begin{equation}\label{eq:mm_pairing}
4\left\langle Q(A),f\right\rangle=\left\langle JA,P(f)\right\rangle
\end{equation}
where~$Q(A)$ is the variation of~$S(J)$ along a path~$J+\varepsilon A+O(\varepsilon^2)$ in~$\mscr{J}$, and~$P(f)=\m{L}_{X_f}J$.

For a given~$A\in T_J\mscr{J}$, we will often denote by~$J'=J+\varepsilon A$ the infinitesimal deformation of the complex structure~$J$ by~$A$; of course this should be thought of as a point of a path~$\gamma:[-\delta,\delta]\to\mscr{J}$ such that~$\gamma(0)=J$ and~$\dot{\gamma}(0)=A$, but it is notationally more convenient to use this notation and work ``to first order in~$\varepsilon$''.

\begin{lemma}\label{lemma:first_ord_identification}
Fix~$J\in\mscr{J}$,~$A\in T_J\mscr{J}$ and let~$J'=J+\varepsilon A$ be an infinitesimal deformation of the almost complex structure~$J$. Then the map
\begin{equation*}
\mu_A:X\mapsto X+\frac{\varepsilon}{2}JAX
\end{equation*}
defines, to first order in~$\varepsilon$, an isometry between~$\left(T^{1,0}_JM,g_J\right)$ and~$\left(T^{1,0}_{J'}M,g_{J'}\right)$.
\end{lemma}
\begin{proof}
Let~$X$ and~$Y$ be real vector fields. Then
\begin{equation*}
\begin{split}
g_{J'}&\!\left(X+\frac{\varepsilon}{2}JAX,Y+\frac{\varepsilon}{2}JAY\right)\!=\!\omega\!\left(X+\frac{\varepsilon}{2}JAX,(J+\varepsilon A)\!\left(Y+\frac{\varepsilon}{2}JAY\right)\!\right)=\\
&=\omega\left(X+\frac{\varepsilon}{2}JAX,JY+\frac{\varepsilon}{2}JJAY+\varepsilon AY\right)+O(\varepsilon^2)=\\
&=\omega\left(X+\frac{\varepsilon}{2}JAX,JY+\frac{\varepsilon}{2}AY\right)+O(\varepsilon^2)=\\
&=g_J\left(X,Y\right)+\frac{\varepsilon}{2}\left(\omega\left(JAX,JY\right)+\omega\left(X,AY\right)\right)+O(\varepsilon^2).
\end{split}
\end{equation*}
So to prove the isometry claim we just have to notice that
\begin{equation*}
\omega\left(JAX,JY\right)+\omega\left(X,AY\right)=\omega\left(AX,Y\right)+\omega\left(X,AY\right)=0.
\end{equation*}
Now, let~$X$ be a~$(1,0)$-vector with respect to~$J$. Then~$X+\frac{\varepsilon}{2}JAX$ is~$X-\frac{\I\varepsilon}{2}AX$, and applying~$J'$ we find
\begin{equation*}
\begin{split}
J'\left(X-\frac{\I\varepsilon}{2}AX\right)&=\left(J+\varepsilon A\right)\left(X-\frac{\I\varepsilon}{2}AX\right)=\I X-\frac{\I\varepsilon}{2}JAX+\varepsilon AX+O(\varepsilon^2)=\\
&=\I X+\frac{\varepsilon}{2}AX+O(\varepsilon^2)=\I\left(X-\frac{\I\varepsilon}{2}AX\right)+O(\varepsilon^2).
\end{split}
\end{equation*}
A similar computation shows also that if~$X\in T^{0,1}_JM$ then~$\mu_A(X)\in T^{0,1}_{J'}M$.
\end{proof}
The map~$\mu_A$ is also invertible, to first order in~$\varepsilon$:~$\mu_A\circ\mu_{-A}=\mathbbm{1}+O(\varepsilon^2)$. Then we have an identification between~$\ext{1,1}_J M$ and~$\ext{1,1}_{J'} M$, using~$\mu_A$. Indeed, by Lemma~\ref{lemma:first_ord_identification}
\begin{equation*}
\alpha\mapsto\mu_A^*\alpha:=\alpha\circ\mu^{-1}_A=\alpha\left(\mu_A^{-1}-,\mu_A^{-1}-\right)
\end{equation*}
maps~$(1,1)$-forms with respect to~$J$ to~$(1,1)$-forms with respect to~$J'$. Moreover,~$\alpha\mapsto\alpha\circ\mu^{-1}$ is also a first-order isometry. This can easily be checked in a system of local coordinates, once we get from Lemma~\ref{lemma:first_ord_identification} the identities
\begin{equation*}
\begin{split}
g(J')_{ij}(\mu_A)\indices{^i_k}(\mu_A)\indices{^j_l}=g(J)_{kl}\\
(\mu_A)\indices{^i_k}g(J)^{kl}=(\mu_A^{-1})\indices{^l_p}g(J')^{pi}.
\end{split}
\end{equation*}

These identifications will be necessary to compute the functional~$Q$ in~\eqref{eq:mm_pairing}. We decompose~$P$ as~$P_2\circ P_1$, where~$P_1$ sends a zero-average function~$H$ to its Hamiltonian vector field~$X_H$, and~$P_2$ sends a vector field~$X$ to the infinitesimal change in complex structure~$\m{L}_XJ$. We may also decompose~$Q$ as follows: given an infinitesimal deformation~$J'=J+\varepsilon A$ of~$J$, we may consider~$\ext{1,0}_J M$ and~$\ext{1,0}_{J'} M$ as the same bundle using the isometry of Lemma~\ref{lemma:first_ord_identification}. Then the two connections~$\nabla_J$ and~$\nabla_{J'}$ may be considered as metric connections on the same vector bundle~$E$, and as such they will differ by some~$\Xi\in\m{A}^1(\mrm{End}(E))$. The trace of~$\Xi$ is a purely imaginary~$1$-form on~$ M$, and it is the difference of the connections induced by~$\nabla_J$ and~$\nabla_{J'}$ on~$\ext{\mrm{top}}E$. Then if we let~$Q_1(A):=-\I\mrm{Trace}(\Xi)$ and define for any~$1$-form~$\psi$
\begin{equation*}
Q_2(\psi)\omega^n=n\mrm{d}\psi\wedge\omega^{n-1}
\end{equation*}
we have~$Q=Q_2\circ Q_1$. So we reformulate equation~\eqref{eq:mm_pairing} as
\begin{equation*}
4\,Q_2\circ Q_1=P_1^*\circ P_2^*\circ J
\end{equation*}
where~$*$ denotes the formal adjoint of an operator.
\begin{lemma}\label{lemma:Q2_aggiunto}
Let~$\flat:T M \to\m{A}^1M$ be the map sending a vector~$X$ to~$g_J(X,-)$. Then
\begin{equation*}
\flat\circ P_1=-Q_2^*.
\end{equation*}
\end{lemma}
\begin{proof}
Fix~$f\in\m{C}^\infty_0$ and~$\vartheta\in\m{A}^1M$. Then~$P_1(f)=J\,\mrm{grad}(f)$, where the gradient is computed with respect to the metric~$g_J$. If~$\sharp$ is the inverse of~$\flat$
\begin{equation*}
\left\langle\vartheta,P_1(f)^\flat\right\rangle=\left\langle\mrm{div}\left(J\,\vartheta^\sharp\right),f\right\rangle.
\end{equation*}
Moreover, for any vector field~$X$ we have
\begin{equation*}
\mrm{div}(X)\omega^n=\m{L}_X\omega^n=\mrm{d}\left(X\lrcorner\omega^n\right)=n\,\mrm{d}\left(X\lrcorner\omega\right)\wedge\omega^{n-1}
\end{equation*}
so for~$J\,\vartheta^\sharp$:
\begin{equation*}
\begin{split}
\mrm{div}\left(J\,\vartheta^\sharp\right)\omega^n&=n\,\mrm{d}\left((J\,\vartheta^\sharp)\lrcorner\omega\right)\wedge\omega^{n-1}=-n\,\mrm{d}\left(\omega(\vartheta^\sharp,J-)\right)\wedge\omega^{n-1}=\\
&=-n\,\mrm{d}\vartheta\wedge\omega^{n-1}=-Q_2(\vartheta)\omega^n.
\end{split}
\end{equation*}
This readily implies~$\left\langle\vartheta,P_1(f)^\flat\right\rangle=-\left\langle Q_2(\vartheta),f\right\rangle$.
\end{proof}
If we substitute the expression for~$P_1^*$ from Lemma~\ref{lemma:Q2_aggiunto} in~\eqref{eq:mm_pairing} we find that to prove~\ref{thm:Donaldson_scalcurv_mm} it is enough to show
\begin{equation*}
-4\,\sharp\circ Q_1=P_2^*\circ J.
\end{equation*}
We now turn to the other two operators,~$P_2$ and~$Q_1$.
\begin{lemma}\label{lemma_P2}
For any vector field~$X$
\begin{equation*}
P_2(X)=4\,\mrm{Im}\left(\nabla^{0,1}X^{1,0}-\frac{1}{4}N\left(X^{0,1},-\right)\right).
\end{equation*}
\end{lemma}
\begin{proof}
Since~$\m{L}_XJ$ is an element of~$T_J\mscr{J}$,~$P_2(X)$ is determined by its~$(1,0)$ part. We compute it by focusing on its transpose,~$P_2(X)^{1,0}:\ext{1,0}M\to\ext{0,1}M$. Fix a~$(1,0)$-form~$\alpha$. Then for any~$Y\in T^{0,1}M$ we have
\begin{equation*}
\begin{split}
\left(P_2(X)(\alpha)\right)(Y)&=\alpha\left(\m{L}_X(JY)-J\m{L}_XY\right)=-\I\alpha\left(\m{L}_XY-\I J\m{L}_XY\right)=\\
&=-2\I\alpha\left((\m{L}_XY)^{1,0}\right)=-2\I\alpha\left(\m{L}_XY\right)=-2\I\m{L}_X\alpha(Y)
\end{split}
\end{equation*}
since~$\alpha$ is of type~$(1,0)$ while~$Y$ is of type~$(0,1)$. This means that~$P_2(X)(\alpha)=-2\I\left(\m{L}_X\alpha\right)^{0,1}$. Using Lemma~\ref{lemma:decomposizione_d} and Cartan's formula we have
\begin{equation*}
\m{L}_X\alpha=X\lrcorner\left(\diff\alpha+\bdiff\alpha-\frac{1}{4}\alpha\circ N_J\right)+\diff\left(\alpha(X)\right)+\bdiff\left(\alpha(X)\right)
\end{equation*}
so its~$(0,1)$ part is
\begin{equation*}
\begin{split}
\left(\m{L}_X\alpha\right)^{0,1}&=-\bdiff\alpha\left(-,X^{1,0}\right)+\bdiff\left(\alpha(X^{1,0})\right)-\frac{1}{4}\alpha\left(N_J(X,-)\right)=\\
&=\alpha\left(\nabla^{0,1}X^{1,0}-\frac{1}{4}N_J(X^{0,1},-)\right).
\end{split}
\end{equation*}
For this last equality it is fundamental to use the Chern connection~$\nabla$ of~$\ext{1,0}_J\!\!M$ with respect to~$\bdiff_J$. In the end we have found that the~$(1,0)$ part of~$\frac{i}{2}P_2(X)$ is~$\nabla^{0,1}X^{1,0}-\frac{1}{4}N_J(X^{0,1},-)$; summing this with its conjugate we obtain the thesis.
\end{proof}
It is slightly more complicated to find an expression for~$Q_1$; we will use some preliminary lemmas. Consider a complex vector bundle~$E\to M$ with a Hermitian fibre metric~$h$ and a metric connection~$\nabla$. Then for each almost complex structure~$J$ on~$M$ we have the usual decomposition~$\nabla=\nabla^{1,0}_J+\nabla^{0,1}_J$. We want to study how an infinitesimal change~$J\mapsto J+\varepsilon A$ of the complex structure affects the other pieces in this picture. First, using the identification of~$T^{1,0}_JM$ and~$T^{1,0}_{J'}M$, we can compare the~$(0,1)$ parts of a connection.
\begin{lemma}\label{lemma:variazione_J}
Fix a metric connection~$\nabla$ on~$(E,h)$, and consider an infinitesimal change~$J\mapsto J'=J+\varepsilon A$ of the almost complex structure of~$ M$. Then the corresponding change in the~$(0,1)$ part of the connection is, to first order in~$\varepsilon$, 
\begin{equation*}
\nabla^{0,1}_{J'}=\nabla^{0,1}_J+\nabla^{1,0}_J\circ\frac{\varepsilon}{2}JA.
\end{equation*}
\end{lemma}
Of course, we can also change the~$(0,1)$ part of the connection while keeping the complex structure fixed. This implies a change of the metric connection itself.
\begin{lemma}\label{lemma:variazione_nabla}
Fix the almost complex structure~$J$. Under a change of the~$(0,1)$ part of the connection~$\nabla^{0,1}\mapsto\nabla^{0,1}+\sigma$ for some~$\sigma\in\m{A}^{0,1}(\mrm{End}(E))$, the connection~$\nabla$ changes by
\begin{equation*}
\nabla\mapsto\nabla+(\sigma-\sigma^*).
\end{equation*}
Here~$\sigma^*$ denotes the formal adjoint of~$\sigma$ with respect to the Hermitian metric.
\end{lemma}
\begin{proof}
This follows from how a connection is defined in terms of its~$(0,1)$-part and the metric, as in Lemma~\ref{lemma:Chern_conn}. It is enough to check that~${\sigma^*}\indices{^i_j}=h_{jk}\overline{\sigma\indices{^k_l}}h^{li}$.
\end{proof}

We use these two lemmas to help us study a more difficult problem: consider the infinitesimal deformation~$J'=J+\varepsilon A$ for some~$A\in T_J\mscr{J}$; then we have the two bundles~$\ext{1,0}_J M$ and~$\ext{1,0}_{J'} M$, each with a Hermitian metric on their fibres and a~$(0,1)$-operator~$\bar{\diff}_J$ and~$\bar{\diff}_{J'}$. On each of these bundle we have a connection~$\nabla_J$,~$\nabla_{J'}$ and we want to compute the difference in the scalar curvatures of these two connections. Notice that we can use the metric isomorphisms between~$\ext{1,0}_JM$ and~$\ext{1,0}_{J'}M$, and between~$\ext{1,1}_JM$ and~$\ext{1,1}_{J'}M$ (c.f. Lemma~\ref{lemma:first_ord_identification}), to think about~$\nabla_J$ and~$\nabla_{J'}$ as metric connections on the same bundle. In doing so our problem becomes to compute the variation of a metric connection under a change of both the complex structure and the~$(0,1)$-part of the connection.

Consider now the following diagram:
\begin{center}
\begin{tikzcd}
\m{A}^{1,0}_J M  \ar[r,dashed]{}{\Theta} \ar{d}{\mu_A^*} & \m{A}^{1,1}_J M \\
\m{A}^{1,0}_{J'} M  \ar{r}{\bar{\diff}_{J'}} & \m{A}^{1,1}_{J'} M  \ar[u,swap]{}{\mu_A^\transpose}
\end{tikzcd}
\end{center}
the two vertical arrows are first-order isomorphisms, and the dashed arrow (the composition of the other three) is, under our identification, the~$(0,1)$-part of the metric connection induced by the new complex structure~$J'$ on the bundle~$\ext{1,0}_J M$. A straightforward computation allows us to write the map~$\bdiff_{J'}$ in terms of~$J$ and the first-order deformation~$A$.
\begin{lemma}
Let~$\chi\in\m{A}^2_{\bb{C}} M$, and consider its decomposition into forms of type~$(p,q)$ according to~$J$:~$\chi=\chi^{2,0}+\chi^{1,1}+\chi^{0,2}$. Then its~$(1,1)$ component with respect to the~$J'$-decomposition is
\begin{equation*}
\chi^{1,1}_{J'}=\mu_A^*\chi^{1,1}+\left(\chi^{2,0}+\chi^{0,2}\right)\circ\mu_A-\left(\chi^{2,0}+\chi^{0,2}\right)+O(\varepsilon^2).
\end{equation*}
\end{lemma}
Let~$\alpha$ be a~$(1,0)$ form with respect to~$J$; using the previous Lemma we can compute~$\bdiff_{J'}\mu_A^*\alpha$:
\begin{equation*}
\begin{split}
\bdiff_{J'}\mu_A^*\alpha=&\mu_A^*(\mrm{d}\,\mu_A^*\alpha)^{1,1}+\left((\mrm{d}\,\mu_A^*\alpha)^{2,0}+(\mrm{d}\,\mu_A^*\alpha)^{0,2}\right)\circ\mu_A-\\
&-\left((\mrm{d}\,\mu_A^*\alpha)^{2,0}+(\mrm{d}\,\mu_A^*\alpha)^{0,2}\right)+O(\varepsilon^2).
\end{split}
\end{equation*}
Now, the decomposition of~$\mrm{d}\mu_A^*\alpha$ with respect to~$J$ is
\begin{equation*}
\begin{split}
\mrm{d}\mu_A^*\alpha&=\left(\diff+\bdiff-\frac{1}{4}N_J^\transpose\right)\left(\alpha\circ\mu_A^{-1}\right)=\\
&=\underbrace{\diff\alpha+\frac{1}{4}\alpha\circ\frac{\varepsilon}{2}JAN_J}_{(2,0)\text{-part}}+\underbrace{\bdiff\alpha-\diff\left(\alpha\circ\frac{\varepsilon}{2}JA\right)}_{(1,1)\text{-part}} \underbrace{-\bdiff\left(\alpha\circ\frac{\varepsilon}{2}JA\right)-\frac{1}{4}\alpha\circ N_J}_{(0,2)\text{-part}}
\end{split}
\end{equation*}
Then~$\Theta\alpha$ can be computed, again to first order in~$\varepsilon$:
\begin{equation}\label{eq:nuova_struttura}
\begin{split}
\Theta\alpha=&\bdiff_J\alpha-\diff\left(\alpha\circ\frac{\varepsilon}{2}JA\right)+\diff\alpha\left(\frac{\varepsilon}{2}JA-,-\right)+\diff\alpha\left(-,\frac{\varepsilon}{2}JA-\right)-\\
&-\frac{1}{4}\alpha\circ\left(N\left(\frac{\varepsilon}{2}JA-,-\right)+N\left(-,\frac{\varepsilon}{2}JA-\right)\right)
\end{split}
\end{equation}
It will be more convenient to let~$\nu:=\left(\frac{\varepsilon}{2}JA\right)^{1,0}$, and define for a~$2$-form~$\beta$
\begin{equation*}
\beta\circ\nu:=\beta(\nu-,-)+\beta(-,\nu-).
\end{equation*}
Then we can rewrite equation~\eqref{eq:nuova_struttura} as
\begin{equation}\label{eq:nuova_struttura_alt}
\Theta\alpha=\bdiff_J\alpha-\diff_J\left(\alpha\circ\nu\right)+\diff_J\alpha\circ\nu-\frac{1}{4}\alpha\circ N\circ\bar{\nu}
\end{equation}
We proceed to rewrite some terms of equation~\eqref{eq:nuova_struttura_alt} in a slightly different way. For any complex vector bundle~$E$ over~$ M$, a connection~$D$ on~$E$ induces a connection also on the complex vector bundle with the conjugate complex structure,~$\bar{E}$; the induced connection~$\bar{D}$ satisfies
\begin{equation*}
\bar{D}_X\sigma=\overline{D_{\bar{X}}\bar{\sigma}}
\end{equation*}
for any~$X\in T_{\bb{C}} M$ and~$\sigma\in\Gamma(\bar{E})$. Here~$\bar{\ }$ denotes both the usual conjugation and the natural map of vector bundles~$\bar{\ }:E\to\bar{E}$. It is then easy to check that~$\bar{D}^{1,0}=\overline{D^{0,1}}$. In particular the term~$\diff_J(\alpha\circ\nu)$ in equation~\eqref{eq:nuova_struttura_alt} is equal to~$\bar{\nabla}^{1,0}(\alpha\circ\nu)$, since~$\alpha\circ\nu$ is a~$(0,1)$-form. Using the Leibniz rule we extend~$\nabla$ also to a connection on~$\mrm{End}(TM)$, and so
\begin{equation*}
\diff(\alpha\circ\nu)=\bar{\nabla}^{1,0}(\alpha\circ\nu)=\alpha\circ\nabla^{1,0}\nu+\nabla^{1,0}\alpha\circ\nu
\end{equation*}
where the composition in~$\nabla^{1,0}\alpha\circ\nu$ indicates that~$\nu$ is acting on the second component of~$\nabla^{1,0}\alpha$. In other words, for any two vector fields~$X^{0,1}$ and~$Y^{1,0}$ we have
\begin{equation}\label{eq:diff_mu_alpha}
\begin{split}
\diff(\alpha\circ\nu)(X^{0,1},Y^{1,0})&=-\diff(\alpha\circ\nu)(Y^{1,0},X^{0,1})=\\
&=-\alpha((\nabla_{Y^{1,0}}\nu)X^{0,1})-\nabla_{Y^{1,0}}\alpha(\nu X^{0,1}).
\end{split}
\end{equation}
Lemma~\ref{lemma:torsion} allows us to rewrite this expression. For any two vector fields~$X,Y\in T^{1,0}_J M$ and any~$\alpha\in\m{A}^{1,0}_J M$ we have
\begin{equation*}
\alpha(\mrm{Tor}^{\nabla}(X,Y))=-\frac{1}{4}\alpha(N(X,Y))=0
\end{equation*}
since~$N$ sends~$X,Y$ to a~$(0,1)$-vector field. On the other hand, however
\begin{equation*}
\begin{split}
\alpha(\mrm{Tor}^{\nabla}(X,Y))&=\alpha(\nabla_XY-\nabla_XY-[X,Y])=\\
=&X(\alpha(Y))-\nabla_X\alpha(Y)-Y(\alpha(X))+\nabla_Y\alpha(X)-\alpha([X,Y])=\\
=&\diff\alpha(X,Y)-(\nabla_X\alpha(Y)-\nabla_Y\alpha(X))
\end{split}
\end{equation*}
and so for every~$X,Y\in T^{1,0}M$ we have
\begin{equation*}
\diff\alpha(X,Y)=\nabla_X\alpha(Y)-\nabla_Y\alpha(X).
\end{equation*}
In particular for~$\nu X^{0,1}$ and~$Y^{1,0}$ we have
\begin{equation}\label{eq:azione_antisimmetrizazione}
(\diff\alpha\circ\nu)(X^{0,1},Y^{1,0})=\diff\alpha(\nu X^{0,1},Y^{1,0})=\nabla_{\nu X^{0,1}}\alpha(Y^{1,0})-\nabla_{Y^{1,0}}\alpha(\nu X^{0,1}).
\end{equation}
Putting together equations~\eqref{eq:diff_mu_alpha} and~\eqref{eq:azione_antisimmetrizazione} we get
\begin{equation}\label{eq:manipolazione}
\left(\diff(\alpha\circ\nu)-(\diff\alpha)\circ\nu\right)(X^{0,1},Y^{1,0})=-\alpha((\nabla_{Y^{1,0}}\nu)X^{0,1})-\nabla_{\nu X^{0,1}}\alpha(Y^{1,0}).
\end{equation}
At this point we can rewrite equation~\eqref{eq:nuova_struttura_alt}:
\begin{equation}\label{eq:nuova_struttura_alt2}
\Theta\alpha=\bdiff_J\alpha+\alpha\circ\nabla^{1,0}\nu+\nabla^{1,0}_\nu\alpha-\frac{1}{4}\alpha\circ N\circ\bar{\nu}
\end{equation}
Notice that the term~$\bar{\diff}\alpha+\nabla^{1,0}_\nu\alpha$ in equation~\eqref{eq:nuova_struttura_alt2} is precisely how the~$(0,1)$-part of the connection would change if we fixed the connection~$\nabla$, c.f. Lemma~\ref{lemma:variazione_J}. Hence by Lemma~\ref{lemma:variazione_nabla} the variation in the connection is
\begin{equation}\label{eq:variazione_connessione}
(\nabla^{1,0}\nu)-\frac{1}{4}N\circ\bar{\nu}-\left[(\nabla^{1,0}\nu)-\frac{1}{4}N\circ\bar{\nu}\right]^*.
\end{equation}
Here, as before,~${}^*$ denotes adjointness with respect to the hermitian metric on the fibres. Since~$Q_1$ is the imaginary part of the variation of the induced metric on~$\ext{n,0}_J M$, to compute~$Q_1$ is enough to take the trace of equation~\eqref{eq:variazione_connessione}. Notice that, for any~$\sigma\in\m{A}^{0,1}(\mrm{End}(E))$, if we write its components as~$\sigma\indices{^j_k}$ then the adjoint is
\begin{equation*}
(\sigma^*)\indices{^i_j}=h_{jk}\overline{\sigma\indices{^k_l}}h^{li}
\end{equation*}
and so, by taking an~$h$-orthonormal frame for~$E$
\begin{equation*}
\mrm{Trace}(\sigma-\sigma^*)=\sigma\indices{^j_j}-\overline{\sigma\indices{^k_k}}=2\I\mrm{Im}(\mrm{Trace}(\sigma)).
\end{equation*}
Summing up, we finally obtain an expression for~$Q_1$.
\begin{lemma}
For~$A\in T_J\mscr{J}$, the deformation of the Chern connection of~$J$ on~$\ext{n,0}_JM$ is 
\begin{equation*}
Q_1(A)(X^{0,1})=2\mrm{Im}\left[\mrm{Trace}\left(\left(\nabla^{1,0}\nu\right)(X^{0,1})-\frac{1}{4}N\left(X^{0,1},\bar{\nu}-\right)\right)\right]
\end{equation*}
\end{lemma}
The next step is to rewrite the trace of~$\nabla^{1,0}\nu$. To do this, it is convenient to use the tensor~$q\in\ext{0,1} M \otimes\ext{0,1} M$ obtained from~$\nu$ by lowering its~$(1,0)$ index: for a local frame~$e_1,\dots,e_n$ of~$T^{1,0}M$ and dual frame~$\epsilon^1,\dots,\epsilon^n$ of~$\ext{1,0}M$,~$q$ is written as
\begin{equation*}
q_{\bar{a}\bar{b}}=g_{c\bar{a}}\nu\indices{^c_{\bar{b}}}.
\end{equation*}
Since~$\nabla$ is a metric connection and~$\nabla^{1,0}=\diff$ on~$\ext{0,1}M$
\begin{equation*}
\mrm{Trace}(\nabla^{1,0}\nu)=\nabla_c\nu\indices{^c_{\bar{b}}}=g^{c\bar{a}}\nabla_c q_{\bar{a}\bar{b}}=g^{c\bar{a}}\diff_c q_{\bar{a}\bar{b}}.
\end{equation*}
We can rewrite this using the operator
\begin{equation*}
\begin{split}
L:\m{A}^{0,1} M &\to\Gamma\left(\ext{1,0}\otimes\ext{0,1}\otimes\ext{0,1}\right)\\
\vartheta&\mapsto\omega\otimes\vartheta
\end{split}
\end{equation*}
that has, like in the integrable case, an adjoint~$\Lambda$ given by contraction with~$\omega$. So we have the nice expression
\begin{equation*}
\mrm{Trace}(\nabla^{1,0}\nu)=\I\Lambda\diff q.
\end{equation*}
In~\cite[Proposition~$16$]{Donaldson_YangMills4man}, Donaldson noted that some K\"ahler identities hold also in the non-integrable case:
\begin{prop}
On any almost K\"ahler manifold, the adjoints of~$\diff$,~$\bdiff$ respectively are
\begin{equation*}
\diff^*=\I[\Lambda,\bdiff],\quad \bdiff^*=-\I[\Lambda,\diff].
\end{equation*}
\end{prop}
Then we can rewrite~$Q_1$ as
\begin{equation*}
Q_1(A)=-2\mrm{Im}\left({\nabla^{0,1}}^*q\right)-\frac{1}{2}\mrm{Im}\left[\mrm{Tr}\left(N(-,\bar{\nu}-)\right)\right]
\end{equation*}
To finish the proof of Theorem~\ref{thm:Donaldson_scalcurv_mm} it is now enough to put together the various expressions we have already computed.
\begin{prop}
For any~$X\in\Gamma(TM)$ and~$A\in\m{A}^1(TM)$ we have
\begin{equation*}
\left\langle P_2(X),A\right\rangle=4\left\langle X,Q_1(JA)^\sharp\right\rangle
\end{equation*}
\end{prop}
\begin{proof}
Recall from Lemma~\ref{lemma_P2} that
\begin{equation*}
\begin{split}
\frac{1}{4}P_2(X)&=\mrm{Im}\left(\nabla^{0,1}X^{1,0}-\frac{1}{4}N\left(X^{0,1},-\right)\right)=\\
&=\mrm{Im}\left(\nabla^{0,1}X^{1,0}\right)+\frac{1}{4}\mrm{Im}\left(N\left(X^{1,0},-\right)\right)
\end{split}
\end{equation*}
while the previous computations give
\begin{equation*}
Q_1(JA)^\sharp=\mrm{Im}\left({\nabla^{0,1}}^*A^{1,0}\right)+\frac{1}{4}\mrm{Im}\left[\mrm{Tr}\left(N(-,A^{0,1}-)\right)^\sharp\right]
\end{equation*}
and now a direct computation shows the two operators are formal adjoints of each other.
\end{proof}

\subsection{The complexified action}\label{sec:complexification_cscK}

Having found a moment map for the action of~$\mscr{G}=\mrm{Ham}(M,\omega)$ on~$\mscr{J}$, the next step to mimic the finite-dimensional picture described in Section~\ref{sec:stability} would be to extend the action of~$\mscr{G}$ to an action of the complexified group~$\mscr{G}^c$; then, the stable~$\mscr{G}^c$-orbits of~$\mscr{J}$ should be in one-to-one correspondence with the set~$\mu^{-1}(0)/\mscr{G}$, i.e. in each stable orbit of the complexified action we should be able to find a compatible almost complex structure of constant scalar curvature.

There is a major issue with this approach: a complexification of~$\mscr{G}$ does not exist, see for example the discussion in~\cite[Remark~$35$]{Wang_Futaki} and~\cite[\S$1.3.3$]{GarciaFernandez_PHD}. However we can certainly complexify the \emph{infinitesimal} action of~$\mrm{Lie}(\mscr{G})\cong\m{C}^\infty_0(M,\bb{R})$ at a point~$J\in\mscr{J}$. Indeed, if~$h\in\m{C}^\infty_0$, the infinitesimal action of~$h$ at~$J$ is
\begin{equation*}
\hat{h}_J=\m{L}_{X_h}J\in T_J\mscr{J}.
\end{equation*}
Since~$\mscr{J}$ is a complex manifold, for any~$h\in\m{C}^\infty_0(M,\bb{R})$ we can define the infinitesimal action of~$\I h$ as
\begin{equation*}
\widehat{\I h}_J:=\bb{J}_J\hat{h}_J=J\m{L}_{X_h}J
\end{equation*}
and so the complexified Lie algebra~$\mrm{Lie}(\mscr{G})^{c}\cong\m{C}^\infty_0(M,\bb{C})$ defines a distribution~$\mscr{D}$ on~$\mscr{J}$ that plays the role of the \emph{infinitesimal complexified action} of~$\mscr{G}$ on~$\mscr{J}$:
\begin{equation*}
\begin{split}
\mscr{D}_J=&\set*{\hat{h}_J\tc h\in\m{C}^\infty(M,\bb{R})}\cup\set*{\hat{\I h}_J\tc h\in\m{C}^\infty(M,\bb{R})}=\\
=&\set*{\m{L}_{X_h}J\tc h\in\m{C}^\infty(M,\bb{R})}\cup\set*{J\m{L}_{X_h}J\tc h\in\m{C}^\infty(M,\bb{R})}.
\end{split}
\end{equation*}
If we are able to prove that~$\mscr{D}$ is an integrable distribution, then the leaves of the distribution can be considered as \emph{complexified orbits} of~$\mscr{G}\curvearrowright\mscr{J}$, even though~$\mscr{G}$ does not admit a complexification. This was shown by Donaldson, for integrable~$J$, in~\cite{Donaldson_scalar} and later in~\cite{Donaldson_SymmKahlerHam} (assuming~$H^1(M,\bb{R})=0$, for simplicity, but the argument can be generalized).

For integrable~$J$, consider the K\"ahler class~$\m{K}_J(\omega)$ of~$\omega$ with respect to the complex structure~$J$. Then, let~$\mscr{Y}_J$ be the set
\begin{equation*}
\mscr{Y}_J=\set*{(f,\omega')\tc f\in\mrm{Diff}_0(M),\ \omega'\in\m{K}_J(\omega)\mbox{ and }f^*\omega'=\omega}.
\end{equation*}
Then we have a~$\mrm{Symp}_0(M,\omega)$-bundle~$\mscr{Y}\to\m{K}_J(\omega)$ given by projection on the second component (the hypothesis~$H^1=0$ is required here to have~$\mrm{Symp}_0=\mrm{Ham}$, see Proposition~\ref{prop:Ham_properties}). Consider now the map
\begin{equation}\label{eq:foglia_integrale_J}
\begin{split}
\Phi_J:\mscr{Y}_J&\to\mscr{J}\\
(f,\omega')&\mapsto f^*\!J.
\end{split}
\end{equation}
\begin{prop}[\cite{Donaldson_SymmKahlerHam}]\label{prop:integral_leaf}
For every integrable compatible complex structure~$J\in\mscr{J}$, the image of~$\Phi_J$ is a leaf of the distribution~$\mscr{D}_J$.
\end{prop}
It is crucial here to assume that~$J$ is integrable, since~$N_J=0$ implies that~$J\m{L}_XJ=\m{L}_{JX}J$ for every vector field~$X$.
\begin{proof}
We'll just show that~$\Phi_J$ is surjective onto the complex part of~$\mscr{D}_J$, i.e.
\begin{equation*}
\mrm{Im}\,\mscr{D}_J=\set*{J\m{L}_{X_f}J\tc f\in\m{C}^\infty_0(M,\bb{R})}.
\end{equation*}
 Assume that~$\omega_\varphi=\omega+\I\diff\bdiff\varphi$ is a K\"ahler form, and consider the path~$\omega_t:=\omega_{t\varphi}$ in~$\m{K}_{J}(\omega)$. Let~$X_t$ be the time-dependent vector field~$X_t=-\frac{1}{2}\mrm{grad}_{g_t}(\varphi)$, and let~$f_t$ be the isotopy defined by~$X_t$, i.e.~$X_t=\diff_tf_t$ and~$f_0=\mathbbm{1}_M$. Notice that
\begin{equation*}
X_t=\frac{1}{2}JX_\varphi(\omega_t),
\end{equation*}
where~$X_\varphi(\omega_t)$ is the Hamiltonian vector field associated to~$\varphi_t$ under the form~$\omega_t$. Then
\begin{equation*}
X_t\lrcorner\omega_t=-\frac{1}{2}\omega_t(X_\varphi(\omega_t),J-)=\frac{1}{2}J\mrm{d}\varphi
\end{equation*}
and by~\cite[Proposition~$6.4$]{DaSilva_lectures_symp} we have
\begin{equation*}
\diff_t\Bigr|_{t=t_0}(f_t^*\omega_t)=f_{t_0}^*(\m{L}_{X_{t_0}}\omega_{t_0}+\I\diff\bdiff\varphi)=f_{t_0}^*(\frac{1}{2}\mrm{d}J\mrm{d}\varphi+\I\diff\bdiff\varphi)=0
\end{equation*}
so~$f_t^*\omega_t=\omega$. Hence we have a path~$(f_t,\omega_t)\in\mscr{Y}_J$, and if we compute the differential of~$\Phi_J$ along this path, again by~\cite[Proposition~$6.4$]{DaSilva_lectures_symp} we get
\begin{equation*}
\diff_t\Bigr|_{t=0}\Phi_J(f_t,\omega_t)=\diff_t\Bigr|_{t=0}(f_t^*J)=\m{L}_{X_0}J=\frac{1}{2}\m{L}_{JX_\varphi}J.\qedhere
\end{equation*}
\end{proof}

\subsubsection{The complexified equation}

So, this map from the~$\mscr{G}$-bundle~$\mscr{Y}$ to~$\mscr{J}_{\mrm{int}}$ defined by~$(f,\omega)\mapsto f^*J$ plays the role of the complex action~$\mscr{G}^{c}\curvearrowright\mscr{J}_{\mrm{int}}$. This can also be used to rephrase the cscK problem in a more familiar setting; the point of view we have adopted here in describing the cscK equation as a moment map equation for the action~$\mscr{G}\curvearrowright\mscr{J}$ is quite different from the constant scalar curvature problem that is usually considered in K\"ahler geometry, namely, finding K\"ahler metrics of constant scalar curvature inside a fixed K\"ahler class. Indeed, so far in the moment map equation
\begin{equation*}
S(\omega,J)-\hat{S}=0
\end{equation*}
the symplectic form is \emph{fixed}, while we let instead~$J$ vary in~$\mscr{J}_{\mrm{int}}$. However, the two points of view are closely related by the map~$\Phi:\mscr{Y}\to\mscr{J}_{\mrm{int}}$. 

\begin{rmk}
For any K\"ahler form~$\omega$ on~$ M$ and any almost complex structure~$J$ compatible with~$\omega$, write~$g(J,\omega)$ for the Riemannian metric defined by~$J$ and~$\omega$. Assume that~$J'=\varphi.J$ for some diffeomorphism~$\varphi$. Then 
\begin{equation*}
g(J',\omega)=\varphi^*(g(J,(\varphi^{-1})^*\omega))
\end{equation*}
and for the curvature we have
\begin{equation*}
S\left(g(J',\omega)\right)=\varphi^*\left(S\left(g(J,(\varphi^{-1})^*\omega)\right)\right).
\end{equation*}
\end{rmk}

We can apply these consideration to compute the moment map along the complexified orbit of a complex structure~$J\in\mscr{J}$. Let~$(f,\omega')\in\mscr{Y}$, and consider
\begin{equation*}
S(\Phi(f,\omega'))=S\left(g(f^*J,\omega)\right)=f^*\left(S\left(g(J,\omega'\right)\right)
\end{equation*}
so~$f^*J$ is a zero of the moment map if and only if~$g(J,\omega')$ is a metric of constant scalar curvature. In other words, looking for a solution to the moment map equation along a complexified orbit~$\mscr{G}^c.J$ is equivalent to finding a constant scalar curvature K\"ahler form in the K\"ahler class of~$\omega$ with respect to~$J$. By analogy with finite-dimensional GIT, Theorem~\ref{thm:moment_map} suggests that there should be some notion of stability for a K\"ahler class~$[\omega]$ such that in each stable orbit for the action of~$\mscr{G}^c$ there is exactly one element in the zero set of the moment map, i.e. a K\"ahler form in~$[\omega]$ of constant scalar curvature.

Of course these considerations are not sufficient on their own to prove any theorem, but they might serve as guiding principles for obtaining new results on the cscK problem, and they give a new perspective on many classical features of the cscK equation. For example Matsushima's criterion~\cite{Matsushima_obstruction}, the Futaki invariant~\cite{Futaki_obstruction} and Mabuchi's~$K$-energy can be almost directly defined starting from this interpretation of the cscK equation: see for example~\cite[Chapter~$1$]{GarciaFernandez_PHD} and~\cite{Wang_Futaki}, where it is shown how these objects arise from the Hamiltonian action.

\chapter{Hyperk\"ahler extension of the Donaldson-Fujiki moment map}\label{chap:HyperkahlerReduction}

\chaptertoc{}

\bigskip

This chapter contains the proofs of our main results, Theorem~\ref{thm:HKThmIntro} and Theorem~\ref{thm:HamiltonianHyperkIntro}. The general idea for the proof of Theorem~\ref{thm:HKThmIntro} is that~$\mscr{J}$ is described locally in terms of maps from~$M$ to the K\"ahler symmetric space~$\f{H}\cong\mrm{Sp}(2n)/U(n)$. As we have seen in Section~\ref{sec:J} the K\"ahler structure on~$\mscr{J}$ is ``pointwise'' induced from that of~$\f{H}$, using the fact that the~$\mrm{Sp}(2n)$-action on~$\f{H}$ is by holomorphic isometries and that the transition functions between two local descriptions of~$\mscr{J}$ act as symplectomorphisms. Roughly stated, the result of~\cite{Biquard_Gauduchon} is that the cotangent bundle of a K\"ahler symmetric space~$G/H$ admits a hyperk\"ahler structure that, crucially, is invariant under the action of~$G$ on~$G/H$. In our situation this suggests that we can induce on~$\cotJ$ a hyperk\"ahler structure locally by the identification of~$\mscr{J}$ with~$\m{C}^\infty(M,\f{H})$, and then checking that these local structures glue together to give a hyperk\"ahler metric on the whole~$\cotJ$.

Section~\ref{sec:BiquardGauduchon} gives a more detailed account of the results of Biquard and Gauduchon~\cite{Biquard_Gauduchon} for symmetric spaces of non-compact type, and in particular we find an explicit expression for the Biquard-Gauduchon hyperk\"ahler potential of the space~$\m{AC}^+$. The Section begins with a couple of technical results characterizing hyperk\"ahler manifolds that will be needed in the proof of Theorem~\ref{thm:HKThmIntro}.

In Section~\ref{sec:cotJ_hyperk} we prove Theorem~\ref{thm:HKThmIntro} and Theorem~\ref{thm:HamiltonianHyperkIntro} (in Section~\ref{sec:moment_map}), and then we find an explicit expression for the moment map equations that arise from the Hamiltonian action. This is quite complicated for the real moment map equation, and leads to the Hitchin-cscK system~\eqref{eq:HcscK_system}. The discussion of the Hamiltonian action then continues in Section~\ref{sec:complexification}, where we show how to formally complexify the action and we consider the moment map equations under this complexification. The final result is not as satisfactory as one would have hoped from the classical case of the cscK equation, and leads to a complexification of the system just in a more formal sense, see the discussion around equation~\eqref{eq:HCSCK_complex_relaxed}.

We close the chapter considering the moment map equations in complex dimension~$1$ and~$2$, in Section~\ref{sec:curves_surfaces}. For curves we show how our approach is essentially equivalent to that in~\cite{Donaldson_hyperkahler}, that instead found a hyperk\"ahler metric on the cotangent space of Poincaré's upper half plane inspired by the existence result of hyperk\"ahler metrics due to Feix~\cite{Feix_hyperkahler}. The first objects for which our approach gives a novel result are complex surfaces, and in Section~\ref{sec:complex_surface} we find an expression of the HcscK system equivalent to~\eqref{eq:HCSCK_original} but slightly simpler to study analytically. Examples of solutions to the HcscK system on surfaces will be given in Chapter~\ref{chap:examples}.

\section{A result of Biquard and Gauduchon}\label{sec:BiquardGauduchon}

\subsection{Characterisations of hyperk\"ahler manifolds}

\begin{definition}\label{def:hyperkahler}
Let~$(M,g)$ be a Riemannian manifold, and let~$I$,~$J$ be two almost complex structures on~$M$ such that
\begin{enumerate}
\item~$IJ=-JI$;
\item~$g(I-,I-)=g(J-,J-)=g(-,-)$;
\item~$\forall x\in M,\,\forall v\in T_xM\quad g(Iv,Jv)=0$.
\end{enumerate}
Then~$(M,g,I,J)$ is a \emph{hyperk\"ahler manifold} if~$(M,g,I)$ and~$(M,g,J)$ are both K\"ahler.
\end{definition}
In this case, if we let~$K:=IJ$ then~$I$,~$J$ and~$K$ satisfy the usual quaternionic relations, giving a quaternionic structure on~$TM$. Moreover, for any~$\bm{u}\in\bb{S}^2$ also~$(M,g,u_1I+u_2J+u_3K)$ is K\"ahler. This is what prompted Calabi~\cite{Calabi_hyperkahler} to call these manifolds \emph{hyper}-K\"ahler. We refer the reader to~\cite{Hitchin_hyperkahler} for a general introduction to hyperk\"ahler manifolds.

The standard notation is to call~$\omega_1$,~$\omega_2$ and~$\omega_3$ (or~$\omega_I$,~$\omega_J$ and~$\omega_K$) the three~$2$-forms defined respectively by~$g\circ I$,~$g\circ J$ and~$g\circ K$. Consider also the complex-valued~$2$-form~$\omega_c:=\omega_2+\I\omega_3$; an important remark is that~$\omega_c$ is a symplectic form on~$M$ of type~$(2,0)$ relatively to the complex structure~$I$. Symplectic forms with these properties are called \emph{holomorphic-symplectic forms}.

The following Lemma gives us a useful criterion to prove that some structures are hyperk\"ahler.
\begin{lemma}[Lemma~$6.8$ in~\cite{Hitchin_self_duality}]\label{lemma:closed_hyperkahler}
Let~$(M,g)$ be a Riemannian manifold, and assume that~$I,J$ are almost complex structures on~$M$ satisfying conditions~$1,2,3$ of Definition~\ref{def:hyperkahler}. Then~$(M,g,I,J)$ is hyperk\"ahler if and only if
\begin{equation*}
\mrm{d}\omega_1=\mrm{d}\omega_2=\mrm{d}\omega_3=0.
\end{equation*}
\end{lemma}
In other words, the three forms being closed is enough to ensure the integrability of~$I$,~$J$ and~$K$. We remark that this conditions follows from an algebraic manipulation of the Newlander-Nirenberg criterion, so it holds also in the infinite-dimensional setting -- guaranteeing at least the \emph{formal} integrability of the complex structures.

To construct the hyperk\"ahler structure on~$\cotJ$ we will need an additional criterion, taken from the discussion in~\cite{Biquard_Gauduchon}.
\begin{lemma}\label{lemma:hyperekahler_criterio}
Let~$(M,g)$ be a Riemannian manifold, and let~$I$ be an almost complex structure on~$M$, compatible with~$g$. Assume also that there is a~$(2,0)$ symplectic form on~$M$,~$\omega_c$. Then we can always define a tensor~$J$ on~$M$ by the condition~$g(J-,-)=\mrm{Re}\,{\omega_c}(-,-)$. If
\begin{enumerate}
\item~$\mrm{d}\omega_1=0$, for~$\omega_1=g(I-,-)$
\item~$J^2=-\mathbbm{1}$
\end{enumerate}
then~$(M,g,I,J)$ is a hyperk\"ahler manifold, and the three~$2$-forms defined by~$g$ and~$I$,~$J$ and~$K=IJ$ are, respectively,~$\omega_1$,~$\mrm{Re}\,\omega_c$ and~$\mrm{Im}\,\omega_c$.
\end{lemma}

\begin{proof}
First of all notice that~$\omega_2$ is closed, since~$\omega_c$ is closed and by definition~$\omega_2=\mrm{Re}\,\omega_c$. We proceed to check that~$g$,~$I$ and~$J$ satisfy the algebraic identities of Definition~\ref{def:hyperkahler}.

To check the compatibility of~$J$ with~$g$, fix two tangent vectors~$v$,~$w$. Then we have, using the (anti-)symmetries of~$g$ and~$\omega_c$ 
\begin{equation*}
g(Jv,Jw)=\mrm{Re}(\omega_c(v,Jw))=-\mrm{Re}(\omega_c(Jw,v))=-g(JJw,v)=g(v,w).
\end{equation*}
The anti-commutativity of~$I$ and~$J$ is equivalent to~$g(IJv+JIv,w)=0$ for every pair of tangent vectors~$v$,~$w$. From the definition of~$J$ we have
\begin{equation*}
\begin{split}
g(IJv+JIv,w)=&-g(Jv,Iw)+g(JIv,w)=\\
=&-\mrm{Re}(\omega_c(v,Iw))+\mrm{Re}(\omega_c(Iv,w))
\end{split}
\end{equation*}
and since~$\omega_c$ is of type~$(2,0)$ with respect to~$I$
\begin{equation*}
-\mrm{Re}(\omega_c(v,Iw))+\mrm{Re}(\omega_c(Iv,w))=-\mrm{Re}(\I\omega_c(v,w))+\mrm{Re}(\I\omega_c(v,w))=0.
\end{equation*}
From the definition of~$J$ it is also easy to check that~$g(Iv,Jv)=0$ for any tangent vector~$v$:
\begin{equation*}
g(Iv,Jv)=\mrm{Re}\omega_c(v,Iv)=\mrm{Re}\left(\I\,\omega_c(v,v)\right)=0.
\end{equation*}
By Lemma~\ref{lemma:closed_hyperkahler}, to conclude the proof we should check that, if~$K:=IJ$ and~$\omega_3:=g(K-,-)$, we have~$\mrm{d}\omega_3=0$. Since~$\omega_c$ is of type~$(2,0)$ with respect to~$I$ we have, from the definition of~$J$
\begin{equation*}
\begin{split}
g(Kv,w)&=g(IJv,w)=-g(JIv,w)=-\mrm{Re}(\omega_c(Iv,w))=\mrm{Im}(\omega_c(v,w))
\end{split}
\end{equation*}
so the closedness of~$\omega_3$ follows from that of~$\omega_c$.
\end{proof}
It is important to highlight the fact that the proof of Lemma~\ref{lemma:hyperekahler_criterio} is purely algebraic, provided that~$\omega_c$ and~$\omega_1$ are closed; we do not need to resort to computations in local coordinates. Hence, this criterion for checking the hyperk\"ahler condition also holds in the infinite-dimensional setting, where we intend to apply it in Section~\ref{sec:cotJ_hyperk}.

\subsection{Hyperk\"ahler structures on symmetric spaces}

We recall the construction of Biquard and Gauduchon in~\cite{Biquard_Gauduchon} of a hyperk\"ahler metric on the cotangent bundle of any Hermitian symmetric space~$\Sigma=G/H$. Assume that~$\Sigma$ has a complex structure~$I$ and a Hermitian metric~$h$. For~$x\in\Sigma$ we have an identification of~${T^{1,0}_x}^*\Sigma$ and~$T_x\Sigma$ given by taking the metric dual of the real part of~$\xi\in{T^{1,0}_x}^*\Sigma$. Under this identification, for every~$\xi\in{T^{1,0}}^*_x\Sigma$, we have an endomorphism~$IR(I\xi,\xi)$ of~$T_x\Sigma$ associated to the Riemann curvature tensor~$R$. Since this is self-adjoint we can consider its spectral functions; we are interested in particular in the function~$f:\bb{R}_{>0}\to\bb{R}$ defined by
\begin{equation}\label{eq:funzione_f}
f(x):=\frac{1}{x}\left(\sqrt{1+x}-1-\log\frac{1+\sqrt{1+x}}{2}\right).
\end{equation}
\begin{thm}[\cite{Biquard_Gauduchon}]\label{thm:BiquardGauduchon}
Let~$\left(\Sigma=G/H,I,h\right)$ be a Hermitian symmetric space of compact type, and let~$\omega_c$ be the canonical symplectic form on~${T^{1,0}}^*\Sigma$. Then there is a unique~$G$-invariant hyperk\"ahler metric on~$({T^{1,0}}^*\Sigma,I,\omega_c)$ whose restriction to the zero-section of~${T^{1,0}}^*\Sigma$ coincides with the Hermitian metric of~$\Sigma$.

Moreover, we have an explicit expression for this metric: if we identify~$T^*\Sigma$ and~$T\Sigma$ using the metric on the base, the K\"ahler form is given by~$\omega_I=\pi^*\omega_{\Sigma}+\mrm{d}\mrm{d}^c\rho$, where~$\rho$ is the function on~$T\Sigma$ defined by
\begin{equation}\label{eq:rho_def}
\rho(x,\xi)=h_x\left(f(-IR(I\xi,\xi))\xi,\xi\right).
\end{equation}%
\nomenclature[rho]{$\rho$}{the Biquard-Gauduchon function of a Hermitian symmetric space}%
Here~$f$ is the function defined by~\eqref{eq:funzione_f}, evaluated on the self-adjoint endomorphism~$-IR(I\xi,\xi)$.

If instead~$\Sigma$ is of non-compact type, the same statement holds in an open neighbourhood~$N\subseteq {T^{1,0}}^*\Sigma$ of the zero section. This neighbourhood is the set of all~$\xi$ such that the modulus of the eigenvalues of~$-IR(I\xi,\xi)$ is less than~$1$.
\end{thm}

The quotient~$\mrm{Sp}(2n)/\mrm{U}(n)$ considered in Section~\ref{sec:Siegel} is a K\"ahler symmetric space of non-compact type, to which we can apply Theorem~\ref{thm:BiquardGauduchon}: the space~${T^{1,0}}^*(\mrm{Sp}(2n)/\mrm{U}(n))$ has a hyperk\"ahler metric, at least in a neighbourhood of the zero section. From Section~\ref{sec:Siegel} we know that~$\m{AC}^+$ is diffeomorphic to~$\mrm{Sp}(2n)/U(n)$, and the K\"ahler structure on~$\m{AC}^+$ is induced from the one of~$\mrm{Sp}(2n)/\mrm{U}(n)$ using this isomorphism. Then we can also carry the hyperk\"ahler structure of~${T^{1,0}}^*(\mrm{Sp}(2n)/\mrm{U}(n))$ to~${T^{1,0}}^*\!\m{AC}^+$.

Let's denote by~$(g,I,\omega)$ the K\"ahler structure of~$\m{AC}^+$; we also denote by~$I$ the complex structure on~${T^{1,0}}^*\!\m{AC}^+$, and we let~$\theta$ %
be the canonical symplectic form, the form denoted by~$\omega_c$ in Theorem~\ref{thm:BiquardGauduchon}. Theorem~\ref{thm:BiquardGauduchon} guarantees that~$\hat{g}:=\pi^*\omega+\mrm{d}\mrm{d}^c\rho$ is a hyperk\"ahler metric on~$T^*\!\m{AC}^+$.

\begin{rmk}
Biquard and Gauduchon consider the full cotangent bundle; for notational reasons, for us it will be more convenient to just consider the \emph{holomorphic} cotangent bundle of~$\m{AC}^+$ and~$\mscr{J}$, but this won't cause issues, thanks to the usual canonical identifications of the two. Moreover, Biquard and Gauduchon in~\cite{Biquard_Gauduchon} use for the curvature tensor the convention~$R(X,Y)=\nabla_{[X,Y]}-[\nabla_X,\nabla_Y]$, rather than the more usual~$R(X,Y)=[\nabla_X,\nabla_Y]-\nabla_{[X,Y]}$, that is the one we are going to use. This is why there is a minus sign in equation~\eqref{eq:rho_def}.
\end{rmk}

\subsubsection{The Biquard-Gauduchon function for~$\m{AC}^+$}\label{sec:BiquardGauduchon_mAC+}

We have to compute the Biquard-Gauduchon function~$\rho$ of~$\m{AC}^+$. For a fixed vector field~$A$ on~$\m{AC}^+$ we consider the endomorphism~$\Xi(A)$ of~$T\m{AC}^+$ defined by
\begin{equation*}
\Xi(A):B\mapsto -\bb{J}\left(R(\bb{J}A,A)(B)\right).
\end{equation*}
By Proposition~\ref{prop:curvatura_AC}, at a point~$J\in\m{AC}^+$,~$\Xi(A)$ can be written as:
\begin{equation}\label{eq:da_diagonalizzare}
\Xi_{J}(A)(B)=-J\left(-\frac{1}{4}\Big[[J\,A,A],B\Big]\right)=-\frac{1}{2}\left(A^2B+B\,A^2\right).
\end{equation}
\begin{rmk}
The~$\bb{C}$-linear extension of~$J\in\m{AC}^+\subset\mrm{End}(\bb{R}^{2n})$ to the vector space~$\bb{C}^{2n}$ can always be diagonalized. We decompose~$\bb{C}^{2n}$ as~$V'\oplus V''$, where~$V'$ and~$V''$ are the eigenspaces for~$J$ relative to the eigenvalues~$\I$ and~$-\I$ respectively. The positive-definite matrix~$\Omega_0 J$ induces on~$V'$ a Hermitian product, that we denote by~$\beta$. Then any~$A\in T_J\m{AC}^+$ is written as~$A=\begin{pmatrix}0 & A' \\ A'' & 0\end{pmatrix}$, with respect to this decomposition of~$\bb{C}^{2n}$. The two complex matrices~$A'$,~$A''$ satisfy~$A''=\overline{A'}$. Moreover,~$A'=\beta^{-1}\sigma(A)$ for a symmetric complex matrix~$\sigma(A)$, and these two properties characterize the tangent space of~$\m{AC}^+$ at a point~$J$.
\end{rmk}
\begin{prop}\label{prop:BiquardGauduchonAC+}
The Biquard-Gauduchon function of~$\m{AC}^+$ is
\begin{equation*}
\rho(J,A)=\sum_{i=1}^n\delta(i)f(-\delta(i))=\sum_{i=1}^n 1-\sqrt{1-\delta(i)}+\log\frac{1+\sqrt{1-\delta(i)}}{2}
\end{equation*}
where~$\delta(i)$ are the eigenvalues of~$A' A''$, that are all real and non-negative.
\end{prop}
\begin{proof}
Since~$\beta$ is a Hermitian matrix, there is a unitary matrix~$S$ such that~$\beta^{-1}=\bar{S}^\transpose\Lambda S$, for a diagonal matrix~$\Lambda$ with real, positive eigenvalues. Then we can decompose~$A'$ as
\begin{equation*}
A'=\bar{S}^\transpose\Lambda S\sigma(A)=\bar{S}^\transpose\Lambda^{\frac{1}{2}}\Lambda^{\frac{1}{2}} S\sigma(A) S^\transpose\Lambda^{\frac{1}{2}}\Lambda^{-\frac{1}{2}}\bar{S}.
\end{equation*}
The matrix~$R:=\Lambda^{\frac{1}{2}} S\sigma(A) S^\transpose\Lambda^{\frac{1}{2}}$ is symmetric, and the Takagi-Autonne factorization (see~\cite{Takagi_Factor} and~\cite{Autonne_Factor}) tells us that there is a unitary matrix~$U$ such that~$R=U D U^\transpose$ for a real diagonal matrix~$D$. Let~$\Delta:=D^2$. Then we have
\begin{equation*}
A'=\bar{S}^\transpose\Lambda^{\frac{1}{2}}U D U^\transpose \Lambda^{-\frac{1}{2}}\bar{S}
\end{equation*}
and for the product matrix~$A'A''$
\begin{equation*}
A'A''=\bar{S}^\transpose\Lambda^{\frac{1}{2}}U \Delta \bar{U}^\transpose \Lambda^{-\frac{1}{2}}S.
\end{equation*}
This shows that the eigenvalues of~$A'A''$ are the diagonal entries of~$\Delta$, and in particular they are all real and non-negative. From now on let~$\Delta=\mrm{diag}\left(\delta(1),\dots,\delta(n)\right)$, and let~$Q:=\bar{S}^\transpose\Lambda^{\frac{1}{2}}U$, so that~$A'A''=Q\Delta Q^{-1}$. For the map~$\Xi$ we find
\begin{equation*}
\begin{split}
A^2B+BA^2=&\begin{pmatrix}0 & Q\\ \bar{Q} & 0\end{pmatrix}
\begin{pmatrix}\Delta & 0\\ 0 &\Delta\end{pmatrix}
\begin{pmatrix}0 & \bar{Q}^{-1}\\ Q^{-1} & 0\end{pmatrix}
\begin{pmatrix}0& B'\\ B''&0\end{pmatrix}+\\
&+\begin{pmatrix}0& B'\\ B''&0\end{pmatrix}
\begin{pmatrix}0 & Q\\ \bar{Q} & 0\end{pmatrix}
\begin{pmatrix}\Delta & 0\\ 0 &\Delta\end{pmatrix}
\begin{pmatrix}0 & \bar{Q}^{-1}\\ Q^{-1} & 0\end{pmatrix}.
\end{split}
\end{equation*}
Notice that~$P(B):=\bar{Q}^{-1}B''Q$ is a symmetric matrix:
\begin{equation*}
P(B)=U^\transpose\Lambda^{-\frac{1}{2}}\bar{S}
S^\transpose\Lambda\bar{S}\sigma(B)
\bar{S}^\transpose\Lambda^{\frac{1}{2}}U=
U^\transpose\Lambda^{\frac{1}{2}}\bar{S}\,\sigma(B)\,\bar{S}^\transpose\Lambda^{\frac{1}{2}}U.
\end{equation*}
The linear map~$\Xi(B)$ has a much simpler expression in these new coordinates for~$T_J\m{AC}^+$:
\begin{equation*}
\Xi(B)=-\frac{1}{2}\!
\begin{pmatrix}0 &Q\\ \bar{Q} &0\end{pmatrix}\!\!
\begin{pmatrix}0 &\Delta P(B)+P(B)\Delta\\ \Delta\overline{P(B)}+\overline{P(B)}\Delta &0\end{pmatrix}\!\!
\begin{pmatrix}0 &\bar{Q}^{-1}\\ Q^{-1} &0\end{pmatrix}
\end{equation*}
and in fact the map~$P\mapsto \Delta P+P\Delta$ on the space of complex symmetric matrices is quite easy to diagonalize: consider the matrices~$E_{ij}$ for~$i\leq j$ defined by
\begin{equation*}
\left(E_{ij}\right)_{pq}:=\frac{1}{2}\left(\delta_{ip}\delta_{jq}+\delta_{iq}\delta_{jp}\right)
\end{equation*}
then~$\Delta E_{ij}+E_{ij}\Delta=\left(\delta(i)+\delta(j)\right)E_{ij}$. We are almost ready to compute the Biquard-Gauduchon function; notice first that~$P(A)=D$ is a real diagonal matrix, so that in the basis~$E_{ij}$ of the space of complex symmetric matrices~$P(A)=\sum_{i=1}^nD_{ii}E_{ii}$, and recall that~$D_{ii}^2=\delta(i)$.
\begin{equation}\label{eq:rhoAC_acc}
\begin{gathered}
\frac{1}{2}\mrm{Tr}\left(f(\Xi(A))(A)\!\cdot\!A\right)=\\
=\frac{1}{2}\mrm{Tr}\left(\begin{pmatrix}
0 & \sum_{i=1}^nf(-\delta(i))D_{ii}E_{ii}
\\ \sum_{i=1}^nf(-\delta(i))D_{ii}E_{ii} & 0
\end{pmatrix}\!\!
\begin{pmatrix}
0 & P(A) \\ \overline{P(A)} & 0
\end{pmatrix}
\right)=\\
=\mrm{Tr}\left(\left(\sum_{i=1}^nf(-\delta(i))D_{ii}E_{ii}\right)\left(\sum_{i=1}^nD_{ii}E_{ii}\right)\right).
\end{gathered}
\end{equation}
To finish the computation use that~$E_{ij}$ is an orthogonal basis with respect to the trace, and in particular
\begin{equation*}
\mrm{Tr}(E_{ii}E_{jj})=\sum_{k,l}\delta_{ik}\delta_{il}\delta_{jk}\delta_{jl}=\sum_l\delta_{il}\delta_{ij}\delta_{jl}=\delta_{ij}.
\end{equation*}
We can conclude the proof from the computation in~\eqref{eq:rhoAC_acc}
\begin{equation*}
\rho(J,A)=\mrm{Tr}\left(\left(\sum_{i=1}^nf(-\delta(i))D_{ii}E_{ii}\right)\left(\sum_{i=1}^nD_{ii}E_{ii}\right)\right)=\sum_{i=1}^nf(-\delta(i))\delta(i).\qedhere
\end{equation*}
\end{proof}

The hyperk\"ahler metric of Theorem~\ref{thm:BiquardGauduchon} does not depend on~$\rho$, but rather on its differential~$\mrm{d}\rho$. To compute explicitly the hyperk\"ahler moment map in Section~\ref{sec:real_momentmap} it will be useful to have an expression for the derivative of~$\rho$ along a path~$(J_t,A_t)$ in~$T\!\m{AC}^+$. To this end we should first compute the differential of the function
\begin{equation*}
A\mapsto\set*{\delta_i(A'A'')\tc 1\leq i\leq n}
\end{equation*}
that assigns to~$A$ the set of eigenvalues of~$A'A''$, where the type decomposition of~$A$ is computed with respect to~$J$. This is potentially problematic, mainly because the eigenvalues are smooth functions of the matrix if and only if they are all distinct. However, it turns out that~$\rho$ is a differentiable function of~$(J,A)$, even if the eigenvalues of~$A'A''$ are not all distinct.
\begin{lemma}\label{lemma:differenziale_BiquardGauduchon}
Let~$(J_t,A_t)$ be a path in~$T\!\m{AC}^+(2n)$. Then
\begin{equation*}
\diff_t\rho(J_t,A_t)=\mrm{Tr}\left(A\dot{A}\,\sum_{i=1}^n\frac{1}{2}\left(1+\sqrt{1-\delta(i)}\right)^{-1}\prod_{j\not=i}\frac{\delta(j)\mathbbm{1}-A^2}{\delta(j)-\delta(i)}\right)
\end{equation*}
where~$\delta(i)$ are as in Proposition~\ref{prop:BiquardGauduchonAC+}.
\end{lemma}
\begin{proof}
This is a consequence of Theorem~$3$ in~\cite{Magnus_differential_eigenvalues}: noting that the eigenvalues of~$A'A''$ are real (and equal to the eigenvalues of~$A''A'$) from~\cite[Theorem~$3$]{Magnus_differential_eigenvalues} we find, assuming that the~$\delta(i)$s are all distinct,
\begin{equation*}
\begin{gathered}
\diff_t\delta_t(i)=\frac{1}{2}\mrm{Tr}\left(\!
\diff_t\!\left(A'_tA''_t\right)\prod_{\stackrel{j=1}{j\not=i}}^n\frac{\delta(j)\mathbbm{1}-A'A''}{\delta(j)-\delta(i)}+\diff_t\!\left(A''_tA'_t\right)\prod_{\stackrel{j=1}{j\not=i}}^n\frac{\delta(j)\mathbbm{1}-A''A'}{\delta(j)-\delta(i)}\!\right)=\\
=\frac{1}{2}\mrm{Tr}\left(\diff_t\left(A^2_t\right)
\begin{pmatrix}
\prod_{\stackrel{j=1}{j\not=i}}^n\frac{\delta(j)\mathbbm{1}-A'A''}{\delta(j)-\delta(i)} & 0\\
0 & \prod_{\stackrel{j=1}{j\not=i}}^n\frac{\delta(j)\mathbbm{1}-A''A'}{\delta(j)-\delta(i)}
\end{pmatrix}
\right)=\\
=\mrm{Tr}\left(A\,\dot{A}\prod_{j=1,\,j\not=i}^n\frac{\delta(j)\mathbbm{1}-A^2}{\delta(j)-\delta(i)}\right).
\end{gathered}
\end{equation*}
To compute the differential of the Biquard-Gauduchon function along the path~$(J_t,A_t)$ now it is enough to compose this expression for~$\diff_t\delta(i)$ with
\begin{equation*}
D\rho=\sum_i\frac{D\delta(i)}{2\left(1+\sqrt{1-\delta(i)}\right)}.
\end{equation*}
It is easy to check that this expression is well-defined even if~$\delta(i)=\delta(j)$ for some~$i$,~$j$.
\end{proof}
\begin{rmk}\label{rmk:traccia_tripla}
If~$A$,~$B$ and~$C$ are elements of~$T_J\m{AC}^+$ then for every~$i\in\set{1,\dots,n}$ this trace vanishes:
\begin{equation*}
\mrm{Tr}\Big(A\,B\,C\prod_{j\not=i}\left(\delta(j)\mathbbm{1}-A^2\right)\Big).
\end{equation*}
Indeed,~$A$,~$B$ and~$C$ anti-commute with~$J$, while~$Q=\prod_{j\not=i}\left(\delta(j)\mathbbm{1}-A^2\right)$ commutes with~$J$. Then
\begin{equation*}
\mrm{Tr}\left(A\,B\,C\,Q\right)=-\mrm{Tr}\left(J\,A\,B\,C\,Q\,J\right)=-\mrm{Tr}\left(A\,B\,C\,Q\right).
\end{equation*}
We will need this vanishing when computing the moment map equations for our hyperk\"ahler reduction in Section~\ref{sec:real_momentmap}.
\end{rmk}

\subsection{The cotangent space of complex structures}

The main object of interest in this chapter is the total space of the holomorphic cotangent bundle of~$\mscr{J}$; in order to use the results of Biquard and Gauduchon to obtain a hyperk\"ahler structure on this space it will be useful to describe it as the space of sections of a bundle over~$M$, with fibres isomorphic to~${T^{1,0}}^*\!\m{AC}^+$. Once we have such a description of~${T^{1,0}}^*\!\!\!\mscr{J}$ we will use it to induce a hyperk\"ahler structure, similarly to what was done in Section~\ref{sec:J} to define a K\"ahler structure on~$\mscr{J}$. The details of this construction will be given in Section~\ref{sec:cotJ_hyperk}.

We first briefly describe the real cotangent bundle of~$\mscr{J}$. From the description of~$\mscr{J}$ as the space of sections of~$\m{E}$ we have
\begin{equation*}
T_J\mscr{J}=\Gamma(M,J^*(\mrm{Vert}\,\m{E}))\mbox{ and }T^*_J\mscr{J}=\Gamma(M,J^*(\mrm{Vert}\,\m{E}^*)).
\end{equation*}
A more explicit description can be obtained by locally trivializing the bundle in a system of Darboux coordinates and identifying the fibres of~$\m{E}$ with~$\m{AC}^+$:
\begin{equation*}
T^*_J\!\mscr{J}=\set*{\alpha\in\mrm{End}(T^*M)\tc J^\transpose\circ\alpha+\alpha\circ J^\transpose=0\mbox{ and } g_J(\alpha-,-)\mbox{ is symmetric}}
\end{equation*}
Here~$J^\intercal:T^*M\to T^*M$ denotes the transpose of~$J:TM\to TM$. In other words, every element~$\alpha\in T^*_J\!\mscr{J}$ is a map from~$\bb{R}^{2n}$ to~$T^*_J\m{AC}^+$, once we fix a system of Darboux coordinates. Hence we have a ``pointwise'' description of the cotangent space of~$\mscr{J}$:~$T^*_J\m{AC}^+$ is the space of all~$\alpha:(\bb{R}^{2n})^\vee\to(\bb{R}^{2n})^\vee$ such that 
\begin{equation*}
\begin{split}
J\alpha^\transpose+&\alpha^\transpose J=0\\
J^\transpose\Omega_0\alpha^\transpose+&\alpha\Omega_0 J^\transpose=0
\end{split}
\end{equation*}
and the pairing between~$T^*_J\m{AC}^+$ and~$T_J\m{AC}^+$ is
\begin{equation*}
\langle\alpha,A\rangle=\frac{1}{2}\mrm{Tr}\left(\alpha^\transpose\,A\right).
\end{equation*}
We can realise the total space~$\cotJ$ as the space of sections of a~$\mrm{Sp}(2n)$-bundle, with fibres diffeomorphic to~$T^*\!\m{AC}^+$. The action by conjugation of~$\mrm{Sp}(2n)$ on~$\m{AC}^+$ induces an action on~$T^*\!\m{AC}^+$, again by conjugation. More precisely, for~$h\in\mrm{Sp}(2n)$ and~$(J,\alpha)\in T^*\!\m{AC}^+$ we have
\begin{equation*}
h.(J,\alpha)=(h\,J\,h^{-1},(h^{-1})^{\transpose}\alpha\,h^\transpose)
\end{equation*}
and that is precisely also the change that the matrices associated to~$(J,\alpha)\in\cotJ$ in a Darboux coordinate system undergo under a change to another Darboux coordinate system. Hence, as was the case for~$\mscr{J}$, we can write~$\cotJ$ as the space of sections of some~$\mrm{Sp}(2n)$-bundle~$\hat{\m{E}}\xrightarrow{\pi}M$.

There is a canonical~$\mrm{Sp}(2n)$-bundle map~$F:\hat{\m{E}}\to\m{E}$, covering the identity on~$M$, that is induced by the projection~$p:T^*\!\m{AC}^+\to\m{AC}^+$. It is defined as follows: for~$\xi\in\hat{\m{E}}$, let~$x=\pi(\xi)$ and fix a system of Darboux coordinates~$\bm{u}:U\to\bb{R}^{2n}$ around~$x$; consider then the trivializations~$\Phi_{\bm{u}}:\hat{\m{E}}_{\restriction U}\to U\times T^*\!\m{AC}^+$ and~$\phi_{\bm{u}}:\m{E}_{\restriction U}\to U\times\m{AC}^+$ and let
\begin{equation*}
F(\xi):=\phi_{\bm{u}}^{-1}\circ (\mathbbm{1}\times p)\circ\Phi_{\bm{u}}(\xi).
\end{equation*}
The definition of~$F$ does not depend upon the choice of Darboux coordinates on~$M$, since the action of the symplectic group on~$T^*\!\m{AC}^+$ is induced by the action on~$\m{AC}^+$. This map accounts for the fact that from a section~$s$ of~$\hat{\m{E}}$ we can always get a section~$J=F(s)$ of~$\m{E}$ and a section~$\alpha$ of~$J^*(\mrm{Vert}\,\m{E}^*)$.

A similar discussion holds for the \emph{holomorphic} cotangent bundle~${T^{1,0}_J}^*\!\!\mscr{J}$: it is locally identified with maps taking values in~${T^{1,0}_J}^*\!\m{AC}^+$, i.e. functions~$x\to\alpha(x)\in(\bb{R}^{2n})^\vee\otimes\bb{C}$ satisfying
\begin{enumerate}
\item~$J^\transpose\alpha=\I\alpha$;
\item~$\alpha\,J^\transpose=-\I\alpha$;
\item~$J^\transpose\Omega_0\mrm{Re}(\alpha)^\transpose+\mrm{Re}(\alpha)\Omega_0 J=0$.
\end{enumerate}
In what follows we will mostly study just the holomorphic cotangent bundle of~$\mscr{J}$, so notationally it will be more convenient to denote it just by~$\cotJ$; %
\nomenclature[cotangent J]{$\cotJ$}{holomorphic cotangent space of~$\mscr{J}$}%
the context will make clear what space we are working on.
\begin{rmk}
If~$J$ is integrable and we fix a system of coordinates on~$M$ that are holomorphic with respect to~$J$, then an element~$\alpha\in {T^{1,0}}^*_J\mscr{J}$ in these coordinates is written as
\begin{equation*}
\alpha=\alpha\indices{_a^{\bar{b}}}\,\diff_{z^{\bar{b}}}\otimes\mrm{d}z^a
\end{equation*}
while for~$A\in T_J\mscr{J}$, if we decompose it as~$A=A^{1,0}+A^{0,1}$ then
\begin{equation*}
A^{1,0}=A\indices{^a_{\bar{b}}}\mrm{d}\bar{z}^b\otimes\diff_{z^a}
\end{equation*}
so the notation~$A^{1,0}$ and~${T^{1,0}}^*\!\!\mscr{J}$ might seem to be contradictory. This is caused by the fact that the metric identification of the tangent and cotangent bundles of a K\"ahler manifold sends~$(1,0)$-vectors to~$(0,1)$-forms and vice versa.
\end{rmk}
We know from Theorem~\ref{thm:BiquardGauduchon} that~${T^{1,0}}^*\!\m{AC}^+$ carries a hyperk\"ahler structure, and one of the complex structures in the hyperk\"ahler family is the \emph{canonical complex structure}, i.e. the complex structure induced by that of~$\m{AC}^+$. To obtain an expression for the complex structure we can either directly pull back the complex structure of~$T^*\f{H}$ using the identification of~$\m{AC}^+$ with~$\f{H}$, or we can consider~$T_J\m{AC}^+$ as the space of infinitesimal deformations of~$\m{AC}^+$; for this approach, see Section~\ref{sec:complexification} and in particular equation~\eqref{eq:TAC_complex_str}.

The tangent space~$T_{J,\alpha}\left({T^{1,0}}^*\!\m{AC}^+\right)$ is the set of all pairs~$(\dot{J},\dot{\alpha})$ with~$\dot{J}\in T_J\m{AC}^+$ and~$\dot{\alpha}\in (\bb{R}^{2n})^\vee\otimes\bb{C}$ such that
\begin{description}
\item[$1'$.]~$\dot{J}^\transpose\alpha+J^\transpose\dot{\alpha}=\I\dot{\alpha}$
\item[$2'$.]$\alpha\,\dot{J}^\transpose+\dot{\alpha}\,J^\transpose=-\I\dot{\alpha}$
\item[$3'$.]~$\dot{J}^\transpose\Omega_0\mrm{Re}(\alpha)^\transpose+J^\transpose\Omega_0\mrm{Re}(\dot{\alpha})^\transpose+\mrm{Re}(\dot{\alpha})\Omega_0 J+\mrm{Re}(\alpha)\Omega_0\dot{J}=0$.
\end{description}
Using the identification between~$\m{AC}^+$ and~$\f{H}$, we can describe the canonical complex structure of~${T^{1,0}}^*\!\!\m{AC}^+$ as
\begin{equation}\label{eq:canonical_compl_str}
\begin{split}
T_{J,\alpha}\left({T^{1,0}}^*\!\!\m{AC}^+\right)&\longrightarrow T_{J,\alpha}\left({T^{1,0}}^*\!\!\m{AC}^+\right)\\
(\dot{J},\dot{\alpha})&\mapsto(J\dot{J},\dot{\alpha}J^\transpose+\dot{J}^\transpose\alpha).
\end{split}
\end{equation}
As we mentioned earlier, this complex structure can be directly computed, but it can also be seen in a more conceptual way, c.f. equation~\eqref{eq:TAC_complex_str}. For now, we just check that this map squares to minus the identity:
\begin{equation*}
\begin{split}
(\dot{J},\dot{\alpha})\mapsto(J\dot{J},\dot{\alpha}J^\transpose+\dot{J}^\transpose\alpha)\mapsto&(-\dot{J},(\dot{\alpha}J^\transpose+\dot{J}^\transpose\alpha)J^\transpose+(J\dot{J})^\transpose\alpha)=\\
&=(-\dot{J},-\dot{\alpha}+\dot{J}^\transpose\alpha J^\transpose+\dot{J}^\transpose J^\transpose\alpha)=(-\dot{J},-\dot{\alpha}).
\end{split}
\end{equation*}

\section{An infinite-dimensional hyperk\"ahler reduction}\label{sec:cotJ_hyperk}

In this section we prove Theorem~\ref{thm:HKThmIntro}, constructing a hyperk\"ahler structure on a neighbourhood of the zero section~$\mscr{J}\hookrightarrow\cotJ$ that restricted to~$\mscr{J}$ coincides with the Donaldson-Fujiki K\"ahler metric on~$\mscr{J}$ of Section~\ref{sec:J}. The structure will be constructed using Theorem~\ref{thm:BiquardGauduchon}, taking advantage of the fact that the metric on~$T^*\left(\mrm{Sp}(2n)/U(n)\right)$ of Theorem~\ref{thm:BiquardGauduchon} is invariant under the~$\mrm{Sp}(2n)$-action.

With a view to applying Lemma~\ref{lemma:hyperekahler_criterio}, we introduce the following tensors on~$\cotJ~$:
\begin{enumerate}
\item[$\cdot$] a Riemannian metric~$\bm{G}$;
\item[$\cdot$] a complex structure~$\bm{I}$ compatible with~$\bm{G}$;
\item[$\cdot$] a symplectic form~$\bm{\Theta}$ of type~$(2,0)$ with respect to~$\bm{I}$.
\end{enumerate}
By Lemma~\ref{lemma:hyperekahler_criterio}, to prove that this defines a hyperk\"ahler structure on~$\cotJ$ it suffices to show that
\begin{enumerate}
\item~$\bm{J}^2=-1$, where~$\bm{J}$ satisfies~$\bm{G}(\bm{J}-,-)=\mrm{Re}(\bm{\Theta})$;
\item~$\mrm{d}\bm{\Omega_I}=0$, where~$\bm{\Omega_I}(-,-) =\bm{G}(\bm{I}-,-)$.
\end{enumerate}

Since~$\mscr{J}$ already has a complex structure~$\bb{J}$, we define~$\bm{I}$ as the complex structure induced on~$\cotJ~$ by~$\bb{J}$; explicitly, from equation~\eqref{eq:canonical_compl_str} we have, for a point~$(J,\alpha)\in\cotJ$ and a tangent vector~$(A,\varphi)\in T_{(J,\alpha)}(\cotJ)$ 
\begin{equation*}
\bm{I}_{(J,\alpha)}(A,\varphi):=(JA,\varphi J^\transpose+A^\transpose\alpha). 
\end{equation*}
Let~$(I,\theta,\hat{g})$ be the triple of a complex structure, canonical~$2$-form and hyperk\"ahler metric on~$T^*\!\m{AC}^+$ described in Section~\ref{sec:BiquardGauduchon}. The~$2$-form~$\bm{\Theta}$ on~$\cotJ~$ will be  
\begin{equation}\label{eq:2forma_canonica}
\bm{\Theta}_{(J,\alpha)}(v,w):=\int_{x\in M}\theta_x(v_x,w_x)\,\frac{\omega^n}{n!}
\end{equation}
for all~$(J,\alpha)\in\cotJ$ and~$v,w\in T_{(J,\alpha)}(\cotJ )$, where as usual we are taking around each~$x\in M$ a trivialization of the fibre bundle (i.e. a system of Darboux coordinates). It's not obvious that this expression is actually independent from the choice of the trivialization; it will be shown in Lemma~\ref{lemma:2canonica_J}. A point to remark is that~$\bm{\Theta}$ is automatically of type~$(2,0)$ with respect to~$\bm{I}$, since~$\theta$ is of type~$(2,0)$ with respect to the canonical complex structure~$I$ of~$T^*\!\m{AC}^+$.

The natural candidate to be the hyperk\"ahler metric is the metric~$\bm{G}$ induced on~$\cotJ$ from the Biquard-Gauduchon metric on~$T^*\!\m{AC}^+$
\begin{equation}\label{eq:metrica_hyper_J}
\bm{G}_{(J,\alpha)}(v,w):=\int_{x\in M}\hat{g}_x(v_x,w_x)\,\frac{\omega^n}{n!}
\end{equation}
but again we should check that this expression is independent from the choice of Darboux coordinates around each point. Assuming for the moment that it is, the fact that~$\bm{I}$ and~$\bm{G}$ are compatible follows immediately from the compatibility of~$I$ and~$\hat{g}$ on~$T^*\!\m{AC}^+$; moreover, the~$2$-form~$\bm{\Omega}_{\bm{I}}$ is
\begin{equation}\label{eq:forma_hyper_J}
{\bm{\Omega}_{\bm{I}}}_{(J,\alpha)}(v,w):=\int_{x\in M}(\omega_I)_x(v_x,w_x)\,\frac{\omega^n}{n!}
\end{equation}
where~$\omega_I$ is the~$2$-form defined in Theorem~\ref{thm:BiquardGauduchon}. Notice also that it is enough to check that~\eqref{eq:forma_hyper_J} does not depend on the choice of coordinates to guarantee that also~\eqref{eq:metrica_hyper_J} does not. Again under the (provisional) assumption that~\eqref{eq:forma_hyper_J} is well-defined, we notice that condition~$(1)$ above is automatically satisfied. Indeed, the complex structure~$\bm{J}$ is pointwise induced from the analogue complex structure~$\bb{J}$ of~$T^*\!\m{AC}^+$, from which it inherits algebraic properties like~$\bb{J}^2=-1$.

Summing up these considerations, to prove Theorem~\ref{thm:HKThmIntro} we just have to verify that~$\bm{\Theta}$ and~$\bm{\Omega_I}$ are well-defined and closed. The closedness of both forms is guaranteed by Theorem~\ref{thm:differenziazione}. Even though that result is stated for~$\mscr{J}$, it is quite easy to see that the same argument can be used for~$\cotJ$, since the construction works in a much more general setting. Again, we refer to the arguments in~\cite{Koiso_complex_structure}.

First we prove the well-definedness of~$\bm{\Omega_I}$. Notice that, since the action of~$\mrm{Sp}(2n)$ on~$\m{AC}^+$ is isometric and holomorphic, both~$\rho$ and~$\mrm{d}\mrm{d}^c\rho$ are~$\mrm{Sp}(2n)$-invariant.
\begin{lemma}
The~$2$-form~$\bm{\Omega}_{\bm{I}}$ of equation~\eqref{eq:forma_hyper_J} is well-defined.
\end{lemma}
\begin{proof}
Choose~$(J,\alpha)\in\cotJ$,~$v,w\in T_{(J,\alpha)}(\cotJ)$ and a Darboux coordinate system~$\bm{u}$. In this coordinate system the bundle~$\hat{\m{E}}$ trivializes, and we have to check that, for~$x\in\mrm{dom}(\bm{u})$, the expression
\begin{equation*}
\pi^*\omega_{(J(x),\alpha(x))}(v(x),w(x))+\mrm{d}\mrm{d}^c\rho_{(J(x),\alpha(x))}(v(x),w(x))
\end{equation*}
does not depend upon the choice of the coordinate system~$\bm{u}$. If~$\bm{v}$ is a different Darboux coordinate system, the matrix~$\varphi:=\frac{\diff\bm{v}}{\diff\bm{u}}$ is a~$\mrm{Sp}(2n)$-valued function and the previous expression becomes, in the new coordinate system,
\begin{equation*}
\begin{split}
\pi^*\omega_{\varphi(x).(J(x),\alpha(x))}&(\varphi(x).v(x),\varphi(x).w(x))+\\
&+\mrm{d}\mrm{d}^c\rho_{\varphi(x).(J(x),\alpha(x))}(\varphi(x).v(x),\varphi(x).w(x)).
\end{split}
\end{equation*}
Since both terms are~$\mrm{Sp}(2n)$-invariant this proves the claim.
\end{proof}
Another consequence of Theorem~\ref{thm:differenziazione} is that~$\bm{\Theta}$ has a more natural description, and in particular it is well-defined, concluding the proof of Theorem~\ref{thm:HKThmIntro}.
\begin{lemma}\label{lemma:2canonica_J}
The~$2$-form~$\bm{\Theta}$ defined in~\eqref{eq:2forma_canonica} is the canonical~$2$-form of~$\cotJ$.
\end{lemma}
\begin{proof}
We recall that for any manifold~$X$, the tautological~$1$ form~$\tau_X$ is a~$1$-form defined on the total space of~$T^*X\xrightarrow{\pi}X$ by
\begin{equation*}
\begin{split}
T^*_{x,\alpha}X&\to\bb{R}\\
v&\mapsto\alpha(\pi_*v)
\end{split}
\end{equation*}
and is related to the canonical~$2$-form~$\theta_X$ of~$T^*X$ by~$\theta_X=-\mrm{d}\tau_X$. Denote simply by~$\tau$ the tautological~$1$-form of~$T^*\!\m{AC}^+$, just as~$\theta$ is the canonical~$2$-form. Let also~$\bm{\tau}$ be the tautological form of~$\cotJ$. Then from the definitions it follows immediately that for any~$(J,\alpha)\in\cotJ$ and~$(A,\varphi)\in T_{(J,\alpha)}(\cotJ)$
\begin{equation*}
\bm{\tau}_{(J,\alpha)}((A,\varphi))=\alpha(A)=\int_{x\in M}\frac{1}{2}\mrm{Tr}(\alpha_x A_x^\transpose)\,\frac{\omega^n}{n!}=\int_{x\in M}\tau_{(J_x,\alpha_x)}(A_x,\varphi_x)\,\frac{\omega^n}{n!}.
\end{equation*}
By Theorem~\ref{thm:differenziazione} it is clear that this shows~$\bm{\Theta}=-\mrm{d}\bm{\tau}$.
\end{proof}

\subsection{The infinite-dimensional Hamiltonian action}\label{sec:moment_map}

Let~$(M,\omega)$ be a compact K\"ahler manifold. In this Section we prove Theorem~\ref{thm:HamiltonianHyperkIntro}, showing that the action of~$\mscr{G} =\mrm{Ham}(M,\omega)$ induced on~$\cotJ$ from the action on~$\mscr{J}$ is Hamiltonian with respect to both the real symplectic form~$\bm{\Omega}_{\bm{I}}$ and the complex symplectic form~$\bm{\Theta}$. 

Recall from Section~\ref{sec:J} that the group~$\mscr{G}$ acts on~$\mscr{J}$ by pull-backs as
\begin{equation*}
\varphi.J =(\varphi^{-1})^*J=\varphi_*\circ J\circ\varphi_*^{-1}.
\end{equation*}
As we saw in Chapter~\ref{chap:classicalScalCurv} this action preserves the K\"ahler structure of~$\mscr{J}$, and is Hamiltonian. The action induced by~$\mscr{G}$ on~$\cotJ$ is given by
\begin{equation*}
\varphi.(J,\alpha)=\left((\varphi^{-1})^*J,(\varphi^{-1})^*\alpha\right)=\left(\varphi_*\circ J\circ\varphi_*^{-1},(\varphi^{-1})^*\circ\alpha\circ\varphi^*\right)
\end{equation*}
and again it preserves~$\bm{\Theta}$,~$\bm{I}$ and~$\bm{\Omega}_{\bm{I}}$. For a function~$h\in\m{C}^\infty_0(M)=\mrm{Lie}(\mscr{G})$, the infinitesimal action of~$h$ on~$\cotJ$ is
\begin{equation}\label{eq:azione_infinitesima}
\hat{h}_{(J,\alpha)}=\left(\m{L}_{X_h}J,\m{L}_{X_h}\alpha\right)\in T_{(J,\alpha)}(\cotJ).
\end{equation}

First we recall a simple result that will be used to prove Theorem~\ref{thm:HamiltonianHyperkIntro}.
\begin{lemma}\label{lemma:moment_map_exact}
Let~$G$ be a Lie group acting on the left on a manifold~$X$, and assume that the action preserves a~$1$-form~$\chi$; let also~$\eta=\mrm{d}\chi$. Then the map
\begin{equation*}
\begin{split}
X&\to Lie(G)^*\\
x&\mapsto\f{m}_x
\end{split}
\end{equation*}
defined by~$\f{m}_x(a) =\chi_x(\hat{a}_x)$ satisfies
\begin{equation*}
\mrm{d}(\f{m}(a))=-\hat{a}\lrcorner\eta.
\end{equation*}
Moreover,~$\f{m}$ is~$G$-equivariant with respect to the action of~$G$ on~$X$ and the co-adjoint action on~$\mrm{Lie}(G)^*$. In particular if~$\eta$ is a symplectic form then~$\f{m}$ is a moment map for~$G\curvearrowright X$.
\end{lemma}

\begin{proof}
The first part is a simple consequence of Cartan's formula.
\begin{equation*}
0=\m{L}_{\hat{a}}\chi=\hat{a}\lrcorner\mrm{d}\chi+\mrm{d}(\hat{a}\lrcorner\chi)=\hat{a}\lrcorner\eta+\mrm{d}(\f{m}_x(a)).
\end{equation*}
As for the~$G$-equivariance, fix~$g\in G$ and~$a\in\mrm{Lie}(G)$. Then for every~$x\in X$ (here~$\sigma$ denotes the left action~$G\curvearrowright X$)
\begin{equation*}
\begin{split}
\f{m}_{g.x}(a)&=\chi_{g.x}(\hat{a}_{g.x})=\chi_{g.x}\left((\mrm{d}\sigma_g)_x\left(\widehat{\mrm{Ad}_{g^{-1}}(a)}_x\right)\right)=\\
&=\left(\sigma_g^*\chi\right)_x\left(\widehat{\mrm{Ad}_{g^{-1}}(a)}_x\right)=\chi_x\left(\widehat{\mrm{Ad}_{g^{-1}}(a)}_x\right)=\mrm{Ad}_{g^{-1}}^*\f{m}_x(a)
\end{split}
\end{equation*}
where we have used again the fact that~$\chi$ is~$G$-invariant.
\end{proof}

As a consequence, we obtain the following results for the action of~$\mscr{G}$ on the hyperk\"ahler manifold~$\cotJ$.

\begin{lemma}\label{lemma:moment_map_theta}
The action~$\mscr{G}\curvearrowright\cotJ$ is Hamiltonian with respect to the canonical symplectic form~$\bm{\Theta}$; a moment map~$\f{m}_{\bm{\Theta}}$ is given by
\begin{equation}\label{eq:moment_map_theta}
{\f{m}_{\bm{\Theta}}}_{(J,\alpha)}(h)=-\int_M\frac{1}{2}\mrm{Tr}(\alpha^\transpose\m{L}_{X_h}J)\,\frac{\omega^n}{n!}.
\end{equation}
\end{lemma}

\begin{proof}
Since~$\bm{\Theta}=-\mrm{d}\bm{\tau}$ and~$\mscr{G}$ preserves~$\bm{\tau}$, we can apply Lemma~\ref{lemma:moment_map_exact} to find that~$-\bm{\tau}_{(J,\alpha)}(\hat{h})$ is a moment map for~$\mscr{G}\curvearrowright(\cotJ,\bm{\Theta})$.
\end{proof}
Let us now consider the action with respect to the real symplectic form.
\begin{lemma}\label{lemma:mappa_momento_OmegaI}
The action~$\mscr{G}\curvearrowright(\cotJ,\bm{\Omega}_{\bm{I}})$ is Hamiltonian; a moment map~$\f{m}_{\bm{\Omega}_{\bm{I}}}$ is given by
\begin{equation*}
\f{m}_{\bm{\Omega}_{\bm{I}}} =\mu\circ\pi+\f{m}
\end{equation*}
where~$\mu$ is the moment map for the action~$\mscr{G}\curvearrowright(\mscr{J},\Omega)$,~$\pi:\cotJ\to J$ is the projection and~$\f{m}:\cotJ\to\mrm{Lie}(\mscr{G})^*$ is defined by 
\begin{equation}\label{eq:moment_map_omega}
\f{m}_{(J,\alpha)}(h)=\int_{x\in M}\mrm{d}^c\rho_{(J(x),\alpha(x))}\left(\m{L}_{X_h}J,\m{L}_{X_h}\alpha\right)\,\frac{\omega^n}{n!}.
\end{equation}
\end{lemma}
Recall from Theorem~\ref{thm:moment_map} that the moment map~$\mu$ for~$\mscr{G}\curvearrowright(\mscr{J},\Omega)$ is
\begin{equation*}
\mu(J)=2\,S(J)-2\,\hat{S}.
\end{equation*}
\begin{proof}[Proof of Lemma~\ref{lemma:mappa_momento_OmegaI}]
Since~$\bm{\Omega}_{\bm{I}}=\pi^*\Omega+\int_M\mrm{d}\mrm{d}^c\rho\,\frac{\omega^n}{n!}$, from Theorem~\ref{thm:moment_map} we just have to show that for all~$h\in\m{C}^\infty_0$
\begin{equation*}
\mrm{d}(\f{m}(h))=-\hat{h}\lrcorner\int_M\mrm{d}\mrm{d}^c\rho\,\frac{\omega^n}{n!}.
\end{equation*}
To prove this, we can use Lemma~\ref{lemma:moment_map_exact} and Theorem~\ref{thm:differenziazione}. Indeed, if we define~$\chi =\int_M\mrm{d}^c\rho\,\frac{\omega^n}{n!}$ then~$\mrm{d}\chi=\int_M\mrm{d}\mrm{d}^c\rho\,\frac{\omega^n}{n!}$. We already saw that the action of~$\mscr{G}$ preserves~$\chi$, and so Lemma~\ref{lemma:moment_map_exact} tells us that~$\f{m}$ defined by
\begin{equation*}
\f{m}_{(J,\alpha)}(h)=\int_{x\in M}\mrm{d}^c\rho_{(J(x),\alpha(x))}\left(\m{L}_{X_h}J,\m{L}_{X_h}\alpha\right)\frac{\omega^n}{n!}
\end{equation*}
has the properties we need.
\end{proof}
The results of Lemma~\ref{lemma:moment_map_theta} and Lemma~\ref{lemma:mappa_momento_OmegaI} show that the action of~$\mscr{G}$ on~$\cotJ$ is Hamiltonian with respect to the symplectic forms in the hyperk\"ahler family. To conclude the proof of Theorem~\ref{thm:HamiltonianHyperkIntro} we will obtain more explicit expressions for the moment maps under the natural~$L^2$-pairing of functions. This is not too difficult for the \emph{complex} moment map, at least if~$J$ is integrable. 
\begin{lemma}
Suppose~$J$ is integrable. Then we have
\begin{equation*}
{\f{m}_{\bm{\Theta}}}_{(J,\alpha)}(h) = \left\langle h,-\mrm{div}\left(\bdiff^*\bar{\alpha}\right)\right\rangle.
\end{equation*}
\end{lemma}
\begin{proof}
It is just a standard computation using the Divergence Theorem:
\begin{equation*}
\begin{split}
{\f{m}_{\bm{\Theta}}}_{(J,\alpha)}(h)&=-\frac{1}{2}\int_M\alpha\indices{_a^{\bar{b}}}\left(\m{L}_{X_h}J\right)\indices{^a_{\bar{b}}}\,\frac{\omega^n}{n!}=-\I\int_M\alpha\indices{_a^{\bar{b}}}\nabla_{\bar{b}}(X_h)\indices{^a}\,\frac{\omega^n}{n!}=\\
&=\I\int_M(X_h)\indices{^a}\nabla_{\bar{b}}\alpha\indices{_a^{\bar{b}}}\,\frac{\omega^n}{n!}=-\int_Mg^{a\bar{c}}\nabla_{\bar{c}}h\,\nabla_{\bar{b}}\alpha\indices{_a^{\bar{b}}}\,\frac{\omega^n}{n!}=\\
&=\int_M h\,g^{a\bar{b}}\nabla_{\bar{c}}\nabla_{\bar{b}}\alpha\indices{_a^{\bar{c}}}\,\frac{\omega^n}{n!}=\left\langle h,-\mrm{div}\left({\nabla^{0,1}}^*\bar{\alpha}\right)\right\rangle.\qedhere
\end{split}
\end{equation*}
\end{proof}
This shows that we can identify the complex moment map for the action of $\mscr{G}$ with $\mrm{div}\left(\diff^*\alpha\right)$. It is interesting to compare this with the expression in Section \ref{sec:thm:moment_map_proof} for the differential of the scalar curvature map along a path of integrable complex structures. By letting $A=\mrm{Re}(\alpha)^\transpose$, the differential $D(J\mapsto S(J,\omega))$ computed on $A$ was written as $Q(A)$ in Section~\ref{sec:thm:moment_map_proof}. We can see that, for an integrable complex structure $J$
\begin{equation*}
Q(A)=Q_2\circ Q_1(A)=-\mrm{div}\left(-JQ_1(A)^\sharp\right)=\mrm{Im}\,\mrm{div}\left(\diff^*\alpha\right).
\end{equation*}

It is slightly more complicated to obtain an explicit expression for the \emph{real} moment map. We carry out the details of the computation in the next section.

\subsection{An expression for the real moment map}\label{sec:real_momentmap}

In this Section we will express the map~$\f{m}$ of Lemma~\ref{lemma:mappa_momento_OmegaI} as the~$L^2$-pairing of~$h$ with some zero-average function. Recall that~$\f{m}$ is defined as
\begin{equation*}
\f{m}_{(J,\alpha)}(h)=\int_{M}\mrm{d}^c\rho_{(J,\alpha)}\left(\m{L}_{X_h}J,\m{L}_{X_h}\alpha\right)\frac{\omega^n}{n!}.
\end{equation*}
From Lemma~\ref{lemma:differenziale_BiquardGauduchon} we know how to compute the differential of the Biquard-Gauduchon function of~$\m{AC}^+$. We can use this to find the Biquard-Gauduchon form~$\mrm{d}^c\rho$ on~${T^{1,0}}^*\!\!\m{AC}^+$, by considering the identification of the space~$T\m{AC}^+$ with~${T^{1,0}}^*\!\m{AC}^+$ under the map
\begin{equation*}
\begin{split}
{T^{1,0}}^*\!\m{AC}^+&\to T\m{AC}^+\\
\alpha&\mapsto\mrm{Re}(\alpha)^\transpose.
\end{split}
\end{equation*}
So, we should compute 
\begin{equation*}
\begin{gathered}
\mrm{d}^c\rho_{J,\mrm{Re}(\alpha)^\transpose}\left(\m{L}_{X_h}J,\mrm{Re}\left(\m{L}_{X_h}\alpha\right)^\transpose\right)=\mrm{d}\rho_{J,\mrm{Re}(\alpha)^\transpose}\Big(-J\m{L}_{X_h}J,\ \mrm{Re}\left((\m{L}_{X_h}J)^\transpose\alpha+(\m{L}_{X_h}\alpha)J^\transpose\right)^\transpose\Big)
\end{gathered}
\end{equation*}
with~$\mrm{d}\rho$ as in~\ref{lemma:differenziale_BiquardGauduchon}, since the identification between~$T\m{AC}^+$ and~${T^{1,0}}^*\!\!\m{AC}^+$ is conjugate-linear in the second component. It is more convenient to write~$A=\mrm{Re}(\alpha)^\transpose\in T_J\mscr{J}$, so that Lemma~\ref{lemma:differenziale_BiquardGauduchon} gives
\begin{equation*}
\begin{gathered}
\mrm{d}^c\rho_{J,\mrm{Re}(\alpha)^\transpose}\left(\m{L}_{X_h}J,\mrm{Re}\left(\m{L}_{X_h}\alpha\right)^\transpose\right)=\\
=\sum_{i=1}^n\frac{1}{2}\left(1+\sqrt{1-\delta(i)}\right)^{-1}\mrm{Tr}\left(A
\left(A\,\m{L}_{X_h}J+J\,\m{L}_{X_h}A\right)\prod_{j\not=i}\frac{\delta(j)\mathbbm{1}-A^2}{\delta(j)-\delta(i)}
\right).
\end{gathered}
\end{equation*}
Notice however that~$\m{L}_{X_h}J\in T_J\mscr{J}$, so Remark~\ref{rmk:traccia_tripla} allows us to ignore that term when taking the trace:
\begin{equation}\label{eq:real_mm_implicit}
\f{m}_{(J,\alpha)}(h)=\sum_{i=1}^n\int_{M}
\frac{1}{2}\left(1+\sqrt{1-\delta(i)}\right)^{-1}\mrm{Tr}\left(A\,J\,\m{L}_{X_h}A\prod_{j\not=i}\frac{\delta(j)\mathbbm{1}-A^2}{\delta(j)-\delta(i)}
\right)
\frac{\omega^n}{n!}.
\end{equation}
We now show how to write the integrals in equation~\eqref{eq:real_mm_implicit} as the~$L^2$ pairing of~$h$ with a function~$f$. More precisely, for~$i\in\set{1,\dots,n}$ we will find a zero-average function~$f(i)=f(i,A,J)$ such that
\begin{equation*}
\int_{M}
\frac{\mrm{Tr}\left(A\,J\,\m{L}_{X_h}A\prod_{j\not=i}(\delta(j)\mathbbm{1}-A^2)
\right)}{2\left(1+\sqrt{1-\delta(i)}\right)\prod_{j\not=i}(\delta(j)-\delta(i))}
\frac{\omega^n}{n!}
=\int_M f(i)\,h\frac{\omega^n}{n!}.
\end{equation*}
This of course allows us to identify the moment map~$\f{m}$ with~$\sum_if(i)$. For notational convenience, we introduce the functions
\begin{equation}\label{eq:psi_A_i}
\begin{gathered}
\psi(i):=\frac{1}{2}\left(1+\sqrt{1-\delta(i)}\right)^{-1}\\
A(i):=\prod_{j\not=i}\frac{\delta(j)\mathbbm{1}-A^2}{\delta(j)-\delta(i)}.
\end{gathered}
\end{equation}
So we can write~$\f{m}(J,\alpha)$ in a more compact way as
\begin{equation*}
\f{m}_{(J,\alpha)}(h)=-\sum_{i=1}^n\int_{M}\psi(i)\,\mrm{Tr}\left((\m{L}_{X_h}A)\,JA(i)A\right)\frac{\omega^n}{n!}.
\end{equation*}
Consider the function
\begin{equation}\label{eq:F_F1F2F3}
\begin{split}
F^i:\m{C}^\infty_0(M)&\to\m{C}^\infty_0(M)\\
h&\mapsto\mrm{Tr}\left((\m{L}_{X_h}A)J\,A(i)\,A\right).
\end{split}
\end{equation}
We have~$\f{m}_{(J,\alpha)}(h)=-\sum_i\left\langle\psi(i),F^i(h)\right\rangle$; so if we find a formal adjoint~$F^i{}^*$ of~$F^i$, we can write~$\f{m}=-\sum_iF^i{}^*\left(\psi(i)\right)$. Notice that we can write~$F^i$ as a composition~$F^i=F^i_3\circ F^i_2\circ F^i_1$, with
\begin{align*}
F^i_1:\m{C}^\infty_0(M)&\to\Gamma(M,TM)\\
h&\mapsto X_h\\
F^i_2:\Gamma(M,TM)&\to\Gamma(M,\mrm{End}(TM))\\
X&\mapsto\m{L}_XA\\
F^i_3:\Gamma(M,\mrm{End}(TM))&\to\m{C}^\infty_0(M)\\
P&\mapsto \mrm{Tr}(P\,J\,A(i)\,A).
\end{align*}
Moreover the formal adjoints of~$F_1^i$ and~$F_3^i$ with respect to the pairing induced by the metric~$g_J:=\omega_0(-,J-)$ are given explicitly by
\begin{align*}
F_1^i{}^*(X)=\mrm{div}(JX);\\
F_3^i{}^*(f)=-f\,A\,A(i)\,J.
\end{align*}
It remains to compute the formal adjoint of~$F_2^i$.

\begin{lemma}\label{lemma:aggiunto_F2}
For any~$Q\in\Gamma(\mrm{End}(TM))$,~$X\in\Gamma(TM)$ and~$A\in T_J\mscr{J}$ we have
\begin{equation*}
\langle\m{L}_XA,Q\rangle=\langle Q,\nabla_XA\rangle+\langle AQ-QA,\nabla X\rangle.
\end{equation*}
\end{lemma}
Here the pairings and the connection are those defined by the metric~$g_J$.
\begin{proof}
Fix an element~$Q$ of~$\Gamma(\mrm{End}(TM))$, and consider the product
\begin{equation}\label{eq:pairing_LX}
\begin{split}
g_J(\m{L}_XA,Q)=&g\indices{^i^j}g\indices{_k_l}Q\indices{^l_j}X\indices{^m}\diff\indices{_m}A\indices{^k_i}-g\indices{^i^j}g\indices{_k_l}Q\indices{^l_j}A\indices{^m_i}\diff\indices{_m}X\indices{^k}+\\
&+g\indices{^i^j}g\indices{_k_l}Q\indices{^l_j}A\indices{^k_m}\diff\indices{_i}X\indices{^m}.
\end{split}
\end{equation}
We can exchange the usual derivatives with covariant derivatives (using the Levi-Civita connection of~$g_J$), but we have to introduce Christoffel symbols; the proof consists in showing that the sum of all the terms that must be introduced in fact vanishes, and this is done recalling that~$g_J(-,A-)$ is symmetric (cf. equation~\eqref{eq:indentita_tangenti}). The first right hand side term of equation~\eqref{eq:pairing_LX} can then be written as
\begin{equation}\label{eq:LX_1}
\begin{split}
g^{ij}g_{kl}Q\indices{^l_j}X^m\diff_m A\indices{^k_i}=&g^{ij}g_{kl}Q\indices{^l_j}X^m\nabla_m A\indices{^k_i}-g^{ij}g_{kl}Q\indices{^l_j}X^m A\indices{^p_i}\Gamma^k_{mp}+\\
&+g^{ij}g_{kl}Q\indices{^l_j}X^m A\indices{^k_q}\Gamma^q_{mi}
\end{split}
\end{equation}
while the other two terms are written as
\begin{equation}\label{eq:LX_2}
\begin{split}
-g\indices{^i^j}g\indices{_k_l}Q\indices{^l_j}A\indices{^m_i}\diff\indices{_m}X\indices{^k}&=-g\indices{^i^j}g\indices{_k_l}Q\indices{^l_j}A\indices{^m_i}\nabla\indices{_m}X\indices{^k}+g\indices{^i^j}g\indices{_k_l}Q\indices{^l_j}A\indices{^m_i}X\indices{^p}\Gamma^k_{pm}
\end{split}
\end{equation}
\begin{equation}\label{eq:LX_3}
\begin{split}
g\indices{^i^j}g\indices{_k_l}Q\indices{^l_j}A\indices{^k_m}\diff\indices{_i}X\indices{^m}&=g\indices{^i^j}g\indices{_k_l}Q\indices{^l_j}A\indices{^k_m}\nabla\indices{_i}X\indices{^m}-g\indices{^i^j}g\indices{_k_l}Q\indices{^l_j}A\indices{^k_m}X\indices{^p}\Gamma^m_{ip}.
\end{split}
\end{equation}
Adding up equations~\eqref{eq:LX_1},~\eqref{eq:LX_2} and~\eqref{eq:LX_3} we find
\begin{equation*}
g_J(\m{L}_XA,Q)=g_J(Q,\nabla_XA)
-g\indices{^i^j}g\indices{_k_l}Q\indices{^l_j}A\indices{^m_i}\nabla\indices{_m}X\indices{^k}
+g\indices{^i^j}g\indices{_k_l}Q\indices{^l_j}A\indices{^k_m}\nabla\indices{_i}X\indices{^m}.
\end{equation*}
Recall now that $g$ and $A$ are compatible, i.e. $g(A-,-)$ is a symmetric tensor. This implies that $g^{ij}A\indices{^m_i}=g^{im}A\indices{^j_i}$ and $g_{kl}A\indices{^k_m}=g_{km}A\indices{^k_l}$, so in the end we get
\begin{equation*}
g_J(\m{L}_XA,Q)=g_J(Q,\nabla_XA)-g_J(QA,\nabla X)+g_J(AQ,\nabla X).\qedhere
\end{equation*}
\end{proof}

\begin{cor}\label{cor:aggiunto_F2}
The formal adjoint of~$F^i_2$ is
\begin{equation*}
\begin{split}
F^i_2{}^*:\Gamma(\mrm{End}(TM))&\to\Gamma(TM)\\
Q&\mapsto C^2_1\left((\nabla A)Q\right)^\sharp+\nabla^*([A,Q]).
\end{split}
\end{equation*}
\end{cor}
Here~$\nabla^*$ is the formal adjoint of~$\nabla$,~$\nabla^*Q=-g^{ij}\nabla_iQ\indices{^k_j}\diff_k$, while~$C^2_1$ denotes the contraction of the first lower index with the second upper index. More explicitly
\begin{equation*}
C^2_1\left((\nabla A)Q\right)^\sharp=g^{np}\left(\nabla_pA\indices{^j_i}\right)Q\indices{^i_j}\diff_n.
\end{equation*}
We are finally in a good position to write the real moment map. Our computations so far show
\begin{equation*}
\begin{split}
\f{m}_{(J,\alpha)}(h)&=-\sum_i\left\langle\psi(i),F^i(h)\right\rangle=-\sum_i\left\langle F_1^i{}^*F_2^i{}^*F_3^i{}^*(\psi(i)),h\right\rangle=\\
&=\sum_i\left\langle\mrm{div}\left[\psi(i)\,J\,C^2_1\left((\nabla A)A\,A(i)\,J\right)^\sharp+2\,J\nabla^*(\psi(i)J A^2A(i))\right],\,h\right\rangle
\end{split}
\end{equation*}
so we identify the function~$\f{m}$ with
\begin{equation}\label{eq:mm_reale_esplicita}
\f{m}(J,\alpha)=\sum_i\mrm{div}\left[\psi(i)\,J\,C^2_1\left((\nabla A)A\,A(i)\,J\right)^\sharp+2\,J\nabla^*(\psi(i)J A^2A(i))\right].
\end{equation}
Notice that this expression implies that~$\f{m}_{(J,\alpha)}$ is a zero-average function, as we expected.

We can make some simplifications to equation~\eqref{eq:mm_reale_esplicita} under the assumption that~$J$ is an integrable complex structure. Then a first simplification is
\begin{equation*}
J\nabla^*(\psi(i)J A^2A(i))=-\nabla^*(\psi(i)A^2A(i))
\end{equation*}
since when~$J$ is integrable~$g_J$ is a K\"ahler metric, so that~$\nabla J=0$. Now, fix holomorphic coordinates with respect to~$J$. Then we have
\begin{equation*}
\begin{gathered}
J\left(C^2_1\,\left((\nabla A)A\,A(i)\,J\right)^\sharp\right)^{1,0}=\\=g^{a\bar{b}}\left(g(\nabla_{\bar{b}}A^{0,1},A^{1,0}A(i)^{0,1})-g(\nabla_{\bar{b}}A^{1,0},A^{0,1}A(i)^{1,0})\right)\diff_{z^a}=\\
=g\left(\nabla^aA^{0,1}-\nabla^aA^{1,0},\,A\,A(i)\right)\diff_{z^a}.
\end{gathered}
\end{equation*}
so we can rewrite~\eqref{eq:mm_reale_esplicita} as
\begin{equation*}
\begin{split}
\sum_i\mrm{div}\Big[\psi(i)\,&2\,\mrm{Re}\left(g\left(\nabla^aA^{0,1}-\nabla^aA^{1,0},\,A\,A(i)\right)\diff_{z^a}\right)-2\nabla^*(\psi(i)A^2A(i))\Big].
\end{split}
\end{equation*}
We can write this expression in a slightly more compact way.
\begin{lemma}\label{lemma:psi_spectral}
Let~$\psi$ be the function defined in~\eqref{eq:funzioni_accessorie_HcscK},~$\psi(x)=\frac{1}{2}\left(1+\sqrt{1-x}\right)^{-1}$. Then we have
\begin{equation*}
\hat{A}:=\sum_i\psi(i)\prod_{j\not=i}\frac{\delta(j)\mathbbm{1}-A^{0,1}A^{1,0}}{\delta(j)-\delta(i)}=\psi(A^{0,1}A^{1,0}).
\end{equation*}
\end{lemma}
\begin{proof}
It is just a matter of linear algebra. In a basis for which~$A^{0,1}A^{1,0}$ is diagonal, with
\begin{equation*}
A^{0,1}A^{1,0}=\mrm{diag}\left(\delta(1),\dots,\delta(n)\right)
\end{equation*} 
one has
\begin{equation*}
\frac{\delta(j)\mathbbm{1}-A^{0,1}A^{1,0}}{\delta(j)-\delta(i)}=\mrm{diag}\Bigg(\frac{\delta(j)-\delta(1)}{\delta(j)-\delta(i)},\dots,\!\!\!\!\underset{i\mbox{{\footnotesize-th place}}}{1}\!\!\!\!,\dots,\!\!\!\!\underset{j\mbox{{\footnotesize-th place}}}{0}\!\!\!\!,\dots,\frac{\delta(j)-\delta(n)}{\delta(j)-\delta(i)}\Bigg)
\end{equation*}
so when we take the product of these terms for~$j\not=i$ we get a matrix whose only non-zero entry is a~$1$ on the~$i$-th place of the diagonal. Summing over~$i$ finally gives
\begin{equation*}
\hat{A}=\mrm{diag}\big(\psi(\delta(1)),\dots,\psi(\delta(n))\big).\qedhere
\end{equation*}
\end{proof}
Our previous computations allow us to write~$\f{m}(J,\alpha)$ as
\begin{equation}\label{eq:mappa_momento_reale}
\begin{split}
2\,\mrm{div}\,\mrm{Re}\left(g\left(\nabla^aA^{0,1},A^{1,0}\,\hat{A}\right)\diff_{z^a}-g\left(\nabla^{\bar{b}}A^{0,1},A^{1,0}\,\hat{A}\right)\diff_{\bar{z}^b}-2\nabla^*(A^{0,1}A^{1,0}\hat{A})\right)
\end{split}
\end{equation}
so from Lemma~\ref{lemma:psi_spectral} we can finally obtain the complete expression for the real moment map (c.f. Lemma~\ref{eq:moment_map_omega}) appearing in the HcscK system~\eqref{eq:HcscK_system}:
\begin{equation}\label{eq:mm_reale_completa}
\begin{gathered}
\f{m}_{\bm{\Omega}_{\bm{I}}}(J,\alpha)=2\,S(J)-2\,\hat{S}+\\
+2\,\mrm{div}\,\mrm{Re}\Big[g\left(\nabla^aA^{0,1}\!,A^{1,0}\hat{A}\right)\diff_{z^a}-g\left(\nabla^{\bar{b}}A^{0,1}\!,A^{1,0}\hat{A}\right)\diff_{\bar{z}^b}-2\nabla^*(A^{0,1}A^{1,0}\hat{A})\Big].
\end{gathered}
\end{equation}

\section{Formal complexification of the action}\label{sec:complexification}

Having a set of moment map equations, one would hope to link the existence of solutions to these equations to a stability condition of algebraic nature. As we already remarked for the classical cscK equation, the correspondence given by the Kempf-Ness Theorem between Marsden-Weinstein reductions and GIT quotients cannot be used in our situation, as there is no complexification of~$\mscr{G}$. However we can generalize the discussion of Section~\ref{sec:complexification_cscK} to formally complexify the action of~$\mscr{G}$ on~$\cotJ$.

We will see that there are some important differences between our problem on~$\cotJ$ and the classical description of the complexified orbits of~$\mscr{G}\curvearrowright\mscr{J}$; most notably, it turns out that our moment map equations can be complexified just in a more formal sense, c.f. Remark~\ref{rmk:int_Higgs_class}.

The infinitesimal action of~$h\in\mrm{Lie}(\mscr{G})=\m{C}^{\infty}(M,\bb{R})$ on~$\cotJ$ is
\begin{equation*}
\hat{h}_{J,\alpha}=\left(\m{L}_{X_h}J,\m{L}_{X_h}\alpha\right)
\end{equation*}
and since~$\cotJ$ has a complex structure~$\bm{I}$ we can infinitesimally complexify the action of~$\mscr{G}$ by setting for~$h\in\mrm{Lie}(\mscr{G})^{c}=\m{C}^\infty(M,\bb{C})$
\begin{equation*}
\widehat{h}_{J,\alpha}:=\widehat{\mrm{Re}(h)}_{J,\alpha}+\bm{I}_{J,\alpha}\widehat{\mrm{Im}(h)}_{J,\alpha}
\end{equation*}
So we can define a distribution on~$\cotJ$, similarly to what we did for~$\mscr{J}$ in Section~\ref{sec:complexification_cscK}, as 
\begin{equation*}
\mscr{D}_{J,\alpha}=\set*{\hat{h}_{J,\alpha}\tc h\in\m{C}^\infty(M,\bb{R})}\cup\set*{\widehat{\I h}_{J,\alpha}\tc h\in\m{C}^\infty(M,\bb{R})}
\end{equation*}
We would like to prove that this distribution is integrable, so that the integral leaves of~$\mscr{D}_{J,\alpha}$ can be considered as complexified orbits of the action~$\mscr{G}\curvearrowright\cotJ$. We will need some preliminary results, the analogues of Remark~\ref{rmk:integrable_submanifoldJ} and Lemma~\ref{lemma:integrable_submanifoldJ} on the cotangent space~$\cotJ$. To obtain these results we will use a slightly different description of~$\cotJ$.

Consider the ring~$\m{E}nd=\Gamma(M,\mrm{End}(TM))$ and the ring~$\m{E}nd\llbracket\varepsilon\rrbracket$ of formal series obtained by adding a formal parameter~$\varepsilon$ to~$\m{E}nd$; we can quotient it by the (two-sided) ideal~$(\varepsilon^2)$, and define 
\begin{equation*}
\begin{split}
\mscr{R}=\Big\lbrace& J+\varepsilon A+(\varepsilon^2)\in\m{E}nd\llbracket\varepsilon\rrbracket/(\varepsilon^2)\mid
\left(J+\varepsilon A\right)^2+1\in(\varepsilon^2)\mbox{ and }\omega\circ(J+\varepsilon A)-\omega\in(\varepsilon^2)\Big\rbrace
\end{split}
\end{equation*}
where we are considering the map~$\m{E}nd\llbracket\varepsilon\rrbracket\to\m{A}^2(M)\llbracket\varepsilon\rrbracket$ defined by
\begin{equation*}
P(\varepsilon)\mapsto\omega\circ P(\varepsilon):=\omega(P(\varepsilon)-,P(\varepsilon)-).
\end{equation*}
There is a diffeomorphism~$k:\mscr{R}\to\cotJ$, given by
\begin{equation*}
k\left(J+\varepsilon A+(\varepsilon^2)\right)=\left(J,(A+\I JA)^\transpose\right)=\left(J,2\,(A^{0,1})^\transpose\right)
\end{equation*}
with inverse
\begin{equation*}
k^{-1}(J,\alpha)=J+\varepsilon\mrm{Re}(\alpha)^\transpose+(\varepsilon^2).
\end{equation*}
The pull-back of the complex structure of~$\cotJ$ to~$\mscr{R}$ via~$k$ has a rather natural expression; compare it with the complex structure of~$\m{AC}^+$ described in Proposition~\ref{prop:metrica_AC+_Siegel}.
\begin{equation}\label{eq:TAC_complex_str}
\begin{split}
\left(k^*\bm{I}\right)_{J+\varepsilon A+(\varepsilon^2)}:T_{J+\varepsilon A+(\varepsilon^2)}\mscr{R}&\longrightarrow T_{J+\varepsilon A+(\varepsilon^2)}\mscr{R}\\
\dot{J}+\varepsilon\dot{A}+(\varepsilon^2)&\longmapsto (J+\varepsilon A)(\dot{J}+\varepsilon\dot{A})+(\varepsilon^2).
\end{split}
\end{equation}

The pull-back of the distribution~$\mscr{D}_{(J,\alpha)}$ to~$J+\varepsilon A+(\varepsilon^2)=k^{-1}(J,\alpha)$ is
\begin{equation*}
\begin{split}
\mscr{D}_{J+\varepsilon A+(\varepsilon^2)}=&\set*{\m{L}_{X_h}J+\varepsilon\m{L}_{X_h}A+(\varepsilon^2)\tc h\in\m{C}^\infty(M)}\cup\\
&\cup\set*{(J+\varepsilon A)(\m{L}_{X_h}J+\varepsilon\m{L}_{X_h}A)+(\varepsilon^2)\tc h\in\m{C}^\infty(M)}.
\end{split}
\end{equation*}

The action of~$\mscr{G}$ on~$\cotJ$ carries over to an action on~$\mscr{R}$; if~$\varphi$ is a Hamiltonian diffeomorphism,
\begin{equation*}
\begin{split}
k^{-1}(\varphi.(J,\alpha))=&k^{-1}((\varphi^{-1})^*J,(\varphi^{-1})^*\alpha)=(\varphi^{-1})^*J+\varepsilon\mrm{Re}((\varphi^{-1})^*\alpha)^\transpose+(\varepsilon^2)=\\
=&(\varphi^{-1})^*\left(J+\varepsilon\mrm{Re}(\alpha)^\transpose\right)+(\varepsilon^2)
\end{split}
\end{equation*}
so the induced action~$\mscr{G}\curvearrowright\mscr{R}$ is~$\varphi.\left(J+\varepsilon\,A+(\varepsilon^2)\right)=(\varphi^{-1})^*(J+\varepsilon\,A)+(\varepsilon^2)$.

We will consider only~$J+\varepsilon A+(\varepsilon^2)\in\mscr{R}$ that are (first-order) integrable in the sense of Definition~\ref{def:integrable}, i.e.~$N\left(J+\varepsilon A\right)=O(\varepsilon^2)$ where~$N$ is the Nijenhuis operator.
\begin{lemma}
The distribution~$J+\varepsilon A+(\varepsilon^2)\mapsto\mscr{D}_{J+\varepsilon A+(\varepsilon^2)}$ is tangent to the subspace of integrable objects in~$\mscr{R}$.
\end{lemma}
\begin{proof}
Let~$\mscr{D}^0_{J+\varepsilon A+(\varepsilon^2)}$ be the distribution \[\set*{\m{L}_{X_h}J+\varepsilon\m{L}_{X_h}A+(\varepsilon^2)\tc h\in\m{C}^\infty(M)}.\] This is the tangent space to the orbits of the action~$\mscr{G}\curvearrowright\mscr{R}$. A computation using the naturality of the Lie bracket shows that if~$N_{J+\varepsilon\,A+(\varepsilon^2)}\in(\varepsilon^2)$ then also~$N_{f^*(J+\varepsilon\,A)+(\varepsilon^2)}\in(\varepsilon^2)$ for any diffeomorphism~$f$, so~$\mscr{D}^0$ is indeed tangent to the space of integrable first-order deformations of complex structures.

Since~$\mscr{D}=\mscr{D}^0+\bm{I}\mscr{D}^0$, the next step is to show that if~$J(\varepsilon)$ is integrable and~$DN_{J(\varepsilon)}(Q)\in(\varepsilon^2)$ then also~$DN_{J(\varepsilon)}(\bm{I}_{J(\varepsilon)}Q)\in(\varepsilon^2)$. For any two vector fields~$X,Y$ on~$M$, we have
\begin{equation}\label{eq:diff_nijenhuis_compl}
\begin{split}
DN_{J(\varepsilon)}(\bm{I}_{J(\varepsilon)}Q)(X,Y)=&J(\varepsilon)\left(\left[J(\varepsilon)QX,Y\right]+\left[X,J(\varepsilon)QY\right]\right)+\\
&+J(\varepsilon)Q\left(\left[J(\varepsilon)X,Y\right]+\left[X,J(\varepsilon)Y\right]\right)-\\
&-\left[J(\varepsilon)X,J(\varepsilon)QY\right]-\left[J(\varepsilon)QX,J(\varepsilon)Y\right].
\end{split}
\end{equation}
Since~$J(\varepsilon)$ is integrable and~$DN_{J(\varepsilon)}(Q)=0$, up to~$\varepsilon^2$-terms we have
\begin{equation*}
\begin{split}
Q\left(\left[J(\varepsilon)X,Y\right]+\left[X,J(\varepsilon)Y\right]\right)=&-J(\varepsilon)\left([QX,Y]+[X,QY]\right)+\\
&+[J(\varepsilon)X,QY]+[QX,J(\varepsilon)Y];\\
\left[J(\varepsilon)X,J(\varepsilon)QY\right]=&\left[X,QY\right]+J(\varepsilon)\Big([J(\varepsilon)X,QY]+[X,J(\varepsilon)QY]\Big)
\end{split}
\end{equation*}
and substituting these expressions in~\ref{eq:diff_nijenhuis_compl} we get~$DN_{J(\varepsilon)}(\bm{I}_{J(\varepsilon)}Q)(X,Y)=0$ up to~$\varepsilon^2$-terms, for all vector fields~$X,Y$ on~$M$.
\end{proof}
Analogous computations allow us to get the following result.
\begin{lemma}\label{lemma:integrable_epsilon}
If~$J+\varepsilon A+(\varepsilon^2)\in\mscr{R}$ is integrable, for every vector field~$X$
\begin{equation*}
\m{L}_{(J+\varepsilon A)X}(J+\varepsilon A)=(J+\varepsilon A)\m{L}_X(J+\varepsilon A)+(\varepsilon^2).
\end{equation*}
\end{lemma}

We can now describe a complexification of the action~$\mscr{G}\curvearrowright\cotJ$. We fix an element~$(J,\alpha)$ of~$\cotJ$ such that~$J$ is an integrable complex structure and~$\alpha$ defines an integrable first-order deformation of~$J$. We will show that the parametrization in Proposition~\ref{prop:integral_leaf} of the complex~$\mscr{G}$-orbit of~$J$ in~$\mscr{J}$ can be used to parametrize the complex orbit of~$(J,\alpha)$, by modifying the map~$\Phi_J:\mscr{Y}_J\to\mscr{J}$ of~\eqref{eq:foglia_integrale_J}.

\begin{prop}\label{prop:complessificazione_cotJ}
Fix an integrable element~$(J,\alpha)$ of~$\cotJ$. Then, there is a map
\begin{equation*}
\Phi_{J,\alpha}:\m{K}(\omega)\to\cotJ
\end{equation*}
defined on the K\"ahler class of~$\omega$, such that~$\Phi_{J,\alpha}(\omega)=(J,\alpha)$, and such that its image is an integral leaf of the distribution
\begin{equation*}
\set*{\left(J\m{L}_{X_h}J,(\m{L}_{X_h}\alpha)J^\transpose+(\m{L}_{X_h}J^\transpose)\alpha\right)\tc h\in\m{C}^\infty(M,\bb{R})}.
\end{equation*}
\end{prop}
In other words, the map~$\Phi_{J,\alpha}$ gives a way to parametrize the ``complex directions'' of the complexified orbit of~$(J,\alpha)$ in~$\cotJ$. For convenience, we will prove this Proposition in the case~$H^1(M)=0$.
\begin{proof}
For a K\"ahler potential~$h$, consider~$\omega_t=\omega+\I\diff\bdiff th$ and the vector field~$Y_t=\frac{1}{2}JX^{\omega_t}_h$, where~$X^{\omega_t}_h$ is the Hamiltonian vector field associated to~$h$ under~$\omega_t$. In the proof of Theorem~\ref{prop:integral_leaf} we showed that, if~$f_t$ is the isotopy of the time-dependent vector field~$Y_t$, then~$(f_t,\omega_t)\in\mscr{Y}_{J}$ and~$\diff_tf_t^*J$ lies in the ``complex part'' of~$\mscr{D}_J$.

We will show that there is a path~$\alpha_t\in\m{A}^{1,0}(T^{0,1}M)$ such that
\begin{equation*}
t\mapsto\left(f_t^*J,f_t^*\alpha_t\right)
\end{equation*}
is a curve in~$\cotJ$ that is tangent to the complex part of the distribution~$\mscr{D}_{J,\alpha}$. These conditions can be rephrased as
\begin{equation}\label{eq:condizioni_complessificazione}
\begin{split}
1.\quad&\diff_tf_t^*\alpha_t=\frac{1}{2}f_t^*\left((\m{L}_{X^{\omega_t}_h}\alpha_t)J^\transpose+(\m{L}_{X^{\omega_t}_h}J^\transpose)\alpha_t\right);\\
2.\quad&\omega_t\left(\alpha_t-,J-\right)\in\mrm{Sym}\left(T^{0,1}M\right).
\end{split}
\end{equation}
For a generic path~$\alpha_t$, the variation of~$f_t^*\alpha_t$ is
\begin{equation*}
\diff_tf_t^*\alpha_t=f_t^*\left(\frac{1}{2}\m{L}_{JX^{\omega_t}_h}\alpha_t+\dot{\alpha}_t\right).
\end{equation*}
By Lemma~\ref{lemma:integrable_epsilon}, this can be rewritten as
\begin{equation*}
\begin{split}
\diff_tf_t^*\alpha_t=f_t^*\left(\frac{1}{2}\left(\m{L}_{X^{\omega_t}_h}\alpha_t\right)J^\transpose+\frac{1}{2}\left(\m{L}_{X^{\omega_t}_h}J\right)^\transpose\alpha_t-\frac{1}{2}\m{L}_{\alpha_t^\transpose(X^{\omega_t}_h)}J^\transpose
+\dot{\alpha}_t\right).
\end{split}
\end{equation*} 
Then, condition~$1$ in~\eqref{eq:condizioni_complessificazione} holds if and only if
\begin{equation}\label{eq:condizione_compl_alpha_t}
\dot{\alpha}_t=\frac{1}{2}\m{L}_{\alpha_t^\transpose(X^{\omega_t}_h)}J^\transpose=\diff\left(\alpha_t(\nabla_th)\right).
\end{equation}
This condition implies that any~$\alpha_t$ satisfying the first condition in~\eqref{eq:condizioni_complessificazione} can be written as~$\alpha_t=\alpha-2\I\diff V_t$, for some time-dependent vector field~$V_t\in\Gamma(T^{0,1}M)$. For equation~\eqref{eq:condizione_compl_alpha_t} to hold it is sufficient for the vector field $V_t$ to satisfy the PDE
\begin{equation}\label{eq:condizione_compl_Vt}
\dot{V}_t=\frac{1}{2}\alpha_t(X^{\omega_t}_h).
\end{equation}
The second condition in~\eqref{eq:condizioni_complessificazione} instead can be rephrased as
\begin{equation*}
\diff\left(t\alpha(\bdiff h)+2\I V_t^{\flat_t}\right)=0.
\end{equation*}
Hence, bearing on the assumption~$H^1=0$, there must be a family of functions~$\varphi_t\in\m{C}^\infty$ such that
\begin{equation}\label{eq:Vt}
V_t=\frac{\I}{2}\left(t\alpha(\bdiff h)-\diff\varphi_t\right)^{\sharp_t}.
\end{equation}
We should show that we can choose~$\varphi_t$ in such a way that~$\alpha_t=\alpha-2\I\diff V_t$ satisfies~\eqref{eq:condizioni_complessificazione}, for~$V_t$ defined by~\eqref{eq:Vt}. Notice that we should set~$\varphi_0=0$, so that~$\alpha_0=\alpha$. Rephrasing equation~\eqref{eq:condizione_compl_Vt} in terms of~$\varphi_t$ we can see that $\varphi_t$ and $V_t$ satisfy the relation
\begin{equation*}
\dot{\varphi}_t=2\I V_t(h)
\end{equation*}
that, together with \eqref{eq:Vt}, implies that $\varphi_t$ must satisfy the PDE
\begin{equation}\label{eq:ft_pde}
\dot{\varphi}_t=g_t\left(\diff\varphi_t-t\alpha(\bdiff h),\bdiff h\right).
\end{equation}
Using the method of characteristics, it is not difficult to check that this equation, together with the starting condition~$\varphi_0=0$, can be solved by a smooth function. The solution is in fact unique, up to the addition of a constant. We have not been able to write an explicit solution to~\eqref{eq:ft_pde}, and this has some consequence on the the complexification of the HcscK system. We will get back to this in the next section.

Tracing back the computations, the path~$\alpha_t$ is uniquely determined by~$\alpha$ and~$\omega_t$ as
\begin{equation*}
\alpha_t=\alpha+t\diff\left(\alpha(\bdiff h)^{\sharp_t}\right)-\diff\nabla_t\varphi_t
\end{equation*}
for a solution~$\varphi_t$ of~\eqref{eq:ft_pde}. 
\end{proof}

\subsection{The complexified equations}

Following the classical case of the cscK equation, the ``formal complexification'' of the orbits of~$\mscr{G}\curvearrowright\cotJ$ would make it natural to regard our system
\begin{equation*}
\begin{cases}
\f{m}_{\bm{\Omega}_{\bm{I}}}(\omega,J,\alpha)=0\\
\f{m}_{\bm{\Theta}}(\omega,J,\alpha)=0
\end{cases}
\end{equation*}
as equations for a form~$\omega$, to be found in some prescribed set, keeping instead the complex structure~$J$ fixed.

From the discussion of the cscK problem in Section~\ref{sec:complexification_cscK} we know that to achieve this we should consider how the moment maps change when we move~$(J,\alpha)$ along a ``complexified direction'' of the orbit, i.e. under the transformation
\begin{equation*}
(J,\alpha)\mapsto f_t^*(J,\alpha_t)
\end{equation*}
where, for a function~$h\in\m{C}^\infty(M)$,~$f_t$~and~$\alpha_t$ are described in the proof of Proposition~\ref{prop:complessificazione_cotJ}. So, we can write \emph{complexified moment map equations} for~$\f{m}_{\bm{\Omega_I}}$ and~$\f{m}_{\bm{\Theta}}$, as
\begin{equation}\label{eq:complessificate_1}
\begin{cases}
\f{m}_{\bm{\Omega_I}}(f_t^*J,f_t^*\alpha_t)=0;\\
\f{m}_{\bm{\Theta}}(f_t^*J,f_t^*\alpha_t)=0.
\end{cases}
\end{equation}
Focusing on the~$\bm{\Omega_I}$-equation, recall that the moment map is computed using the metric~$g(J,\omega)$ defined from the background symplectic form and the complex structure. Making the dependence on~$\omega$ explicit we find
\begin{equation*}
\begin{split}
\f{m}_{\bm{\Omega_I}}(f_t^*J,f_t^*\alpha_t)=&\f{m}_{\bm{\Omega_I}}\left(g(f_t^*J,\omega),f_t^*\alpha_t\right)=f_t^*\left(\f{m}_{\bm{\Omega_I}}\left(g(J,f_t^{-1}{}^*\omega),\alpha_t\right)\right)=\\
=&f_t^*\left(\f{m}_{\bm{\Omega_I}}\left(g(J,\omega_t),\alpha_t\right)\right).
\end{split}
\end{equation*}
So, the equation~$\f{m}_{\bm{\Omega_I}}(f_t^*J,f_t^*\alpha_t)=0$ is equivalent to
\begin{equation*}
\f{m}_{\bm{\Omega_I}}\left(g(J,\omega_t),\alpha_t\right)=0
\end{equation*}
hence looking for a solution to the moment map equations along the complexified orbit of~$(J,\alpha)$ is equivalent to keeping~$J$ fixed, moving~$\omega$ in its K\"ahler class to~$\omega+2\I\diff\bdiff t h$ and simultaneously moving~$\alpha$ to~$\alpha+t\diff\left(\alpha(\bdiff h)^{\sharp_t}\right)-\diff\nabla_t\varphi_t$, for a solution~$\varphi_t$ of~\eqref{eq:ft_pde}. 

The same considerations can be made also for~$\f{m}_{\bm{\Theta}}$, so that the \emph{complexified HcscK equations} are
\begin{equation}\label{eq:complessificate_definitive}
\begin{cases}
\f{m}_{\bm{\Omega_I}}\left(g(J,\omega_t),\alpha_t\right)=0;\\
\f{m}_{\bm{\Theta}}\left(g(J,\omega_t),\alpha_t\right)=0.
\end{cases}
\end{equation}

\begin{rmk}\label{rmk:int_Higgs_class}
Since we are assuming that~$\alpha$ is an integrable deformation of the complex structure~$J$,~$\alpha$~determines a class~$[\bar{\alpha}]\in H^1(T^{1,0}M)$. The formal complexification moves~$\alpha$ to the deformation~$\alpha_t:=\alpha-2\I\diff V^{0,1}_t$ that lies in the same class,~$[\bar{\alpha}_t]=[\bar{\alpha}]$. So, the formally complexified system~\eqref{eq:complessificate_definitive} can be considered as a system for compatible~$\omega$ and~$\alpha$ belonging to a fixed K\"ahler class and a fixed deformation class respectively.
\end{rmk}
It seems quite difficult to find an explicit expression for the general solution~$\varphi_t$ of the PDE~\eqref{eq:ft_pde}. There is, however, a very specific case in which we can characterize solutions to~\eqref{eq:ft_pde}: if~$\alpha(\bdiff h)=0$, then~$\alpha$ is compatible with~$\omega_t:=\omega+\I\diff\bdiff th$ and the unique solution of~\eqref{eq:ft_pde} is~$\varphi_t=0$. Notice however that, even if $\alpha$ is already compatible with $\omega_h$, i.e. $\diff\left(\alpha\left(\bdiff h\right)\right)=0$, we still have to move $\alpha$ along a nonconstant path $\alpha_t$ to obtain the complexified orbit.

The problem of finding an explicit expression for the function~$f_t$ perhaps could be avoided by complexifying the equations not by moving~$(J,\alpha)$ along the complex direction of the orbit that we have described, 
but rather by moving~$(J,\alpha)$ to a point of~$\mscr{G}^{c}.(J,\alpha)$ along a different path, or along a different parametrization of the same transversal path that described in the proof of Proposition~\ref{prop:complessificazione_cotJ}.

As we are not yet able to address this issue, we propose to study a slightly different version of system~\eqref{eq:complessificate_definitive}, obtained by decoupling the ``Higgs term'' and the K\"ahler form. For a fixed complex structure~$J$ on the compact manifold~$M$ and a K\"ahler class~$[\omega]$ on~$M$, we look for a K\"ahler form~$\omega'\in[\omega]$ and a deformation of the complex structure~$\alpha\in\m{A}^{1,0}(T^{0,1}M)$ such that
\begin{equation}\label{eq:HCSCK_complex_relaxed}
\begin{cases}
\omega'(J-,\alpha^\transpose-)+\omega'(\alpha^\transpose-,J-)=0;\\
\f{m}_{\bm{\Omega_I}}\left(\omega',\alpha\right)=0;\\
\f{m}_{\bm{\Theta}}\left(\omega',\alpha\right)=0.
\end{cases}
\end{equation}
Every solution to~\eqref{eq:complessificate_definitive} for some K\"ahler form~$\omega$ on~$M$ and some~$t>0$ gives a solution to~\eqref{eq:HCSCK_complex_relaxed}. The vice-versa does not hold, since for an integrable Higgs term~$\alpha$ we know by Remark~\ref{rmk:int_Higgs_class} that we should also impose the condition that~$\alpha$ belongs to a fixed class in~$H^1(T^{0,1}M)$. System~\eqref{eq:HCSCK_complex_relaxed} has the advantage over~\eqref{eq:complessificate_definitive} of not depending on the background metric~$\omega$ but just on the K\"ahler class, and the results in Chapter~\ref{chap:examples} about solutions to the system~\eqref{eq:HCSCK_complex_relaxed} suggest that it might be interesting to also consider non-integrable deformations of the complex structure. For the rest of the thesis we will refer to~\eqref{eq:HCSCK_complex_relaxed} as the \emph{HcscK system}.

\section{The HcscK system on curves and surfaces}\label{sec:curves_surfaces}

In this Section we examine the moment map equations when the base manifold~$M$ has complex dimension~$1$ and~$2$. In these cases the equations are easier to study, and in particular we will recover Donaldson's equations on a complex curve (c.f.~\cite{Donaldson_hyperkahler}) by considering the complexified equations~\eqref{eq:HCSCK_complex_relaxed}. We then show how the real moment map equation can be written in an alternative way for complex surfaces, and we show how the equation becomes simpler when considering deformations of the complex structure of non-maximal rank.

\subsection{The real moment map on a curve}\label{section:moment_map_curve}

In dimension~$1$ the expression for the real moment map, equation~\eqref{eq:mm_reale_completa}, becomes much simpler. Indeed~$A^{1,0}A^{0,1}$ is just the multiplication by some number, so that
\begin{equation*}
A^{0,1}A^{1,0}=\mrm{Tr}\left(A^{1,0}A^{0,1}\right)\mathbbm{1}
\end{equation*}
and the matrix~$\hat{A}$ of Lemma~\ref{lemma:psi_spectral}, that is needed to compute the real moment map, is just
\begin{equation*}
\hat{A}=\frac{1}{2}\left(1+\sqrt{1-\norm{A^{1,0}}^2_{g_J}}\right)\mathbbm{1}.
\end{equation*}
Then on a curve we can rewrite some of the terms in~\eqref{eq:mm_reale_completa}:
\begin{equation*}
-2\nabla^*(A^{0,1}A^{1,0}\hat{A})=2\,\mrm{grad}\left(\psi\,\norm{A^{1,0}}^2\right)
\end{equation*}
and we also have
\begin{equation*}
\begin{gathered}
\left(g(\nabla^aA^{0,1},A^{1,0})-g(\nabla^aA^{1,0},A^{0,1})\right)\diff_{a}=\\=-\mrm{grad}\left(\norm{A^{1,0}}^2\right)^{1,0}+2\,g(\nabla^aA^{0,1},A^{1,0})\diff_{z^a}.
\end{gathered}
\end{equation*}
The map~$\f{m}$ of equation~\eqref{eq:mappa_momento_reale} becomes, if~$\psi:=\psi(\norm{A^{1,0}}^2)$
\begin{equation}\label{eq:mappa_momento_RS_parziale}
\begin{split}
\f{m}(J,\alpha)=\mrm{div}\Big[-\psi\,\mrm{grad}\left(\norm{A^{1,0}}^2\right)+2\,\mrm{grad}\left(\psi\,\norm{A^{1,0}}^2\right)+4\,\psi\,\mrm{Re}\big(g(\nabla^aA^{0,1},A^{1,0})\diff_{z^a}\big)\Big].
\end{split}
\end{equation}
A direct computation shows that
\begin{equation*}
-\psi\,\mrm{grad}\left(\norm{A^{1,0}}^2\right)+2\,\mrm{grad}\left(\psi\,\norm{A^{1,0}}^2\right)=\mrm{grad}\left(\mrm{log}\left(1+\sqrt{1-\norm{A^{1,0}}^2}\right)\right)
\end{equation*}
and we find a simple expression for~$\f{m}$:
\begin{equation*}
\f{m}(J,\alpha)=\Delta\left(\mrm{log}\left(1+\sqrt{1-\norm{A^{1,0}}^2}\right)\right)+\mrm{div}\left(4\,\psi\,\mrm{Re}\left(g(\nabla^aA^{0,1},A^{1,0})\diff_a\right)\right)
\end{equation*}
where we recall that~$A=\mrm{Re}(\alpha)^\transpose$, i.e.~$A^{0,1}=\frac{1}{2}\alpha^\transpose$. Summing up these computations, the real moment map on a Riemann surface is (c.f. equation~\eqref{eq:mm_reale_completa})
\begin{equation}\label{eq:momento_omega_alpha}
\begin{split}
\f{m}_{\bm{\Omega}_{\bm{I}}}(J,\alpha)=&2\,S(J)-2\,\hat{S}+\Delta\left(\mrm{log}\left(1+\sqrt{1-\norm{A^{1,0}}^2}\right)\right)+\\
&+\mrm{div}\left(4\,\psi\,\mrm{Re}\left(g(\nabla^aA^{0,1},A^{1,0})\diff_a\right)\right).
\end{split}
\end{equation}

\subsubsection{A system of equations studied by Donaldson}

We turn now to the complexified system of equations~\eqref{eq:HCSCK_complex_relaxed}, so we fix a complex structure~$J$ on~$M$ and a K\"ahler class~$[\omega_0]$. Notice first that in dimension~$1$ every~$\alpha\in T_J^*\!\mscr{J}$ is compatible with any K\"ahler form, since~$\omega(\alpha-,J-)$ is certainly symmetric. Then we look for a metric~$\omega\in[\omega_0]$ and a ``Higgs field''~$\alpha\in\mrm{Hom}({T^{0,1}}^*M,{T^{1,0}}^*M)$ that satisfy the system of equations:
\begin{equation}\label{eq:equazioni_curva_alpha}
\begin{cases}\frac{1}{4}\norm{\alpha}^2<1;\\
\mrm{div}(\diff^*\alpha)=0;\\
2\,S(\omega)+\Delta\left(\mrm{log}\left(1+\sqrt{1-\frac{1}{4}\norm{\alpha}^2}\right)\right)+\mrm{div}\left(\psi\,Q(\omega,\alpha)\right)=2\,\hat{S},
\end{cases}
\end{equation}
where all the metric quantities are computed from the metric defined by~$\omega$ and~$J$, and~$Q(\omega,\alpha):=\mrm{Re}\left(g(\nabla^a\alpha,\bar{\alpha})\diff_a\right)$. It is more convenient to write the equations in~\eqref{eq:equazioni_curva_alpha} not in terms of~$\alpha$ but rather in terms of the \emph{quadratic differential}~$ q~$ defined by
\begin{equation*}
 q :=\frac{1}{2}\alpha\indices{_a^{\bar{b}}}\,g_{\bar{b}c}\,\mrm{d}z^a\odot\mrm{d}z^c;
\end{equation*}
using this object, equations~\eqref{eq:equazioni_curva_alpha} become
\begin{equation}\label{eq:equazioni_curva}
\begin{cases}
\norm{ q }^2<1;\\
\mrm{div}(\diff^* q )^\sharp=0;\\
2\,S(\omega)+\Delta\left(\mrm{log}\left(1+\sqrt{1-\norm{ q }^2}\right)\right)+\mrm{div}\left(\psi\,Q(\omega, q )\right)=2\,\hat{S}.
\end{cases}
\end{equation}
We can make the the second equation in~\eqref{eq:equazioni_curva} more explicit by using holomorphic local coordinates (with respect to the fixed complex structure); recall that we are working on a Riemann surface, so we just have one index, when working in coordinates:
\begin{equation*}
\mrm{div}(\diff^* q )^\sharp=-g^{1\bar{1}}\diff_{\bar{z}}\left(g^{1\bar{1}}\diff_{\bar{z}} q _{11}\right).
\end{equation*}
This shows that the second equation in~\eqref{eq:equazioni_curva} is certainly satisfied when~$ q~$ is a \emph{holomorphic} quadratic differential, i.e. when~$\bdiff q =0$. The space of such objects has dimension~$3\,g(M)-3$, and in particular is non-empty, for~$g(M)>1$. Notice that, while the second equation in~\eqref{eq:equazioni_curva} depends on the choice of~$\omega$ in the fixed K\"ahler class, the simpler condition~$\bdiff q =0$ does not.
\begin{rmk}\label{rmk:hol_quad_diff}
If~$ q~$ is a holomorphic quadratic differential then the vector field~$Q(\omega, q )$ vanishes, for every K\"ahler form~$\omega$. Indeed
\begin{equation*}
g(\nabla^a q ,\bar{ q })\diff_a=g^{a\bar{b}}g^{c\bar{e}}g^{d\bar{f}}\,\nabla_{\bar{b}} q _{cd}\, q _{\bar{e}\bar{f}}\,\diff_a=0.
\end{equation*}
\end{rmk}
So, assuming that~$ q~$ is a holomorphic quadratic differential on the Riemann surface~$M$, the complexified HcscK system reduces to the equation
\begin{equation}\label{eq:F_tau_seconda}
2\,S(\omega)-2\,\hat{S}+\Delta\left(\mrm{log}\left(1+\sqrt{1-\norm{ q }^2}\right)\right)=0.
\end{equation}
As was already mentioned in the Introduction, this equation has been already studied by Donaldson in~\cite{Donaldson_hyperkahler} and by Hodge in~\cite{Hodge_phd_thesis} (see also~\cite{traut}). In particular in~\cite{Hodge_phd_thesis} it is shown that if the norm of~$ q~$ and its derivative is small enough with respect to the hyperbolic metric~$\omega_0$ of~$M$, then there is a unique solution~$\omega$ of equation~\eqref{eq:F_tau_seconda} in the conformal class of~$\omega_0$.

\subsection{The real moment map on complex surfaces}\label{sec:complex_surface}

In this Section we will describe an alternative and more explicit expression for the real moment map equation on a complex surface, alternative to the one found in equation~\eqref{eq:mappa_momento_reale}. This is because the matrix~$\hat{A}=\frac{1}{2}\left(1+\sqrt{1-A^{0,1}A^{1,0}}\right)^{-1}$ is not easy to compute explicitly, in general, in dimension greater than~$1$. However in dimension~$2$ we can write the differential of the Biquard-Gauduchon functional in a different way.

On a surface the real moment map depends on the two eigenvalues~$\delta^+$,~$\delta^-$ of~$A^{1,0}A^{0,1}$, and the differential of the Biquard-Gauduchon function~$\rho$ is, according to Lemma~\ref{lemma:differenziale_BiquardGauduchon}
\begin{equation*}
\mrm{d}\rho_{J,A}(\dot{J},\dot{A})=\frac{\mrm{Tr}\left(A\,\dot{A}(\delta^+\mathbbm{1}-A^2)\right)}{2\left(1+\sqrt{1-\delta^-}\right)(\delta^+-\delta^-)}-\frac{\mrm{Tr}\left(A\,\dot{A}(\delta^-\mathbbm{1}-A^2)\right)}{2\left(1+\sqrt{1-\delta^+}\right)(\delta^+-\delta^-)}.
\end{equation*}
However in complex dimension~$2$ the eigenvalues can be expressed as
\begin{equation}\label{eq:autovalori_dim2}
\delta^{\pm}(A)=\frac{1}{2}\left(\mrm{Tr}\left(A^{1,0}A^{0,1}\right)\pm\sqrt{\left(\mrm{Tr}\left(A^{1,0}A^{0,1}\right)\right)^2-4\,\mrm{det}(A^{1,0}A^{0,1})}\right)
\end{equation}
or, equivalently, as
\begin{equation*}
\delta^{\pm}(A)=\frac{1}{2}\left(\frac{\mrm{Tr}(A^2)}{2}\pm\sqrt{\left(\frac{\mrm{Tr}(A^2)}{2}\right)^2-4\,\mrm{det}(A)}\right).
\end{equation*}
These expressions allows us to compute the derivative of~$\delta^\pm$ along a path~$(J_t,A_t)\in T\!\!\mscr{J}$ without using the results of~\cite{Magnus_differential_eigenvalues}, thus finding an alternative expression for~$\mrm{d}\rho$ in terms of the \emph{adjugate matrix} of~$A$, which we denote provisionally by~$\tilde{A}:=\mrm{adj}(A)$. %
\nomenclature[adjugate]{$\mrm{adj}(A)$}{adjugate matrix to~$A$, the transpose of the cofactor matrix of~$A$, satisfying~$A\,\mrm{adj}(A)=\mrm{adj}(A)\,A=\mrm{det}(A)\mathbbm{1}$}%
This is the transpose of the cofactor matrix of~$A$, and it appears when computing the differential of the determinant, by Jacobi's formula. Then the differential of~$\rho$ can be expressed as
\begin{equation}\label{eq:diff_rho_alt}
\begin{split}
\mrm{d}\rho_{J,A}(\dot{J},\dot{A})=&\frac{\mrm{Tr}(A\dot{A})}{2\left(\sqrt{1-\delta^+}+\sqrt{1-\delta^-}\right)}-\\
&-\frac{\mrm{Tr}(\mrm{adj}(A)\dot{A})}{2\left(\sqrt{1-\delta^+}+\sqrt{1-\delta^-}\right)\left(1+\sqrt{1-\delta^+}\right)\left(1+\sqrt{1-\delta^-}\right)}.
\end{split}
\end{equation}
To see that the two expressions are the same, one should check that in dimension~$2$ the following identity holds:~$\mrm{adj}(A)=\frac{1}{2}\mrm{Tr}(A^2)\,A-A^3$.

The real moment map can also be rewritten using this alternative expression for~$\mrm{d}\rho$. To do so, it will be convenient to introduce the quantities
\begin{equation}\label{eq:psi_psitilde}
\begin{split}
\psi_1(A)&=\frac{1}{2}\left(\sqrt{1-\delta^+(A)}+\sqrt{1-\delta^-(A)}\right)^{-1};\\
\psi_2(A)&=\psi_1(A)\left(1+\sqrt{1-\delta^+(A)}\right)^{-1}\left(1+\sqrt{1-\delta^-(A)}\right)^{-1}.
\end{split}
\end{equation}
The same process used to obtain equation~\eqref{eq:mappa_momento_reale} can also be carried out with a few differences starting from~\eqref{eq:diff_rho_alt}, giving this form of~$\f{m}(J,\alpha)$:
\begin{equation}\label{eq:HCSCK_surface_complete}
\begin{split}
&\f{m}(J,\alpha)=\mrm{div}\left[2\,\psi_1\,\mrm{Re}\Big(g(\nabla^aA^{0,1}-\nabla^aA^{1,0},A)\diff_{z^a}\Big)-2\,\nabla^*(\psi_1\,A^2)\right]-\\
&\phantom{\f{m}(J,\alpha)}
-\mrm{div}\left[2\,\psi_2\,\mrm{Re}\left(g\left(\nabla^aA^{0,1}-\nabla^aA^{1,0},\tilde{A}\right)\diff_{z^a}\right)+2\,\mrm{grad}(\psi_2\,\mrm{det}(A))\right]=\\
&=\mrm{div}\left[-\psi_1\,\mrm{grad}\left(\norm{A^{1,0}}^2\right)+4\,\psi_1\,\mrm{Re}\left(g(\nabla^aA^{0,1},A^{1,0})\diff_{z^a}\right)
-2\,\nabla^*(\psi_1\,A^2)\right]-\\
&\quad-\mrm{div}\left[2\,\psi_2\,\mrm{Re}\left(g\left(\nabla^aA^{0,1}-\nabla^aA^{1,0},\tilde{A}\right)\diff_{z^a}\right)+2\,\mrm{grad}(\psi_2\,\mrm{det}(A))\right].
\end{split}
\end{equation}
\begin{rmk}
Since~$\tilde{A}$ is the adjugate of~$A$ and is also an element of~$T_J\mscr{J}$,
\begin{equation*}
g(A^{1,0},\tilde{A}^{0,1})=\mrm{Tr}\left(A^{1,0}\tilde{A}^{0,1}\right)=2\,\mrm{det}(A)=g(A^{0,1},\tilde{A}^{1,0}).
\end{equation*}
This identity can be used in some situations to further simplify~\eqref{eq:HCSCK_surface_complete}.
\end{rmk}
There are some conditions under which~$\f{m}(J,\alpha)$ becomes much simpler. If~$A$ does not have maximal rank then~$\mrm{det}(A)=0$; moreover, since the rank of~$A$ is even (the kernel of~$A$ is~$J$-invariant), if~$\mrm{rk}(A)$ is not maximal then actually~$\mrm{rk}(A)=0$ or~$2$, so also~$\mrm{adj}(A)=0$. If~$A$ does not have maximal rank then~$\delta^-(A)=0$, and~$\delta^+(A)=\frac{1}{2}\mrm{Tr}(A^2)=\norm{A^{1,0}}_{g_J}^2$. Hence
\begin{equation*}
\psi_1=\frac{1}{2}\left(\sqrt{1-\delta^+(A)}+\sqrt{1-\delta^-(A)}\right)^{-1}=\frac{1}{2}\left(1+\sqrt{1-\norm{A^{1,0}}^2}\right)^{-1}.
\end{equation*}
In the low-rank case~$\f{m}(J,\alpha)$ becomes
\begin{equation}\label{eq:low_rank}
\mrm{div}\left[-\psi_1\,\mrm{grad}\left(\norm{A^{1,0}}^2\right)+4\,\psi_1\,\mrm{Re}\left(g(\nabla^aA^{0,1},A^{1,0})\diff_a\right)-2\nabla^*\left(\psi_1\,A^2\right)\right]
\end{equation}
The resulting moment map is remarkably similar to the one we had in the Riemann surface case, c.f. equation~\eqref{eq:mappa_momento_RS_parziale}.

\chapter{Examples of solutions to the HcscK system}\label{chap:examples}

\chaptertoc{}

\bigskip

In this chapter we collect some results about the existence of solutions to the HcscK system, particularly on curves and surfaces.

The first section, where we generalize some of the results of~\cite{Hodge_phd_thesis} to the HcscK system, is taken from~\cite{ScarpaStoppa_HcscK_curve}. The section is devoted to the proof of Theorem~\ref{thm:curve_existence_intro}, that is based on a change of variables introduced by Donaldson to study the HcscK system on a curve~\eqref{eq:equazioni_curva} in the special case when~$q$ is a holomorphic quadratic differential. The set of solutions described in Theorem~\ref{thm:curve_existence_intro} is obtained essentially by a continuity method that deforms the usual cscK equation to the real moment map equation in~\eqref{eq:equazioni_curva}, so it considers the HcscK system as a perturbation of the cscK equation.

The goal of Section~\ref{sec:ruled_surface} is to find a different kind of solutions to the HcscK system, not given by perturbation around a cscK metric. In fact, we find an example of a surface on which there is no metric of constant scalar curvature, but for which we can find a non-zero Higgs field and a metric solving the HcscK system. The result is not completely satisfactory, since in fact the metric and the Higgs term are not \emph{compatible}, i.e. they do not satisfy the first condition in~\eqref{eq:HCSCK_complex_relaxed}. This poses some difficulty in the interpretation of the real moment map equation, as will be addressed in Remark~\ref{rmk:surf_eq_complex_nonequiv}. The results of this second section already appeared in~\cite{ScarpaStoppa_hyperk_reduction}.

\section{The HcscK system on a curve}\label{sec:curve}

Let~$\Sigma$ be a compact oriented surface of genus~$g(\Sigma)>1$, let~$\omega$ be an area form on~$\Sigma$ and consider the space~$\mscr{J}=\mscr{J}(\omega)$ of complex structures compatible with~$\omega$. In this setting the group~$\mscr{G}$ acting on~$\mscr{J}$ consists of exact area-preserving diffeomorphisms of~$\Sigma$, and as a particular case of the discussion in Chapter~\ref{chap:classicalScalCurv} we have a well-defined K\"ahler reduction~$\m{M}=\mu^{-1}(0)/\mscr{G}$.

It is important to note that~$\m{M}$ is not the Teichm\"uller space~$\m{T}$ of~$\Sigma$, since we are taking the quotient of~$\mscr{J}$ under the group~$\mscr{G}$ rather than the group~$\mscr{G}^+$ of all area-preserving diffeomorphisms isotopic to the identity. The quotient of~$\mscr{G}^+$ by~$\mscr{G}$ is the~$2g$-dimensional torus~$\m{A}_\Sigma=H^1(\Sigma, \mathbb{R})/H^1(\Sigma, \mathbb{Z})$; in fact~$\m{M}$ can be identified with the moduli space of marked Riemann surfaces~$(\Sigma, J)$, together with a choice of a holomorphic line bundle on~$\Sigma$ of fixed degree (see the last paragraph in~\cite[\S$2.2$]{Donaldson_hyperkahler}). The Teichm\"uller space of~$\Sigma$ can then be obtained as the quotient of~$\m{M}$ by~$\m{A}_\Sigma$.

In this section we will study the complexified HcscK system~\eqref{eq:equazioni_curva}, whose solution should allow us to describe a hyperk\"ahler thickening of~$\m{M}$ inside~$T^*\!\m{M}$. More precisely, we fix a marked Riemann surface~$(\Sigma,J)$ together with a K\"ahler class~$\Omega$, and we consider the following system of equations for a quadratic differential~$q\in\Gamma(K^2_\Sigma)$ and a K\"ahler form~$\omega\in\Omega$
\begin{equation}\label{eq:HcscK_curve_originale}
\begin{dcases}
\norm{q}^2_{\omega}<1;\\
{\nabla^{1,0}_{\omega}}^*{\nabla^{1,0}_{\omega}}^*q=0;\\
2\,S(\omega)-2\,\widehat{S(\omega)}+\Delta_{\omega}\,\mrm{log}\left(1+\sqrt{1-\norm{q}_{\omega}^2}\right) +\mrm{div} \frac{2\,\mathrm{Re}\,(g(\bdiff q, \bar{q}))^{\sharp}}{1+\sqrt{1-\norm{q}_{\omega}^2}} =0 
\end{dcases}
\end{equation}
In the original paper~\cite{Donaldson_hyperkahler}, Donaldson was interested in solutions to~\eqref{eq:HcscK_curve_originale} given by a \emph{holomorphic} quadratic differential~$q$, since these special solutions can be used to define a hyperk\"ahler extension of the Weil-Petersson metric on the Teichm\"uller space~$\m{T}$ of~$\Sigma$ to an open subset of~$T^*\m{T}$. Under this holomorphicity condition the system decouples and reduces to equation~\eqref{eq:F_tau_seconda}. From our point of view however the restriction to holomorphic quadratic differentials is not very natural, and we will consider more general solutions to the original coupled system~\eqref{eq:HcscK_curve_originale}.

Let~$\omega_0$ be the K\"ahler form of constant scalar curvature in~$\Omega$. Then we state more precisely our main theorem of existence on curves (c.f. Theorem~\ref{thm:curve_existence_intro}) as
\begin{thm}\label{thm:curve_MainThm}
The system~\eqref{eq:HcscK_curve_originale} admits a set of solutions whose points are in bijection with pairs~$(\tau,\beta)$, consisting of a holomorphic quadratic differential~$\tau$ and a holomorphic~$1$-form~$\beta$, such that~$\norm{\tau}_{\m{C}^{0,\frac{1}{2}}(\omega_0)}\!\!\!\!< c_1$,~$\norm{\beta}_{\m{C}^{1,\frac{1}{2}}(\omega_0)}\!\!\!\!<c_2$ for certain~$c_1$,~$c_2 > 0$. The constants~$c_1$,~$c_2$ depend on~$(\Sigma, J)$ only through a few Sobolev and elliptic constants with respect to the hyperbolic metric~$\omega_0$.
\end{thm}
An application of the Implicit Function Theorem would give quite easily the result above for some~$c_1$,~$c_2 > 0$, but much of the work here goes into proving the stronger characterization in terms of Sobolev and elliptic constants of the hyperbolic metric. The precise constants which play a role will be made clear in the course of the proof, and with a little effort the dependence upon these constants could be made completely explicit.

A consequence of Theorem~\ref{thm:curve_MainThm} is the construction of a hyperk\"ahler structure on an open neighbourhood of the zero section in~$T^*\m{M}$. The hyperk\"ahler thickening of Teichm\"uller space considered by Donaldson is then a quotient of the locus~$\beta=0$ by the torus~$\m{A}_\Sigma$, see~\cite[\S$3.1$, pag.~$185$]{Donaldson_hyperkahler}. The open neighbourhood on which the hyperk\"ahler metric is defined can be controlled in terms of the hyperbolic geometry of~$\Sigma$.

According to our previous discussion of the space~$\m{M}$, the holomorphic cotangent space~$T^*\m{M}$ can be identified with the moduli space of collections consisting of a marked Riemann surface~$(\Sigma, J)$ together with a holomorphic line bundle of fixed degree, a holomorphic quadratic differential~$\tau$ and a holomorphic~$1$-form~$\beta$, so that it is a hyperk\"ahler manifold of complex dimension~$2g+2(3g-3)$: the differential~$\beta$ parametrizes the cotangent space of the torus~$\m{A}_\Sigma$.   

\begin{cor}
There is an open subset of the space of collections~$T^*\m{M} = \{[(\Sigma, J, L, \tau, \beta)]\}$, described by the conditions
\begin{equation*}
\norm{\tau}_{\m{C}^{0,\frac{1}{2}}(\omega_0)} < c_1(\omega_0)\mbox{ and }\norm{\beta}_{\m{C}^{1,\frac{1}{2}}(\omega_0)}<c_2(\omega_0),
\end{equation*}
carrying an incomplete hyperk\"ahler structure, induced by the hyperk\"ahler reduction of~$\cotJ$ by~$\mscr{G}$.  
\end{cor}

The rest of the section is devoted to the proof of Theorem~\ref{thm:curve_MainThm}. In Section~\ref{sec:complex_mm_curve} we first show that solutions to~\eqref{eq:HcscK_curve_originale}, if they exist, are parametrised a priori by pairs~$(\tau,\beta)$ as above. The pair~$(0,0)$ corresponds to the unique hyperbolic metric~$\omega_0$. Then, following an idea of Donaldson, we perform a conformal transformation of the unknown metric~$\omega$ which brings the real moment map equation to a much simpler form. But in our case this has the cost of turning the linear complex moment map equation into a more complicated quasi-linear equation. 

In Section~\ref{sec:ContinuitySec} we introduce a continuity method for solving this equivalent system of equations. It is given simply by deforming a given pair~$(\tau, \beta)$ to~$(t \tau, t\beta)$ for~$t\in[0,1]$. In Section~\ref{sec:Estimates_curve} we proceed to establish~$\m{C}^{2,\frac{1}{2}}(\omega_0)$ a priori estimates on solutions~$\omega_t$,~$q_t$, and to show that the condition~$\norm{q_t}^2_{\omega_t} < 1$ is closed along the continuity path. The latter fact requires to control the growth of the norm~$\norm{\omega_t}_{\m{C}^{0,\frac{1}{2}}(\omega_0)}$, which we can achieve provided the norms~$\norm{\tau}_{\m{C}^{0,\frac{1}{2}}(\omega_0)}$,~$\norm{\beta}_{\m{C}^{1}(\omega_0)}$ are sufficiently small, depending only on a few Sobolev constants of~$\omega_0$, as well as elliptic constants for the Bochner Laplacian~$\nabla^*_{\omega_0}\nabla_{\omega_0}$ acting on~$1$-forms and the Riemannian Laplacian~$\Delta_{\omega_0}$ acting on functions. Finally in Section~\ref{sec:Openness_curve} we show that the linearization of the operator corresponding to our equations is an isomorphism. For this we need to take~$\norm{\beta}_{\m{C}^{1, \frac{1}{2}}(\omega_0)}$ sufficiently small, again in terms of an elliptic constant for the Riemannian Laplacian~$\Delta_{\omega_0}$ on functions. Thus our continuity path is also open, and moreover the parametrization by~$(\tau, \beta)$ is bijective.
 
\subsection{The complex moment map}\label{sec:complex_mm_curve}

Let us focus on the first equation in~\eqref{eq:HcscK_curve_originale}, corresponding to the complex moment map. We want to show that for any fixed~$\omega$ we can parametrize the solutions~$q$ to
\begin{equation*}
\nabla^{1,0}_\omega{}^*\nabla^{1,0}_\omega{}^*q=0.
\end{equation*}
Notice that the kernel of the operator~$\nabla^{1,0}{}^*:\m{A}^{1,0}(\Sigma)\to\m{C}^\infty_0(\Sigma)$ is~$H^0(K_\Sigma)$, since 
\begin{equation*}
\nabla^{1,0}{}^*\beta=-g^{1\bar{1}}\diff_{\bar{z}}\beta_1.
\end{equation*}
So in order to solve the complex moment map equation we can simply fix a holomorphic~$1$-form~$\beta$ and solve
\begin{equation}\label{eq:complex_mm_curve_abeliandiff}
\nabla^{1,0}{}^*q=\beta.
\end{equation}
In equation~\eqref{eq:complex_mm_curve_abeliandiff},~$\nabla^{1,0}{}^*$ is the formal adjoint of
\begin{equation*}
\nabla^{1,0}\!:\m{A}^{1,0}(\Sigma)\to\Gamma(K^2_\Sigma).
\end{equation*}
Since~$\nabla^{1,0}{}^*\!:\Gamma(K^2_\Sigma)\to\Gamma(K_\Sigma)$ is an elliptic operator, by the Fredholm alternative we know that there is a solution~$q$ to equation~\eqref{eq:complex_mm_curve_abeliandiff} if and only if~$\beta$ is orthogonal to the kernel of~$\nabla^{1,0}$.
\begin{lemma}
The kernel of~$\nabla^{1,0}\!:\m{A}^0(K_\Sigma)\to\m{A}^0(K_\Sigma^2)$ is trivial.
\end{lemma}
\begin{proof}
Assume that~$\eta$ is in the kernel of~$\nabla^{1,0}\!:\m{A}^0(K_\Sigma)\to\m{A}^0(K_\Sigma^2)$, and let~$X:=\bar{\eta}^\sharp\in\Gamma(T^{1,0}\Sigma)$. Then~$\nabla^{0,1}\bar{\eta}=0$, but this happens if and only if
\begin{equation*}
0=\nabla_{\bar{1}}\eta_{\bar{1}}\,(\mrm{d}\bar{z})^2=g_{1\bar{1}}\,\nabla_{\bar{1}}X^1\,(\mrm{d}\bar{z})^2
\end{equation*}
if and only if~$X$ is holomorphic. But since~$g(\Sigma)>1$ there are no non-zero holomorphic vector fields on~$\Sigma$, so~$\eta=0$.
\end{proof}
Hence for all fixed~$\beta$ there is a solution to equation~\eqref{eq:complex_mm_curve_abeliandiff}. Moreover, there is a unique solution orthogonal to the kernel of~$\nabla^{1,0}{}^*$, i.e. there is a unique solution to equation~\eqref{eq:complex_mm_curve_abeliandiff} that is in the image of~$\nabla^{1,0}$. We already computed in~\ref{rmk:hol_quad_diff} that the kernel of~$\nabla^{1,0}{}^*$ is the space of holomorphic quadratic differentials, so we deduce that for any holomorphic~$1$-form~$\beta$, any solution~$q$ of~\eqref{eq:complex_mm_curve_abeliandiff} can be written as
\begin{equation}
q=\tau+\nabla^{1,0}\eta(\beta)
\end{equation} 
where~$\tau$ is a holomorphic quadratic differential and~$\eta(\beta)$ is the (unique)~$(1,0)$--form that solves
\begin{equation*}
\nabla^{1,0}{}^*\nabla^{1,0}\eta=\beta.
\end{equation*}
Of course~$\eta(\beta)$ can be written as~$\eta=G(\beta)$, where~$G$ is the Green's operator associated to the elliptic operator~$\nabla^{1,0}{}^*\nabla^{1,0}\!:\Gamma(K_\Sigma)\to\Gamma(K^2_\Sigma)$. So the set of solutions to the complex moment map equation can be written as the~$(4g-3)$-dimensional complex vector space 
\begin{equation*}
\m{V}=\set*{\tau+\nabla^{1,0}G(\beta)\tc\beta\in H^0(K_\Sigma)\mbox{ and }\tau\in H^0(K^2_\Sigma)}.
\end{equation*}
The solutions considered in~\cite{Donaldson_hyperkahler} and~\cite{Hodge_phd_thesis} form a codimension-$g$ vector subspace of~$\m{V}$ and correspond to setting~$\beta = 0$.

Let~$L\!: \Gamma(K_\Sigma) \to \Gamma(K_\Sigma)$ be the self--adjoint elliptic operator defined by~$L(\varphi)=\nabla^{1,0}{}^*\nabla^{1,0}\varphi$. The standard Schauder estimates for elliptic operators on~$\m{C}^{k,\alpha}(\Sigma,\omega)$ tell us that there is a constant~$C=C(\omega,\alpha,k)$ such that
\begin{equation}\label{eq:elliptic_estimate_nablanabla}
\norm{\varphi}_{k,\alpha}\leq C\left(\norm{L\varphi}_{k-2,\alpha}+\norm{\varphi}_0\right),
\end{equation}
so for the Green operator we have
\begin{lemma}\label{lemma:stime_operatore_Green}
Let~$\beta\in\m{A}^{1,0}(\Sigma)$, and let~$\eta\in\m{A}^{1,0}(\Sigma)$ be the unique solution to
\begin{equation*}
\nabla^{1,0}{}^*\nabla^{1,0}\eta=\beta.
\end{equation*}
Then, for every~$k\geq 2$
\begin{equation*}
\norm{\eta}_{k,\alpha}\leq K\norm{\beta}_{k-2,\alpha}
\end{equation*}
for some constant~$K>0$ that does not depend on~$\eta$,~$\beta$.
\end{lemma}
This result is analogous to~\cite[Proposition~$2.3$]{KodairaMorrow}. The proof there is relative to the Green operator associated to the Laplacian, but it also goes through in our situation; the key points are an elliptic estimate, the linearity of the operator and its self-adjointness. We give a proof of Lemma~\ref{lemma:stime_operatore_Green} anyway, for completeness.
\begin{proof}
Let as before~$L:=\nabla^{1,0}{}^*\nabla^{1,0}$, and let~$G$ be the corresponding Green's operator. In the statement of the Lemma we have~$\eta=G(\beta)$, so we have to prove an estimate for the operator~$G$. By the elliptic estimate~\eqref{eq:elliptic_estimate_nablanabla} we have, for any~$\beta$
\begin{equation*}
\norm{G\beta}_{k,\alpha}\leq C\left(\norm{\beta}_{k-2,\alpha}+\norm{G\beta}_0\right)
\end{equation*}
so it will be enough to show that there is a constant~$C'$ such that~$\norm{G\beta}_0\leq C'\norm{\beta}_{k-2,\alpha}$ for every~$\beta$. Assume that this is not the case. Then we can find a sequence~$\beta_n$ such that
\begin{equation*}
\frac{\norm{G\beta_n}_0}{\norm{\beta_n}_{k-2,\alpha}}\to\infty
\end{equation*}
so the sequence~$\psi_n:=\frac{1}{\norm{G\beta_n}_0}\beta_n$ satisfies
\begin{equation*}
\norm{G\psi_n}_0=1\mbox{ and }\norm{\psi_n}_{k-2,\alpha}\to 0.
\end{equation*}
In particular, together with the elliptic estimate, this implies
\begin{equation*}
\norm{G\psi_n}_{k,\alpha}\leq K
\end{equation*}
for some constant~$K$. By the Ascoli-Arzelà Theorem we can assume that there is a~$\vartheta$ such that for every~$h\leq k$ we have uniform convergence~$\nabla^hG\psi_n\to\nabla^h\vartheta$, up to choosing a subsequence of~$\set{\psi_n}$. Then:
\begin{equation*}
\begin{split}
\norm{\vartheta}_{L^2}^2&=\lim \left\langle G\psi_n,\vartheta\right\rangle_{L^2}=\lim\left\langle G\psi_n,LG\vartheta\right\rangle_{L^2}=\\
&=\lim\left\langle LG\psi_n,G\vartheta\right\rangle_{L^2}=\lim\left\langle\psi_n,G\vartheta\right\rangle_{L^2}=0
\end{split}
\end{equation*}
since~$\psi_n\to 0$ in~$\m{C}^{k-2,\alpha}$. But this is a contradiction: indeed~$\norm{\vartheta}_0=\lim\norm{G\psi_n}_0=1$.
\end{proof}

In particular we deduce from Lemma~\ref{lemma:stime_operatore_Green} that for every~$\alpha\in(0,1)$, if~$\nabla^{1,0}{}^*\nabla^{1,0}\eta=\beta$ then
\begin{equation*}
\norm{\nabla^{1,0}_\omega\eta}_0\leq\norm{\eta}_{2,\alpha}\leq \tilde{C}\norm{\beta}_{0,\alpha}.
\end{equation*}
So for~$q=\tau+\nabla^{1,0}\eta$ we see that if for some~$\alpha$ the~$\m{C}^{0,\alpha}(\omega)$--norms of~$\tau$,~$\beta$ are small enough then we also have~$\norm{q}_0^2<1$, as required by the real moment map equation.

\begin{rmk}
Let us consider what happens when~$g(\Sigma) \leq 1$, that is, when~$\Sigma=\bb{CP}^1$ or~$\Sigma=\bb{C}/\Lambda$ for a lattice~$\Lambda<\bb{C}$.

In the first case~$\Sigma=\bb{CP}^1$ there are no holomorphic~$1$-forms or holomorphic quadratic differentials, so the only solution to the complex moment map equation is~$q=0$ and the HcscK system reduces to the cscK equation.

When~$\Sigma$ is a torus, if we consider systems of coordinates on~$\Sigma$ induced by affine coordinates on~$\bb{C}$ via the projection~$\bb{C}\rightarrow\bb{C}/\Lambda$, then holomorphic objects on~$\Sigma$ have constant coefficients. It is immediate then to see, by the Fredholm alternative for~$\nabla^{1,0}{}^*$, that the equation~$\nabla^{1,0}{}^*q=\beta$ can be solved precisely when the holomorphic form~$\beta$ is~$0$. In this case then~$q$ must be a holomorphic quadratic differential. By fixing an affine coordinate~$z$ on the torus then the HcscK system reduces to the equation
\begin{equation*}
\Delta\,\mrm{log}\left(g_{1\bar{1}}\left(1+\sqrt{1-(g^{1\bar{1}})^2q_{11}q_{\bar{1}\bar{1}}}\right)\right)=0
\end{equation*}
since~$g_{1\bar{1}}$ can be regarded as a (global) positive function on~$\Sigma$ and~$q_{11}$ is a constant. But then~$g_{1\bar{1}}\left(1+\sqrt{1-(g^{1\bar{1}})^2q_{11}q_{\bar{1}\bar{1}}}\right)$ must be a constant, and this happens only if~$g$ is the flat metric in its class. So, even for~$g(\Sigma)=1$, the HcscK equations essentially reduce to the cscK equation.
\end{rmk}

\subsection{A change of variables} 

The upshot of the previous section is that a unique solution~$q$ to the complex moment map equation can always be found, for a fixed metric~$\omega$, by prescribing two parameters~$\tau\in H^0(K_\Sigma^2)$,~$\beta\in H^0(K_\Sigma)$. The corresponding~$q(\omega,\tau,\beta)$ is given by
\begin{equation*}
q=\tau+\nabla^{1,0}\eta(\beta)
\end{equation*}
where~$\eta(\beta)$ is the unique solution to~$\nabla^{1,0}{}^*\nabla^{1,0}\eta=\beta$, so our system becomes
\begin{equation}\label{eq:sistema_HCSCK_beta}
\begin{dcases}
\norm{q}^2_\omega<1;\\
\nabla^{1,0}{}^*q=\beta;\\
2\,S(\omega)-2\,\widehat{S}+\Delta\,\mrm{log}\left(1+\sqrt{1-\norm{q}^2}\right)-\mrm{div}\frac{2\,\mrm{Re}\left(\bar{q}\left(-,\beta^\sharp\right)\right)^\sharp}{1+\sqrt{1-\norm{q}^2}}=0.
\end{dcases}
\end{equation}

In order to study the real moment map equation we take an approach analogous to the one in~\cite{Donaldson_hyperkahler}, by performing a change of variables. 

Let~$F:=1+\sqrt{1-\norm{q}^2_\omega}$, and consider the K\"ahler form~$\tilde{\omega}:=F\,\omega$. Notice that~$\omega$ can be recovered from~$\tilde{\omega}$ and~$q$, by~$\omega=\frac{1}{2}\left(1+\norm{q}^2_{\tilde{\omega}}\right)\tilde{\omega}$. Indeed, a quick computation shows that
\begin{equation*}
\frac{2}{1+\sqrt{1-\norm{q}^2_\omega}}=1+\frac{\norm{q}^2_\omega}{\left(1+\sqrt{1-\norm{q}^2_\omega}\right)^2}=1+\norm{q}^2_{\tilde{\omega}}
\end{equation*}
so that~$F^{-1}=\frac{1}{2}(1+\norm{q}^2_{\tilde{\omega}})$. We also have this identity for the scalar curvature:
\begin{equation*}
S(\omega)=F\,S(\tilde{\omega})-\frac{1}{2}\Delta(\mrm{log}\,F)
\end{equation*}
while for any vector field~$X$
\begin{equation*}
\mrm{div}_\omega(X)=\mrm{div}_{\tilde{\omega}}(X)-X(\mrm{log}\,F)
\end{equation*}
so that we get
\begin{equation*}
\mrm{div}_\omega\left(\frac{2}{F}\mrm{Re}\left(\bar{q}\left(-,\beta^\sharp\right)^\sharp\right)\right)=F\,\mrm{div}_{\tilde{\omega}}\left(2\,\mrm{Re}\left(\bar{q}\big(-,\beta^{\tilde{\sharp}}\big)^{\tilde{\sharp}}\right)\right).
\end{equation*}
Using these identities we see that~$\omega$ solves the second equation in~\eqref{eq:sistema_HCSCK_beta} if and only if~$\tilde{\omega}$ solves
\begin{equation*}
2\,S(\tilde{\omega})-\frac{2}{F}\widehat{S}-\mrm{div}_{\tilde{\omega}}\left(2\,\mrm{Re}\left(\bar{q}\big(-,\beta^{\tilde{\sharp}}\big)^{\tilde{\sharp}}\right)\right)=0.
\end{equation*}
These computations show that under this change of variables the HcscK system for~$(\omega,q)$ is equivalent to the following system of equations for~$\tilde{\omega}$ and~$q$
\begin{equation}\label{eq:HcscK_trasformata}
\begin{dcases}
\frac{2}{1+\norm{q}^2_{\tilde{\omega}}}{\nabla^{1,0}_{\tilde{\omega}}}^*q=\beta;\\
2\,S(\tilde{\omega})-\widehat{S}\left(1+\norm{q}^2_{\tilde{\omega}}\right)-\mrm{div}_{\tilde{\omega}}\left(2\,\mrm{Re}\left(\bar{q}\big(-,\beta^{\tilde{\sharp}}\big)^{\tilde{\sharp}}\right)\right)=0;\\
\norm{q}^2_{\tilde{\omega}}<1.
\end{dcases}
\end{equation}
We can use the first equation in~\eqref{eq:HcscK_trasformata} to rewrite the second as
\begin{equation*}
2\,S(\tilde{\omega})+\left(-\widehat{S}+\norm{\beta}^2_{\tilde{\omega}}\right)\left(1+\norm{q}^2_{\tilde{\omega}}\right)
-2\,\mrm{Re}\left(\tilde{g}(\bar{q},\nabla^{1,0}\beta)\right)=0.
\end{equation*}

\subsubsection{The equations for a conformal potential}

In order to solve our equations~\eqref{eq:HcscK_trasformata} we take the standard approach of fixing a reference K\"ahler form, still denoted by~$\tilde{\omega}$, and of looking for solutions in its conformal class, that is, of the form~$e^{f}\tilde{\omega}$. A straightforward computation shows that our equations written in terms of the unknown~$f$ become 
\begin{equation}\label{eq:HcscK_trasf_conf_pot}
\begin{dcases}
\frac{2\,\mrm{e}^{-f}}{1+\norm{\mrm{e}^{-f}q}^2_{\tilde{\omega}}}{\nabla^{1,0}_{\tilde{\omega}}}^*q=\beta;\\
2S(\tilde{\omega})+\Delta_{\tilde{\omega}}f+\left(\norm{\beta}^2_{\tilde{\omega}}-\mrm{e}^f\widehat{S}\right)\!\left(1+\norm{\mrm{e}^{-f}\!q}^2_{\tilde{\omega}}\right)=2\mrm{Re}\,\tilde{g}(\mrm{e}^{-f}\bar{q},\nabla^{1,0}_{\tilde{\omega}}\beta-\beta\otimes\diff f);\\
\norm{\mrm{e}^{-f}q}^2_{\tilde{\omega}}<1.
\end{dcases}
\end{equation}
Here~$\widehat{S}=\widehat{S(\omega)}$ is still computed using the original metric~$\omega$.

Of course we may also do things in the opposite order: we can first write our original system~\eqref{eq:sistema_HCSCK_beta} in terms of a conformal factor and then perform the change of variables described in the previous section. In fact this yields the same equations~\eqref{eq:HcscK_trasf_conf_pot}. To see this write~\eqref{eq:sistema_HCSCK_beta} in terms of a reference K\"ahler form, still denoted by~$\omega$, and a conformal metric~$\omega_f=\mrm{e}^f\omega$, giving 
\begin{equation}\label{eq:HcscK_system_curve_confpot}
\begin{dcases}
{\nabla^{1,0}_{\omega}}^*q=\mrm{e}^f\beta;\\
2\,S(\omega)+\Delta_\omega(f)+\Delta_{\omega}\,\mrm{log}\left(1+\sqrt{1-\norm{\mrm{e}^{-f}q}_{\omega}^2}\right)=\mrm{div}_{\omega}\frac{2\,\mrm{e}^{-f}\,\mrm{Re}\left(\bar{q}\left(-,\beta^\sharp\right)\right)^\sharp}{1+\sqrt{1-\norm{\mrm{e}^{-f}q}_{\omega}^2}}+2\,\mrm{e}^f\widehat{S};\\
\norm{\mrm{e}^{-f}q}^2_{\omega}<1.
\end{dcases}
\end{equation}
Notice first of all that if~$\omega_f$ satisfies the second equation in~\eqref{eq:HcscK_system_curve_confpot} then~$\omega_f$ is necessarily in the same K\"ahler class of~$\omega$, since the constant which appears is~$\widehat{S(\omega)}$ rather than~$\widehat{S(\omega_f)}$. Now we can rewrite this system in terms of~$\omega' =\left(1+\sqrt{1-\norm{\mrm{e}^{-f}q}^2_\omega}\right)\omega$ and a computation shows that this is the same as~\eqref{eq:HcscK_trasf_conf_pot}, with~$\tilde{\omega}$ replaced by~$\omega'$.

The upshot of this observation is that there is a bijection between the solutions to~\eqref{eq:HcscK_trasf_conf_pot} and those of~\eqref{eq:HcscK_system_curve_confpot}, given by mapping~$(q, \mrm{e}^f \tilde{\omega})$ to~$(q, \mrm{e}^f \omega)$, and a solution~$\mrm{e}^f\omega$ is automatically cohomologous to the original metric~$\omega$. In particular the complex moment map equation in~\eqref{eq:HcscK_trasf_conf_pot}, that is
\begin{equation}\label{eq:trasf_complex_mm_confpot}
\frac{2\,\mrm{e}^{-f}}{1+\norm{\mrm{e}^{-f}q}^2_{\tilde{\omega}}}{\nabla^{1,0}_{\tilde{\omega}}}^*q=\beta
\end{equation}
is equivalent to~$\mrm{e}^{-f}{\nabla^{1,0}_\omega}^*q=\beta$, and we know from Section~\ref{sec:complex_mm_curve} that it has solutions. 

\subsection{A continuity method}\label{sec:ContinuitySec}

In the previous section we showed that the original HcscK system is equivalent to~\eqref{eq:HcscK_trasf_conf_pot}. We will solve this system, under appropriate conditions on~$\tau$ and~$\beta$, by using a continuity method. 

It is convenient to change our notation for the background metric appearing in~\eqref{eq:HcscK_trasf_conf_pot}, denoting it simply by~$\omega$. We take the background metric~$\omega$ to have constant negative Gauss curvature. Without loss of generality we can also assume that the constant~$\widehat{S}$ in~\eqref{eq:HcscK_trasf_conf_pot} is equal to~$-2$, and we consider the family of equations~\eqref{eq:HcscK_continuity} parametrized by~$t\in[0,1]$, 
\begin{equation}
\tag{$\star_t$}\label{eq:HcscK_continuity}
\begin{split}
&\Delta f_t+\left(2\,\mrm{e}^{f_t}+\norm{t\beta}^2\right)\!\left(1+\mrm{e}^{-2\,f_t}\norm{q_t}^2\right)=2+2\,\mrm{Re}\left(g\big(\mrm{e}^{-f_t}\bar{q_t},\nabla^{1,0}\left(t\beta\right)-(t\beta)\otimes\diff f_t\big)\right)\\
&\mrm{e}^{-2\,f_t}\norm{q_t}^2<1
\end{split}
\end{equation}
where~$q_t$ is a solution to
\begin{equation*}
\frac{2\,\mrm{e}^{-f_t}}{1+\mrm{e}^{-2f_t}\norm{q_t}^2}\nabla^{1,0}{}^*q_t=t\,\beta.
\end{equation*}
Recall that the complex moment map always has a solution, from Section~\ref{sec:complex_mm_curve}. Moreover, once we fix the projection of~$q_t$ on~$\mrm{ker}\left(\nabla^{1,0}{}^*\right)$ to be~$t\tau\in H^0(K^2_\Sigma)$, there is a unique~$(1,0)$-form~$\eta_t$ such that~$q_t =t\tau+\nabla^{1,0}\eta_t$. Here all metric quantities are computed with respect to the background metric~$\omega$, as usual.
 
For~$t=0$ we have the solution~$f\equiv 0$ to~$(\star_0)$, and we propose to show that, under some boundedness assumptions of~$\tau$,~$\beta$,~$\nabla\beta$, we can find a solution~$f$ to~$(\star_1)$. To prove closedness of the continuity method we need \textit{a priori}~$\m{C}^{k,\alpha}$-estimates on~$f_t$ and~$q_t$, for some~$k\geq 2$ and some~$0<\alpha<1$. Moreover, crucially, we also need to show that the open condition~$\mrm{e}^{-2\,f_t}\norm{q_t}^2<1$ is also closed. 

Ideally, the openness of our continuity method should follow from general principles: one expects that obstructions are given by Hamiltonian Killing vector fields, which are trivial in our case since~$g(\Sigma)>1$. However, it is not clear that this general principle can be used in our problem involving a coupled system of equations. Thus we will give a more direct argument using the Implicit Function Theorem. This requires a further estimate along the continuity method.

As a preliminary step we first establish such estimates on the quadratic differential~$q$, along the continuity path, in terms of given H\"older bounds on~$\tau$,~$\beta$, and a H\"older bound on~$f$. The latter will be then proved in the following sections. In what follows all metric quantities are computed with respect to~$\omega$. We know that a solution~$q$ to~\eqref{eq:trasf_complex_mm_confpot} can be decomposed as~$q=\tau+\nabla^{1,0} \eta$ for some~$\eta\in\m{A}^{1,0}(\Sigma)$ and~$\tau\in H^0(K^2_\Sigma)=\mrm{ker}(\nabla^{1,0}{}^*)$. Thus~$\eta$ solves the equation
\begin{equation*}
\frac{2\,\mrm{e}^{-f}}{1+\mrm{e}^{-2f}\norm{\tau+\nabla^{1,0} \eta}^2 }{\nabla^{1,0} }^*\nabla^{1,0} \eta=\beta.
\end{equation*}
We write this in the form
\begin{equation*}
{\nabla^{1,0} }^* \nabla^{1,0}  \eta=\frac{1}{2}\beta\left(\mrm{e}^f+\mrm{e}^{-f}\norm{\tau+\nabla^{1,0} \eta}^2 \right)
\end{equation*}
and use the standard estimate given in Lemma~\ref{lemma:stime_operatore_Green} to show that for all~$k\geq 2$ and~$\alpha\in(0,1)$ there are constants~$C, C' > 0$ such that
\begin{equation}\label{eq:ineq_acc_curva}
\begin{split}
\norm{\eta}_{k,\alpha}&\leq C\norm{\beta}_{k-2,\alpha}\left(\norm{\mrm{e}^f}_{k-2,\alpha}+\norm{\mrm{e}^{-f}}_{k-2,\alpha}\norm*{\norm{\tau+\nabla^{1,0} \eta}^2 }_{k-2,\alpha}\right)\leq\\
&\leq C\mrm{e}^{\norm{f}_{k-2,\alpha}}\norm{\beta}_{k-2,\alpha}\left(1+C'\norm*{\tau+\nabla^{1,0}\eta}^2_{k-2,\alpha}\right)\leq\\
&\leq C\mrm{e}^{\norm{f}_{k-2,\alpha}}\norm{\beta}_{k-2,\alpha}\left(1+C'\left(\norm{\tau}^2_{k-2,\alpha}\!+\norm{\eta}^2_{k,\alpha}+2\norm{\tau}_{k-2,\alpha}\norm{\eta}_{k,\alpha}\right)\right)\!.
\end{split}
\end{equation}
In this estimate only the constant~$C$ depends on~$\omega$, and the dependence is only through the elliptic constant~$K$ appearing in Lemma~\ref{lemma:stime_operatore_Green}. Note that to go from the first to the second inequality in~\eqref{eq:ineq_acc_curva} one has to bound~$\norm*{\norm{\tau+\nabla^{1,0} \eta}^2 }_{k-2,\alpha}$ in terms of~$\norm{\tau+\nabla^{1,0} \eta}_{k-2,\alpha}$: this is possible since~$\nabla$ is the Levi-Civita connection, so the covariant derivatives of the function~$\norm{\tau+\nabla^{1,0} \eta}^2$ can be written in terms of the covariant derivatives of the tensor~$\tau+\nabla^{1,0} \eta$. Setting 
\begin{equation*}
\begin{split}
a=&C\,C'\,\mrm{e}^{\norm{f}_{k-2,\alpha}}\norm{\beta}_{k-2,\alpha}\\
b=&2\,C\,C'\,\mrm{e}^{\norm{f}_{k-2,\alpha}}\norm{\beta}_{k-2,\alpha}\norm{\tau}_{k-2,\alpha}\\
c=&C\,\mrm{e}^{\norm{f}_{k-2,\alpha}}\norm{\beta}_{k-2,\alpha}\left(1+C'\norm{\tau}^2_{k-2,\alpha}\right)\\
\end{split}
\end{equation*}
we can rewrite the inequality~\ref{eq:ineq_acc_curva} in the form
\begin{equation*}
\norm{\eta}_{k,\alpha}\leq c+b\norm{\eta}_{k,\alpha}+a\norm{\eta}^2_{k,\alpha}.
\end{equation*}
Notice that~$a$,~$b$ and~$c$ become arbitrarily small if~$\norm{\beta}_{k-2,\alpha}$ is small enough, depending on~$K$. So if~$\norm{f}_{k-2,\alpha}$,~$\norm{\beta}_{k-2,\alpha}$ and~$\norm{\tau}_{k-2,\alpha}$ satisfy a suitable bound, which only depends on~$\omega$ through~$K$, then we have~$1-b>0$ and~$(1-b)^2-4ac>0$, and we find
\begin{equation*}
0\leq\norm{\eta}_{k,\alpha}\leq\frac{1-b-\sqrt{(1-b)^2-4ac}}{2a}\mbox{ or }\norm{\eta}_{k,\alpha}\geq\frac{1-b+\sqrt{(1-b)^2-4ac}}{2a}.
\end{equation*}
Since for~$\beta=0$ the only solution to our equation is~$\eta=0$, along the continuity path~\eqref{eq:HcscK_continuity} we obtain the  bounds
\begin{equation}\label{eq:stime_eta_cambiamentocoordinate_sintetica}
\norm{\eta}_{k,\alpha}\leq\frac{1-b-\sqrt{(1-b)^2-4ac}}{2a}.
\end{equation}
In particular, for~$k=2$ we get bounds on~$\norm{\eta}_0$ in terms of the~$\m{C}^{0,\alpha}$-norms of~$\beta$,~$\tau$,~$f$. The bound~\eqref{eq:stime_eta_cambiamentocoordinate_sintetica} on~$\eta$ may be written more explicitly as
\begin{equation}\label{eq:stime_eta_cambiamentocoordinate}
\norm{\eta}_{k,\alpha}\leq C\,\mrm{e}^{\norm{f}_{k-2,\alpha}}\norm{\beta}_{k-2,\alpha}+O(\norm{\beta}_{k-2,\alpha}\norm{\tau}_{k-2,\alpha}),
\end{equation}
(where the~$O$ term depends on the background~$\omega$ only through the constant~$K$), and holds as long as
\begin{equation*}
2\,C\,C'\,\mrm{e}^{\norm{f}_{k-2,\alpha}}\norm{\beta}_{k-2,\alpha}\norm{\tau}_{k-2,\alpha}<1
\end{equation*}
and
\begin{equation*}
\begin{split}
4\,C\,C'\,\mrm{e}^{\norm{f}_{k-2,\alpha}}\norm{\beta}_{k-2,\alpha}&\left(C\,\mrm{e}^{\norm{f}_{k-2,\alpha}}\norm{\beta}_{k-2,\alpha}+\norm{\tau}_{k-2,\alpha}\right)<1.
\end{split}
\end{equation*}

\subsection{Estimates along the continuity method}\label{sec:Estimates_curve}

We now proceed to establish H\"older bounds on the conformal potential~$f$.

\paragraph{$\m{C}^0$-estimates.} Let~$(q,f)$ be a solution to~\eqref{eq:HcscK_trasf_conf_pot}. Then, at a point at which~$f$ attains its maximum we have 
\begin{equation*}
-2+\left(2\,\mrm{e}^f+\norm{\beta}^2\right)\left(1+\mrm{e}^{-2\,f}\norm{q}^2\right)-2\,\mrm{Re}\left(g\left(\mrm{e}^{-f}\bar{q},\nabla^{1,0}\beta\right)\right)\leq 0
\end{equation*}
(recall that our convention is~$\Delta=-\mrm{div}\,\mrm{grad}$, so that~$\Delta(f)$ is \emph{positive} where~$f$ attains its maximum). As we are assuming~$\mrm{e}^{-2\,f}\norm{q}^2<1$, by the Cauchy--Schwarz inequality we have
\begin{equation*}
\abs*{\mrm{Re}\left(g\left(\mrm{e}^{-f}\bar{q},\nabla^{1,0}\beta\right)\right)}\leq\norm{\nabla^{1,0}\beta}
\end{equation*}
and so, at a maximum of~$f$
\begin{equation*}
\begin{split}
0&\geq -2+\left(2\,\mrm{e}^f+\norm{\beta}^2\right)\left(1+\mrm{e}^{-2\,f}\norm{q}^2\right)-2\,\mrm{Re}\left(g\left(\mrm{e}^{-f}\bar{q},\nabla^{1,0}\beta\right)\right)\geq\\
&\geq-2+\left(2\,\mrm{e}^f+\norm{\beta}^2\right)-2\norm{\nabla^{1,0}\beta}
\end{split}
\end{equation*}
hence we find that
\begin{equation*}
\mrm{e}^f\leq 1+\norm{\nabla^{1,0}\beta}.
\end{equation*}
Similarly, at a point of minimum of~$f$ we find
\begin{equation*}
-2+\left(2\,\mrm{e}^f+\norm{\beta}^2\right)\left(1+\mrm{e}^{-2\,f}\norm{q}^2\right)-2\,\mrm{Re}\left(g\left(\mrm{e}^{-f}\bar{q},\nabla^{1,0}\beta\right)\right)\geq 0.
\end{equation*}
The same estimates then imply
\begin{equation*}
\begin{split}
0&\leq-2+\left(2\,\mrm{e}^f+\norm{\beta}^2\right)\left(1+\mrm{e}^{-2\,f}\norm{q}^2\right)-2\,\mrm{Re}\left(g\left(\mrm{e}^{-f}\bar{q},\nabla^{1,0}\beta\right)\right)\leq\\
&\leq-2+2\left(2\,\mrm{e}^f+\norm{\beta}^2\right)+2\,\norm{\nabla^{1,0}\beta}
\end{split}
\end{equation*}
so that
\begin{equation*}
2\,\mrm{e}^f\geq 1-\norm{\nabla^{1,0}\beta}-\norm{\beta}^2.
\end{equation*}
If~$\beta$ is chosen in such a way that~$\norm{\nabla^{1,0}\beta}+\norm{\beta}^2\leq 1-2\varepsilon$ then~$\mrm{e}^f$ is uniformly bounded away from~$0$ by~$\varepsilon$, and we have a~$\m{C}^0$-bound for solutions to~$(\star_1)$ (and similarly for solutions to any~\eqref{eq:HcscK_continuity}).

\paragraph{$L^4$-bounds on the gradient and the Laplacian.} Our~$\m{C}^0$-bound on~$f$ can be used to obtain an estimate for the~$L^2$-norm of~$\mrm{d}f$. Since~$f$ solves 
\begin{equation*}
\begin{split}
\Delta(f)-2&+\left(2\,\mrm{e}^f+\norm{\beta}^2\right)\left(1+\mrm{e}^{-2\,f}\norm{q}^2\right)-
2\,\mrm{Re}\left(g\left(\mrm{e}^{-f}\bar{q},\nabla^{1,0}\beta-\beta\otimes\diff f\right)\right)=0
\end{split}
\end{equation*}
the identity
\begin{equation*}
\int_{\Sigma}\norm{\mrm{d}f}^2\omega=\int_{\Sigma} f\Delta(f)\,\omega
\end{equation*}
shows that we have 
\begin{equation*}
\begin{split}
\norm{\mrm{d}f}^2_{L^2}=\int f\Big[2&-\left(2\,\mrm{e}^f+\norm{\beta}^2\right)\left(1+\mrm{e}^{-2\,f}\norm{q}^2\right)+2\,\mrm{Re}\left(g\left(\mrm{e}^{-f}\bar{q},\nabla^{1,0}\beta-\beta\otimes\diff f\right)\right)\Big]\omega.
\end{split}
\end{equation*}
Expanding out the product in the integrand, we see that the first three terms can be bounded explicitly in terms of~$\norm{\beta}_{0}$ and~$\norm{\nabla\beta}_{0}$ using the~$\m{C}^0$-bound on~$f$. As for the last term, we have by Cauchy--Schwarz
\begin{equation*}
\begin{gathered}
\int f\,2\,\mrm{Re}\left(g\left(\mrm{e}^{-f}\bar{q},\beta\otimes\diff f\right)\right)\omega=2\,\mrm{Re}\left\langle \beta\otimes\diff f,f\,\mrm{e}^{-f}q\right\rangle_{L^2}\leq\\
\leq 2\,\norm*{f\,\mrm{e}^{-f}q}_{L^2}\norm*{\beta}_{L^2}\norm*{\diff f}_{L^2}<\sqrt{2}\,\norm{f}_{L^2}\norm*{\beta}_{L^2}\norm*{\mrm{d}f}_{L^2}.
\end{gathered}
\end{equation*}
So there are some positive constants~$C_1$ and~$C_2$ that depend explicitly on our~$\m{C}^0$-bound for~$f$ and a bound for~$\norm{\beta}_0$, such that
\begin{equation*}
\norm*{\mrm{d}f}^2_{L^2}<C_1+C_2\norm*{\mrm{d}f}_{L^2},
\end{equation*}
which clearly gives a bound on the~$L^2$-norm of~$\mrm{d}f$.

Now we write our equation as
\begin{equation*}
\Delta(f)=2-\left(2\,\mrm{e}^f+\norm{\beta}^2\right)\left(1+\mrm{e}^{-2\,f}\norm{q}^2\right)+2\,\mrm{Re}\left(g\left(\mrm{e}^{-f}\bar{q},\nabla^{1,0}\beta-\beta\otimes\diff f\right)\right).
\end{equation*}
Using the~$\m{C}^0$-estimate, the condition~$\mrm{e}^{-f}\norm{q} <1$ and the Cauchy--Schwarz inequality we get
\begin{equation*}
\abs*{\Delta(f)}\leq C_3+2\abs*{\mrm{Re}\left(g\left(\mrm{e}^{-f}\bar{q},\beta\otimes\diff f\right)\right)}\leq C_3+C_4\norm{\mrm{d}f} 
\end{equation*}
for positive constants~$C_3$,~$C_4$ that depend on~$\norm{\beta}_0$,~$\norm{\nabla^{1,0}\beta}_0$ and the~$\m{C}^0$--estimate on~$f$. This implies
\begin{equation*}
\norm{\Delta f}_{L^2}\leq\norm*{C_3+C_4\norm{\mrm{d}f}_\omega}_{L^2}\leq C_3+C_4\norm{\mrm{d}f}_{L^2}
\end{equation*}
so the~$L^2$-bound on~$\mrm{d}f$ gives us a~$L^2$-bound on~$\Delta f$. The same reasoning actually shows that~$L^p$-bounds on~$\mrm{d}f$ will imply~$L^p$-bounds on~$\Delta f$.

Recall the Sobolev inequality
\begin{equation*}
\norm{u}_{L^2}\leq K_1\norm{u}_{W^{1,1}}.
\end{equation*}
In particular for~$u=\norm{\mrm{d}f}^2_\omega$ we find
\begin{equation*}
\norm*{\norm{\mrm{d}f}^2_\omega}_{L^2}\leq K_1\norm*{\norm{\mrm{d}f}^2_\omega}_{W^{1,1}}= K_1\left(\norm*{\norm{\mrm{d}f}^2_\omega}_{L^1}+\norm*{\nabla\norm{\mrm{d}f}^2_\omega}_{L^1}\right)
\end{equation*}
Now,~$\norm*{\norm{\mrm{d}f}^2_\omega}_{L^1}=\norm{\mrm{d}f}^2_{L^2}$ and~$\nabla\norm{\mrm{d}f}^2_\omega=2\,g(\nabla\mrm{d}f,\mrm{d}f)$, so by Cauchy--Schwarz
\begin{equation*}
\begin{split}
\norm*{\nabla\norm{\mrm{d}f}^2_\omega}_{L^1}=&\int\norm*{2\,g(\nabla\mrm{d}f,\mrm{d}f)}_\omega\omega=2\int\norm*{\nabla\mrm{d}f}_\omega\norm*{\mrm{d}f}_\omega\omega=\\
=&2\left\langle\norm{\nabla\mrm{d}f}_{\omega},\norm{\mrm{d}f}_\omega\right\rangle_{L^2}\leq 2\norm{\nabla\mrm{d}f}_{L^2}\norm{\mrm{d}f}_{L^2}.
\end{split}
\end{equation*}
By elliptic estimates (c.f.~\cite[Theorem~$5.2$]{SpinGeometry_book}) we have
\begin{equation*}
\norm{\nabla d f}_{L^2}\leq K_2\left(\norm{f}_{L^2}+\norm{\Delta f}_{L^2}\right).
\end{equation*}
Thus we find
\begin{equation*}
\norm*{\norm{\mrm{d}f}^2_\omega}_{L^2}\leq K_1\left(\norm{\mrm{d}f}^2_{L^2}+2\,K_2\,\norm{\mrm{d}f}_{L^2}\left(\norm{f}_{L^2}+\norm{\Delta f}_{L^2}\right)\right).
\end{equation*}
Since~$\norm*{\norm{\mrm{d}f}^2_\omega}_{L^2}=\norm{\mrm{d}f}^2_{L^4}$, from the~$L^2$-bound on~$\mrm{d}f$ and~$\Delta(f)$ that we already have we deduce an~$L^4$-bound on~$\mrm{d}f$. 

Our previous discussion then shows that we can actually obtain (explicit)~$L^4$-bounds on~$\Delta f$, in terms of~$\norm{\beta}_0$,~$\norm{\nabla^{1,0}\beta}_0$, the Sobolev constant~$K_1$ and the elliptic constant~$K_2$.

\paragraph{$\m{C}^{k,\alpha}$-bounds.} Recall Morrey's inequality for~$n=2$,~$p=4$ (c.f.~\cite[\S$5.6.2$]{Evans_elliptic_PDE}):
\begin{equation*}
\norm{f}_{ 0,\frac{1}{2} }\leq K_3\norm{f}_{W^{1,4}}.
\end{equation*}
By our~$L^4$ bound on~$df$ this implies a~$\m{C}^{0,\frac{1}{2}}$-estimate on~$f$ in terms of the~$\m{C}^0$-estimate on~$f$, the Sobolev constants~$K_1, K_3$ and the elliptic constant~$K_2$. 

Moreover, the Sobolev inequality for~$n=2$,~$p=4$ (c.f.~\cite[\S$5.6.3$]{Evans_elliptic_PDE}) tells us that
\begin{equation*}
\norm{f}_{ 3,\frac{1}{2} }\leq K_4\norm{f}_{W^{2,4}}
\end{equation*}
so our previous~$L^4$ bound on~$\Delta f$ gives  \emph{a priori} estimates for the~$\m{C}^{3,\frac{1}{2}}$-norm of~$f$ solving~$(\star_1)$ (or~\eqref{eq:HcscK_continuity} substituting~$t\beta$ to~$\beta$ in the previous discussion).

\subsubsection{Closedness}\label{sec:Closedness_curve}

We can now complete the proof of closedness for our continuity path.

Our~$\m{C}^{3,\frac{1}{2}}$-estimate for~$f$ is enough to pass to the limit as~$t \to \bar{t} \leq 1$ in the equation 
\begin{equation*}
\begin{split}
\Delta(f_t)-2+&\left(2\,\mrm{e}^{f_t}+\norm{t\beta}^2\right)\left(1+\mrm{e}^{-2\,f_t}\norm{q_t}^2\right)-2\,\mrm{Re}\left(g\left(\mrm{e}^{-f_t}\bar{q_t},\nabla^{1,0}t\beta-t\beta\otimes\diff f_t\right)\right)=0
\end{split}
\end{equation*}
for~$q_t = q(t\tau,f_t,t\beta)$. Bootstrapping then shows that the set of~$t\in[0,1]$ for which this equation has a smooth solution is closed. Moreover, the~$\m{C}^{3,\frac{1}{2}}$-estimate for~$f$ follows from the~$\m{C}^{0}$-estimate, which only requires the assumption~$\norm{\beta}_1 < 1$. 

What remains to be checked is that the quantity~$\norm{\mrm{e}^{-f_t}q_t}_0$ stays uniformly bounded away from~$1$ along the continuity path. This is where the more refined control on the growth of~$\norm{f_t}_{ 0,\frac{1}{2} }$ is required.

Our estimate~\eqref{eq:stime_eta_cambiamentocoordinate} on~$\eta$ for~$k = 2$,~$\alpha = 1/2$ immediately gives a bound on~$q$ of the form
\begin{equation*}
\begin{split}
\norm{q}_0&\leq  \norm{\tau}_0+\norm{\nabla^{1,0}\eta }_0\leq  \norm{\tau}_0+\norm{\eta }_{2,\frac{1}{2}} \leq\\
&\leq  \norm{\tau}_0+C\mrm{e}^{\norm{f}_{0,\frac{1}{2}}} \norm{\beta}_{0,\frac{1}{2}}+O( \norm{\beta}_{0,\frac{1}{2}}\norm{\tau}_{0,\frac{1}{2}}).
\end{split}
\end{equation*}
The~$O( \norm{\beta}_{0,\frac{1}{2}}\norm{\tau}_{0,\frac{1}{2}})$-terms depend on the background~$\omega$ only through the elliptic constant~$K$, and the inequality holds provided~$\norm{f}_{0,\frac{1}{2}}$,~$\norm{\beta}_{0,\frac{1}{2}}$ and~$\norm{\tau}_{0,\frac{1}{2}}$ are sufficiently small, also in terms of~$K$. But we showed that there is a uniform a priori bound on~$\norm{f}_{0,\frac{1}{2}}$, depending on the condition~$\norm{\beta}_1 < 1$, the Sobolev constants~$K_1, K_3$ and the elliptic constant~$K_2$.

It follows that we if choose~$\norm{\tau}_{0,\frac{1}{2}}$,~$\norm{\beta}_1$ small enough, depending only on the Sobolev constants~$K_1, K_3$ and the elliptic constants~$K, K_2$, then we can make sure that for all~$t \in [0, 1]$ the norm~$\norm{q_t}_0$ is sufficiently small so that the required bound
\begin{equation*}
\norm{\mrm{e}^{-f_t} q_t}_0 \leq  \mrm{e}^{\norm{f_t}_{0, \frac{1}{2}}} \norm{q_t}_{0} < 1
\end{equation*}
holds uniformly.

\subsection{Openness}\label{sec:Openness_curve}

We complete our analysis of the continuity path~\eqref{eq:HcscK_continuity} by showing that the set of times~$t \in [0,1]$ for which there is a smooth solution is open. We will see that openness requires control of a further elliptic constant, namely the~$\m{C}^{2,\frac{1}{2}}$ Schauder estimate for the Riemannian Laplacian of the hyperbolic metric acting on functions.

It is convenient to write our equations in the form 
\begin{equation}\label{eq:sistema_openness_f}
\begin{dcases}
\frac{2}{1+\norm{q}_f^2}{\nabla^{1,0}_f}^*q=\beta\\
-2\,\mrm{e}^{-f}+\Delta_f(f)+(1+\norm{q}^2_f)(2+\norm{\beta}^2_f)-2\,\mrm{Re}\left(g_f(\bar{q},\nabla^{1,0}_f\beta)\right)=0\\
 \norm{q}^2_f<1.
\end{dcases}
\end{equation}
where the notation underlines that metric quantities are now computed with respect to the metric~$\omega_f$. The last condition is clearly open, so we focus on the first two equations. These can be regarded as the zero-locus equations for the functional
\begin{equation*}
\begin{split}
\m{F}:H^0(K^2_\Sigma)\times H^0(K_\Sigma)\times\m{A}^{0}(K_\Sigma)\times\m{C}^\infty(\Sigma,\bb{R})\to\m{A}^0(K_\Sigma)\times\m{C}^\infty(\Sigma,\bb{R})&\\
(\tau,\beta,\eta,f)\mapsto(\m{F}^1(\tau,\beta,\eta,f),\m{F}^2(\tau,\beta,\eta,f))&
\end{split}
\end{equation*}
defined as 
\begin{equation*}
\m{F}^1(\tau,\beta,\eta,f)=\frac{2}{1+\norm*{\tau+\nabla^{1,0}_f\eta}_f^2}{\nabla^{1,0}_f}^*\nabla^{1,0}_f\eta-\beta
\end{equation*}
\begin{equation*}
\begin{split}
\m{F}^2(\tau,\beta,\eta,f)=-2\,\mrm{e}^{-f}+\Delta_f(f)&+\left(1+\norm*{\tau+\nabla^{1,0}_f\eta}_f^2\right)(2+\norm{\beta}^2_f)-\\&-2\,\mrm{Re}\,g_f\left(\bar{\tau}+\nabla^{0,1}_f\bar{\eta},\nabla^{1,0}_f\beta\right).
\end{split}
\end{equation*}
Assume that~$\m{F}(\tau,\beta,\eta,f)=(0,0)$. We want to show that if~$\tau',\beta'$ are close enough to~$\tau,\beta$ then we can also find~$\eta',f'$ such that~$\m{F}(\tau',\beta',\eta',f')=(0,0)$. To use the Implicit Function Theorem we should show that
\begin{equation*}
\begin{split}
\m{A}^0(K_\Sigma)\times\m{C}^\infty(\Sigma,\bb{R})&\to\m{A}^0(K_\Sigma)\times\m{C}^\infty(\Sigma,\bb{R})\\
(\dot{\eta},\varphi)&\mapsto D\m{F}_{(\tau,\beta,\eta,f)}(0,0,\dot{\eta},\varphi)
\end{split}
\end{equation*}
is surjective (on some appropriate Banach subspaces). We will show that in fact it is an isomorphism. For the rest of this section we will compute all metric quantities with respect to~$\omega_f$, unless we specify otherwise, so we will drop the subscript~$f$. As usual we write~$q=\tau+\nabla^{1,0} \eta$.

Using~$\m{F}^1(\tau,\beta,\eta,f)=0$, we compute 
\begin{equation*}
\begin{gathered}
D\m{F}^1(\dot{\eta},\varphi)=-\frac{\beta}{1+\norm{q}^2}\left(-2\,\varphi\,\norm{q}^2+2\,\mrm{Re}\left\langle\nabla^{1,0}_0\dot{\eta},\bar{q}\right\rangle\right)-\varphi\,\beta+\frac{2\,\nabla^{1,0}{}^*\nabla^{1,0} \dot{\eta}}{1+\norm{q}^2}=\\
=\varphi\,\beta\frac{\norm{q}^2-1}{1+\norm{q}^2}+\frac{2\,\nabla^{1,0}{}^*\nabla^{1,0}_0\dot{\eta}}{1+\norm{q}^2}-\frac{\beta}{1+\norm{q}^2}2\,\mrm{Re}\left\langle\nabla^{1,0}\dot{\eta},\bar{q}\right\rangle.
\end{gathered}
\end{equation*}
Similarly, using~$\m{F}^2(\tau,\beta,\eta,f)=0$, we compute
\begin{equation*}
\begin{split}
D\m{F}^2(\dot{\eta},\varphi)=&\Delta(\varphi)+2\,\varphi\left(1-\norm{q}^2(1+\norm{\beta}^2)+2\mrm{Re}\langle\bar{q},\nabla\beta\rangle\right)+\\
&+2\,\mrm{Re}\left\langle(2+\norm{\beta}^2)\bar{q},\nabla^{1,0}_0\dot{\eta}\right\rangle-2\,\mrm{Re}\left\langle\nabla\bar{\beta},\nabla^{1,0} \dot{\eta}\right\rangle.
\end{split}
\end{equation*}

To prove that~$(\dot{\eta},\varphi)\mapsto D\m{F}(\dot{\eta},\varphi)$ is an isomorphism we have to show that for any fixed~$(\sigma,h)\in\m{A}^0(K_\Sigma)\times\m{C}^{\infty}(\Sigma)$ there is a unique pair~$(\dot{\eta},\varphi)$ such that
\begin{equation}\label{eq:sistema_apertura}
\begin{dcases}
D\m{F}^1(\dot{\eta},\varphi)=\sigma\\
D\m{F}^2(\dot{\eta},\varphi)=h.
\end{dcases}
\end{equation}
Our strategy to prove this is to regard~\eqref{eq:sistema_apertura} as a deformation of the system
\begin{equation}\label{eq:sistema_deformato}
\begin{dcases}
\frac{2\,\nabla^{1,0}{}^*\nabla^{1,0} \dot{\eta}}{1+\norm{q}^2} + \varphi\,\beta\frac{\norm{q}^2-1}{1+\norm{q}^2} =\sigma\\
\Delta(\varphi)+2\,\varphi(1-\norm{q}^2)=h.
\end{dcases}
\end{equation}
Since the two operators~$\dot\eta\mapsto\nabla^{1,0}{}^*\nabla^{1,0}\dot{\eta}$ and~$\varphi\mapsto\Delta(\varphi)+2\,\varphi$ are elliptic, self-adjoint and their kernel is trivial, it is straightforward to check that~\eqref{eq:sistema_deformato} has a unique smooth solution~$(\dot{\eta},\varphi)$ for each fixed~$\sigma$,~$h$. Now the equations~\eqref{eq:sistema_apertura} differ from~\eqref{eq:sistema_deformato} from terms which vanish as~$\norm{\beta}_f$,~$\norm{\nabla^{1,0}_f\beta}_f$,~$\norm{\tau}_f$ and~$\norm{\nabla^{1,0}_f\eta}_f$ go to zero; we have shown that all these terms can be bounded in terms of~$\norm{\beta}_0$,~$\norm{\nabla^{1,0} \beta}_0$,~$\norm{\tau}_0$, effectively in terms of certain Sobolev and elliptic constants, so for~$\norm{\beta}_{\m{C}^1}$ and~$\norm{\tau}_0$ small enough we can make sure that~\eqref{eq:sistema_apertura} also have a unique smooth solution. 

\begin{lemma}[Lemma~$7.10$ in~\cite{Fine_phd}]\label{lemma:inversione}
Let~$D:B_1\to B_2$ be a bounded linear map between Banach spaces, with bounded inverse~$D^{-1}$. Then any other linear bounded operator~$L$ such that~$\norm{D-L}\leq (2\,\norm{D^{-1}})^{-1}$ is also invertible, and~$\norm{L^{-1}}\leq 2\,\norm{D^{-1}}$.
\end{lemma}
To apply this result we regard~$\m{F}$ as an operator defined on the product~$\m{C}^{0,\alpha}(\Sigma,K^2_\Sigma)\times\m{C}^{1,\alpha}(\Sigma,K_\Sigma)\times\m{C}^{2,\alpha}(\Sigma,K_\Sigma)\times\m{C}^{2,\alpha}(\Sigma,\bb{R})$ with image contained in~$\m{C}^{0,\alpha}(\Sigma,K_\Sigma)\times\m{C}^{0,\alpha}(\Sigma,\bb{R})$, and we are interested in the invertibility of the linear operator
\begin{equation*}
L(\dot{\eta},\varphi)=\left(D\m{F}^1_{(\tau,\beta,\eta,f)}(\dot{\eta},\varphi),D\m{F}^2_{(\tau,\beta,\eta,f)}(\dot{\eta},\varphi)\right).
\end{equation*}
We compare~$L$ to the auxiliary linear operator
\begin{equation*}
D(\dot{\eta},\varphi)=\left(\frac{2\,\nabla^{1,0}{}^*\nabla^{1,0}\dot{\eta}}{1+\norm{q}^2}+\varphi\,\beta\frac{\norm{q}^2-1}{1+\norm{q}^2},\Delta(\varphi)+2\,\varphi(1-\norm{q}^2)\right).
\end{equation*}

$D$ is invertible, and the norm of~$D^{-1}$ is controlled by the Schauder constants of the Laplacians~$\Delta$ and~$\nabla^{1,0}{}^*\nabla^{1,0}$. The difference between~$D$ and~$L$ is given by the operator
\begin{equation*}
\begin{split}
(D-L)(\dot{\eta}&,\varphi)=\Bigg(\frac{-\beta}{1+\norm{q}^2}2\,\mrm{Re}\left\langle\nabla^{1,0}\dot{\eta},\bar{q}\right\rangle,\\
&2\varphi\left(2\,\mrm{Re}\langle\bar{q},\nabla^{1,0}\beta\rangle-\norm{q}^2\norm{\beta}^2\right)+2\,\mrm{Re}\left\langle(2+\norm{\beta}^2)\bar{q}-\nabla\bar{\beta},\nabla^{1,0}\dot{\eta}\right\rangle\!\Bigg)
\end{split}
\end{equation*}
and we can estimate
\begin{equation*}
\begin{split}
\norm{D-L}\leq&\norm{\beta}_{1,\alpha}\left(1+\norm{q}_{0,\alpha}\right)\left(1+\norm{\beta}_{1,\alpha}\norm{q}_{0,\alpha}\right)+2\norm{q}_{0,\alpha}\left(1+\norm{\beta}_{1,\alpha}\right).
\end{split}
\end{equation*}
It is important to recall that in the present context all these norms are computed using the conformal metric~$\omega_f$. However, our H\"older estimates on the conformal potential~$f$ along the continuity method tell us that these norms are uniformly equivalent to those computed using the background hyperbolic metric~$\omega$. By also using our H\"older estimates on~$q$, it follows that we can control the norm of~$D - L$, for~$\alpha=1/2$, in terms of the norms~$\norm{\beta}_{\m{C}^{1,\frac{1}{2}}(\omega)}$,~$\norm{\tau}_{\m{C}^{0,\frac{1}{2}}(\omega)}$. If these are small enough, then by Lemma~\ref{lemma:inversione} the operator~$L$ is invertible. Finally, bootstrapping shows that a solution to~\eqref{eq:sistema_openness_f} in~$\m{C}^{2,\alpha}$ is actually smooth.

\section{The HcscK system on a ruled surface}\label{sec:ruled_surface}

Let~$\Sigma$ be a Riemann surface of genus~$g(\Sigma)\geq 2$ and assume that~$L\to\Sigma$ is a holomorphic line bundle equipped with a Hermitian fibre metric~$h$. In this section we study our equations on the ruled surface~$M=\bb{P}(\m{O}\oplus L)$ (the completion of~$L$) using the \emph{momentum construction}; our main reference for this technique is~\cite{Hwang_Singer}; see also~\cite[chapter~$5$]{Szekelyhidi_phd}). 

After this initial study we solve a ``complexified'' version of the equations in the particular case when~$L$ is the anti-canonical bundle of~$\Sigma$. We remark that we solve just a subset of equations of the complexified HcscK system~\eqref{eq:HCSCK_complex_relaxed}, namely for a fixed complex structure~$J$ we will find a K\"ahler form~$\omega_\phi$ and a Higgs field~$\alpha$ that are a zero of the moment maps, but such that~$\alpha$ and~$\omega_\phi$ are \emph{not} compatible. A consequence of the lack of compatibility is that we have at least two different equations to consider when studying the complexified real moment map equation. We will discuss this point in Remark~\ref{rmk:surf_eq_complex_nonequiv}.

Section~\ref{sec:def_complex_bundle} and Section~\ref{sec:ruled_complex_mm} are devoted to describing a set of first-order deformations of the complex structure and to characterize the deformations that give solutions to the complex moment map equation. Essentially we consider deformations of the complex structure of~$\bb{P}(\m{O}\oplus L)$ that are induced from a deformation of the bundle~$\m{O}\oplus L$. We proceed to study the real moment map equation in Section~\ref{sec:ruled_surface_realmm}, for a particular choice of a deformation of complex structure solving the complex moment map equation. Our goal is to prove that when the fibres of~$\bb{P}(\m{O}\oplus L)\to\Sigma$ are sufficiently small compared to the base then there is a solution to the real moment map equation. The study of the constant scalar curvature equation in this \emph{adiabatic limit} is a well-developed subject (see for example~\cite{Hong_fibered}), and we follow in particular the approach of~\cite{Fine_phd}.

For a fixed K\"ahler form~$\omega_\Sigma$ on~$\Sigma$, we consider K\"ahler forms on the total space of the bundle
\begin{equation*}
\bb{P}(\m{O}\oplus L)\xrightarrow{\pi}\Sigma
\end{equation*}
that satisfy the \emph{Calabi ansatz}, i.e. we consider a form~$\omega$ of the form
\begin{equation}\label{eq:metrica_ruled_surface}
\omega=\pi^*\omega_\Sigma+\I\,\diff\bar{\diff}f(t)
\end{equation}
where~$t$ is the logarithm of the fibre-wise norm function, and~$f$ is a suitably convex real function. More explicitly, we fix a system of holomorphic coordinates~$(z,\zeta)$ on~$M$ that are adapted to the bundle structure, i.e.~$z$ is a holomorphic coordinate on~$\Sigma$ while~$\zeta$ is a linear coordinate on the fibres of~$L\to\Sigma$. Let~$a(z)$ denote the local function on~$\Sigma$ such that the Hermitian metric~$h$ on~$L$ is given by~$h=a(z)\,\mrm{d}\zeta\,\mrm{d}\bar{\zeta}$; then~$t:=\mrm{log}(a(z)\,\zeta\bar{\zeta})$ is a well-defined function on~$L\setminus\Sigma$, and if~$f$ satisfies some conditions on its second derivative then~$\I\diff\bar{\diff}f(t)$ is a (globally) well-defined real~$2$-form on the total space of~$L$, that in some cases can be extended to~$M$.

Let~$F(h)$ be the curvature form of~$h$. We choose~$h$ such that~$F(h)=-\omega_\Sigma$. Then in bundle-adapted holomorphic coordinates~$\bm{w}=(z,\zeta)$ we have 
\begin{equation}\label{eq:metrica_CalabiAnsatz}
\begin{split}
\pi^*\omega_\Sigma&+\I\,\diff\bar{\diff}f(t)=(1+f'(t))\pi^*\omega_\Sigma+\\
+&\I\,f''(t)\left[\diff_zt\,\diff_{\bar{z}}t\,\mrm{d}z\wedge\mrm{d}\bar{z}+\frac{\diff_zt}{\bar{\zeta}}\mrm{d}z\wedge\mrm{d}\bar{\zeta}+\frac{\diff_{\bar{z}}t}{\zeta}\mrm{d}\zeta\wedge\mrm{d}\bar{z}+\frac{1}{\zeta\,\bar{\zeta}}\mrm{d}\zeta\wedge\mrm{d}\bar{\zeta}\right].
\end{split}
\end{equation}

It will be useful to change point of view to describe the curvature properties of the metric~$\omega$. Rather than working with~$f$ and~$t$, define~$\tau$ to be the function~$\tau=f'(t)$, and let~$F$ be the Legendre transform of~$f$. If we define~$\phi:=\frac{1}{F''}$, then we have
\begin{align*}
\tau=f'(t)\\
t=F'(\tau)\\
F(\tau)+f(t)=t\,\tau\\
f''(t)=\phi(\tau)
\end{align*}
so that the metric~$\omega_\phi:=\omega$ is, with the notation of~\eqref{eq:metrica_CalabiAnsatz}
\begin{equation}\label{eq:metrica_CalabiAnsatz_phi}
\begin{split}
\omega_\phi=&(1+\tau)\pi^*\omega_\Sigma+\\
&+\I\phi(\tau)\left(\diff_zt\,\diff_{\bar{z}}t\,\mrm{d}z\wedge\mrm{d}\bar{z}+\frac{\diff_zt}{\bar{\zeta}}\mrm{d}z\wedge\mrm{d}\bar{\zeta}+\frac{\diff_{\bar{z}}t}{\zeta}\mrm{d}\zeta\wedge\mrm{d}\bar{z}+\frac{1}{\zeta\,\bar{\zeta}}\mrm{d}\zeta\wedge\mrm{d}\bar{\zeta}\right)
\end{split}
\end{equation}
In particular, the matrices of the metric and its inverse in this system of coordinates are
\begin{equation*}
\left(g_{a\bar{b}}\right)_{1\leq a,b\leq 2}=\begin{pmatrix}(1+\tau)g_\Sigma+\phi(\tau)\,\diff_zt\,\diff_{\bar{z}}t & \phi(\tau)\frac{\diff_zt}{\bar{\zeta}} \\
\phi(\tau)\frac{\diff_{\bar{z}}t}{\zeta} & \frac{\phi(\tau)}{\zeta\,\bar{\zeta}} \end{pmatrix}
\end{equation*}
\begin{equation*}
\left(g^{a\bar{b}}\right)_{1\leq a,b\leq 2}=\begin{pmatrix}\frac{1}{(1+\tau)g_\Sigma}& -\frac{\bar{\zeta}\,\diff_{\bar{z}}t}{(1+\tau)g_\Sigma}\\
-\frac{\zeta\,\diff_zt}{(1+\tau)g_\Sigma} & \frac{\zeta\,\bar{\zeta}}{\phi(\tau)}+\frac{\zeta\,\bar{\zeta}\,\diff_zt\,\diff_{\bar{z}}t}{(1+\tau)g_\Sigma} \end{pmatrix}.
\end{equation*}
The main reason for using~$\phi(\tau)$ rather than~$f(t)$ is that it is easier to characterize the conditions under which our Calabi ansatz metric extends to the two possibly singular sections~$\Sigma_0$ and~$\Sigma_\infty$, see Proposition~\ref{prop:omega_phi_estende}. Moreover the scalar curvature of~$\omega_\phi$ has a nice expression in terms of the momentum profile~$\phi(\tau)$, as it often happens for the scalar curvature under Legendre duality (see Chapter~\ref{chap:symplectic_coords}). We will see this formula in Proposition~\ref{prop:curv_scal_phi}. Note that we are only stating a particular case of the more general results of Hwang-Singer in~\cite{Hwang_Singer}.
\begin{prop}[\cite{Hwang_Singer}, see also~\cite{Szekelyhidi_libro}]\label{prop:omega_phi_estende}
Assume that~$\phi:[a,b]\to[0,\infty)$ is a function positive on the interior of~$[a,b]$. Then~$\omega_\phi$ defines a smooth metric on~$M\setminus\Sigma_\infty$ if and only if~$\phi(a)=0$,~$\phi'(0)=1$. Moreover,~$\omega_\phi$ extends to the whole of~$M$ if and only if~$\phi(a)=\phi(b)=0$ and~$\phi'(a)=1$,~$\phi'(b)=-1$.
\end{prop}

Then it will be useful to assume that~$\tau$ takes values in an interval~$[a,b]$. The convexity assumptions on~$f$ imply that actually~$\tau$ is increasing (as a function of~$t$), and that~$\tau_{\restriction_{\Sigma_0}}=a$,~$\tau_{\restriction_{\Sigma_\infty}}=b$. Up to translations, we can assume that in fact~$[a,b]=[0,m]$ for some~$m\in\bb{R}_{>0}$. This~$m$ has a direct geometric interpretation:

\begin{lemma}\label{lemma:volume_fibre}
The volume of a fibre of~$\bb{P}(\m{O}\oplus L)\to\Sigma$ is~$2\,\pi\,m$.
\end{lemma}

\begin{proof}
We just have to compute~$\int_{F}i^*\omega_\phi$, where~$F$ is a fibre of~$\bb{P}(\m{O}\oplus L)\to\Sigma$ and~$i:F\hookrightarrow\bb{P}(\m{O}\oplus L)$ is the inclusion. Fix a system of bundle-adapted coordinates~$(z,\zeta)$ on~$\bb{P}(\m{O}\oplus L)$, and let~$r=\abs{\zeta}$. Then~$\diff_r\tau=2\,\phi(\tau)\,r^{-1}$, so
\begin{equation*}
\begin{split}
\int_{F}i^*\omega_\phi=&\int_{\zeta\in\bb{C}}\I\,\frac{\phi(\tau)}{r^2}\,\mrm{d}\zeta\,\mrm{d}\bar{\zeta}=\int_{\bb{R}^2}2\,\frac{\phi(\tau)}{r^2}\,\mrm{d}x\,\mrm{d}y=\\
=&\int_{[0,2\,\pi]\times\bb{R}}\diff_r\tau\,\mrm{d}\tau\,\mrm{d}\vartheta=2\,\pi\,m.\qedhere
\end{split}
\end{equation*}
\end{proof}

\begin{prop}[\cite{Hwang_Singer}, see also~\cite{Szekelyhidi_libro}]\label{prop:curv_scal_phi}
With the previous notation, the scalar curvature of~$\omega_\phi$ is
\begin{equation*}
s(\omega_\phi)=\frac{1}{1+\tau}\pi^*s(\omega_\Sigma)-\phi''(\tau)-\frac{2}{1+\tau}\phi'(\tau).
\end{equation*}
\end{prop}
To study the real moment map equation we also need an explicit expression for~$\widehat{s(\omega_\phi)}$.
\begin{lemma}\label{lemma:media_curvatura}
If~$\phi$ defines a K\"ahler metric on the whole ruled surface~$\bb{P}(\m{O}\oplus L)$ then
\begin{equation*}
\widehat{s(\omega_\phi)}=\frac{2}{m+2}\widehat{s(\omega_\Sigma)}+\frac{2}{m}.
\end{equation*}
\end{lemma}
\begin{proof}
We use the same notation of the proof of Lemma~\ref{lemma:volume_fibre}. First notice that
\begin{equation*}
\omega_\phi^2=-(1+\tau)g_\Sigma\frac{\phi(\tau)}{r^2}\,\mrm{d}z\wedge\mrm{d}\bar{z}\wedge\mrm{d}\zeta\wedge\mrm{d}\bar{\zeta}
\end{equation*}
so that the volume of~$M=\bb{P}(L\oplus\m{O})$ is
\begin{equation*}
\begin{split}
\mrm{Vol}_\phi(M)=&\int_M\frac{\omega_\phi^2}{2}=-\frac{1}{2}\int_\Sigma\mrm{d}z\,\mrm{d}\bar{z}\left[g_0\int_{\bb{C}}(1+\tau)\frac{\phi(\tau)}{r^2}\mrm{d}\zeta\,\mrm{d}\bar{\zeta}\right]=\\
=&\int_\Sigma\mrm{d}z\,\mrm{d}\bar{z}\,g_0\left[\pi\,\I\left(1+\frac{m}{2}\right)m\right]=\pi\frac{m(2+m)}{2}\,\mrm{Vol}_{\omega_\Sigma}(\Sigma).
\end{split}
\end{equation*}
To compute the integral of~$s(\omega_\phi)$ we use the expression in Proposition~\ref{prop:curv_scal_phi} 
\begin{equation*}
\begin{split}
\int_M s(\omega_\phi)\frac{\omega_\phi^2}{2}=&-\frac{1}{2}\int_M\!\!\!(1+\tau)g_\Sigma\frac{\phi(\tau)}{r^2}\left(\frac{s(\omega_\Sigma)}{1+\tau}-\phi''(\tau)-2\frac{\phi'(\tau)}{1+\tau}\right)\mrm{d}z\,\mrm{d}\bar{z}\,\mrm{d}\zeta\,\mrm{d}\bar{\zeta}=\\
=&-\frac{1}{2}\int_\Sigma\mrm{d}z\,\mrm{d}\bar{z}\,g_\Sigma\,s(\omega_\Sigma)\left[\int_{\bb{C}}\frac{\phi(\tau)}{r^2}\mrm{d}\zeta\,\mrm{d}\bar{\zeta}\right]+\\
&+\frac{1}{2}\int_\Sigma\mrm{d}z\,\mrm{d}\bar{z}\,g_\Sigma\left[\int_{\bb{C}}\frac{2\,\phi(\tau)\,\phi'(\tau)}{r^2}\mrm{d}\zeta\,\mrm{d}\bar{\zeta}\right]\\
&+\frac{1}{2}\int_\Sigma\mrm{d}z\,\mrm{d}\bar{z}\,g_\Sigma\left[\int_{\bb{C}}\frac{(1+\tau)\phi(\tau)\,\phi''(\tau)}{r^2}\mrm{d}\zeta\,\mrm{d}\bar{\zeta}\right].
\end{split}
\end{equation*}
We split the computation in three parts. To compute the integrals over~$\bb{C}$, we use polar coordinates.
\begin{equation*}
\int_{\bb{C}}\frac{\phi(\tau)}{r^2}\mrm{d}\zeta\,\mrm{d}\bar{\zeta}=-\I\int_0^{2\pi}\mrm{d}\vartheta\left[\int_0^\infty2\frac{\phi(\tau)}{r}\mrm{d}r\right]=-2\,\pi\,\I\,m;
\end{equation*}
\begin{equation*}
\begin{split}
\int_{\bb{C}}\frac{2\,\phi(\tau)\,\phi'(\tau)}{r^2}\mrm{d}\zeta\,\mrm{d}\bar{\zeta}=&-2\,\I\int_0^{2\pi}\mrm{d}\vartheta\left[2\int_0^\infty\frac{\phi(\tau)\,\phi'(\tau)}{r}\mrm{d}r\right]=\\
=&-2\,\I\int_0^{2\pi}\mrm{d}\vartheta\left[\phi(\tau)\right]^\infty_0=0;
\end{split}
\end{equation*}
\begin{equation*}
\begin{split}
\int_{\bb{C}}&\frac{(1+\tau)\phi(\tau)\,\phi''(\tau)}{r^2}\mrm{d}\zeta\,\mrm{d}\bar{\zeta}=-\,\I\int_0^{2\pi}\!\!\!\!\mrm{d}\vartheta\left[2\int_0^\infty\frac{(1+\tau)\phi(\tau)\,\phi''(\tau)}{r}\mrm{d}r\right]=\\
&=-\I\int_0^{2\pi}\!\!\!\!\mrm{d}\vartheta\left[\int_0^\infty\diff_r\phi'(\tau)\mrm{d}r\right]-\I\int_0^{2\pi}\!\!\!\!\mrm{d}\vartheta\left[\int_0^\infty\diff_r(\phi'(\tau)\,\tau)-\diff_r\phi(\tau)\mrm{d}r\right]=\\
&=4\,\pi\,\I+2\,\pi\,\I\,m.
\end{split}
\end{equation*}
Putting everything together:
\begin{equation*}
\begin{split}
\int_M s(\omega_\phi)\frac{\omega_\phi^2}{2}=&-\frac{1}{2}\int_\Sigma\mrm{d}z\,\mrm{d}\bar{z}\,g_\Sigma\,s(\omega_\Sigma)\left[-2\,\pi\,\I\,m\right]+\\
&+\frac{1}{2}\int_\Sigma\mrm{d}z\,\mrm{d}\bar{z}\,g_\Sigma\left[4\,\pi\,\I+2\,\pi\,\I\,m\right]=\\
=&\pi\,m\int_\Sigma s(\omega_\Sigma)\,\omega_\Sigma+(2\,\pi+\pi\,m)\,\mrm{Vol}_{\omega_\Sigma}(\Sigma)
\end{split}
\end{equation*}
and finally we find
\begin{equation*}
\widehat{s(\omega_\phi)}=\frac{\int_M s(\omega_\phi)\,\frac{\omega^2_\phi}{2}}{\mrm{Vol}_\phi(M)}
=\frac{2}{2+m}\widehat{s(\omega_\Sigma)}+\frac{2}{m}.\qedhere
\end{equation*}
\end{proof}

An analogous computation will give the K\"ahler class of~$\omega_\phi$.

\begin{lemma}[See \S~$4.4$ in~\cite{Szekelyhidi_libro}]\label{lemma:classe_phi}
Consider on~$\bb{P}(\m{O}\oplus L)$ the classes of a fibre~$\m{C}$ and the infinity section~$\Sigma_\infty$. Then the Poincaré dual to~$[\omega_\phi]$ is
\begin{equation*}
\m{L}_m:=2\,\pi\left(\m{C}+m\,\Sigma_\infty\right).
\end{equation*}
\end{lemma}

\paragraph*{Transversally normal coordinates.} For many of the computations in Section~\ref{sec:ruled_surface_realmm} it will be convenient to choose bundle-adapted holomorphic coordinates~$\bm{w}=(z,\zeta)$ such that, for a fixed point~$p\in\Sigma$,~$\left(\diff_zt\right)(p)=0$. For brevity, we will call coordinates with these properties \emph{transversally normal at}~$p$. Such a system of coordinates always exists, they are essentially just normal coordinates for the bundle metric~$h$. In these coordinates the metric~$\omega_\phi$ becomes (c.f. equation~\eqref{eq:metrica_CalabiAnsatz_phi})
\begin{equation*}
\omega(p)=(1+\tau)\pi^*\omega_\Sigma+\I\,\frac{\phi(\tau)}{\zeta\bar{\zeta}}\mrm{d}\zeta\wedge\mrm{d}\bar{\zeta}.
\end{equation*}
In particular, it will be convenient to use transversally normal coordinates whenever we have to compute objects that involve the Christoffel symbols of~$\omega_\phi$, since in these coordinates~$g_\phi$ and its inverse are diagonal.

\begin{lemma}\label{lemma:christoffel}
Let~$\Gamma(\Sigma)$ be the Christoffel symbol of~$g_\Sigma$. Then the Christoffel symbols of~$\omega_\phi$ are
\begin{align*}
\Gamma^1_{11}=&2\frac{\phi(\tau)}{1+\tau}\diff_zt+\Gamma(\Sigma); & \Gamma^2_{11}=&\zeta\diff_zt\left(\!\left(\phi'(\tau)-2\frac{\phi(\tau)}{1+\tau}\right)\diff_zt-\Gamma(\Sigma)\right)+\zeta\,\diff^2_zt;\\
\Gamma^1_{21}=&\frac{\phi(\tau)}{(1+\tau)\zeta}; & \Gamma^2_{21}=&\diff_zt\left(\phi'(\tau)-\frac{\phi(\tau)}{1+\tau}\right);\\
\Gamma^1_{22}=&0; & \Gamma^2_{22}=&\frac{\phi'(\tau)-1}{\zeta}.
\end{align*}
\end{lemma}
In particular, if we fix a point~$p\in\Sigma$ and a system of transversally normal coordinates around it, all the Christoffel symbols of~$\omega_\phi$ at the point~$p$ vanish, except for
\begin{equation}\label{christoffel}
\nonumber \Gamma^1_{11}=\Gamma(\Sigma),\quad\Gamma^1_{21}=\frac{\phi(\tau)}{(1+\tau)\zeta},\quad\Gamma^2_{22}=\frac{\phi'(\tau)-1}{\zeta}.
\end{equation}

\subsection{First-order deformations of projective bundles}\label{sec:def_complex_bundle}

The HcscK equations involve both a K\"ahler metric and a deformation of the complex structure. While in this ruled surface case we have already chosen to use K\"ahler metrics satisfying the Calabi ansatz~\eqref{eq:metrica_ruled_surface}, we still have to choose which deformations of~$\bb{P}(\m{O}\oplus L)$ to consider. The natural choice is to consider a deformation of the~$\bdiff$-operator of~$E:=\m{O}\oplus L$, so a matrix-valued form~$\beta\in\m{A}^{0,1}(\mrm{End}(E))$; this~$1$-form will induce in turn a deformation~$A\in\mrm{End}(TE)$ of the complex structure of the total space (which we still denote by~$E$).

First, recall how a~$\bdiff_E$-operator determines the complex structure~$J_E$, see~\cite[Proposition~$1.3.7$]{Kobayashi_bundles}. Fix a local holomorphic coordinate~$z$ on~$\Sigma$ and a local frame~$(s_1,s_2)$ on~$E$. If we let~$(w^1,w^2)$ be the usual coordinates on~$\mathbb{C}^2$, by the choice of the local frame we can use~$(z,w^1,w^2)$ as local complex coordinates on~$E$. Denote by
\begin{equation*}
T\indices{^i_j}:=T\indices{_{\bar{1}}^i_j}\mathrm{d}\bar{z}
\end{equation*}
the local representative of the~$\bdiff_E$-operator. A complex structure on~$E$ is uniquely determined by a decomposition~$T_{\mathbb{C}}E=T^{1,0}E\oplus T^{0,1}E$; we define
\begin{equation*}
T^{1,0}E:=\mrm{span}_{\mathbb{C}}\left(\diff_{w^1},\diff_{w^2},\diff_z-\overline{T\indices{^i_j}}(\diff_z)\bar{w}^j\diff_{\bar{w}^i}\right).
\end{equation*}
A different choice of a local frame does not change this bundle; moreover, the integrability of~$\bdiff_E$ (i.e.~$\bdiff_E^2=0$) is equivalent to that of~$T^{1,0}E$ (i.e.~$[T^{1,0}E,T^{1,0}E]\subseteq T^{1,0}E$.)

Consider now the case in which we already have a holomorphic structure~$\bdiff_E$, and we are deforming it as~$\bdiff'_E:=\bdiff_E+\beta$ for some~$\beta\in\m{A}^{0,1}(\mrm{End}(E))$. Choose a local~$\bdiff_E$-holomorphic frame~$s_1,s_2$ for~$E$. Then a local representative for~$\bdiff'_E$ in this local frame is just the matrix~$\beta$, and the previous construction gives us
\begin{equation*}
T_{\bdiff_E}^{1,0}E=\mrm{span}_{\mathbb{C}}\left(\diff_{w^1},\diff_{w^2},\diff_z\right),\quad T_{\bdiff'_E}^{1,0}E=\mrm{span}_{\mathbb{C}}\left(\diff_{w^1},\diff_{w^2},\diff_z-\overline{\beta\indices{_{\bar{1}}^i_j}}\bar{w}^j\diff_{\bar{w}^i}\right).
\end{equation*}

Changing point of view,~$\bdiff_E$ defines on the total space of~$E$ a complex structure~$J_E$, and if we slightly deform it to~$J'_E:=J_E+\varepsilon\, A$ for some~$A\in\Gamma(E,\mrm{End}(TE))$, to first order in~$\varepsilon$ the holomorphic tangent bundle of~$E$ with respect to~$J'_E$ can be described as
\begin{equation*}
T_{J'_E}^{1,0}E=\left\lbrace v-\frac{\mrm{i}\,\varepsilon}{2}A(v)\mid v\in T_{J_E}^{1,0}E\right\rbrace.
\end{equation*}
Comparing the spaces~$T_{J'_E}^{1,0}E$ and~$T_{\bdiff'_E}^{1,0}E$, we see that~$A$ induces the same deformation of~$J_E$ as~$\beta$ if and only if
\begin{equation}
\begin{split}
A^{1,0}(\diff_{\bar{w}^i})=&0\\
A^{1,0}(\diff_{\bar{z}})=&2\,\mrm{i}\,\beta\indices{^i_j}(\diff_{\bar{z}})w^j\diff_{w^i};
\end{split}
\end{equation}
we let~$A(\beta)$ be the deformation of the complex structure defined by these equations.

The next step is to see how a deformation of~$\bdiff_E$,~$\beta\in\m{A}^{0,1}(\mrm{End}(E))$ induces a deformation of the complex structure of~$\bb{P}(E)$. From the previous discussion, we have a canonical way to induce a first-order deformation~$A(\beta)\in\Gamma(\mrm{End}(TE))$ of the complex structure of~$E$. Recall the definition of~$\bb{P}(E)$ from the usual~$\bb{C}^*$-action on the fibres
\begin{equation*}
\bb{P}(E):=\left(E\setminus M\right)/\bb{C}^*.
\end{equation*}
\begin{lemma}
Let~$p:E\setminus M\to\bb{P}(E)$ be the usual projection, and fix~$\beta\in\m{A}^{0,1}(\mrm{End}(E))$. Then~$A=A(\beta)$ induces a deformation of the complex structure of~$\bb{P}(E)$ as follows: for~$[x]\in\bb{P}(E)$ and~$v\in T_{[x]}\bb{P}(E)$ choose a~$p$-lift~$\hat{v}\in T_xE$ of~$v$, and let
\begin{equation*}
A_{[x]}(v):=p_*A_{x}(\hat{v}).
\end{equation*}
\end{lemma}
\begin{proof}
We have to check that this expression does not depend upon the choice of the preimage of~$[x]$ and of the lift~$\hat{v}$ of~$v$. Fix holomorphic local frames of~$\m{O}$ and~$L$, so that we can locally describe~$E$ as~$M\times\bb{C}^2$, with coordinates~$w^1,w^2$ on the fibres. We get homogeneous coordinates on the fibres of~$\bb{P}(E)$ as~$[w^1:w^2]$. If we fix a holomorphic coordinate~$z$ on~$M$, on the open subset of~$\bb{P}(E)$ where~$w^1\not=0$ we have local holomorphic coordinates~$(z,\zeta)$, with~$\zeta=w^2/w^1$.

In this system of local coordinates the projection~$p$ is written as~$p(z,\bm{w})=\left(z,\frac{w^2}{w^1}\right)$, and (the~$(1,0)$ part of) its differential is
\begin{equation*}
\mrm{d}p_{(z,w^1,w^2)}=\begin{pmatrix}
1 & 0 & 0 \\
0 & -\frac{w^2}{(w^1)^2}& \frac{1}{w^1}
\end{pmatrix}.
\end{equation*}
We have to check that for all~$[x]\in\bb{P}(E)$ and all~$\lambda\in\bb{C}^*$, if~$\hat{v}_1\in T^{0,1}_{x}E$ and~$\hat{v}_2\in T^{0,1}_{\lambda\,x}E$ are such that~$p_*\hat{v}_1=p_*\hat{v}_2$, then also
\begin{equation*}
p_*A_{x}(\hat{v}_1)=p_*A_{\lambda\,x}(\hat{v}_2).
\end{equation*}
If~$x=(z,w^1,w^2)$ and~$\hat{v}_1=V\,\diff_{\bar{z}}+U^{\bar{i}}\diff_{\bar{w}^i}$ then
\begin{equation*}
\begin{split}
p_*A_{x}(\hat{v}_1)&=p_*\left(2\,\I\,V\,\beta\indices{^i_j}(\diff_{\bar{z}})\,w^j\diff_{w^i}\right)=\\
&=2\,\I\,V\left(-\beta\indices{^1_j}(\diff_{\bar{z}})\,w^j\frac{w^2}{(w^1)^2}+\beta\indices{^2_j}(\diff_{\bar{z}})\,w^j\frac{1}{w^2}\right)\diff_\zeta
\end{split}
\end{equation*}
while, if~$\hat{v}_2=\tilde{V}\diff_{\bar{z}}+\tilde{U}^{\bar{i}}\diff_{\bar{w}^i}$
\begin{equation*}
\begin{split}
p_*A_{\lambda\,x}(\hat{v}_2)&=p_*\left(2\,\I\,\tilde{V}\,\beta\indices{^i_j}(\diff_{\bar{z}})\,w^j\diff_{w^i}\right)=\\
&=2\,\I\,\tilde{V}\left(-\beta\indices{^1_j}(\diff_{\bar{z}})\,w^j\frac{w^2}{(w^1)^2}+\beta\indices{^2_j}(\diff_{\bar{z}})\,w^j\frac{1}{w^2}\right)\diff_\zeta
\end{split}
\end{equation*}
but if~$\hat{v}_1$ and~$\hat{v}_2$ have the same image under~$p_*$,~$V=\tilde{V}$.
\end{proof}
Let~$v=v^{\bar{1}}\diff_{\bar{z}}+v^{\bar{2}}\diff_{\bar{\zeta}}\in T^{0,1}_{(z,\zeta)}\bb{P}(E)$, and consider~$\hat{v}=v^{\bar{1}}\diff_{\bar{z}}+v^{\bar{2}}\diff_{\bar{w}^2}\in T^{0,1}_{(z,1,\zeta)}(E)$. By our definition,
\begin{equation*}
\begin{split}
p_*A(\hat{v})=&2\,\I\,v^{\bar{1}}\left(-\beta\indices{^1_1}(\diff_{\bar{z}})\,\zeta-\beta\indices{^1_2}(\diff_{\bar{z}})\,\zeta^2+\beta\indices{^2_1}(\diff_{\bar{z}})+\beta\indices{^2_2}(\diff_{\bar{z}})\,\zeta\right)\diff_\zeta.
\end{split}
\end{equation*}
So, if we denote still by~$A$ the deformation of the complex structure of~$\bb{P}(E)$ we have
\begin{equation}\label{eq:deformazione_proiettivo}
A^{1,0}=2\,\I\left[(\beta\indices{_{\bar{1}}^2_2}-\beta\indices{_{\bar{1}}^1_1})\,\zeta-\beta\indices{_{\bar{1}}^1_2}\,\zeta^2+\beta\indices{_{\bar{1}}^2_1}\right]\mrm{d}\bar{z}\otimes\diff_\zeta.
\end{equation}

Notice that when we decompose~$\beta\in\m{A}^{0,1}(\m{O}\oplus L)$ as
\begin{equation*}
\beta=\begin{pmatrix}
\beta\indices{^1_1} & \beta\indices{^1_2}\\
\beta\indices{^2_1} &\beta\indices{^2_2}
\end{pmatrix}
\end{equation*}
then
\begin{center}
\begin{tabular}{l l}
$\beta\indices{^1_1}\in\m{A}^{0,1}(\m{O})\cong\m{A}^{0,1}(\Sigma,\bb{C})$, &~$\beta\indices{^1_2}\in\m{A}^{0,1}(L^*)$,\\
$\beta\indices{^2_1}\in\m{A}^{0,1}(L)$, &~$\beta\indices{^2_2}\in\m{A}^{0,1}(\mrm{End}(L))\cong\m{A}^{0,1}(\Sigma,\bb{C})$.
\end{tabular}
\end{center}
The expression~\eqref{eq:deformazione_proiettivo} for~$A^{1,0}$ holds just on the set~$\bb{P}(\m{O}\oplus L)\setminus\Sigma_\infty$. If instead we change coordinates to~$\bb{P}(\m{O}\oplus L)\setminus\Sigma_0$, we simply have to exchange the roles of~$\beta\indices{^1_2}$ and~$\beta\indices{^2_1}$. Indeed, equation~\eqref{eq:deformazione_proiettivo} was obtained by fixing a system of bundle-adapted holomorphic coordinates~$(z,\zeta)$ on~$L$; if we perform the change of variables~$\eta=\zeta^{-1}$ we obtain
\begin{equation*}
A^{1,0}=-2\,\I\left[\left(\beta\indices{_{\bar{1}}^2_2}-\beta\indices{_{\bar{1}}^1_1}\right)\eta-\beta\indices{_{\bar{1}}^1_2}+\beta\indices{_{\bar{1}}^2_1}\eta^2\right]\mrm{d}\bar{z}\otimes\diff_{\eta}.
\end{equation*}
After all, the construction of~$\bb{P}(\m{O}\oplus L)$ can be interpreted as glueing the total spaces of~$L$ and~$L^*$ along their open subsets~$L\setminus\Sigma$ and~$L^*\setminus\Sigma$.

\begin{rmk}
Choosing the Higgs term~$A$ as in~\eqref{eq:deformazione_proiettivo} guarantees that the first-order deformation of the complex structure is integrable (c.f. Definition~\ref{def:integrable}). The integrability condition for~$A$ in~\eqref{eq:deformazione_proiettivo} is written as
\begin{equation*}
\diff_{\bar{a}}A\indices{^2_{\bar{c}}}=\diff_{\bar{c}}A\indices{^2_{\bar{a}}};
\end{equation*}
the only possibly non-vanishing term of~$A$ is~$A\indices{^2_{\bar{1}}}$, so the integrability reduces to~$\diff_{\bar{\zeta}}A\indices{^2_{\bar{1}}}=0$. It is immediate to check that any~$A$ defined by~\eqref{eq:deformazione_proiettivo} satisfies this condition.
\end{rmk}

\begin{rmk}\label{rmk:surf_eq_complex_nonequiv}
Our choice of deformation of the complex structure~$A$ is not compatible with~$\omega_\phi$ for any~$\phi$. Indeed~$A^2=A^{1,0}A^{0,1}+A^{0,1}A^{1,0}=0$, and if~$A$ and~$\omega_\phi$ were compatible then we would find 
\begin{equation*}
\norm{A}^2_{g_\phi}=\mrm{Tr}(A^2)=0
\end{equation*}
but~$A\not=0$. Hence, in this Section we study the complexified equations
\begin{equation*}
\begin{cases}
\f{m}_{\bm{\Omega_I}}\left(\omega,A(\beta)\right)=0;\\
\f{m}_{\bm{\Theta}}\left(\omega,A(\beta)\right)=0.
\end{cases}
\end{equation*}
for~$A(\beta)$ as in~\ref{eq:deformazione_proiettivo}, and so we'll find a solution to the complexified system~\eqref{eq:HCSCK_complex_relaxed} without the compatibility condition.

Hence, we are tacitly assuming that we have extended the moment maps~$\f{m}_{\bm{\Omega_I}}$,~$\f{m}_{\bm{\Theta}}$ to the space of metrics~$g$ for which~$g(\alpha^\transpose-,-)$ is not necessarily symmetric. We have shown above that~${\f{m}_{\bm{\Theta}}}_{(J,\alpha)}(h) = \left\langle h,-\mrm{div}\left(\bdiff^*\bar{\alpha}^\transpose\right)\right\rangle$, which clearly has a tautological extension to all~$g$ in the K\"ahler class. But the choice of an extension of the real moment map~$\f{m}_{\bm{\Omega_I}}$ is more flexible. 

The crucial point is that, by Lemma~\ref{lemma:mappa_momento_OmegaI},~$\f{m}_{\bm{\Omega_I}}$ is computed in terms of a spectral function of~$A = \mrm{Re}(\alpha^\transpose)$. This function can be expressed in several different, equivalent ways by using a compatible metric~$g$, that is, one for which~$g(\alpha^\transpose-,-)$ is symmetric. In our present situation where this compatibility condition might not hold, these equivalent expressions give rise to potentially different extensions of~$\f{m}_{\bm{\Omega_I}}$. A simple example is given by the spectral quantity~$\mrm{Tr}(A^2)$. We know that for \emph{compatible}~$g$ this may be expressed equivalently as~$\norm{A}^2_g$, so when~$g$ and~$A$ are not compatible~$\norm{A}^2_g$ gives an alternative extension of the spectral quantity~$\mrm{Tr}(A^2)$.
\end{rmk}
\paragraph*{Choice of complexification.} The expression for~$\f{m}$ appearing in~\eqref{eq:low_rank} has been derived in close analogy to the case of curves. However in Remark~\ref{rmk:surf_eq_complex_nonequiv} we saw that we have to extend~$\f{m}$ to the space of first-order deformations of the complex structure that are not necessarily compatible with the metric, and this can be done at least in two ways:~$\f{m}$ is the divergence of the vector field
\begin{equation}\label{eq:low_rank_alternative}
\begin{split}
&-\frac{\mrm{grad}\left(\norm{A^{1,0}}^2\right)}{2\left(1+\sqrt{1-\norm{A^{1,0}}^2}\right)}+\frac{2\,\mrm{Re}\,g(\nabla^aA^{0,1},A^{1,0})\diff_a}{1+\sqrt{1-\norm{A^{1,0}}^2}}-\nabla^*\left(\frac{A^2}{1+\sqrt{1-\norm{A^{1,0}}^2}}\right)\\
&=-\frac{\mrm{grad}\left(\frac{\mrm{Tr}(A^2)}{2}\right)}{2\left(1+\sqrt{1-\frac{\mrm{Tr}(A^2)}{2}}\right)}+\frac{2\,\mrm{Re}\,g(\nabla^aA^{0,1},A^{1,0})\diff_a}{1+\sqrt{1-\frac{\mrm{Tr}(A^2)}{2}}}-\nabla^*\left(\frac{A^2}{1+\sqrt{1-\frac{\mrm{Tr}(A^2)}{2}}}\right).
\end{split}
\end{equation}
This leads to a few different possibilities for the formal complexification. In the rest of this paper we examine the natural choices given by the two expressions in~\eqref{eq:low_rank_alternative}.

\subsection{The complex moment map}\label{sec:ruled_complex_mm}

In this Section we find sufficient conditions on~$\beta\in\m{A}^{0,1}(\mrm{End}(\m{O}\oplus L))$ such that the pair~$\left(\omega_\phi,A(\beta)\right)$ satisfies the complex moment map equation. We work with a fixed metric~$\omega_\phi$ for a prescribed (arbitrary) momentum profile~$\phi$. 

Our strategy is to carry out the necessary computations \emph{without} assuming that~$A=A(\beta)$, but rather for some arbitrary~$A^{1,0}=A\indices{^2_{\bar{1}}}\mrm{d}\bar{z}\otimes\diff_\zeta$. At the end of this Section we show that, when~$L$ is the anti-canonical bundle and for suitable choices of~$A=A(\beta)$, we can find some solutions to the complex moment map equation on~$\bb{P}(\m{O}\oplus L)$.

Recall that, for a deformation of complex structures~$\dot{J}_0$ and a K\"ahler form~$\omega$, the complex moment map equation is
\begin{equation*}
\mrm{div}\left(\bdiff^*\!\dot{J}_0^{1,0}\right)=0.
\end{equation*}
\begin{lemma}
With the previous notation, 
\begin{equation*}
\begin{split}
\bdiff^*A^{1,0}=&-\frac{\phi(\tau)A\indices{^2_{\bar{1}}}}{\zeta\,g_0\,(1+\tau)^2}\diff_z-\frac{1}{(1+\tau)g_0}\left(\diff_zA\indices{^2_{\bar{1}}}+A\indices{^2_{\bar{1}}}\,\diff_zt\,\left(1-\frac{\phi(\tau)}{1+\tau}\right)-\zeta\,\diff_zt\,\diff_{\zeta}A\indices{^2_{\bar{1}}}\right)\diff_\zeta.
\end{split}
\end{equation*}
\end{lemma}
\begin{proof}
It's just a matter of computing carefully, starting from
\begin{equation*}
\bdiff^*A^{1,0}=-g^{a\bar{b}}\nabla_aA\indices{^c_{\bar{b}}}\,\diff_c.
\end{equation*}
The only possibly non-vanishing covariant derivatives of~$A$ are
\begin{equation*}
\begin{split}
\nabla_1A\indices{^1_{\bar{1}}}=A\indices{^2_{\bar{1}}}\Gamma^1_{12},\quad
\nabla_1A\indices{^2_{\bar{1}}}=\diff_zA\indices{^2_{\bar{1}}}+A\indices{^2_{\bar{1}}}\Gamma^2_{21},\quad
\nabla_2A\indices{^2_{\bar{1}}}=\diff_{\zeta}A\indices{^2_{\bar{1}}}+A\indices{^2_{\bar{1}}}\Gamma^2_{22}
\end{split}
\end{equation*}
So by Lemma~\ref{lemma:christoffel} we can rewrite~$\bdiff^*A^{1,0}$ as 
\begin{equation*}
\begin{split}
\bdiff^*A^{1,0}=&-g^{1\bar{1}}\nabla_1A\indices{^c_{\bar{1}}}\,\diff_c-g^{2\bar{1}}\nabla_2A\indices{^2_{\bar{1}}}\,\diff_\zeta=\\
=&-g^{1\bar{1}}\nabla_1A\indices{^1_{\bar{1}}}\diff_z-\left(g^{1\bar{1}}\nabla_1A\indices{^2_{\bar{1}}}+g^{2\bar{1}}\nabla_2A\indices{^2_{\bar{1}}}\right)\!\diff_\zeta=\\
=&-\frac{\phi(\tau)A\indices{^2_{\bar{1}}}\diff_z}{\zeta\,g_0\,(1+\tau)^2}-\frac{1}{(1+\tau)g_0}\left(\diff_zA\indices{^2_{\bar{1}}}+A\indices{^2_{\bar{1}}}\,\diff_zt\,\left(1-\frac{\phi(\tau)}{1+\tau}\right)-\zeta\,\diff_zt\,\diff_{\zeta}A\indices{^2_{\bar{1}}}\right)\diff_\zeta.\qedhere
\end{split}
\end{equation*}
\end{proof}
We proceed to calculate the divergence of~$\bdiff^*A^{1,0}$. By definition
\begin{equation}\label{eq:div_ruled_complex}
\mrm{div}(\bdiff^*A^{1,0})=\nabla_a(\bdiff^*A^{1,0})^a=\diff_a(\bdiff^*A^{1,0})^a+(\bdiff^*A^{1,0})^b\Gamma^a_{ab}.
\end{equation}
We compute the two terms separately. We will need the quantities
\begin{equation*}
\begin{split}
D_1(\tau):=&-\frac{\phi(\tau)}{(1+\tau)^2}=\phi(\tau)\,\diff_\tau\left(\frac{1}{1+\tau}\right)\\
D_2(\tau):=&\phi(\tau)\,\diff_\tau D_1(\tau).
\end{split}
\end{equation*}
The first term in~\eqref{eq:div_ruled_complex} is the sum of 
\begin{equation*}
\begin{split}
\diff_1(\bdiff^*A^{1,0})^1=&\diff_z\left(D_1(\tau)\,\frac{A\indices{^2_{\bar{1}}}}{\zeta\,g_0}\right)=\\
=&D_2(\tau)\,\frac{\diff_zt}{\zeta\,g_0}A\indices{^2_{\bar{1}}}-D_1(\tau)\,\frac{\Gamma(\Sigma)}{\zeta\,g_0}A\indices{^2_{\bar{1}}}+D_1(\tau)\,\frac{1}{\zeta\,g_0}\diff_z A\indices{^2_{\bar{1}}}
\end{split} 
\end{equation*}
and 
\begin{equation*}
\begin{split}
&\diff_2(\bdiff^*A^{1,0})^2=\\
=&\diff_\zeta\left(-\frac{\diff_zA\indices{^2_{\bar{1}}}}{(1+\tau)g_0}-\frac{A\indices{^2_{\bar{1}}}\,\diff_zt}{(1+\tau)g_0}-D_1(\tau)\,\frac{A\indices{^2_{\bar{1}}}\,\diff_zt}{g_0}+\frac{\zeta\,\diff_zt}{(1+\tau)g_0}\diff_{\zeta}A\indices{^2_{\bar{1}}}\right)=\\
=&-D_1(\tau)\,\frac{\diff_zA\indices{^2_{\bar{1}}}}{\zeta\,g_0}-\frac{\diff_\zeta\diff_zA\indices{^2_{\bar{1}}}}{(1+\tau)g_0}-D_1(\tau)\frac{A\indices{^2_{\bar{1}}}\,\diff_zt}{\zeta\,g_0}-D_2(\tau)\,\frac{A\indices{^2_{\bar{1}}}\,\diff_zt}{\zeta\,g_0}+\frac{\zeta\,\diff_zt}{(1+\tau)g_0}\diff_{\zeta}\diff_{\zeta}A\indices{^2_{\bar{1}}}.
\end{split}
\end{equation*}
The sum is given by 
\begin{equation*}
\begin{split}
\diff_a(\bdiff^*A^{1,0})^a=&-D_1(\tau)\,\frac{\Gamma(\Sigma)}{\zeta\,g_0}A\indices{^2_{\bar{1}}}-\frac{\diff_\zeta\diff_zA\indices{^2_{\bar{1}}}}{(1+\tau)g_0}-D_1(\tau)\,\frac{A\indices{^2_{\bar{1}}}\,\diff_zt}{\zeta\,g_0}+\frac{\zeta\,\diff_zt}{(1+\tau)g_0}\diff_{\zeta}\diff_{\zeta}A\indices{^2_{\bar{1}}}=\\
=&-(\bdiff^*A^{1,0})^1\Gamma(\Sigma)-\frac{\diff_\zeta\diff_zA\indices{^2_{\bar{1}}}}{(1+\tau)g_0}-D_1(\tau)\,\frac{A\indices{^2_{\bar{1}}}\,\diff_zt}{\zeta\,g_0}+\frac{\zeta\,\diff_zt}{(1+\tau)g_0}\diff_{\zeta}\diff_{\zeta}A\indices{^2_{\bar{1}}}.
\end{split}
\end{equation*}
On the other hand the second term in~\eqref{eq:div_ruled_complex} is given by
\begin{equation*}
\begin{split}
(\bdiff^*A^{1,0})^b\Gamma^a_{ab}=&(\bdiff^*A^{1,0})^1\Gamma^1_{11}+(\bdiff^*A^{1,0})^1\Gamma^2_{21}+(\bdiff^*A^{1,0})^2\Gamma^1_{12}+(\bdiff^*A^{1,0})^2\Gamma^2_{22}=\\
=&(\bdiff^*A^{1,0})^1\Gamma(\Sigma)+D_1(\tau)\,\frac{\diff_zt}{\zeta\,g_0}\,A\indices{^2_{\bar{1}}}\,\left(\phi'(\tau)+\frac{\phi(\tau)}{1+\tau}\right)+\\
&+D_1(\tau)\,\frac{\diff_zA\indices{^2_{\bar{1}}}}{\zeta\,g_0}+D_1(\tau)\,\frac{\diff_zt}{\zeta\,g_0}A\indices{^2_{\bar{1}}}-D_1(\tau)\,\frac{\phi(\tau)\,\diff_zt}{\zeta\,g_0}A\indices{^2_{\bar{1}}}-\\
&-D_1(\tau)\,\frac{\diff_zt\,\diff_\zeta A\indices{^2_{\bar{1}}}}{g_0}+(\bdiff^*A^{1,0})^2\frac{\phi'(\tau)-1}{\zeta}.
\end{split}
\end{equation*}
These computations show that we can write~$\mrm{div}(\bdiff^*A^{1,0})$ as
\begin{equation*}
\begin{split}
&-\frac{\diff_\zeta\diff_zA\indices{^2_{\bar{1}}}}{(1+\tau)g_0}+\frac{\zeta\,\diff_zt}{(1+\tau)g_0}\diff_{\zeta}\diff_{\zeta}A\indices{^2_{\bar{1}}}-\frac{\diff_zA\indices{^2_{\bar{1}}}}{(1+\tau)g_0\,\zeta}\left(\frac{\phi(\tau)}{1+\tau}+\phi'(\tau)-1\right)+\\
&\quad+\left(\frac{\phi(\tau)}{1+\tau}+\phi'(\tau)-1\right)\frac{\diff_zt\,\diff_\zeta A\indices{^2_{\bar{1}}}}{(1+\tau)g_0}-\frac{\diff_zt\,A\indices{^2_{\bar{1}}}}{(1+\tau)\zeta\,g_0}\left(\phi'(\tau)-1+\frac{\phi(\tau)}{1+\tau}\right)=\\
=&\frac{-\diff_\zeta\diff_zA\indices{^2_{\bar{1}}}+\zeta\,\diff_zt\,\diff_{\zeta}\diff_{\zeta}A\indices{^2_{\bar{1}}}}{(1+\tau)g_0}+\frac{1}{(1+\tau)g_0}\left(\frac{\diff_zA\indices{^2_{\bar{1}}}}{\zeta}-\diff_zt\,\diff_\zeta A\indices{^2_{\bar{1}}}+\frac{\diff_zt\,A\indices{^2_{\bar{1}}}}{\zeta}\right)-\\
&-\frac{1}{(1+\tau)g_0}\diff_\zeta\left[\mrm{log}\,\phi(\tau)(1+\tau)\right]\left(\diff_zA\indices{^2_{\bar{1}}}-\diff_zt\,\zeta\diff_\zeta A\indices{^2_{\bar{1}}}+\diff_zt\,A\indices{^2_{\bar{1}}}\right).
\end{split}
\end{equation*}
This quantity vanishes precisely when
\begin{equation*}
\begin{split}
-\zeta\diff_\zeta\diff_zA\indices{^2_{\bar{1}}}&+\zeta^2\,\diff_zt\,\diff_{\zeta}\diff_{\zeta}A\indices{^2_{\bar{1}}}-\left(-\diff_zA\indices{^2_{\bar{1}}}+\diff_zt\,\zeta\,\diff_\zeta A\indices{^2_{\bar{1}}}-\diff_zt\,A\indices{^2_{\bar{1}}}\right)+\\&+\zeta\diff_\zeta\left[\mrm{log}\,\phi(\tau)(1+\tau)\right]\left(-\diff_zA\indices{^2_{\bar{1}}}+\diff_zt\,\zeta\,\diff_\zeta A\indices{^2_{\bar{1}}}-\diff_zt\,A\indices{^2_{\bar{1}}}\right)=0.
\end{split}
\end{equation*}
Notice that
\begin{equation*}
-\zeta\diff_\zeta\diff_zA\indices{^2_{\bar{1}}}+\zeta^2\,\diff_zt\,\diff_{\zeta}\diff_{\zeta}A\indices{^2_{\bar{1}}}=\zeta\diff_\zeta\left(-\diff_zA\indices{^2_{\bar{1}}}+\diff_zt\,\zeta\diff_\zeta A\indices{^2_{\bar{1}}}-\diff_zt\,A\indices{^2_{\bar{1}}}\right).
\end{equation*}
Thus, introducing the \emph{locally defined function} 
\begin{equation}\label{eq:def_k}
k:=-\diff_zA\indices{^2_{\bar{1}}}+\diff_zt\,\zeta\diff_\zeta A\indices{^2_{\bar{1}}}-\diff_zt\,A\indices{^2_{\bar{1}}},
\end{equation}
the complex moment map equation~$\mrm{div}(\bdiff^*A^{1,0})=0$ may be expressed \emph{locally} as
\begin{equation*}
\zeta\diff_\zeta k+\left(\zeta\diff_\zeta\left[\mrm{log}\,\phi(\tau)(1+\tau)\right]-1\right)k=0.
\end{equation*}
This condition can be rewritten as
\begin{equation*}
\diff_\zeta k+\diff_\zeta\left(\mrm{log}\frac{\phi(\tau)(1+\tau)}{\zeta\,\bar{\zeta}}\right)k=0.
\end{equation*}
This equation can be integrated; so we see that the equation~$\mrm{div}(\bdiff^*A^{1,0})=0$ is satisfied locally if and only if the function~$k$ defined by equation~\eqref{eq:def_k} satisfies
\begin{equation}\label{eq:complex_mm_k_sol}
k=c\,\frac{\zeta\,\bar{\zeta}}{\phi(\tau)(1+\tau)}
\end{equation}
for some function~$c=c(z,\zeta)$ such that~$\diff_\zeta c=0$.

\paragraph*{Choosing~$c=0$.} Let's consider the case in which the function~$c$ in~\eqref{eq:complex_mm_k_sol} is identically~$0$. In this case,~$A^{1,0}$ satisfies 
\begin{equation}\label{eq:complessa_c0}
-\diff_zA\indices{^2_{\bar{1}}}+\zeta\,\diff_zt\,\diff_\zeta A\indices{^2_{\bar{1}}}-\diff_zt\,A\indices{^2_{\bar{1}}}=0.
\end{equation}
If we now choose~$A=A(\beta)$, i.e.
\begin{equation*}
A\indices{^2_{\bar{1}}}=2\I\left(\zeta(\beta\indices{_{\bar{1}}^2_2}-\beta\indices{_{\bar{1}}^1_1})-\zeta^2\beta\indices{_{\bar{1}}^1_2}+\beta\indices{_{\bar{1}}^2_1}\right)
\end{equation*}
for~$\beta\indices{^1_1},\beta\indices{^2_2}\in\m{A}^{0,1}(\Sigma,\bb{C})$,~$\beta\indices{^1_2}\in\m{A}^{0,1}(L^*)$,~$\beta\indices{^2_1}\in\m{A}^{0,1}(L)$, we can get an interesting consequence from equation~\eqref{eq:complessa_c0}. Indeed, on the divisor~$\Sigma=\Sigma_0=\set{\zeta=0}$ we get, from equation~\eqref{eq:complessa_c0}
\begin{equation*}
-\diff_z\beta\indices{_{\bar{1}}^2_1}-\diff_zt\,\beta\indices{_{\bar{1}}^2_1}=0
\end{equation*}
and recalling that~$\diff_zt=\diff_z\mrm{log}(a(z))$, were~$a(z)$ is the local representative of the fibre metric on~$L$, this tells us that
\begin{equation*}
\beta\indices{_{\bar{1}}^2_1}=\frac{q(z)}{a(z)}
\end{equation*}
for some function~$q$ over~$\Sigma$ such that~$\diff_zq=0$. Consider instead what equation~\eqref{eq:complessa_c0} tells us for~$\zeta=\infty$, i.e. on the zero-set of~$\eta=\zeta^{-1}$; after the change of coordinates, equation~\eqref{eq:complessa_c0} becomes  
\begin{equation*}
\diff_zA\indices{^2_{\bar{1}}}(\eta)+\diff_zt\left(\eta\diff_\eta A\indices{^2_{\bar{1}}}-A\indices{^2_{\bar{1}}}(\eta)\right)=0
\end{equation*}
where~$A\indices{^2_{\bar{1}}}(\eta)=-2\I\left(\eta(\beta\indices{_{\bar{1}}^2_2}-\beta\indices{_{\bar{1}}^1_1})-\beta\indices{_{\bar{1}}^1_2}+\eta^2\,\beta\indices{_{\bar{1}}^2_1}\right)$. Setting~$\eta=0$ we find
\begin{equation*}
\diff_z\beta\indices{_{\bar{1}}^1_2}-\diff_zt\,\beta\indices{_{\bar{1}}^1_2}=0
\end{equation*}
and so
\begin{equation*}
\beta\indices{_{\bar{1}}^1_2}=a(z)\,\tilde{q}(z)
\end{equation*}
for some function~$\tilde{q}$ over~$\Sigma$ such that~$\diff_z\tilde{q}=0$. With these choices, the matrix associated to~$\beta\in\m{A}^{0,1}(\mrm{End}(\m{O}\oplus L))$ in a local holomorphic frame for~$L$ is
\begin{equation*}
\begin{pmatrix}
\beta\indices{^1_1} & \tilde{q}(z)\,a(z)\,\mrm{d}\bar{z}\\
\frac{q(z)}{a(z)}\,\mrm{d}\bar{z} & \beta\indices{^2_2}
\end{pmatrix}.
\end{equation*}
It is useful to notice the identity~$\zeta\,\diff_\zeta A\indices{^2_{\bar{1}}}-A\indices{^2_{\bar{1}}}=-2\I\left(\zeta^2\,\beta\indices{_{\bar{1}}^1_2}+\beta\indices{_{\bar{1}}^2_1}\right)$. Plugging this into~\eqref{eq:complessa_c0} the equation can be rewritten as
\begin{equation*}
-\zeta\diff_z\left(\beta\indices{_{\bar{1}}^2_2}-\beta\indices{_{\bar{1}}^1_1}\right)-\zeta^2\,\diff_z\beta\indices{_{\bar{1}}^1_2}+\diff_z\beta\indices{_{\bar{1}}^2_1}-\diff_zt\left(\zeta^2\,\beta\indices{_{\bar{1}}^1_2}+\beta\indices{_{\bar{1}}^2_1}\right)=0,
\end{equation*}
which reduces to
\begin{equation*}
\diff_z\left(\beta\indices{_{\bar{1}}^2_2}-\beta\indices{_{\bar{1}}^1_1}\right)=0.
\end{equation*}
So equation~\eqref{eq:complessa_c0} is satisfied if and only if
\begin{equation}\label{eq:condizione_coeff_c0}
\begin{split}
\beta\indices{_{\bar{1}}^1_2}&=a(z)\,\tilde{q}(z)\quad\mbox{with }\diff_z\tilde{q}=0;\\
\beta\indices{_{\bar{1}}^2_1}&=\frac{q(z)}{a(z)}\quad\mbox{with }\diff_zq=0;\\
\diff\big(\beta\indices{^2_2}&-\beta\indices{^1_1}\big)=0.
\end{split}
\end{equation}
The first two conditions in equation~\eqref{eq:condizione_coeff_c0} are still just local ones. However we can globalize them by choosing~$L$ to be the anti-canonical bundle of~$\Sigma$,~$L=T^{1,0}\Sigma$. Indeed, recall that~$\beta\indices{^1_2}\in\m{A}^{0,1}(L^*)$,~$\beta\indices{^2_1}\in\m{A}^{0,1}(L)$, so that if~$L=T^{1,0}\Sigma$ then~$\beta\indices{^1_2}$ must be an element of~$\m{A}^{0,1}({T^{1,0}}^*\Sigma)$, while~$\beta\indices{^2_1}$ must be an element of~$\m{A}^{0,1}(T^{1,0}\Sigma)$. Then we can choose the quantity~$\tilde{q}$ of equation~\eqref{eq:condizione_coeff_c0} to be a constant, and the local condition on~$\beta\indices{^1_2}$ becomes the global condition~$\beta\indices{^1_2}=\tilde{q}\,h$. This is compatible with~$\beta\indices{^1_2}\in\m{A}^{0,1}({T^{1,0}}^*\Sigma)$, since~$h$ is a Hermitian metric on~$T^{1,0}\Sigma$. In the same way, if~$q$ is the local representative of a global holomorphic quadratic differential on~$\Sigma$ (that we denote still by~$q$), then the local condition on~$\beta\indices{^2_1}$ globalizes to~$\beta\indices{^2_1}=q^{\sharp_h}$, i.e.~$\beta\indices{^2_1}$ should be the quadratic differential with one index raised by~$h$.

Let us summarise the results of this Section. Suppose that~$L=K_\Sigma^*=T^{1,0}\Sigma$ and that~$\beta$ satisfies the globally defined equations
\begin{equation}\label{eq:cond_globali_coeff_c0}
\begin{split}
\beta\indices{^1_2}&=\tilde{q}\,h\quad\mbox{ for some constant }\tilde{q};\\
\beta\indices{^2_1}&=q^{\sharp_h}\quad\mbox{ for some holomorphic quadratic differential }q;\\
\diff\big(\beta\indices{^2_2}&-\beta\indices{^1_1}\big)=0.
\end{split}
\end{equation}
Then the complex moment map equation is satisfied. From now we always assume that~$L$ and~$\beta$ are of this form.

\subsection{The real moment map}\label{sec:ruled_surface_realmm}

In this section we will prove that there exists a solution to the HcscK equations on our ruled surface, at least when the fibres have sufficiently small volume. We will work with the two possible choices of formal complexification given by the expressions in~\eqref{eq:low_rank_alternative}. First we reformulate Theorem~\ref{thm:solution_mm_reale} using the notation introduced in the last few sections.
\begin{thm}\label{thm:HCSCK_rigata_esistenza}
Let~$\Sigma$ be a Riemann surface of genus~$g\geq 2$, and consider the ruled surface~$M=\bb{P}(\m{O}\oplus K_\Sigma^*)$. Then, for all sufficiently small~$m>0$, there exists a K\"ahler metric~$\omega$ in the class dual to~$2\,\pi\left(\m{C}+m\,\Sigma_\infty\right)$ and a Higgs field~$\alpha\in\mrm{Hom}({T^{1,0}}^*M,{T^{0,1}}^*M)$ such that the system of equations
\begin{equation*}
\begin{dcases}
\mrm{div}\left(\diff^*\alpha\right)=0\\
2\,S(\omega)-2\,\widehat{S(\omega)}+\mrm{div}\left[X(\omega,\alpha)\right]=0
\end{dcases}
\end{equation*}
is satisfied, where the vector field~$X(\omega,\alpha)$ is given by one of the expressions in~\eqref{eq:low_rank_alternative}.
\end{thm}
We will choose~$A=\mrm{Re}(\alpha^\transpose) = A(\beta)$, for a form~$\beta\in\m{A}^{0,1}(\mrm{End}(\m{O}\oplus L))$. Then the complex moment map equation holds provided~$\beta$ satisfies the conditions~\eqref{eq:cond_globali_coeff_c0}. As we already noted, any~$A(\beta)$ will be such that~$A^2=0$, and in particular we are in the low-rank situation described at the end of Section~\ref{sec:complex_surface}.

We will present the details of the proof of Theorem~\ref{thm:HCSCK_rigata_esistenza} for the choice of complexification given by the second expression in~\eqref{eq:low_rank_alternative}, namely 
\begin{equation*}
\begin{split}
\f{m}=\mrm{div}\Bigg[\psi\left(\frac{\mrm{Tr}(A^2)}{2}\right)
\left(4\mrm{Re}\left(\langle\nabla^aA^{0,1},A^{1,0}\rangle\diff_a\right)-\mrm{grad}\left(\frac{\mrm{Tr}(A^2)}{2}\right)\right)-2\nabla^*\left(\psi\left(\frac{\mrm{Tr}(A^2)}{2}\right)A^2\right)&\Bigg]
\end{split}
\end{equation*}
for~$\psi(x)=\frac{1}{2}\left(1+\sqrt{1-x}\right)^{-1}$, as in~\eqref{eq:funzioni_accessorie_HcscK}. The proof for the alternative complexified equation, i.e.
\begin{equation}\label{eq:low_rank_norm} 
\begin{split}
\f{m}=\mrm{div}\Bigg[\psi\left(\norm{A^{1,0}}^2\right)
\left(4\mrm{Re}\left(\langle\nabla^a A^{0,1} ,A^{1,0}\rangle\diff_a\right)-\mrm{grad}\left(\norm{A^{1,0}}^2\right)\right)-2\nabla^*\left(\psi\left(\norm{A^{1,0}}^2\right)A^2\right)&\Bigg]
\end{split}
\end{equation}
is essentially the same, but some of the computations are more involved. We will point out the key differences in the course of the proof.

With our current choice of complexification we find, from~$A(\beta)^2=0$
\begin{equation*}
\begin{split}
\f{m}(\omega_\phi,A(\alpha))=&\mrm{div}_{\phi}\,\mrm{Re}\left(g_\phi(\nabla_\phi^aA^{0,1},A^{1,0})\diff_a\right).
\end{split}
\end{equation*}
In the rest of this Section we fix~$L = K^*_{\Sigma}$ and choose~$\beta$ so that the complex moment map vanishes, i.e. we assume that~$A^{1,0}$ satisfies equation~\eqref{eq:complessa_c0}. Then~$\beta$ should satisfy the conditions in equation~\eqref{eq:cond_globali_coeff_c0}, and in particular~$\beta\indices{_{\bar{1}}^1_2}=\tilde{q}\,a(z)$ for some constant~$\tilde{q}$. We can also make some additional assumptions on the Hermitian metric on~$L$ and on~$\beta\in\m{A}^{0,1}\left(\mrm{End}(\m{O}\oplus L)\right)$. First, since~$L=K^*_\Sigma$ we can assume that~$a(z)$ is a local representative for~$g_\Sigma$, so that~$\beta\indices{^1_2}=\tilde{q}g_\Sigma$. We will prove the following result.
\begin{lemma}\label{lemma:calcolo_div}
Assume that~$\beta\indices{^2_2}=\beta\indices{^1_1}$ and~$\beta\indices{^2_1}=0$, so that the matrix of~$1$-forms associated to~$\beta$ is upper triangular. Then
\begin{equation*}
\mrm{div}_{\phi}\,\mrm{Re}\left(g_\phi(\nabla_\phi^aA^{0,1},A^{1,0})\diff_a\right)=\norm{A^{1,0}}^2_{\phi}\left(\frac{(\phi'+1)^2}{\phi}+\phi''\right).
\end{equation*}
\end{lemma}
Under the assumptions of Lemma~\ref{lemma:calcolo_div},~$A^{1,0}=A\indices{^2_{\bar{1}}}\mrm{d}\bar{z}\otimes\diff_\zeta$ is written as
\begin{equation*}
A\indices{^2_{\bar{1}}}=-2\I\,\zeta^2\,\tilde{q}\,g_0
\end{equation*}
so we can easily compute that, for some positive constant~$\lambda$
\begin{equation}\label{eq:norma_A}
\norm{A^{1,0}}^2_\phi=\lambda\frac{\phi(\tau)}{1+\tau}\zeta\,\bar{\zeta}\,g_0=\lambda\frac{\mrm{e}^t\phi(\tau)}{1+\tau}.
\end{equation}
\begin{proof}[Proof of Lemma~\ref{lemma:calcolo_div}]
We will first obtain an expression for the vector field
\begin{equation*}
g_\phi(\nabla_\phi^aA^{0,1},A^{1,0})\diff_a.
\end{equation*}
We will then compute the divergence of this vector field working in a system of transversally normal coordinates at a point~$p\in M$. First, in any system of bundle-adapted coordinates~$(z,\zeta)$ we have
\begin{equation*}
\begin{split}
g_\phi(\nabla_\phi^aA^{0,1},A^{1,0})\diff_a=&g^{a\bar{b}}\nabla_{\bar{b}}A\indices{^{\bar{c}}_d}\,A\indices{^e_{\bar{f}}}\,g_{e\bar{c}}g^{d\bar{f}}=\\
=&g^{a\bar{b}}\left(\diff_{\bar{b}}A\indices{^{\bar{c}}_d}+A\indices{^{\bar{q}}_d}\Gamma^{\bar{c}}_{\bar{q}\bar{b}}\right)A\indices{^e_{\bar{f}}}\,g_{e\bar{c}}g^{d\bar{f}}=\\
=&g^{a\bar{b}}\diff_{\bar{b}}A\indices{^{\bar{2}}_1}\,A\indices{^2_{\bar{1}}}\,g_{2\bar{2}}g^{1\bar{1}}+g^{a\bar{b}}A\indices{^{\bar{2}}_1}\Gamma^{\bar{c}}_{\bar{2}\bar{b}}A\indices{^2_{\bar{1}}}\,g_{2\bar{c}}g^{1\bar{1}}.
\end{split}
\end{equation*}
a direct computation gives
\begin{equation*}
\begin{split}
g^{a\bar{b}}\diff_{\bar{b}}A\indices{^{\bar{2}}_1}=&\left(g^{a\bar{1}}\diff_{\bar{z}}A\indices{^{\bar{2}}_1}+g^{a\bar{2}}\diff_{\bar{\zeta}}A\indices{^{\bar{2}}_1}\right)A\indices{^2_{\bar{1}}}\,g_{2\bar{2}}g^{1\bar{1}}=\\
=&\norm{A^{1,0}}^2_\phi\left(g^{a\bar{1}}\diff_{\bar{z}}t+2g^{a\bar{2}}\diff_{\bar{\zeta}}t \right).
\end{split}
\end{equation*}
While for the term involving Christoffel symbols
\begin{equation*}
\begin{split}
g^{a\bar{b}}A\indices{^{\bar{2}}_1}\Gamma^{\bar{c}}_{\bar{2}\bar{b}}A\indices{^2_{\bar{1}}}\,g_{2\bar{c}}g^{1\bar{1}}&=A\indices{^{\bar{2}}_1}A\indices{^2_{\bar{1}}}g^{1\bar{1}}\left(g^{a\bar{1}}\frac{\phi\,\phi'}{\zeta\bar{\zeta}}\diff_{\bar{z}}t+g^{a\bar{2}}\frac{\phi}{\zeta\bar{\zeta}}(\phi'-1)\diff_{\bar{\zeta}}t\right)=\\
&=\norm{A^{1,0}}^2_\phi\left(g^{a\bar{1}}\phi'\diff_{\bar{z}}t+g^{a\bar{2}}(\phi'-1)\diff_{\bar{\zeta}}t\right).
\end{split}
\end{equation*}
Putting these expressions together we get, after some simplifications
\begin{equation*}
\begin{split}
g_\phi(\nabla_\phi^aA^{0,1},A^{1,0})\diff_a=&\norm{A^{1,0}}^2_\phi(\phi'+1)\left(g^{a\bar{1}}\diff_{\bar{z}}t+g^{a\bar{2}}\diff_{\bar{\zeta}}t \right)\diff_a=\\
=&\norm{A^{1,0}}^2_\phi(\phi'+1)\frac{\zeta}{\phi}\diff_\zeta.
\end{split}
\end{equation*}
We should now compute the divergence of this vector field. We do so in transversally normal coordinates around a point of~$\bb{P}(\m{O}\oplus L)$; in such a system of coordinates we have
\begin{equation*}
\diff_\zeta\norm{A^{1,0}}^2_\phi=\norm{A^{1,0}}^2_\phi\left(1+\phi'-\frac{\phi}{1+\tau}\right)\diff_\zeta t
\end{equation*}
so we can compute
\begin{equation*}
\begin{split}
&\nabla_a\left(g_\phi(\nabla_\phi^aA^{0,1},A^{1,0})\right)=\\
&=\diff_\zeta\left(\norm{A^{1,0}}^2_\phi(\phi'+1)\frac{\zeta}{\phi}\right)+\norm{A^{1,0}}^2_\phi(\phi'+1)\left(\frac{\phi'-1}{\phi}+\frac{1}{1+\tau}\right)=\\
&=\norm{A^{1,0}}^2_{\phi}\left(\frac{(\phi'+1)^2}{\phi}+\phi''\right).
\end{split}
\end{equation*}
In particular~$\mrm{div}\left(g_\phi(\nabla_\phi^aA^{0,1},A^{1,0})\diff_a\right)$ is a real function, and this concludes the proof.
\end{proof}
Substituting~\eqref{eq:norma_A} in the result of Lemma~\ref{lemma:calcolo_div} we find
\begin{equation*}
\mrm{div}_{\phi}\,\mrm{Re}\left(g_\phi(\nabla_\phi^aA^{0,1},A^{1,0})\diff_a\right)=\lambda\frac{\mrm{e}^t\phi}{1+\tau}\left(\frac{(\phi'+1)^2}{\phi}+\phi''\right).
\end{equation*}
We can finally write the zero-locus equation of the real moment map, using Proposition~\ref{prop:curv_scal_phi} and Lemma~\ref{lemma:media_curvatura}: since we are choosing a metric on~$\Sigma$ that has constant scalar curvature equal to~$-1$, the equation is 
\begin{equation}\label{eq:mappa_reale}
\phi''+\frac{2\,\phi'+1}{1+\tau}+\frac{4}{m(2+m)}=\frac{c}{m^2}\frac{\mrm{e}^t\phi}{1+\tau}\left(\frac{(\phi'+1)^2}{\phi}+\phi''\right)
\end{equation}
where we are collecting in~$c\,m^{-2}$ the various constants. The reason for introducing the factor~$m^{-2}$ in the equation is that in the next sections we will find a solution to equation~\eqref{eq:mappa_reale} in the \emph{adiabatic limit} when~$m\to 0$, and to do this we will have to expand the equation with respect to~$m$. This~$m^{-2}$ factor has been chosen precisely in such a way that the expansion  in~$m$ will have the appropriate form.

Let us summarise our computations so far. We showed that with all our assumptions, in particular those of Lemma~\ref{lemma:calcolo_div}, the complex moment map vanishes automatically, while the real moment map equation reduces to the problem 
\begin{equation}\label{eq:HCSCK_original}
\begin{gathered}
\phi''+\frac{2\,\phi'+1}{1+\tau}+\frac{4}{m(2+m)}=\frac{c}{m^2}\frac{\mrm{e}^t\phi}{1+\tau}\left(\frac{(\phi'+1)^2}{\phi}+\phi''\right)\\
\phi(0)=\phi(m)=0\\
\phi'(0)=-\phi'(m)=1
\end{gathered}
\end{equation}
to be solved for a positive function~$\phi(\tau)$ on~$[0,m]$ and a positive real number~$c$. Here the function~$t$ is a primitive of~$\frac{1}{\phi(\tau)}$; we might fix the starting point of integration as~$m/2$, since the choice of a different point can be absorbed by the constant~$c$. From now on then we will consider~$t$ as
\begin{equation*}
t(\tau)=\int_{\frac{m}{2}}^{\tau}\frac{1}{\phi(x)}\mrm{d}x
\end{equation*}
hence equation~\eqref{eq:HCSCK_original} becomes an ordinary integro-differential equation for~$\phi$ and~$c$.
\begin{rmk}
Essentially the same computations show that for the alternative choice of complexification~\eqref{eq:low_rank_norm}, the real moment map equation reduces to the problem 
\begin{equation*}
\begin{split}
&\left(1+\sqrt{1-\frac{c}{m^2}\frac{\phi(\tau)}{1+\tau}\mrm{e}^t}\right)\left(\phi''(\tau)+\frac{1+2\,\phi'(\tau)}{1+\tau}\right)+\\
&+\frac{8}{m(2+m)}+\frac{\frac{c}{m^2}\frac{\phi(\tau)}{1+\tau}\mrm{e}^t}{2\sqrt{1-\frac{c}{m^2}\frac{\phi(\tau)}{1+\tau}\mrm{e}^t}}\left(\frac{\phi(\tau)}{(1+\tau)^2}-\frac{(1+\phi'(\tau))^2}{\phi(\tau)}\right)=0
\end{split}
\end{equation*}
with the same boundary and positivity conditions, and the same definition of~$t(\tau)$.
\end{rmk}

\subsection{Approximate solutions}

We may regard~\eqref{eq:HCSCK_original} as a family of integro-differential equations parame\-trized by~$m\in\bb{R}_{>0}$. Our aim is to show that for sufficiently small values of this parameter (i.e. in the limit when the fibres of~$\bb{P}(\m{O}\oplus L)$ are very small) there is a solution to the equation. Notice however that~$m$ appears both in the equation and in the domain of definition of~$\phi(\tau)$, since~$\tau$ takes values in~$[0,m]$. It will then more convenient to first change variables, letting~$\tau=\lambda m$, so that~$\lambda$ takes values in the fixed interval~$[0,1]$. If we rewrite the problem~\eqref{eq:HCSCK_original} in terms of~$\phi(\lambda)$ we get the equivalent equation
\begin{equation*}
\begin{split}
&\phi''+\frac{2\,m\,\phi'+m^2}{1+\lambda m}+\frac{4\,m}{2+m}=\frac{c}{m^2}\,\frac{\mrm{exp}\!\left(\int_{\frac{1}{2}}^{\lambda}\frac{m}{\phi(x)}\mrm{d}x\right)}{1+\lambda m}\left((\phi'+m)^2+\phi\,\phi''\right)
\end{split}
\end{equation*}
with boundary conditions
\begin{equation*}
\begin{split}
\phi(0)=\phi(1)&=0\\
\phi'(0)=-\phi'(1)&=m,
\end{split}
\end{equation*}
to be solved for a momentum profile~$\phi(\lambda)$ and a constant~$c >0$. 
\begin{rmk} The corresponding equation for~\eqref{eq:low_rank_norm} is given by
\begin{equation*}
\begin{split}
&\left(1+\left(1-\frac{c}{m^2}\frac{\phi}{1+\lambda m}\mrm{exp}\!\left(\int_{1/2}^\lambda\frac{m}{\phi}\mrm{d}x\right)\right)^{\frac{1}{2}}\right)\left(\phi''+\frac{2\,m\,\phi'+m^2}{1+\lambda m}\right)+\\
&+\frac{8\,m}{2+m}+\frac{\frac{c}{m^2}\frac{\phi}{1+\lambda m}\mrm{exp}\!\left(\int_{1/2}^\lambda\frac{m}{\phi}\mrm{d}x\right)}{2\left(1-\frac{c}{m^2}\frac{\phi}{1+\lambda m}\mrm{exp}\!\left(\int_{1/2}^\lambda\frac{m}{\phi}\mrm{d}x\right)\right)^{\frac{1}{2}}}\left(\frac{m^2\,\phi}{(1+\lambda m)^2}-\frac{(m+\phi')^2}{\phi}\right)=0
\end{split}
\end{equation*}
with the same boundary conditions.
\end{rmk}
Introduce the space
\begin{equation*}
\m{V}_m:=\set*{\phi\in\m{C}^\infty([0,1])\tc\begin{matrix}\phi>0\mbox{ in }(0,1)\\
\phi(0)=\phi(1)=0\\
\phi'(0)=-\phi'(1)=m
\end{matrix}
}.
\end{equation*}
Our problem is equivalent to showing that the integro-differential operator
\begin{equation*}
\m{F}_m:\m{V}_m\times\bb{R}_{>0}\to\m{C}^{\infty}_0([0,1])
\end{equation*}
has a zero for some choice of~$m$, where~$\m{F}_m$ is defined as 
\begin{equation}\label{eq:operatore_F_m}
\begin{split}
\m{F}_m(\phi,c):=&\phi''+\frac{2\,m\,\phi'+m^2}{1+\lambda m}+\frac{4\,m}{2+m}-\frac{c}{m^2}\,\frac{\mrm{exp}\!\left(\int_{\frac{1}{2}}^{\lambda}\frac{m}{\phi}\mrm{d}x\right)}{1+\lambda m}\left((\phi'+m)^2+\phi\,\phi''\right).
\end{split}
\end{equation}
The reason why the image of~$\m{F}_m$ lies inside the space of zero-average functions is that in its original form the real HcscK equation is of the form
\begin{equation*}
\mbox{scalar curvature }-\mbox{ its average }+\mbox{ divergence of a vector field }=0.
\end{equation*}
In fact we will show that~$\m{F}_m$ has a zero for all sufficiently small~$m>0$. 

We follow the well-developed approach of \emph{adiabatic limits} and in particular the excellent reference~\cite{Fine_phd}. In this approach one first constructs a sufficiently good approximate solution and then perturbs this to a genuine solution by using a suitable quantitative version of the Implicit Function Theorem. An ``approximate solution'' in this context is just the data of a function~$\tilde{\phi}\in\m{V}_m$ and a positive constant~$\tilde{c}$ such that
\begin{equation*}
\m{F}_m(\tilde{\phi},\tilde{c})=O(m^n)
\end{equation*}
for some~$n>0$, in a purely formal sense. It is in fact possible to find approximate solutions up to every order, but we'll just need the first one
\begin{equation*}
\begin{split}
\phi_0(\lambda)=&\frac{\lambda m(1-\lambda)}{2(2+m)}\big(4+2\,m-m(4+3\,m)\lambda(1-\lambda)\big);\\
c_0=&2\,m^2.
\end{split}
\end{equation*}
For this choice of~$\phi$ and~$c$ we have
\begin{equation*}
\m{F}_m(\phi_0,c_0)=O(m^3)
\end{equation*}
moreover,
\begin{equation}\label{eq:prima_approssimazione}
\begin{split}
\phi_0''+2\,m\frac{\phi_0'}{1+\lambda m}+\frac{m^2}{1+\lambda m}+\frac{4\,m}{2+m}=O(m^2)\\
\frac{\mrm{exp}\!\left(\int_{\frac{1}{2}}^{\lambda}\frac{m}{\phi_0}\mrm{d}x\right)}{1+\lambda m}\left((\phi_0'+m)^2+\phi_0\phi_0''\right)=O(m^2).
\end{split}
\end{equation}
\begin{rmk}
Precisely the same choice of approximate solution works for the more complicated equation corresponding to~\eqref{eq:low_rank_norm}; the choice of~$c_0$ instead is~$8\,m^2$. 
\end{rmk}
\paragraph*{Linearization around the approximate solution.} We wish to study the differential of~$\m{F}_m$ around our approximate solution,~$(\phi_0,c_0)$. Introduce the space
\begin{equation*}
\m{V}:=\left\lbrace\phi\in\m{C}^\infty([0,1])\mid \phi(0)=\phi'(0)=\phi(1)=\phi'(1)=0\right\rbrace.
\end{equation*}
The tangent space to~$\m{V}_m\times\bb{R}_{>0}$ is~$\m{V}\times\bb{R}$. The linearization 
\begin{equation*}
\left(D\m{F}_m\right)_{(\phi,c)}:\m{V}\times\bb{R}\to\m{C}^{\infty}_0([0,1])
\end{equation*}
around a point~$(\phi,c)\in\m{V}_m\times\bb{R}_{>0}$ is
\begin{equation}\label{eq:linearizzazione}
\begin{split}
\left(D\m{F}_m\right)_{(\phi,c)}(u,k&)=u''+\frac{2\,m\,u'}{1+\lambda m}-\frac{k\,\mrm{exp}\!\left(\int_{\frac{1}{2}}^{\lambda}\frac{m}{\phi}\mrm{d}x\right)}{m^2(1+\lambda m)}\left((\phi'+m)^2+\phi\,\phi''\right)+\\
&+\frac{c}{m^2}\,\frac{\mrm{exp}\!\left(\int_{\frac{1}{2}}^{\lambda}\frac{m}{\phi}\mrm{d}x\right)}{1+\lambda m}\left(\int_{\frac{1}{2}}^\lambda\frac{m\,u}{\phi^2}\mrm{d}x\right)\left((\phi'+m)^2+\phi\,\phi''\right)-\\
&-\frac{c}{m^2}\,\frac{\mrm{exp}\!\left(\int_{\frac{1}{2}}^{\lambda}\frac{m}{\phi}\mrm{d}x\right)}{1+\lambda m}\left(2(\phi'+m)u'+u\,\phi''+\phi\,u''\right).
\end{split}
\end{equation}
Consider now the linearization around the approximate solution~$(\phi_0,c_0)$. Taking into account~\eqref{eq:prima_approssimazione} and the fact that~$\phi_0(\lambda)\in O(m)$, we have for the various terms in the linearized operator:
\begin{equation*}
\begin{gathered}
u''+2\,m\frac{u'}{1+\lambda m}=u''+O(m);\\
\frac{k}{m^2}\frac{\mrm{exp}\!\left(\int_{\frac{1}{2}}^{\lambda}\frac{m}{\phi_0}\mrm{d}x\right)}{1+\lambda m}\left((\phi_0'+m)^2+\phi_0\,\phi_0''\right)=-2\,k\,(3\,\lambda^2-2\,\lambda)+O(m);\\
\frac{c_0}{m^2}\frac{\mrm{exp}\!\left(\int_{\frac{1}{2}}^{\lambda}\frac{m}{\phi_0}\mrm{d}x\right)}{1+\lambda m}\left(\int_{\frac{1}{2}}^\lambda\frac{m\,u}{\phi_0^2}\mrm{d}x\right)\left((\phi_0'(\lambda)+m)^2+\phi_0\,\phi_0''\right)=O(m);\\
\frac{c_0}{m^2}\,\frac{\mrm{exp}\!\left(\int_{\frac{1}{2}}^{\lambda}\frac{m}{\phi_0}\mrm{d}x\right)}{1+\lambda m}\left(2(\phi_0'+m)u'+u\,\phi_0''+\phi_0\,u''\right)=O(m).
\end{gathered}
\end{equation*}
Hence we see that the differential of~$\m{F}_m$ at the point~$(\phi_0,c_0)$ is
\begin{equation*}
\left(D\m{F}_m\right)_{(\phi_0,c_0)}(u,k)=u''+2\,k(3\,\lambda^2-2\,\lambda)+O(m).
\end{equation*}
\begin{lemma}\label{lemma:linearizzazione_primo_ordine}
The following map is an isomorphism:
\begin{equation}\label{eq:linearizzazione_primo_ordine}
\begin{split}
D:\m{V}\times\bb{R}&\to\m{C}^\infty_0([0,1])\\
(u,k)&\mapsto u''+2\,k(3\,\lambda^2-2\,\lambda).
\end{split}
\end{equation}
\end{lemma}
\begin{proof}
Fix~$f\in\m{C}^\infty([0,1])$ and consider
\begin{equation*}
u''(\lambda)+2\,k(3\,\lambda^2-2\,\lambda)=f(\lambda)
\end{equation*}
as a differential equation for~$u(\lambda)$. The general solution is given by
\begin{equation*}
u(\lambda)=\int_0^\lambda\left(\int_0^yf(x)\mrm{d}x\right)\mrm{d}y-2\,k\left(\frac{\lambda^4}{4}-\frac{\lambda^3}{3}\right)+C_1\,\lambda+C_2
\end{equation*}
for constants~$C_1$,~$C_2$. There is a \textit{unique} choice of~$k$,~$C_1$,~$C_2$ such that this solution~$u$ lies in~$\m{V}$, namely
\begin{equation*}
C_1=C_2=0\mbox{ and }k=-6\,\int_0^1\left(\int_0^yf(x)\mrm{d}x\right)\mrm{d}y.\qedhere
\end{equation*}
\end{proof}
The proof of Lemma~\ref{lemma:linearizzazione_primo_ordine} gives in particular an explicit inverse to the zeroth-order part of~$\left(D\m{F}_m\right)_{(\phi_0,c_0)}$.
\begin{rmk}
The linearisation of the more complicated equation corresponding to the choice of complexification~\eqref{eq:low_rank_norm} is in fact just the same as~$D\m{F}_m$, up to~$O(m)$ terms, so Lemma~\ref{lemma:linearizzazione_primo_ordine} also applies to that case.  
\end{rmk}

\subsection{Solutions in the adiabatic limit}

Since the linearization around the approximate solution~$(\phi_0,c_0)$ is an isomorphism up to~$O(m)$-terms, to obtain an exact solution we can use Lemma~\ref{lemma:inversione} together with the following result, a quantitative version of the Inverse Function Theorem.
\begin{lemma}[Theorem~$5.3$ in~\cite{Fine_phd}]\label{lemma:funzione_inversa_quantitativo}
Let~$F:B_1\to B_2$ be a differentiable map between Banach spaces, with derivative~$DF:B_1\to B_2$ at~$0$. Assume that~$DF$ is an isomorphism, with inverse~$P$, and let~$\delta$ be such that~$F-DF$ is Lipschitz on the ball~$B(0,\delta)$ with a Lipschitz constant~$l\leq (2\norm{P})^{-1}$. Then, for any~$y\in B_2$ such that~$\norm{y-F(0)}<\delta\,(2\norm{P})^{-1}$ there is a unique~$x$ in~$B_1$ such that~$\norm{x}<\delta$ and~$F(x)=y$.  
\end{lemma}
In order to apply these results we embed~$\m{V}\times\bb{R}$ and~$\m{C}^{\infty}_0([0,1])$ into Banach spaces as follows:
\begin{itemize}
\item the first Banach space is the closure~$\overline{\m{V}}$ of~$\m{V}$ in~$\m{C}^{l+2,\beta}([0,1])$, with the usual H\"older norm, for~$l$ large enough and~$0<\beta< 1$. We can then take the direct sum of this space with~$\bb{R}$, and we let~$\left(\overline{\m{V}}\times{\bb{R}},\norm{.}\right)$ be the resulting Banach space;
\item for~$\m{C}^{\infty}_0([0,1])$, we'll just consider it as a subset of~$\m{C}^{l,\beta}_0([0,1])$.
\end{itemize}
We have the following estimate for the norm of the operator~$D$ defined in equation~\eqref{eq:linearizzazione_primo_ordine} (that is the zeroth-order part of the linearization of~$\m{F}_m$ around the approximate solution~$(\phi_0,c_0)$):
\begin{equation*}
\begin{split}
\norm{D(u,k)}_{\m{C}^{l,\beta}}&\leq\norm{u''}_{\m{C}^{l,\beta}}+2\,\abs{k}\,\norm{3\,\lambda^2-2\,\lambda }_{\m{C}^{l,\beta}}\leq\\
&\leq\norm{u}_{\m{C}^{l+2,\beta}}+22\,\abs{k}\leq 22\,\norm{(u,k)}.
\end{split}
\end{equation*}
In order to prove a similar estimate for the inverse, fix~$f\in \m{C}^{l,\beta}_0([0,1])$ and let~$(u_0,k_0):=D^{-1}(f)$. Then
\begin{equation*}
\abs{k_0}=\abs*{6\,\int_0^1\left(\int_0^yf(x)\mrm{d}x\right)\mrm{d}y}\leq 3\,\sup{f}\leq 3\,\norm{f}_{\m{C}^{l,\beta}}
\end{equation*}
\begin{equation*}
\begin{split}
\norm{u_0}_{\m{C}^{l+2,\beta}}=&\norm*{\int_0^\lambda\left(\int_0^yf(x)\mrm{d}x\right)\mrm{d}y+2\,k_0\left(\frac{\lambda^4}{4}-\frac{\lambda^3}{3}\right)}_{\m{C}^{l,\beta}}\leq\\
\leq&\norm*{\int_0^\lambda\left(\int_0^yf(x)\mrm{d}x\right)\mrm{d}y}_{\m{C}^{l,\beta}}+2\,\abs{k_0}\,\norm*{\frac{\lambda^4}{4}-\frac{\lambda^3}{3}}_{\m{C}^{l,\beta}}< 70\,\norm{f}_{\m{C}^{l,\beta}}
\end{split}
\end{equation*}
This shows
\begin{equation*}
\norm{D^{-1}(f)}< 73\,\norm{f}_{\m{C}^{l,\beta}}.
\end{equation*}

\begin{lemma}
For all sufficiently small~$m > 0$ the map~$\left(D\m{F}_m\right)_{(\phi_0,c_0)}$ is a linear isomorphism of Banach spaces. Moreover the norm of its inverse is less than~$146$.
\end{lemma}
\begin{proof}
We can use Lemma~\ref{lemma:inversione}; indeed, we know that~$\left(D\m{F}_m\right)_{(\phi_0,c_0)} - D\in O(m)$ so for~$m$ small enough we'll have that the norm of the difference is less than~$\frac{1}{146}$, as is needed to apply the Lemma.
\end{proof}
\begin{rmk}
In fact precise estimates for the norm of~$\left(D\m{F}_m\right)_{(\phi_0,c_0)}$ and its inverse are not needed. We only require the norm of the inverse to be controlled by a quantity which is independent of~$m$ and~$l$. In what follows we will write simply~$N$ for the norm of~$\left(D\m{F}_m\right)_{(\phi_0,c_0)}^{-1}$.
\end{rmk}

We showed that for~$m$ small enough we have an approximate solution~$(\phi_0,c_0)$, depending on~$m$, to the equation~$\m{F}_m=0$, such that
\begin{equation*}
\m{F}_m(\phi_0,c_0)=O(m^3). 
\end{equation*}
Moreover, we know that the differential of~$\m{F}$ around this approximate solution is an isomorphism of Banach spaces. Our next step is to use Lemma~\ref{lemma:funzione_inversa_quantitativo} to show that for small enough~$m$ we have a genuine solution to~$\m{F}_m=0$.

Let~$\m{G}_m:\overline{\m{V}}\times\bb{R}\to L^\infty_0([0,1])$ be defined as
\begin{equation*}
\m{G}_m(u,c):=\m{F}_m(\phi_0+u,c_0+c).
\end{equation*}
The differential of~$\m{G}_m$ at~$0$ is just~$\left(D\m{F}_m\right)_{(\phi_0,c_0)}$, so it is an isomorphism. Then Lemma~\ref{lemma:funzione_inversa_quantitativo} tells us that, if~$\delta$ is the radius of a ball over which~$\m{G}_m-D\m{G}_m$ is Lipschitz with a constant that is less than~$\frac{1}{N}$, then for any~$y$ such that~$\norm{y-\m{G}_m(0)}\leq\frac{\delta}{N}$ there is a unique~$x$ such that~$\norm{x}<\delta$ and~$\m{G}_m(x)=y$.

As~$\m{G}_m(0)=O(m^3)$, in order to apply the result, we need to show that~$\delta$ can be chosen to vanish slower than~$m^3$ as~$m \to 0$. 

However we also want~$(\Phi,C)$ to satisfy some additional conditions:~$\Phi$ should be strictly positive in the interior of~$[0,1]$, and~$C$ should be positive. The approximate solution satisfies these conditions, however~$\phi_0\in O(m)$ and~$c_0\in O(m^2)$; so in order to preserve positivity we need to choose a radius~$\delta$ that goes to~$0$ faster than~$m^2$ as~$m \to 0$ 

The next result shows that we can choose~$\delta$ as required.

\begin{lemma}\label{lemma:costante_lipschitz}
Let~$k\geq 2$. If~$\delta\in O(m^k)$ then for~$m$ small enough~$\m{G}_m-D\m{G}_m$ is Lipschitz on~$B(0,\delta)\subset\overline{\m{V}}\times\bb{R}$ with Lipschitz constant smaller than~$\frac{1}{N}$.
\end{lemma}

This tells us that for a small enough~$m$ we can choose~$\delta$ in such a way that the solution to the equation that we have found satisfies the positivity conditions; it is enough to use Lemma~\ref{lemma:costante_lipschitz} for~$k=2+\frac{1}{2}$.

\begin{proof}[Proof of Lemma~\ref{lemma:costante_lipschitz}]
Let~$\m{N}_m:=\m{G}_m-D\m{G}_m$ be the non-linear part of~$\m{G}_m$. For~$a,b\in B(0,\delta)$, the Mean Value Theorem implies
\begin{equation*}
\norm{\m{N}_m(a)-\m{N}_m(b)}_{\m{C}^{l,\beta}}\leq \norm{a-b}_{\m{C}^{l+2,\beta}}\cdot\mrm{sup}_{z\in B(0,\delta)}\norm{\left(D\m{N}_m\right)_z}.
\end{equation*}
For~$z\in B(0,\delta)$,
\begin{equation*}
\begin{split}
\left(D\m{N}_m\right)_z(\varphi)=&\left(D\m{G}_m\right)_z(\varphi)-\left(D\m{G}_m\right)_0(\varphi)=\\
=&\left(D\m{F}_m\right)_{(\phi_0,c_0)+z}(\varphi)-\left(D\m{F}_m\right)_{(\phi_0,c_0)}(\varphi).
\end{split}
\end{equation*}
We will show that this quantity is~$O(m)$, if~$\delta\in O(m^2)$. Since~$O(m^k)\subseteq O(m^2)$ for~$k\geq 2$, this will give us the thesis.

To prove the claim, let~$z=:(\tilde{y},\tilde{c})$; if~$\delta$ is~$O(m^2)$, since~$\norm{z}\leq\delta$ also~$\tilde{y}=O(m^2)$ and~$\tilde{c}=O(m^2)$. The linearization of~$\m{F}_m$ at~$(\phi,c):=(\phi_0,c_0)+z$ is given by equation~\eqref{eq:linearizzazione}
\begin{equation*}
\begin{split}
\left(D\m{F}_m\right)_{(\phi,c)}(u,k)=&u''+\frac{2\,m\,u'}{1+\lambda m}-\frac{k\,\mrm{exp}\!\left(\int_{\frac{1}{2}}^{\lambda}\frac{m}{\phi}\mrm{d}x\right)}{m^2(1+\lambda m)}\left((\phi'+m)^2+\phi\,\phi''\right)+\\
+&\frac{c}{m^2}\frac{\mrm{exp}\!\left(\int_{\frac{1}{2}}^{\lambda}\frac{m}{\phi}\mrm{d}x\right)}{1+\lambda m}\left(\int_{\frac{1}{2}}^\lambda\frac{m\,u}{\phi^2}\mrm{d}x\right)\left((\phi'+m)^2+\phi\,\phi''\right)-\\
&-\frac{c}{m^2}\,\frac{\mrm{exp}\!\left(\int_{\frac{1}{2}}^{\lambda}\frac{m}{\phi}\mrm{d}x\right)}{1+\lambda m}\left(2(\phi'+m)u'+u\,\phi''+\phi\,u''\right).
\end{split}
\end{equation*}
Let us consider the series expansions:
\begin{equation*}
\begin{split}
\mrm{exp}\!\left(\int_{\frac{1}{2}}^{\lambda}\frac{m}{\phi}\mrm{d}x\right)=&\mrm{exp}\!\left(\int_{\frac{1}{2}}^{\lambda}\frac{m}{\phi_0+\tilde{y}}\mrm{d}x\right)=\mrm{exp}\!\left(\int_{\frac{1}{2}}^{\lambda}\frac{m}{\phi_0}-\frac{m}{\phi_0^2}\tilde{y}+\dots\mrm{d}x\right)=\\
=&\mrm{exp}\!\left(\int_{\frac{1}{2}}^{\lambda}\frac{m}{\phi_0}+O(m)\,\mrm{d}x\right)=\mrm{exp}\!\left(\int_{\frac{1}{2}}^{\lambda}\frac{m}{\phi_0}\mrm{d}x\right)+O(m),
\end{split}
\end{equation*}
\begin{equation*}
\begin{split}
&(\phi'(\lambda)+m)^2+\phi(\lambda)\,\phi''(\lambda)=(\phi_0'+\tilde{y}'+m)^2+(\phi_0+\tilde{y})(\phi_0''+\tilde{y}'')=\\
&=(\phi_0'+m)^2+(\tilde{y}')^2+2(\phi_0'+m)\tilde{y}'+\phi_0\,\phi_0''+\tilde{y}\,\phi_0''+\phi_0\,\tilde{y}''+\tilde{y}\,\tilde{y}''=\\
&=(\phi_0'+m)^2+\phi_0\,\phi_0''+O(m^3).
\end{split}
\end{equation*}
So we have
\begin{equation*}
\begin{split}
\frac{k}{m^2}&\frac{\mrm{exp}\!\left(\int_{\frac{1}{2}}^{\lambda}\frac{m}{\phi}\mrm{d}x\right)}{1+\lambda m}\left((\phi'+m)^2+\phi\,\phi''\right)=\\
&=\frac{k}{m^2}\frac{\mrm{exp}\!\left(\int_{\frac{1}{2}}^{\lambda}\frac{m}{\phi_0}\mrm{d}x\right)}{1+\lambda m}\,\left((\phi_0'+m)^2+\phi_0\,\phi_0''\right)+O(m).
\end{split}
\end{equation*}
For the other terms, recalling that~$c=c_0+\tilde{c}=O(m^2)$, we simply have
\begin{equation*}
\frac{c}{m^2}\frac{\mrm{exp}\!\left(\int_{\frac{1}{2}}^{\lambda}\frac{m}{\phi}\mrm{d}x\right)}{1+\lambda m}\left(\int_{\frac{1}{2}}^\lambda\frac{m\,u}{\phi^2}\mrm{d}x\right)\left((\phi'+m)^2+\phi\,\phi''\right)=O(m)
\end{equation*}
\begin{equation*}
\frac{c}{m^2}\,\frac{\mrm{exp}\!\left(\int_{\frac{1}{2}}^{\lambda}\frac{m}{\phi}\mrm{d}x\right)}{1+\lambda m}\left(2(\phi'+m)u'+u\,\phi''+\phi\,u''\right)=O(m).
\end{equation*}
As a consequence we find that, up to~$O(m)$-terms
\begin{equation*}
\begin{split}
\left(D\m{F}_m\right)_{(\phi,c)}(u,k)=&u''+\frac{2m\,u'}{1+\lambda m}-\frac{k}{m^2}\frac{\mrm{exp}\!\left(\int_{\frac{1}{2}}^{\lambda}\!\frac{m}{\phi_0}\mrm{d}x\right)}{1+\lambda m}\left((\phi_0'+m)^2+\phi_0\phi_0''\right)=\\
=&\left(D\m{F}_m\right)_{(\phi_0,c_0)}(u,k).
\end{split}
\end{equation*}
Then for~$z\in B(0,\delta)$,~$\norm{\left(D\m{N}_m\right)_z}$ is~$O(m)$. Hence for~$m$ small enough, on a ball of radius~$m^2$ the Lipschitz constant of~$\m{N}_m$ will be smaller than~$N^{-1}$.
\end{proof}
Lemma~\ref{lemma:costante_lipschitz} settles the problem of existence of a solution~$(\Phi,C)\in\m{C}^{l+2,\beta}([0,1])\times\bb{R}$ of~$\m{F}_m(\phi,c)=0$. To prove smoothness we consider the existence result we just showed for increasing values of~$l$, with corresponding solutions~$\Phi_l$. The uniqueness statement in Lemma~\ref{lemma:funzione_inversa_quantitativo}, together with the fact that~$\norm{u}_{\m{C}^{l,\beta}}\leq\norm{u}_{\m{C}^{l+1,\beta}}$, implies that actually all the various~$\Phi_l$ are the same function, that is of course smooth.
\begin{rmk}
Given our previous remarks, it is straightforward to check that the same proof works for the more complicated equation corresponding to~\eqref{eq:low_rank_norm}.  
\end{rmk}

\chapter{The HcscK system in symplectic coordinates}\label{chap:symplectic_coords}

\chaptertoc{}

\bigskip

In the first part of this Chapter we study the HcscK system in the special case of a compact complex torus, focusing in particular on the~$2$-dimensional case. In the second part instead we study the system on a toric manifold, i.e. a compact K\"ahler~$n$-fold~$(M,\omega)$ equipped with a Hamiltonian~$\bb{T}^n$-action. We study the two situations together because they share an important feature, the existence of symplectic coordinates on a dense open subset of~$M$. The idea is to generalize Abreu's formula for the scalar curvature of a K\"ahler metric in symplectic coordinates, obtaining an interesting expression for the HcscK system under some simplifying assumptions. Our first result, Theorem~\ref{thm:HcscKSymplCoordThm}, shows that the equations~\eqref{eq:HCSCK_complex_relaxed} become more explicit and treatable when written in symplectic coordinates, mainly because they can be \emph{decoupled} by this simple change of variables. These results generalize those of~\cite{ScarpaStoppa_HcscK_abelian} to abelian varieties and toric manifolds of arbitrary dimension.

Section~\ref{sec:toric} focuses on \emph{obstructions} to the solvability of the HcscK system on a toric surface: we will see for example that the complex moment map can be solved just for a \emph{non-integrable} Higgs term~$\alpha$. More interestingly, we will show that a necessary condition for having a solution of the (toric) HcscK system is that the surface is~$K$-stable, so we have a first obstruction to the existence of solutions to the system, at least on a toric surface. This result is a consequence of an ``integration by parts formula'', inspired by a result in~\cite{Donaldson_stability_toric} that is crucial to study the cscK equation on toric varieties, see the discussion in~\cite[\S$6$]{Ross_toric}.

It is easier to start the study of the HcscK system in symplectic coordinates on a complex torus, following the approach of~\cite{FengSzekelyhidi_abelian} to the study of the prescribed scalar curvature problem. Without loss of generality, we can assume that the torus is in fact the abelian variety
\begin{equation*}
M=\mathbb{C}^n/(\mathbb{Z}^n+\I\mathbb{Z}^n),
\end{equation*}
on which we have complex coordinates~$\bm{z}=\bm{x}+\I\bm{w}$ induced by the projection~$\bb{C}^n\to M$. We fix the flat K\"ahler form 
\begin{equation*}
\omega_0=\I\sum_{a}\mrm{d}z^a\wedge\mrm{d}\bar{z}^a
\end{equation*}
and we want to find solutions to~\eqref{eq:HCSCK_complex_relaxed} for a K\"ahler form in~$\omega\in[\omega_0]$ and a first-order deformation~$\alpha$ of the standard complex structure induced by the projection~$\bb{C}^n\to M$. Even on this simple geometry the equations~\eqref{eq:HcscK_system} are highly non-trivial, so we will work under some additional assumptions. There is a real torus~$\bb{T}\subset \operatorname{Symp}(M,\omega_0)$, with~$\bb{T}\cong \bb{R}^n/\bb{Z}^n$, acting on~$M$ by translations   
\begin{equation*}
\bm{t}\cdot(\bm{x}+\I\bm{w})=\bm{x}+\I(\bm{w}+\bm{t}).
\end{equation*} 
We will study a particular set of solutions to the HcscK system such that the symplectic form~$\omega$ and the endomorphism~$\alpha$ are both~$\mathbb{T}$-invariant.

We can assume that the K\"ahler form~$\omega$ has a local (real) potential of the form
\begin{equation*}
f(\bm{z})=\abs{\bm{z}}^2+4\,h(\bm{x})
\end{equation*}
for a function~$h$ that is~$\bb{Z}^2$-periodic in~$\bm{x}$. Notice that, if we define~$v(\bm{x})$ as
\begin{equation*}
v(\bm{x})=\frac{1}{2}\abs{\bm{x}}^2+h(\bm{x})
\end{equation*}
then the matrix of the metric defined by~$\omega$ and~$J$ is the Hessian of~$v(\bm{x})$:
\begin{equation*}
\diff_{z^a}\diff_{\bar{z}^b}f(\bm{z})=\diff_{x^a}\diff_{x^b}v(\bm{x}).
\end{equation*}
If we let~$v^{ab}$ be the inverse matrix of~$\mrm{Hess}_{\bm{x}}v$ then we have the following formula for the scalar curvature of our metric:
\begin{equation*}
S(\omega)=-g^{a\bar{b}}\diff_{\bar{b}}\Gamma^c_{ca}=-\frac{1}{4}v^{ab}\diff_{x^b}\left(v^{cd}\diff_{x^c}v_{,ad}\right).
\end{equation*}
It is an observation of Abreu~\cite{Abreu_toric} that, under the change of coordinates given by the Legendre transform of~$v$, this expression of the scalar curvature greatly simplifies. Consider the system
\begin{equation*}
\begin{cases}
\bm{x}=\diff_{\bm{y}}u\\
\bm{y}=\diff_{\bm{x}}v\\
u(\bm{y})+v(\bm{x})=\bm{y}\cdot\bm{x}
\end{cases}
\end{equation*}
defining the Legendre transform~$u(\bm{y})$ of~$v(\bm{x})$. Then the identities
\begin{center}
\begin{tabular}{l l}
$\mrm{Hess}_{\bm{y}}u(\bm{y})=(\mrm{Hess}_{\bm{x}}v(\bm{x}))^{-1}$&$v^{ab}(\bm{x})\diff_{x^a}=\diff_{y^b}$\\
$\mrm{Hess}_{\bm{x}}v(\bm{x})=(\mrm{Hess}_{\bm{y}}u(\bm{y}))^{-1}$&$u^{ab}(\bm{y})\diff_{y^a}=\diff_{x^b}$
\end{tabular}
\end{center}
allow us to write the scalar curvature of~$(\omega,J)$ as
\begin{equation}\label{eq:Abreu_scalar}
S(\omega)=-\frac{1}{4}v^{ab}\diff_{x^b}\left(v^{cd}\diff_{x^c}v_{,ad}\right)=-\frac{1}{4}\diff_{y^a}\diff_{y^d}u^{ad}=-\frac{1}{4}\left(u^{ab}\right)_{,ab}.
\end{equation}
In these coordinates~$(\bm{y},\bm{w})$ the K\"ahler form~$\omega$ becomes the standard symplectic form
\begin{equation*}
\omega=\I\sum_{a=1}^n\mrm{d}y^a\wedge\mrm{d}w^a
\end{equation*}
while the complex structure~$J$ takes the form
\begin{equation*}
J(\bm{y},\bm{w}) = J(\bm{y}) = \begin{pmatrix} 0 & -(\mrm{Hess}_{\bm{y}}u)^{-1} \\
\mrm{Hess}_{\bm{y}}u & 0
\end{pmatrix}.
\end{equation*}
So we could equivalently fix the flat metric on~$M$ and consider instead the complex structure~$J$ as a variable in the HcscK system. In other words, the duality between complex and symplectic coordinates gives an alternative description of the complexification of the action~$\mscr{G}\curvearrowright\mscr{J}$. We refer to~\cite[\S$3.1$]{Donaldson_stability_toric} for more details on this point of view.

In Section~\ref{sec:symplectic_coords_proof} we will write the HcscK system~\eqref{eq:HCSCK_complex_relaxed} in symplectic coordinates, obtaining an expression similar to~\eqref{eq:Abreu_scalar}. To do so we will write the HcscK system in terms of the quadratic differential~$q$ defined by~$\alpha$ and~$\omega$, similarly to what we did in Section~\ref{sec:curve}. Under the assumption that~$\alpha$ is~$\bb{T}$-invariant, we have
\begin{equation*}
\alpha=q_{ac}(\bm{x})v^{cb}(\bm{x})\diff_{\bar{z}^b}\otimes\mrm{d}z^a
\end{equation*}
and we let~$\xi^{ab}(\bm{y})$ be the image under Legendre duality of the matrix~$q_{ab}(\bm{x})$, so that
\begin{equation*}
\alpha=\sum_c\xi^{ac}(\bm{y})u_{cb}(\bm{y})\diff_{\bar{z}^b}\otimes\mrm{d}z^a.
\end{equation*}
The HcscK system can thus be written in terms of~$\mrm{Hess}(u)$ and~$\xi$; these are functions on the real torus~$\mathbb{T}$ with values in complex symmetric matrices, and~$\mrm{Hess}(u)$ is real and positive-definite. In what follows we occasionally refer to~$u$ as the \emph{symplectic potential} of the metric, and we will often use~$G$ or~$D^2u$ to denote the Hessian matrix of~$u$.

Notice that with our new set of variables~$(u,\xi)$ we can express the morphism~$\alpha\bar{\alpha}$ as~$\xi G\bar{\xi}G$, so the eigenvalues of~$\alpha\bar{\alpha}$ are in fact the eigenvalues of~$\xi G\bar{\xi}G$. The morphism~$\hat{\alpha}$ can also be written in terms of~$u$ and~$\xi$, since it is computed as a spectral function of~$\alpha\bar{\alpha}$. If we let~$\check{\alpha}$ be the matrix representing~$\hat{\alpha}$ in a system of symplectic coordinates, i.e.
\begin{equation*}
\hat{\alpha}=\sum_{a,b}\check{\alpha}\indices{^a_b}(\bm{y})\diff_{z^b}\otimes\mrm{d}z^a
\end{equation*}
then the matrix~$\check{\alpha}$ can be expressed as
\begin{equation}\label{eq:alpha_symp_check}
\check{\alpha}=\sum_i\frac{1}{2}\left(1+\sqrt{1-\delta(i)}\right)^{-1}\prod_{j\not=i}\frac{\delta(j)\mathbbm{1}-\xi G\bar{\xi}G}{\delta(j)-\delta(i)}=\psi(\xi\,G\,\bar{\xi}\,G).
\end{equation}
\begin{thm}\label{thm:HcscKSymplCoordThm}
The HcscK equations~\eqref{eq:HCSCK_complex_relaxed} on~$M$ for~$\mathbb{T}$-invariant~$\omega$,~$J$ and~$\alpha$, are equivalent to the system of \emph{uncoupled} partial differential equations  
\begin{equation}\label{eq:HcscKSymplRealEq}
\begin{cases}
\left(\xi^{ab}\right)_{,ab}=0\\
 \left(-\frac{1}{2}G^{ab}+(\check{\alpha}\xi G\bar{\xi})^{ab}\right)_{,ab}=0, 
\end{cases}
\end{equation}
corresponding to the vanishing of the complex and real moment map respectively.
\end{thm}
\begin{rmk}
Since~$x\psi(x)=\frac{1}{2}(1-\sqrt{1-x})$, we could also write the real moment map equation in~\eqref{eq:HcscKSymplRealEq} as
\begin{equation*}
\left(\left(1-\xi\,G\,\bar{\xi}\,G\right)^{\frac{1}{2}}G^{-1}\right)^{ab}_{,ab}=0
\end{equation*}
obtaining the expression in~\eqref{eq:symplectic_HcscK_intro}.
\end{rmk}
To study the HcscK system on an abelian variety we fix a solution~$\xi$ to the first equation in~\eqref{eq:HcscKSymplRealEq}, and we will study the second equation for a symplectic potential~$u$. In~\cite{ScarpaStoppa_HcscK_abelian} we have shown that, in complex dimension~$2$, for every small enough~$\xi$ there is a (unique) solution to the real moment map equation, using results developed in~\cite{FengSzekelyhidi_abelian} to study Abreu's equation on abelian varieties.

This point of view for studying the HcscK system in symplectic coordinates is also motivated by a variational characterisation of the real moment map equation in~\eqref{eq:HcscKSymplRealEq} in terms of the Biquard-Gauduchon function. To describe this, let~$H(u,\xi)$ be the Biquard-Gauduchon function of the corresponding first-order deformation of the complex structure~$\alpha$, c.f. Proposition~\ref{prop:BiquardGauduchonAC+}
\begin{equation*}
H(u,\xi):=\rho(\alpha)=\sum_{i=1}^n 1-\sqrt{1-\delta(i)}+\log\frac{1+\sqrt{1-\delta(i)}}{2}
\end{equation*}
where~$\delta(i)$ are the eigenvalues of~$\alpha\bar{\alpha}=\xi\,G\,\bar{\xi}\,G$, and define the \emph{Biquard-Gauduchon functional} as
\begin{equation}\label{eq:BGFunc}
\m{H}(u, \xi) = \frac{1}{2} \int_{\mathbb{T}}H(u,\xi)\,\mrm{d}\mu
\end{equation}
where~$\mrm{d}\mu$ is the Lebesgue measure. The \emph{periodic K-energy} (see~\cite[\S$2$]{FengSzekelyhidi_abelian}) is given by
\begin{equation}\label{eq:K_energy_periodic}
\m{F}(u) = -\frac{1}{2} \int_{\mathbb{T}}\log\det(D^2 u)\,\mrm{d}\mu.
\end{equation}
We define the \emph{periodic HK-energy} as
\begin{equation}\label{eq:HK_energy}
\widehat{\m{F}}(u, \xi) = \m{F}(u) + \m{H}(u,\xi).
\end{equation}
Note that, for fixed~$\xi$, these functionals are well defined on the set
\begin{equation*}
\mscr{A}(\xi) =\set*{u\mbox{ symplectic potential}\tc \SpecRad(u, \xi) < 1}
\end{equation*}
where~$\SpecRad(u,\xi)$ is the \emph{spectral radius} of~$\alpha\bar{\alpha}=\xi G\bar{\xi}G$, i.e. the largest eigenvalue of~$\alpha\bar{\alpha}$. Pairs~$(u,\xi)$ corresponding to points of~$\cotJ$ satisfy~$u\in\mscr{A}(\xi)$. 
\begin{rmk}
The set~$\mscr{A}(\xi)$ is not convex, in general, even if we assume that~$\xi$ is definite.
\end{rmk}
\begin{prop}\label{HKEulLagr}
For a fixed Higgs tensor~$\xi$, the real moment map equation in~\eqref{eq:HcscKSymplRealEq} is the Euler-Lagrange equation, with respect to variations of the symplectic potential~$u$, for the periodic HK-energy~$\widehat{\m{F}}(u,\xi)$.
\end{prop}
\begin{proof}
Let~$u(t) =u+t\,\varphi$ be a path of symplectic potentials in~$\mscr{A}(\xi)$, for a periodic function~$\varphi$, and compute the derivative of the Biquard-Gauduchon function (c.f. the proof of Proposition~\ref{prop:BiquardGauduchonAC+})
\begin{equation*}
\diff_t\rho(u_t,\xi)=\sum_i\frac{\diff_t\delta(i)}{2\left(1+\sqrt{1-\delta(i)}\right)}=\sum_i\frac{\mrm{Tr}\left( \underset{j\not=i}{\prod}\frac{\delta(j)\mathbbm{1}-\xi\,G\,\bar{\xi}\,G}{\delta(j)-\delta(i)}\ \diff_t\left(\xi\,G_t\,\bar{\xi}\,G_t\right)\right)}{2\left(1+\sqrt{1-\delta(i)}\right)}
\end{equation*}
and from the definition of~$\check{\alpha}$ we have~$\diff_t\rho(u_t,\xi)=\mrm{Tr}\left(\check{\alpha}\,\diff_t\left(\xi\,G_t\,\bar{\xi}\,G_t\right)\right)$. A quick computation shows that~$\hat{\alpha}\,\alpha=\alpha\,\bar{\hat{\alpha}}$ and~$\bar{\alpha}\,\hat{\alpha}=\bar{\hat{\alpha}}\,\bar{\alpha}$, so we get
\begin{equation*}
\begin{split}
\diff_t\rho(u_t,\xi)=&\mrm{Tr}\left(\check{\alpha}\,\xi\,\dot{G}\,\bar{\xi}\,G\right)+\mrm{Tr}\left(\check{\alpha}\,\xi\,G\,\bar{\xi}\,\dot{G}\right)=\\
=&\mrm{Tr}\left(\bar{\check{\alpha}}\,\bar{\xi}\,G\,\xi\,\dot{G}\right)+\mrm{Tr}\left(\check{\alpha}\,\xi\,G\,\bar{\xi}\,\dot{G}\right).
\end{split}
\end{equation*}
Since~$G=\mrm{Hess}(u)$ is a real matrix we find
\begin{equation*}
\diff_t\rho(u_t,\xi)=2\,\mrm{Re}\,\mrm{Tr}\left(\check{\alpha}\,\xi\,G\,\bar{\xi}\,\dot{G}\right).
\end{equation*}
At this point computing the derivative of the Biquard-Gauduchon functional~\eqref{eq:BGFunc} is straightforward, since every~$\varphi\in T_u\mscr{A}$ is a periodic function:
\begin{equation*}
D\m{H}_{(u,\xi)}(\varphi)=\int\mrm{Re}\left(\check{\alpha}\indices{^a_b}\xi^{bc}u_{cd}\bar{\xi}^{de}\varphi_{,ea}\right)\mrm{d}\mu=\int\mrm{Re}\left(\check{\alpha}\,\xi\,G\,\bar{\xi}\right)^{ab}_{,ab}\varphi\,\mrm{d}\mu.
\end{equation*}
On the other hand the first variation of the (periodic) K-energy gives the scalar curvature term, see~\cite{FengSzekelyhidi_abelian}, so the variation of~$\hat{\m{F}}=\m{F}+\m{H}$ is
\begin{equation*}
D\hat{\m{F}}_{(u,\xi)}(\varphi)=-\frac{1}{2}\int u^{ab}_{,ab}\varphi\,\mrm{d}\mu+\int\mrm{Re}\left(\check{\alpha}\,\xi\,G\,\bar{\xi}\right)^{ab}_{,ab}\varphi\,\mrm{d}\mu.
\end{equation*}
To conclude the proof we note that~$\left(\check{\alpha}\,\xi\,G\,\bar{\xi}\right)^{ab}_{,ab}$ is a real quantity. Indeed, using the expression~\eqref{eq:alpha_symp_check} for~$\check{\alpha}$, we see that~$\check{\alpha}\,\xi\,G\,\bar{\xi}=\xi\,G\,\bar{\xi}\,\overline{\check{\alpha}^{\transpose}}$, from which it follows
\begin{equation*}
\left(\check{\alpha}\,\xi\,G\,\bar{\xi}\right)^{ab}_{,ab}=\left(\xi\,G\,\bar{\xi}\,\overline{\check{\alpha}^{\transpose}}\right)^{ab}_{,ab}=\left(\left(\xi\,G\,\bar{\xi}\,\overline{\check{\alpha}^{\transpose}}\right)^\transpose\right)^{ab}_{,ab}=\left(\overline{\check{\alpha}}\,\bar{\xi}\,G\,\xi\right)^{ab}_{,ab}.
\qedhere
\end{equation*}
\end{proof}
Our main application of this variational characterisation is a uniqueness result for surfaces, which relies in turn on a convexity property of the HK-energy that is proven in Section~\ref{sec:convexity}.

We will provide some examples of solutions to system~\eqref{eq:HcscKSymplRealEq}. First note that the complex moment map equation is~\eqref{eq:HcscKSymplRealEq} is linear. As a consequence the full space of solutions~$\xi$ can be described quite easily in momentum space. Indeed writing~$\xi_k$ for the Fourier modes of~$\xi$, the complex moment map equation is equivalent to the set of conditions 
\begin{equation*}
\bm{k}^\transpose\xi_{\bm{k}}\bm{k} = 0,\, \bm{k}\in\bb{Z}^n.
\end{equation*}
In~\cite{ScarpaStoppa_HcscK_abelian} we also proved a general existence result for the real moment map equation on surfaces.
\begin{thm}\label{thm:QuantPertThm}
In complex dimension~$2$ there exists a constant~$K > 0$, depending only on well-known elliptic estimates on the real torus~$\mathbb{T}$ with respect to the flat metric~$g_{J_0}$, such that if~$\xi$ satisfies~$\norm{\xi}_{u_0} < K$ then there exists a unique solution~$u$ to the real moment map equation in~\eqref{eq:HcscKSymplRealEq}, up to an additive constant.
\end{thm}
This result is quite implicit, but we discuss in detail the special case when the Higgs tensor~$\xi$ depends only on a single variable, say~$y^1$, and satisfies~$\det(\xi) = 0$. We show that in this case the real moment map equation is ``integrable'', i.e. it reduces to an algebraic condition, and use this to prove an effective existence result in Section~\ref{sec:1dim_solutions}. These results have been obtained just in complex dimension~$2$, since on a complex surface we have an alternative expression for the real moment map, as we explained in Section~\ref{sec:complex_surface}. We will write the resulting equivalent for of the real moment map in Section~\ref{sec:convexity}.

\begin{rmk}
We will show in Proposition~\ref{prop:integrability_xi} that the integrability condition for the deformation of complex structure~$\alpha$ is equivalent to~$\xi$ being of the form
\begin{equation*}
\xi=(D^2u)^{-1}\,D^2\varphi\,(D^2u)^{-1},
\end{equation*}
for a function~$\varphi$ whose Hessian is a periodic function of~$\bm{y}$. Thus a general solution to the complex moment map equation does not correspond to an integrable deformation of the complex structure, but rather to an almost K\"ahler deformation. We will also check that the translation-invariant solutions to Section~\ref{sec:1dim_solutions} are integrable only if~$\xi$ is constant.
\end{rmk}

\section{The HcscK system on abelian varieties}\label{sec:symplectic_coords_proof}

The proof of Theorem~\ref{thm:HcscKSymplCoordThm} relies on the interplay of complex and symplectic coordinates on a K\"ahler manifold, as in the computation of Abreu's equation~\eqref{eq:Abreu_scalar} for the scalar curvature. The key ingredient is the following computation.
\begin{lemma}\label{lemma:divergenza_coord_sympl}
Let~$X$ be a~$(1,0)$-vector field on~$M$ that is invariant under the~$\mathbb{T}^2$-action. Then in symplectic coordinates the divergence of~$X$ may be expressed as 
\begin{equation*}
\operatorname{div}(X)=\frac{1}{2}\diff_{y^a}\left(u^{ab}\chi_b(y)\right)
\end{equation*}
where we write~$\chi_b(y) =X^b(x)$ for the image under the Legendre duality of~$X$.
\end{lemma}
\begin{proof}
Using the fundamental relations of Legendre duality, we have
\begin{align*}
\nabla_aX^a=\frac{1}{2}\left(\diff_{x^a}X^a(x)+X^b(x)v^{ac}\diff_{x^a}v_{bc}\right)&=\frac{1}{2}\left(u^{ab}\diff_{y^b}\chi_a(y)+\chi_b(y)\diff_{y^c}u^{bc}\right)\\
&=\frac{1}{2}\diff_{y^a}\left(u^{ab}\chi_b(y)\right).\qedhere
\end{align*}
\end{proof}
Let us consider the complex moment map equation in~\eqref{eq:HcscK_system}, i.e. the linear PDE
\begin{equation*} 
\mrm{div}\left(\diff^*\alpha\right)=0
\end{equation*}
where~$\diff^*$ denotes the formal adjoint with respect to~$g_J=\omega(-,J-)$. With the previous notation we have~$\alpha\indices{_a^{\bar{b}}}(\bm{x})=q_{ac}(\bm{x})v^{cb}(\bm{x})$, and
\begin{equation*}
\left(\diff^*\alpha\right)^{\bar{b}}=-g^{a\bar{c}}\nabla_{\bar{c}}\alpha\indices{_a^{\bar{b}}}=-g^{a\bar{c}}g^{\bar{b}d}\diff_{\bar{z}^c}q_{ad}=-\frac{1}{2}v^{ac}v^{bd}\diff_{x^c}q_{ad}
\end{equation*}
so in symplectic coordinates we find
\begin{equation*}
\left(\diff^*\alpha\right)^{\bar{b}}(\bm{x})=-\frac{1}{2}u_{bd}\diff_{y^a}\xi^{ad}
\end{equation*}
and by Lemma~\ref{lemma:divergenza_coord_sympl} we get
\begin{equation*}
\operatorname{div}\left(\diff^*\alpha\right)=-\frac{1}{4}\diff_{y^c}\left(u^{bc}u_{bd}\diff_{y^a}\xi^{ad}\right)=-\frac{1}{4}\diff_{y^a}\diff_{y^b}\xi^{ab}=-\frac{1}{4}\left(\xi^{ab}\right)_{,ab}.
\end{equation*}
Thus the vanishing condition~$m^{\bb{C}}_{(J,\alpha)}(h) = 0, \, h \in \mathfrak{ham}(M, \omega)$ is equivalent to the linear PDE~$\left(\xi^{ab}\right)_{,ab}=0$ appearing in~\eqref{eq:HcscKSymplRealEq}.

The analogous computation for the real moment map is slightly more complicated. By Abreu's formula we have~$S(g_J) = -\frac{1}{4}(u^{ab})_{,ab}$, so we just have to compute the other terms of~\eqref{eq:HcscK_system} in symplectic coordinates. As these are in divergence form this can be achieved by applying Lemma~\ref{lemma:divergenza_coord_sympl}.

In order to compute the divergence term in the real moment map in symplectic coordinates we will use the following fact: if~$X$ is a real vector field and we decompose it as~$X=X^{1,0}+X^{0,1}$, we have~$X^{0,1}=\overline{X^{1,0}}$, so~$X=2\,\mathrm{Re}(X^{1,0})$ and
\begin{equation*}
\operatorname{div}(X)=\operatorname{div}\left(2\,\operatorname{Re}(X^{1,0})\right)=2\,\operatorname{Re}\left(\operatorname{div}(X^{1,0})\right).
\end{equation*}
Now the~$(1,0)$-part of the vector field~$X(\alpha)$ appearing in~\eqref{eq:HcscK_system} is
\begin{equation*}
X^a=g(\nabla^a\alpha,\bar{\alpha}\,\hat{\alpha})-g(\nabla^a\bar{\alpha},\alpha\,\bar{\hat{\alpha}})-2\nabla^*(\alpha\bar{\alpha}\,\hat{\alpha})^a
\end{equation*}
We examine the three terms separately:
\begin{equation*}
g(\nabla^a\alpha,\bar{\alpha}\,\hat{\alpha})=\frac{1}{2}\diff_{y^a}\xi^{bc}\,u_{bd}u_{ep}\bar{\xi}^{dp}\check{\alpha}\indices{^e_c};
\end{equation*}
\begin{equation*}
\begin{split}
g(\nabla^a\bar{\alpha},\alpha\,\bar{\hat{\alpha}})=&g(\nabla^a\bar{\alpha},\hat{\alpha}\alpha)=\\
=&\frac{1}{2}\left(\diff_{y^a}\bar{\xi}^{bc}\,u_{bd}\check{\alpha}\indices{^d_e}\xi^{ep}u_{pc}+2\diff_{y^a}u_{bd}\,\bar{\xi}^{bc}\check{\alpha}\indices{^d_e}\xi^{ep}u_{pc}\right);\\
\nabla^*(\alpha\bar{\alpha}\hat{\alpha})^a=&-\frac{1}{2}\diff_{y^c}\left(\xi^{cp}u_{dp}\bar{\xi}^{db}u_{be}\check{\alpha}\indices{^e_a}\right).
\end{split}
\end{equation*}
Putting everything together we have
\begin{equation}\label{eq:symplectic_acc1}
\begin{split}
2\,X^a(\bm{x})=&\diff_{a}\xi^{cb}\,u_{bd}(\bar{\xi}G)\indices{^d_e}\,\check{\alpha}\indices{^e_c}-\diff_{a}\bar{\xi}^{bc}\,u_{bd}\,\check{\alpha}\indices{^d_e}(\xi G)\indices{^e_c}-\\
&-2\diff_{y^a}u_{bd}\,\bar{\xi}^{bc}\check{\alpha}\indices{^d_e}(\xi G)\indices{^e_c}+2\diff_{y^c}\left((\xi G)\indices{^c_d}(\bar{\xi}G)\indices{^d_e}\check{\alpha}\indices{^e_a}\right).
\end{split}
\end{equation}
Since~$\check{\alpha}\,\xi\,G=\xi\,G\,\bar{\check{\alpha}}$ we can rewrite the first two terms in~\eqref{eq:symplectic_acc1} as
\begin{equation*}
\begin{gathered}
\diff_{a}\xi^{cb}\,u_{bd}(\bar{\xi}G)\indices{^d_e}\,\check{\alpha}\indices{^e_c}-\diff_{a}\bar{\xi}^{bc}\,u_{bd}\,\check{\alpha}\indices{^d_e}(\xi G)\indices{^e_c}=\\=\diff_{a}\xi^{cb}\,u_{bd}(\bar{\xi}G)\indices{^d_e}\,\check{\alpha}\indices{^e_c}-\diff_{a}\bar{\xi}^{bc}\,u_{bd}\,\bar{\check{\alpha}}\indices{^e_c}\,(\xi G)\indices{^d_e}
\end{gathered}
\end{equation*}
and in particular this quantity is purely imaginary; we have to compute the divergence of~$X^{1,0}+X^{0,1}$, so we can ignore these two terms in the computation. For the other terms in~\eqref{eq:symplectic_acc1} instead
\begin{equation*}
\begin{gathered}
-2\diff_{y^a}u_{bd}\,(\check{\alpha}\xi G\bar{\xi})^{db}+2\diff_{y^c}(\xi G\bar{\xi}G\check{\alpha})\indices{^c_a}=\\
=-2\diff_{y^a}u_{bd}\,(\check{\alpha}\xi G\bar{\xi})^{db}+2\diff_{y^c}(\check{\alpha}\xi G\bar{\xi}G)\indices{^c_a}=\\
=-2\diff_{y^a}u_{bd}\,(\check{\alpha}\xi G\bar{\xi})^{db}
+2\,u_{ea}\,\diff_{y^c}(\check{\alpha}\xi G\bar{\xi})^{ce}
+2\,\diff_{y^c}u_{ea}\,(\check{\alpha}\xi G\bar{\xi})^{ce}=\\
=2\,u_{ea}\,\diff_{y^c}(\check{\alpha}\xi G\bar{\xi})^{ce}
\end{gathered}
\end{equation*}
and finally we conclude the computation of the real moment map using Lemma~\ref{lemma:divergenza_coord_sympl}:
\begin{equation*}
\mrm{div}(X(\alpha))=\mrm{Re}\left(\diff_{y^a}\diff_{y^b}\left(\check{\alpha}\,\xi\,G\,\bar{\xi}\right)^{ab}\right)=\diff_{y^a}\diff_{y^b}\left(\check{\alpha}\,\xi\,G\,\bar{\xi}\right)^{ab}
\end{equation*}
The last equality comes from the observation that~$\left(\check{\alpha}\,\xi\,G\,\bar{\xi}\right)^{ab}_{,ab}$ is real, as was shown in the proof of Proposition~\ref{HKEulLagr}.

\subsection{Deformations of complex structures in symplectic coordinates}\label{sec:symp_boundary_abelian}

We want to characterize those Higgs fields~$\alpha\in T_J^*\!\mscr{J}$ that are integrable in terms of the matrix~$\xi$ that we used to write the HcscK system in symplectic coordinates. To do so, it will be convenient to let~$\dot{J}=\mrm{Re}(\alpha)^\transpose$ be the first-order deformation of the complex structure defined by~$\alpha$.

In the system of symplectic coordinates~$(\bm{y},\bm{w})$ that we are using on~$M$, the symplectic form~$\omega$ is just the standard symplectic form~$\sum_i\mrm{d}y^i\wedge\mrm{d}w^i$, while the complex structure~$J$ is described by the matrix
\begin{equation*}
J=\begin{pmatrix}
0 & -G\\ G & 0
\end{pmatrix}
\end{equation*}
with~$G=\mrm{Hess}_{\bm{y}}u$; the metric~$g_J$ defined by~$\omega$ and~$J$ is 
\begin{equation*}
g_J=\omega(-,J-)=\begin{pmatrix} G & 0\\ 0& G^{-1}\end{pmatrix}.
\end{equation*}
A system of holomorphic coordinates for~$J$ is described by the vector fields
\begin{equation*}
\diff_{z^a}=\frac{1}{2}\left(G^{ab}\diff_{y^b}-\I\diff_{w^a}\right);
\end{equation*}
indeed, using the fact that~$G$ is a Hessian matrix it is easy to check that~$[\diff_{z^a},\diff_{z^b}]=0$.

The first-order deformation~$\dot{J}$ is also described, in our fixed coordinate system~$(\bm{y},\bm{w})$, by a real matrix-valued function. It must satisfy the relations~$\dot{J}J+J\dot{J}=0$ and~$\omega(J-,\dot{J}-)+\omega(\dot{J}-,J-)=0$. It is not too difficult to see that any such~$\dot{J}$ is of the form
\begin{equation*}
\dot{J}=\begin{pmatrix} G^{-1}A & G^{-1}B G^{-1}\\ B &- A G^{-1}\end{pmatrix}
\end{equation*}
for some real symmetric matrices~$A$ and~$B$. We are interested in the~$(0,1)$-part of~$\dot{J}$, i.e. the restriction of~$\dot{J}$ to~$T^{1,0}_JM$, that can be written as~$\dot{J}\indices{^{\bar{b}}_a}\mrm{d}z^a\otimes\diff_{\bar{z}^b}$, since~$\alpha^\transpose=\dot{J}^{0,1}$.
\begin{lemma}
With the previous notation, the matrix~$\dot{J}\indices{^{\bar{b}}_a}$ is~$G^{ac}\left(A-\I B\right)_{cb}$.
\end{lemma}
\begin{proof}
Using the expression of~$\dot{J}$ in the coordinates~$(\bm{y},\bm{w})$ we see that
\begin{equation*}
\begin{split}
\dot{J}\indices{^{\bar{b}}_a}=&\mrm{d}\bar{z}^b\left(\dot{J}(\diff_{z^a})\right)=\frac{1}{2}\left(G_{bc}\mrm{d}y^c-\I\mrm{d}w^b\right)\left(G^{ad}\dot{J}(\diff_{y^d})-\I\dot{J}(\diff_{w^a})\right)=\\
=&\frac{1}{2}\left(G_{bc}G^{ad}G^{ef}A_{fd}\delta^c_e+\delta^{ef}A_{fp}G^{pq}\delta_{qa}\delta^b_e
\right)-\\
&-\frac{\I}{2}\left(G_{bc}G^{ef}B_{fp}G^{pq}\delta_{aq}\delta^c_e+G^{ad}\delta^{ef}B_{fd}\delta^b_e\right)=\\
=&G^{ad}A_{db}-\I G^{ad}B_{db}.\hspace{2.5cm}\qedhere
\end{split}
\end{equation*}
\end{proof}
\begin{lemma}\label{lemma:xi_caratt_alt}
The quadratic differential~$q:=\omega(-,J\dot{J}^{0,1}-)$ is written in the coordinates~$\bm{z}$ as
\begin{equation*}
q=\sum_{a,b}G^{ac}(A-\I B)_{cd}G^{db}\mrm{d}z^a\otimes\mrm{d}z^b.
\end{equation*}
In particular,~$\xi=G^{-1}(A-\I B)G^{-1}$.
\end{lemma}
\begin{proof}
From the definition, if we let~$g=\omega(-,J-)$ be the metric defined by~$J$, we have~$q=g_{a\bar{b}}\dot{J}\indices{^{\bar{b}}_c}\mrm{d}z^a\otimes\mrm{d}z^c$. Then it is enough to use the matrix expression of~$g$ to find
\begin{equation*}
q=g_{a\bar{b}}\dot{J}\indices{^{\bar{b}}_c}\mrm{d}z^a\otimes\mrm{d}z^c=\sum_{a,c}
G^{ab}G^{cd}\left(A-\I B\right)_{db}\mrm{d}z^a\otimes\mrm{d}z^c.\qedhere
\end{equation*}
\end{proof}
The integrability condition for~$\alpha$ from Section~\ref{sec:J} is
\begin{equation}\label{eq:integrability}
\diff\alpha=0.
\end{equation}
The next result expresses~\eqref{eq:integrability} in terms of our symplectic coordinates.
\begin{prop}\label{prop:integrability_xi}
Let~$J$ be the complex structure associated to a symplectic potential~$u$ in a system of symplectic coordinates, and let~$\xi$ be a Higgs tensor for~$J$. Then~$\xi$ comes from an integrable deformation of the complex structure~$J$ if and only if~$D^2u\,\xi\,D^2u$ is the Hessian matrix of a function.
\end{prop}
\begin{proof}
The integrability condition~\eqref{eq:integrability} is written more explicitly as
\begin{equation*}
\diff_{z^a}\alpha\indices{_c^{\bar{b}}}=\diff_{z^c}\alpha\indices{_a^{\bar{b}}}
\end{equation*}
and in our case it becomes
\begin{equation*}
\diff_{x^a}\left(v^{bd}q_{cd}\right)=\diff_{x^c}\left(v^{bd}q_{ad}\right)
\end{equation*}
i.e.~$\diff_{x^a}\left(v^{bd}q_{cd}\right)$ should be symmetric in the indices~$a$,~$c$. Under Legendre duality this quantity becomes, using the fact that~$u_{ij}$ is a Hessian matrix,
\begin{equation*}
\diff_{x^a}\left(v^{bd}q_{cd}\right)=u^{ae}\diff_{y^e}\left(u_{bd}\xi^{cd}\right)=u^{ae}\diff_{y^e}\left(u_{bd}\xi^{fd}u_{fg}\right)u^{gc}+u_{bd}\xi^{df}\diff_{y^f}u^{ac}.
\end{equation*}
The second expression on the right hand side is symmetric in~$a$ and~$c$, while the first
\begin{equation*}
u^{ae}\diff_{y^e}\left(u_{bd}\xi^{fd}u_{fg}\right)u^{gc}
\end{equation*}
is symmetric in~$a$ and~$c$ if and only if~$\diff_{y^e}\left(u_{bd}\xi^{fd}u_{fg}\right)$ is symmetric in~$e$ and~$g$. Thus letting~$H_{ab}(y):=u_{ac}\xi^{cd}u_{db}$ equation~\eqref{eq:integrability} is equivalent to~$\diff_c H_{ab}$ being symmetric in all indices.
\end{proof}
From Lemma~\ref{lemma:xi_caratt_alt} we know that~$\xi$ can be written as~$G^{-1}HG^{-1}$ for some symmetric matrix-valued function~$H$; Proposition~\ref{prop:integrability_xi} then tells us that~$\xi$ corresponds to an integrable first-order deformation of~$J$ if and only if~$H=\mrm{Hess}_yh$ for some complex-valued function~$h$.

\subsection{Convexity of the periodic HK-energy}\label{sec:convexity}

In this Section we study the Biquard-Gauduchon functional~$\m{H}(u, \xi)$ defined in~\eqref{eq:BGFunc} on an abelian surface, and the corresponding version of the K-energy~$\widehat{\m{F}}(u, \xi)$ defined in~\eqref{eq:HK_energy}, proving a convexity result.
\begin{thm}\label{thm:ConvexThm}
Fix a Higgs tensor~$\xi$ on a~$2$-dimensional torus~$M$, and suppose that either~$\mrm{Re}(\xi)$,~$\mrm{Im}(\xi)$ are (positive or negative) semidefinite, or~$\det(\xi) = 0$. Let~$u_t = (1-t)u_0 + t u_1, \, t\in [0,1]$ be a linear path between symplectic potentials~$u_0$,~$u_1$, such that~$\norm{\xi}^2_{u_i} < 1$,~$i =0, 1$. Then the HK-energy~$\widehat{\m{F}}(u_t, \xi)$ is a well defined and strictly convex function of~$t \in [0,1]$.
\end{thm}
As a consequence, under these assumptions on~$\xi$, there is at most one symplectic potential~$u$, up to an additive constant, solving the real moment map equation in~\eqref{eq:HcscKSymplRealEq} and such that~$\norm{\xi}^2_u < 1$. This condition on the norm of~$\xi$ is equivalent to
\begin{equation*}
\delta^+(\xi\,G\,\bar{\xi}\,G)+\delta^-(\xi\,G\,\bar{\xi}\,G)<1
\end{equation*}
so the symplectic potentials~$u$ such that~$\norm{\xi}^2_u<1$ describe a proper subset of~$\mscr{A}(\xi)$. We claim that it is in fact a convex subset of~$\mscr{A}(\xi)$, if~$\xi$ satisfies the hypothesis of Theorem~\ref{thm:ConvexThm}, namely
\begin{equation*}
\mrm{Re}(\xi),\,\mrm{Im}(\xi) \textrm{ definite, or} \det(\xi) = 0.
\end{equation*}
This follows from the fact that with our assumption~$\norm{\xi}^2_{u}$ is a convex function of~$D^2u$. Indeed we have, for a linear path~$u_t=(1-t)u_0+t\,u_1\in\mscr{A}(\xi)$  
\begin{align*}
&\frac{d^2}{dt^2}\norm{\xi}^2_{u_t} = 2\mrm{Tr}(\xi\,D^2\dot{u}_t\,\bar{\xi}\,D^2\dot{u}_t).  
\end{align*}
If~$\mrm{Re}(\xi)$ is definite, then~$\mrm{Re}(\xi)\,D^2\dot{u}_t$ is similar to a symmetric matrix and so it has real spectrum, thus
\begin{equation*}
\mrm{Tr}(\mrm{Re}(\xi)\,D^2\dot{u}_t\,\mrm{Re}(\xi)\,D^2\dot{u}_t) = \mrm{Tr}(\mrm{Re}(\xi)\,D^2\dot{u}_t)^2 \geq 0.
\end{equation*}
The same happens for~$\mrm{Im}(\xi)$, and both conditions together imply
\begin{equation*}
\mrm{Tr}(\xi\,D^2\dot{u}_t\,\bar{\xi}\,D^2\dot{u}_t)\geq 0.
\end{equation*}
The same conclusion holds when~$\det(\xi) = 0$: it is enough to consider the case when~$D^2\dot{u}_t$ is diagonal with eigenvalues~$\mu_1$,~$\mu_2$, and using the condition~$\mrm{det}(\xi) = 0$ we compute in that case  
\begin{equation*}
\mrm{Tr}(\xi\,D^2\dot{u}_t\,\bar{\xi}\,D^2\dot{u}_t) = (\abs{\xi^{11}}\mu_1 +\abs{\xi^{22}}\mu_2)^2 \geq 0.
\end{equation*}
\begin{rmk}
It is important to point our that, even under the assumptions of Theorem~\ref{thm:ConvexThm}, the Biquard-Gauduchon functional~$\m{H}(u, \xi)$ is not, in general, convex with respect to~$u$. However we show that the periodic K-energy~$\m{F}(u)$ compensates this lack of convexity.
\end{rmk}
\begin{proof}[Proof of Theorem~\ref{thm:ConvexThm}]
Let~$u_0,u_1\in\mscr{A}(\xi)$ be symplectic potentials such that~$\norm{\xi}^2_{u_i}<1$, and consider the linear path~$u_t=(1-t)u_0+t\,u_1$. The convexity in~$t$ of~$\norm{\xi}^2_{u_t}$ implies in particular that the HK-energy is well-defined along~$u_t$.

The differential of~$\hat{\m{F}}$ along the path is
\begin{equation*}
\diff_t\hat{\m{F}}(u_t,\xi)=-\frac{1}{2}\int u_t^{ab} \dot{u}_{,ab}\,\mrm{d}\mu+\frac{1}{2}\int\diff_t\rho(u_t,\xi)\,\mrm{d}\mu 
\end{equation*}
to compute the differential of the Biquard-Gauduchon function we will use the expression~\eqref{eq:autovalori_dim2} for the eigenvalues of~$\xi\,G_t\,\bar{\xi}\,G_t$ on a complex surface. Consider the functions of two real variables
\begin{equation*}
\begin{split}
f^\pm(x,y):=&x\pm\sqrt{x^2-4 y}\\
R(x,y):=&2-\sqrt{1-f^+(x,y)}-\sqrt{1-f^-(x,y)}+\\
&+\mrm{log}\frac{1+\sqrt{1-f^+(x,y)}}{2}+\mrm{log}\frac{1+\sqrt{1-f^-(x,y)}}{2}
\end{split}
\end{equation*}
then the Biquard-Gauduchon function on the torus~$\rho(u_t,\xi)$ is in fact 
\begin{equation*}
\rho(u_t,\xi)=R\left(\norm{\xi}^2_{u_t},\,\abs{\mrm{det}(G_t\xi)}^2\right)
\end{equation*}
where as usual~$G_t:=\mrm{Hess}(u_t)$. Setting also~$\dot{G} =\mrm{Hess}(\dot{u}_t)$ for convenience, we can compute the second variation of the HK-energy in terms of the derivatives of~$\norm{\xi}^2_{u_t}$ and~$\abs{\mrm{det}(G_t\xi)}^2$. These can be expressed as
\begin{equation*}
\frac{d}{dt}\Bigr|_{t=0}\norm{\xi}^2_{u(t)}=2\,\mrm{Tr}\left(\xi\,G_t\,\bar{\xi}\,\dot{G}_t\right)
\end{equation*}
while for the determinant we find
\begin{equation*}
\begin{split}
\frac{d}{dt}\Bigr|_{t=0}\abs{\mrm{det}(G_t\xi)}^2=2\,\abs{\mrm{det}(G_t\xi)}^2u_t^{ab}\varphi_{ab}.
\end{split}
\end{equation*}
We can compute the second variation of the HK-energy. In the following computation the derivatives of~$R$ are all evaluated on~$(\norm{\xi}^2_{u_t},\abs{\mrm{det}(G\xi)}^2)$.
\begin{equation}\label{secondVar}
\begin{split}
&\frac{d^2}{dt^2}\widehat{\m{F}}(u_t,\xi)=\frac{1}{2}\int\!\!\mrm{Tr}\left((G^{-1}\dot{G})^2\right)d\mu+\!\int\!\!D_1R\,\mrm{Tr}(\xi\,\dot{G}\,\bar{\xi}\,\dot{G})d\mu+\\
&+2\!\int\!\! D_2R\,\abs{\mrm{det}(G\,\xi)}^2\mrm{Tr}\left(G^{-1}\dot{G}\right)^2\mrm{d}\mu
-\!\int\!\! D_2R\,\abs{\mrm{det}(G\,\xi)}^2\mrm{Tr}\left((G^{-1}\dot{G})^2\right)\mrm{d}\mu+\\
&+2\int\!\!\begin{pmatrix}\mrm{Tr}(\xi\,G\,\bar{\xi}\,\dot{G})\\\abs{\mrm{det}(G\,\xi)}^2\mrm{Tr}\left(G^{-1}\dot{G}\right)\end{pmatrix}^\transpose\!\!\cdot\mrm{Hess}_{x,y}R\cdot\begin{pmatrix}\mrm{Tr}(\xi\,G\,\bar{\xi}\,\dot{G})\\\abs{\mrm{det}(G\,\xi)}^2\mrm{Tr}\left(G^{-1}\dot{G}\right)\end{pmatrix}\mrm{d}\mu.
\end{split}
\end{equation}
Elementary inequalities show that~$R$ is a convex function on its domain; moreover, the function~$D_1 R(p, q)$ is positive, while~$D_2 R(p, q)$ is negative. It follows that, in the expression above for the second variation, under our assumptions on~$\xi$, all terms are positive except for the negative term
\begin{equation*}
2\int D_2R\,\abs{\det\left(G\,\xi\right)}^2\mrm{Tr}\left(G^{-1}\dot{G}\right)^2 d\mu.
\end{equation*}
We claim that we have in fact
\begin{align*}
\frac{1}{2}\int\mrm{Tr}\left((G^{-1}\dot{G})^2\right)\mrm{d}\mu+2\int D_2R\,\abs{\det\left(G\,\xi\right)}^2\mrm{Tr}\left(G^{-1}\dot{G}\right)^2\mrm{d}\mu-&\\
-\int D_2R\,\abs{\det\left(G\,\xi\right)}^2\mrm{Tr}\left((G^{-1}\dot{G})\right)\mrm{d}\mu&\geq 0. 
\end{align*}
Indeed, if we denote by~$\mu_i$,~$i=1,2$ the eigenvalues of the endomorphism~$G^{-1}\dot{G}$ at a point, then the sum of the integrands in the expression above may be expressed as
\begin{equation}\label{eq:convexity_Kenergyterm}
\left(\frac{1}{2}-D_2R\,\delta^+\delta^-\right)(\mu^2_1 + \mu^2_2) + 2D_2R\,\delta^+\delta^-(\mu_1+\mu_2)^2
\end{equation}
since~$\abs{\mathrm{det}(G\,\xi)}^2=\delta^+\delta^-$.  We will show that, under the condition~$\delta^++\delta^-=\norm{\xi}^2_{u}<1$, we have~$-D_2R\,\delta^+\delta^-<\frac{1}{6}$. Assuming this for the moment, we have
\begin{equation*}
\begin{split}
&\left(\frac{1}{2}-D_2R\,\delta^+\delta^-\right)(\mu^2_1 + \mu^2_2) + 2D_2R\,\delta^+\delta^-(\mu_1+\mu_2)^2\geq\\
&\quad\geq \left(\frac{1}{2}+D_2R\,\delta^+\delta^-\right)(\mu_1^2 + \mu_2^2) + 4D_2R\,\delta^+\delta^-\abs{\mu_1}\,\abs{\mu_2}\geq \\
&\quad\geq \left(\frac{1}{2}+D_2R\,\delta^+\delta^-\right)(\mu_1^2 + \mu_2^2) -2\left(\frac{1}{2}+D_2R\,\delta^+\delta^-\right)\abs{\mu_1}\,\abs{\mu_2}
\end{split}
\end{equation*}
so the expression in~\eqref{eq:convexity_Kenergyterm} is strictly positive unless~$\mu_i = 0$,~$i = 1, 2$, that is,~$\dot{G}\equiv 0$. To check that~$-D_2R\,\delta^+\delta^-<\frac{1}{6}$ holds, we consider the following inequality, which holds for~$\delta^++\delta^-<1$:
\begin{align*}
-D_2R\,\delta^+\delta^-=&\frac{\delta^+\delta^-}{2\left(\sqrt{1-\delta^+}+\sqrt{1-\delta^-}\right)\left(1+\sqrt{1-\delta^+}\right)\left(1+\sqrt{1-\delta^-}\right)}=\\
=&\frac{\left(1-\sqrt{1-\delta^+}\right)\left(1-\sqrt{1-\delta^+}\right)}{2\left(\sqrt{1-\delta^+}+\sqrt{1-\delta^-}\right)}\leq\frac{\left(1-\sqrt{\delta^+}\right)\left(1-\sqrt{1-\delta^+}\right)}{2\left(\sqrt{\delta^+}+\sqrt{1-\delta^+}\right)}.
\end{align*}
This last term achieves its maximum for~$\delta^+=\frac{1}{2}$, and one can check numerically that this is strictly less than~$\frac{1}{6}$.
\end{proof}
Our discussion so far shows that~$\widehat{\m{F}}(u_t, \xi)$ is convex along the linear path. Moreover,~$\widehat{\m{F}}(u_t, \xi)$ is strictly convex unless~$D^2 \dot{u}_t \equiv 0$, i.e. unless~$u_1 - u_0$ is constant. By Proposition~\ref{HKEulLagr}, solutions to the real moment map equation in~\eqref{eq:HcscKSymplRealEq}, for fixed~$\xi$, are precisely critical points of~$\widehat{\m{F}}(u, \xi)$, which only depends on~$D^2 u$. Therefore a symplectic potential in the set
\begin{equation*}
\mscr{A}'(\xi):=\set*{u\in\mscr{A}(\xi)\tc\norm{\xi}^2_u<1}
\end{equation*}
solving~\eqref{eq:HcscKSymplRealEq} is unique up to additive constants. Notice that, in general,~$\mscr{A}'(\xi)$ is strictly included in~$\mscr{A}(\xi)$. A possible case in which the two sets coincide is when the deformation of the complex structure has non-maximal rank.
\begin{rmk}
The same result also holds on curves but it is not particularly interesting in that case, since we already obtained a stronger result for genus~$1$ Riemann surfaces in Section~\ref{sec:curve}. We also conjecture that a convexity result analogous to Theorem~\ref{thm:ConvexThm} should hold in higher dimension, using the convexity of the Biquard-Gauduchon function.
\end{rmk}

\subsection{Translation-invariant periodic solutions}\label{sec:1dim_solutions}

In this Section we describe some solutions to the HcscK system in dimension~$2$ in the very simple situation where both the Higgs tensor and the matrix~$G=\mrm{Hess}(u)$ depend on a single variable, say~$y^1$, and the Higgs field has non-maximal rank.

The complex moment map equation in~\eqref{eq:HcscKSymplRealEq} reduces to~$\xi^{11}_{11} = 0$, i.e. to the condition~$\xi^{11} = c \in \bb{C}$, by periodicity. With our further assumption~$\det(\xi) = 0$, all possible solutions are given by
\begin{equation}\label{eq:sol_lowrank_complex}
\begin{pmatrix}
0 & 0\\ 0& \xi^{22}(y^1)
\end{pmatrix}
\ \mbox{ or }\ 
\begin{pmatrix} c & \xi^{12}(y^1)\\ \xi^{12}(y^1) & \frac{\left(\xi^{12}(y^1)\right)^2}{c}\end{pmatrix}.
\end{equation}
We look for solutions~$u(y)$ of the real moment map equation of the form
\begin{equation*}
u(\bm{y})=\frac{1}{2}\abs{\bm{y}}^2+f(y^1)
\end{equation*}
for a periodic function~$f(y^1)$, so 
\begin{equation*}
D^2(u)=\begin{pmatrix}1+f''(y^1) &0\\ 0& 1\end{pmatrix}.
\end{equation*}

\begin{rmk}
Once we show that there are such solutions, Theorem~\ref{thm:ConvexThm} will guarantee that in fact all possible solutions~$u(\bm{y})$ are of this form. Notice also that under the assumption~$\mathrm{det}(\xi)=0$, the conditions~$\SpecRad(u,\xi)<1$ and~$\norm{\xi}^2<1$ are equivalent. 
\end{rmk}

From the definition~\eqref{eq:alpha_symp_check} of~$\check{\alpha}$, in the low-rank case the real moment map equation becomes
\begin{equation*}
\left(-(1+f'')+\frac{(1+f'')\abs{\xi^{11}}^2+\abs{\xi^{12}}^2}{1+(1-\norm{\xi}^2_{u})^{1/2}}\right)_{,11}=0.
\end{equation*}
Notice however that if~$\xi$ is of the first type in~\eqref{eq:sol_lowrank_complex} then the equation reduces to~$f^{(4)}=0$, i.e. the metric must have constant coefficients. So it is enough to discuss the case
\begin{equation*}
\xi=\begin{pmatrix} c & F(y^1)\\ F(y^1) & \frac{F(y^1)^2}{c}\end{pmatrix}
\end{equation*}
for some periodic function~$F$. Then the equation becomes
\begin{equation}\label{eq:real_mm_twoderivs}
\left(-(1+f'')+\frac{(1+f'')\abs{c}^2+\abs{F}^2}{1+\sqrt{1-\frac{1}{\abs{c}^2}\left((1+f'')\abs{c}^2+\abs{F}^2\right)^2}}\right)''=0
\end{equation}
which is equivalent to the condition 
\begin{equation*}
\frac{(1+f'')\abs{c}^2+\abs{F}^2}{1+\sqrt{1-\frac{1}{\abs{c}^2}\left((1+f'')\abs{c}^2+\abs{F}^2\right)^2}}=1+k+f'' 
\end{equation*}
for some real constant~$k$, or equivalently,~$f''$ must satisfy the algebraic equation 
\begin{equation}\label{eq:real_mm_abelian_alt}
\frac{\abs{F}^2}{\abs{c}^2}+\left(f''+1\right)=\frac{2\left(f''+k+1\right)}{\abs{c}^2+\left(f''+k+1\right)^2}.
\end{equation}
When~$F \equiv 0$ a solution~$f''$ is a suitable constant, corresponding to a constant almost-complex structure and a flat metric. We will prove our existence result by perturbing around~$F \equiv 0$ in a quantitative way. It is convenient to introduce the operators 
\begin{equation*}
\begin{split}
P:C^{\infty}(C^1,\bb{R})&\to C^{\infty}_0(S^1,\bb{R})\\
\varphi&\mapsto\varphi'';
\end{split}
\end{equation*}
\begin{equation*}
\begin{split}
Q: C^{\infty}(S^1,\bb{R})\times C^{\infty}&(S^1,\bb{R})\to C^{\infty}(S^1,\bb{R})\\
(\varphi_1,\varphi_2)&\mapsto \frac{(1+\varphi_1)\abs{c}^2+\varphi_2^2}{1+\sqrt{1-\frac{1}{\abs{c}^2}\left((1+\varphi_1)\abs{c}^2+\varphi_2^2\right)^2}}-(1+\varphi_1);
\end{split}
\end{equation*}
\begin{equation*}
\begin{split}
R: C^\infty(S^1,\bb{R})\times C^\infty(S^1,\bb{C})&\to C^\infty_0(S^1,\bb{R})\\
(f,F)&\mapsto P\left(Q(P(f),\abs{F})\right).
\end{split}
\end{equation*}
We will show that if~$\abs{F}\leq\abs{c}<\frac{3}{10}$, then there exists~$f\in C^\infty(S^1,\bb{R})$, unique up to an additive constant, such that
\begin{equation*}
\begin{gathered}
1+f''>0;\\
(1+f'')\abs{c}^2+\abs{F}^2<\abs{c};\\
R(f,F)=0.
\end{gathered}
\end{equation*}
Clearly~$P$ is surjective, and its kernel are the constant functions, so it is enough to show that with our assumptions there is a unique~$\phi\in C^\infty_0(S^1,\bb{R})$ satisfying the positivity condition
\begin{equation}\label{positivityIneq}
\phi+1>0,
\end{equation}  
as well as the uniform nonsingularity condition
\begin{equation}\label{nonsingIneq}
\left((1+\phi)\abs{c}^2+\abs{F}^2\right)^2\leq \abs{c}^2-\varepsilon\abs{c}^2
\end{equation}
for some \emph{fixed}~$0 < \varepsilon < 1$, and such that
\begin{equation*}
P(Q(\phi,F))=0.
\end{equation*}
For~$F=0$ we have a unique solution,~$\phi=0$. We consider the continuity path
\begin{equation*}
P(Q(\phi,t\,F))=0\mbox{ for }0\leq t\leq 1,
\end{equation*}
and prove openness and closedness, as usual. 

Consider the differential of the operator
\begin{equation*}
\begin{split}
C^{2,\alpha}_0(S^1)&\to C^{0,\alpha}_0(S^1)\\
\phi&\mapsto P(Q(\phi,F)),
\end{split}
\end{equation*}
namely
\begin{equation*}
\dot{\phi}\mapsto\left[\left(\frac{\abs{c}^2}{\sqrt{1-\frac{1}{\abs{c}^2}p^2}\left(1+\sqrt{1-\frac{1}{\abs{c}^2}p^2}\right)}-1\right)\dot{\phi}\right]'',
\end{equation*}
where we set~$p = (1+\phi)\abs{c}^2+\abs{F}^2$ for convenience. This is clearly an isomorphism provided its coefficient is bounded away from~$0$. Since we have~$p^2\leq \abs{c}^2-\varepsilon\abs{c}^2$ by~\eqref{nonsingIneq}, a short computation shows that it is enough to assume
\begin{equation*}
\abs{c}^2<\sqrt{\varepsilon}-\varepsilon.
\end{equation*}
What remains to be seen is that the (automatically open) positivity condition~\eqref{positivityIneq} is also closed along the path, while conversely the (automatically closed) uniform nonsingularity condition~\eqref{nonsingIneq} is also open. Then further regularity would follow by a standard bootstrapping argument. Let us prove both claims at once. Our equation holds if and only if there is a constant~$k$ such that
\begin{equation*}
\frac{(1+\phi)\abs{c}^2+\abs{F}^2}{1+\sqrt{1-\frac{1}{\abs{c}^2}\left((1+\phi)\abs{c}^2+\abs{F}^2\right)^2}}-(1+\phi)=k.
\end{equation*}
Note that this implies~$1+\phi+k\geq 0$. Moreover, since~$\phi$ has zero average, we have
\begin{equation*}
k+1=\int_0^1\frac{p}{1+\sqrt{1-\frac{1}{\abs{c}^2}p^2}} dy^1,\, p = (1+\phi)\abs{c}^2+\abs{F}^2 
\end{equation*}
so we also have~$k+1\geq 0$ and~$k+1\leq p$, hence~$k+1\leq \abs{c}$ along the continuity path. 
Moreover we have~$1+\phi \geq -k \geq 1-\abs{c}$, so if~$\abs{c} < 1$ then~$1+\phi$ is bounded uniformly away from~$0$. This shows that the positivity condition~\eqref{positivityIneq} is closed along the path provided~$\abs{c} <1$. Now as in~\eqref{eq:real_mm_abelian_alt} we may write our equation as
\begin{equation*}
\abs*{\frac{F}{c}}^2+\phi+1=\frac{2(\phi+1+k)}{\abs{c}^2+(\phi+1+k)^2}.
\end{equation*}
Then~$\phi+1+k$ is a positive solution to the equation
\begin{equation*}
\abs{F}^2-k\abs{c}^2+\abs{c}^2x+\left(\abs*{\frac{F}{c}}^2-k\right)x^2+x^3=2x.
\end{equation*}
Note that all the coefficients on the left hand side are positive. Using the fact that~$-k \geq 1-\abs{c}$, we see that a positive solution to the equation cannot be larger than~$1+\abs{c}$. Hence we find
\begin{equation*}
\phi+1\leq1+\abs{c}-k\leq 2+\abs{c}
\end{equation*}
from which we get, assuming~$\abs{F}^2\leq \abs{c}^2$,
\begin{equation*}
(1+\phi)\abs{c}^2+\abs{F}^2<\abs{c}^2\left(3+ \abs{c}\right).
\end{equation*}
It follows that~\eqref{nonsingIneq} certainly holds for~$0 < \varepsilon < 1$ as long as
\begin{equation*}
\abs{c}^2\left(3+\abs{c}\right)^2\leq 1-\varepsilon
\end{equation*}
hence~$\abs{c}$ must be less or equal than the first positive root of the polynomial
\begin{equation}\label{eq:polinomio_limite}
x^4+6 x^3+9 x^2+\varepsilon -1.
\end{equation}
Under these conditions on~$\varepsilon$ and~$\abs{c}$ the uniform nonsingularity condition~\eqref{nonsingIneq} would hold automatically along the path. Finally we need to choose the values of~$\abs{c}$ and~$\varepsilon$ in our argument so that all constraints are satisfied. Recall these are
\begin{equation*}
\abs{c} < \sqrt{\sqrt{\varepsilon}-\varepsilon},\, \abs{c} < 1
\end{equation*}
as well as the fact that~$\abs{c}$ is less or equal than the first positive root of~\eqref{eq:polinomio_limite}. Direct computation shows that~$\varepsilon= 1/100$,~$\abs{c}<\frac{3}{10}$ is an admissible choice. 

We conclude by examining the integrability condition, as characterised in Proposition~\ref{prop:integrability_xi}, for the solutions to the HcscK system we have constructed. Since both~$\xi$ and~$D^2u$ depend just on the variable~$y^1$, the integrability condition on~$\xi$ implies  
\begin{equation*}
\diff_1\left(G\xi G\right)_{22}=\diff_2\left(G\xi G\right)_{12}=0.
\end{equation*}
Hence, if~$\xi$ is of the first type in~\eqref{eq:sol_lowrank_complex}, we find~$F'=0$, while for the second type we must have~$F\,F'=0$. In both cases~$\xi$ must be constant.

\section{The HcscK system on toric surfaces}\label{sec:toric}

Consider a toric manifold~$(M,J,\omega)$ with a Hamiltonian action~$\bb{T}^n\curvearrowright M$, whose moment map~$\mu$ sends~$M$ to a convex polytope~$P\subseteq\bb{R}^n$ by the Atiyah-Guillemin-Sternberg Theorem~\cite{Atiyah_convexity,Guillemin_convexity}. It is well-know that~$P$ is a \emph{Delzant polytope}~\cite{Delzant_polytope}, and that any such polytope defines in turn a standard compact symplectic manifold~$(M_P,\omega_P)$ together with a Hamiltonian~$\bb{T}^n$-action on~$M_P$ such that~$(M_P,\omega_P)$ is equivariantly isomorphic to~$(M,\omega)$. The polytope defines also a standard compatible complex structure~$J_P$, but in general~$(M_P,J_P,\omega_P)$ will \emph{not} be isomorphic to~$(M,J,\omega)$. For the general theory we refer to~\cite{Guillemin_toric,Guillemin_momentmaps} and~\cite{Abreu_toric}.

The moment map gives an alternative way to describe the symplectic structure on~$M$, since it establishes an isomorphism between~$(M^\circ,\omega)$ and~$(P^\circ\times\bb{T}^n,\omega_P)$, where~$M^\circ$ is the open subset of~$M$ where the action is free and~$\omega_P$ is the standard symplectic structure induced by the inclusion in~$\bb{R}^{2n}$. The manifold~$M^\circ$ is~$\bb{T}$-equivariantly biholomorphic to~$\bb{C}^n/2\pi\I\bb{Z}^n\cong\bb{R}^n+\I\bb{T}^n$, where the action is by translations on the~$\bb{T}^n$-factor. We consider the standard coordinates~$\bm{z}=\bm{x}+\I\bm{w}$ on~$\bb{R}^n\times\bb{T}^n$. On~$P^\circ\times\bb{T}^n$ instead we consider coordinates~$(\bm{y},\bm{w})$, with~$\bm{y}=\mu(\bm{x})$.

The symplectic form on~$M^\circ$ is given by~$4\I\diff\bdiff v$ for some~$\bb{T}$-invariant potential function~$v$, so that~$\omega=\I\diff_{x^a}\diff_{x^b}v\,\mrm{d}z^a\wedge\mrm{d}\bar{z}^b$, while~$J$ is just represented by the canonical matrix~$J(\bm{x},\bm{w})=\begin{pmatrix}0 & -\mathbbm{1} \\ \mathbbm{1} & 0\end{pmatrix}$. On the other hand the symplectic structure on~$P$ induced by~$\omega$ via the moment map is the canonical symplectic form~$\sum_{i,j}\mrm{d}y^i\wedge\mrm{d}w^j$ on~$\bb{R}^{2n}$, while the complex structure~$J$ is described by a matrix~$J(\bm{y},\bm{w})=\begin{pmatrix} 0 & -G^{-1} \\ G & 0\end{pmatrix}$. Since~$J$ is integrable, this matrix must be of the form~$G=\mrm{Hess}_{\bm{y}}u$ for some potential~$u(\bm{y})$: moreover, the two coordinate systems and the two functions~$u$ and~$v$ are Legendre dual to each other:
\begin{equation*}
\begin{cases}
\bm{y}=\diff_{\bm{x}}v;\\
\bm{x}=\diff_{\bm{y}}u;\\
u(\bm{y})+v(\bm{x})=\bm{x}\cdot\bm{y}.
\end{cases}
\end{equation*}
This means that the HcscK system~\eqref{eq:HcscK_system} can be expressed in the coordinates~$(\bm{y},\bm{w})$ in a form similar to what we did for abelian varieties in~\eqref{eq:HcscKSymplRealEq}. The only difference is that~$\hat{S}$ does not vanish, in general, for a toric manifold, so the system of equations becomes
\begin{equation}\label{eq:symplectic_HcscK_toric}
\begin{cases}
\left(\xi^{ab}\right)_{,ab}=0\\
\left(\left(1-\xi\,D^2u\,\bar{\xi}\,D^2u\right)^{\frac{1}{2}}D^2u^{-1}\right)^{ab}_{,ab}=-C.
\end{cases}
\end{equation}
Here~$C$ is the topological constant~$4\hat{S}$. The system should be solved for a potential~$u$ and a deformation~$\xi$ of the complex structure. However an important difference between the system on abelian and toric varieties is the boundary conditions that we must impose on~$u$ and~$\xi$. For an abelian variety the conditions are just periodicity of the matrix-valued functions~$D^2u$ and~$\xi$, but for the toric case the situation is slightly more complicated: while the boundary conditions for~$u$ are well-understood from the work of Abreu, it is not yet clear what the boundary behaviour of~$\xi$ should be. We will see that, as a consequence of the compatibility between~$\xi$ and the complex structure, also the boundary conditions for~$\xi$ can be expressed in terms of \emph{Guillemin's boundary conditions}. We recall here in some details some features of these boundary conditions, since they play a major role in what follows.

As was mentioned above, there is a standard complex structure~$J_P$ on~$P^\circ\times\bb{T}^n$, whose boundary behaviour allows it to be extended to a complex structure on the whole manifold~$M_P$; moreover,~$(M_P,J_P)$ is~$\bb{T}^n$-equivariantly biholomorphic to~$(M,J)$. We recall from~\cite{Guillemin_toric} the construction of this complex structure~$J_P$: let~$S_1,\dots,S_r$ be the faces of the polytope~$P$, defined as~$S_r:=\set*{\ell^r(\bm{y})=0}$ for some affine-linear functions~$\ell^r$ such that
\begin{equation*}
P=\set*{\bm{y}\in\bb{R}^n\tc\ell^r(\bm{y})\geq 0\ \forall r}.
\end{equation*}
Then the potential~$u_P$ for the canonical complex structure~$J_P$ is
\begin{equation}\label{eq:can_potential}
u_P:=\sum_r\ell^r(\bm{y})\,\mrm{log}\,\ell^r(\bm{y}).
\end{equation}
The following result of Abreu describes all possible integrable compatible complex structures~$J$ on~$P$; more precisely, it shows that any~$\bb{T}$-invariant complex structure on~$M$ has the same behaviour of~$J_P$ near the boundary of the moment polytope: this boundary behaviour is what we refer as \emph{Guillemin's boundary conditions} for a symplectic potential on the polytope~$P$.
\begin{thm}[\cite{Abreu_toric}, Thm.~$2.8$]\label{thm:complex_str}
Every integrable compatible complex structure~$J$ is given by a potential~$u=u_P+h$, where~$u_P(\bm{y})$ is defined by~\eqref{eq:can_potential}, and~$h$ is a smooth function on the whole polytope such that~$\mrm{Hess}_{\bm{y}}u$ is positive definite on~$P^\circ$ and has determinant of the form
\begin{equation*}
\mrm{det}(\mrm{Hess}_{\bm{y}}(u))=\left(\delta(\bm{y})\prod_{r}\ell^r(\bm{y})\right)^{-1}
\end{equation*}
for some strictly positive function~$\delta\in\m{C}^0(P)$.
\end{thm}
In particular, consider the matrix~$G^{-1}:=\left(\mrm{Hess}_{\bm{y}}(u_P+h)\right)^{-1}$. We claim that it is a continuous and bounded function on the whole polytope; by our assumptions it is invertible in~$P^\circ$, and we claim that as~$\bm{y}$ tends to a face~$S_r$, the matrix~$G^{-1}$ acquires a kernel containing the vector~$\nabla\ell^r$. In particular,~$G^{-1}$ vanishes at the vertices of~$P$.

We write down the details of the proof of this claim just in the~$2$-dimensional case, since it is notationally easier and in what follows we will mainly focus on the complex surfaces.\todo{$2$-dimensional case} 
\begin{proof}\label{boundary_coords}
It is enough to show this for the first edge~$S_1=\set*{\ell^1(\vec{y})=0}$ and the vertex~$p_0$ at the intersection of~$S_1$ and~$S_2$. Consider the linear system of coordinates~$\bm{p}$ centred at~$p_0$ and generated by~$\nabla\ell^1$,~$\nabla\ell^2$; up to reordering the edges of~$P$ we can assume that~$(\nabla\ell^1,\nabla\ell^2)$ is positively oriented, and if we let~$L$ be the matrix whose columns are~$\nabla\ell^1$ and~$\nabla\ell^2$, the two system of coordinates~$\bm{y}$ and~$\bm{p}$ are related by the affine transformation~$\bm{y}=L\bm{p}+p_0$.

Notice that~$\mrm{Hess}_{\bm{y}}u=L^{-1}\mrm{Hess}_{\bm{p}}u\,L^{-1}$: it will be enough to prove our claim for the matrix~$\left(\mrm{Hess}_{\bm{p}}u\right)^{-1}$. More precisely, since in this new coordinate system the side~$S_1$ is described as~${p^1=0}$, we will consider the limit for~$p^1\to 0$ of the matrix~$\left(\mrm{Hess}_{\bm{p}}(u)\right)^{-1}$.
\begin{equation*}
\mrm{Hess}_{\bm{p}}u=\begin{pmatrix}
\frac{1}{p^1}+\sum_{r>2}\frac{\nabla\ell^r_1\,\nabla\ell^r_1}{\ell^r} & \sum_{r>2}\frac{\nabla\ell^r_1\,\nabla\ell^r_2}{\ell^r}\\ \sum_{r>2}\frac{\nabla\ell^r_1\,\nabla\ell^r_2}{\ell^r} & \frac{1}{p^2}+\sum_{r>2}\frac{\nabla\ell^r_2\,\nabla\ell^r_2}{\ell^r}
\end{pmatrix}+\begin{pmatrix}
h_{11} & h_{12} \\ h_{12} & h_{22}
\end{pmatrix}
\end{equation*}
\begin{equation}\label{eq:Hess_det}
\begin{gathered}
p^1p^2\left(\mrm{det}\,\mrm{Hess}_{\bm{p}}u\right)=1+p^1\left(h_{11}+\sum_{r>2}\frac{\nabla\ell^r_1\,\nabla\ell^r_1}{\ell^r}\right)+p^2\left(h_{22}+\sum_{r>2}\frac{\nabla\ell^r_2\,\nabla\ell^r_2}{\ell^r}\right)+\\
+p^1p^2\left[\left(h_{11}+\sum_{r>2}\frac{\nabla\ell^r_1\,\nabla\ell^r_1}{\ell^r}\right)\left(h_{22}+\sum_{r>2}\frac{\nabla\ell^r_2\,\nabla\ell^r_2}{\ell^r}\right)-\left(h_{12}+\sum_{r>2}\frac{\nabla\ell^r_1\,\nabla\ell^r_2}{\ell^r}\right)^{2}\right]
\end{gathered}
\end{equation}
\begin{equation}\label{eq:Hess_inverso}
\left(\mrm{Hess}_{\bm{p}}u\right)^{-1}=\frac{1}{\mrm{det}\,\mrm{Hess}_{\bm{p}}u}\begin{pmatrix}
h_{22}+\frac{1}{p^2}+\underset{r>2}{\sum}\frac{\nabla\ell^r_2\,\nabla\ell^r_2}{\ell^r} & -h_{12}-\underset{r>2}{\sum}\frac{\nabla\ell^r_1\,\nabla\ell^r_2}{\ell^r}\\ -h_{12}-\underset{r>2}{\sum}\frac{\nabla\ell^r_1\,\nabla\ell^r_2}{\ell^r} & h_{11}+\frac{1}{p^1}+\underset{r>2}{\sum}\frac{\nabla\ell^r_1\,\nabla\ell^r_1}{\ell^r}
\end{pmatrix}
\end{equation}
so as~$p^1$ goes to~$0$ we see that~$\left(\mrm{Hess}_{\bm{p}}u\right)^{-1}$ tends to
\begin{equation*}
\frac{1}{p^2h_{22}+1+p^2\sum_{r>2}\frac{\nabla\ell^r_2\nabla\ell^r_2}{\ell^r}}\begin{pmatrix} 0 & 0 \\ 0 & p^2
\end{pmatrix}.
\end{equation*}
Since we know from Theorem~\ref{thm:complex_str} that the denominator is strictly positive on the edges of~$P$, this matrix-valued function is continuous on~$S_1$ and its kernel is generated by~$(1,0)$, i.e.~$\nabla\ell^1$. Moreover,~$\mrm{Hess}^{-1}_{\bm{p}}u$ vanishes at~$p_0$, i.e. for~$p^1,p^2\to0$.

In the higher-dimensional case the situation is analogous; in a system of linear coordinates centred in a vertex~$p_0$ and generated by the vectors~$\set*{\nabla\ell^r\tc p_0=\cap S_r}$ the matrix is of the shape
\begin{equation*}
\mrm{Hess}_{\bm{p}}(u)^{-1}=
\begin{pmatrix}
a_1 	& a_2 	& \dots	& a_n 	\\
a_2 	& 		& 		& 		\\
\vdots 	& 		&	H	& 		\\
a_n 	&		&		& 		\\
\end{pmatrix}
\end{equation*}
where~$a_1,\dots,a_n$ are functions on~$P$ that vanish for~$p^1\to 0$, and~$H$ is a~$(n-1)\times(n-1)$ matrix-valued function on~$P$ that vanishes on the vertices of~$P$.
\end{proof}

\subsection{Boundary conditions for a deformation of the complex structure}\label{sec:symp_boundary_toric}

From Theorem~\ref{thm:complex_str} we can describe all integrable deformations of the complex structure~$J$ in the symplectic coordinates~$\bm{y}$ of~$P$: assume that~$J$ is defined by the potential~$u$, and let~$G:=\mrm{Hess}_{\bm{y}}(u)$. For a function~$\varphi\in\m{C}^\infty(P)$ let~$\Phi:=\mrm{Hess}_{\bm{y}}(\varphi)$ and consider
\begin{equation*}
\dot{J}:=\begin{pmatrix} 0 & G^{-1}\Phi G^{-1}\\ \Phi & 0\end{pmatrix}.
\end{equation*}
Then~$\dot{J}$ is the deformation of the complex structure~$J$ corresponding to the path of potentials~$u+\varepsilon\varphi$. Since~$\mscr{J}$ is a complex manifold, also~$J\dot{J}$ is a first-order deformation of~$J$, so the general form of an integrable deformation of the complex structure is
\begin{equation*}
\dot{J}=\begin{pmatrix} G^{-1}A & G^{-1}B G^{-1}\\ B &- A G^{-1}\end{pmatrix}
\end{equation*}
where~$A$ and~$B$ are the Hessian matrices of two functions in~$\m{C}^\infty(P)$, and in particular they are smooth matrix-valued functions on~$P$. If we do not assume~$A$ and~$B$ to be Hessians, but rather just symmetric matrices, then we obtain deformations of the complex structure that are not first-order integrable. This description of the first-order deformations of~$J$ is completely analogous to the one we had for abelian varieties, see Section~\ref{sec:symplectic_coords_proof}. Correspondingly, the matrix~$\xi$ has the same description as in Lemma~\ref{lemma:xi_caratt_alt}
\begin{equation}\label{eq:def_cs_sympcoord}
\xi=G^{-1}(A+\I B)G^{-1}.
\end{equation}
We see that, essentially, the boundary behaviour of a deformation~$\xi$ of the complex structure is determined by that of~$J$. More precisely, consider a symplectic potential~$u$ satisfying Guillemin's boundary conditions (c.f. Theorem~\ref{thm:complex_str}), that defines a~$\bb{T}$-invariant integrable complex structure~$J$. Then for any~$A,B\in\m{C}^\infty(P,\bb{R}^{n\times n})$ the matrix~$\xi$ defined by~\eqref{eq:def_cs_sympcoord} comes from a~$\bb{T}$-invariant Higgs term~$\alpha\in T^*_J\!\mscr{J}$, and vice-versa.
\begin{rmk}
One of the conditions of the HcscK system is that~$\norm{\xi}^2_u<1$; letting~$G=\mrm{Hess}_{\bm{y}}u$ we have
\begin{equation*}
\norm{\xi}^2_u=\mrm{Tr}\left(\xi\,G\,\bar{\xi}\,G\right).
\end{equation*}
For~$\bm{y}\to\diff S_r$ however, we know that some components of~$G$ blow up along the direction~$\nabla\ell^r$, so the condition~$\norm{\xi}^2_u<1$ can be satisfied in~$P$ only if~$\xi$ vanishes for~$\bm{y}\to\diff P$ along the normal direction. Equation~\eqref{eq:def_cs_sympcoord} implies that this is the case: since~$\xi=G^{-1}\Phi G^{-1}$ for some symmetric matrix-valued function~$\Phi\in\m{C}^\infty(P,\bb{C}^{n\times n})$ then
\begin{equation*}
\norm{\xi}^2_u=\mrm{Tr}\left(\Phi\,G^{-1}\,\bar{\Phi}\,G^{-1}\right)
\end{equation*}
is well-defined (and finite) up to the boundary of~$P$.
\end{rmk}

\subsection{The complex moment map on a toric surface}

In this Section we consider the complex moment map equation in complex dimension~$2$, with the aim of studying the integrability condition for a solution to the HcscK system on a toric surface.

Assume that~$\xi$ is of the form~\eqref{eq:def_cs_sympcoord} for some smooth matrix-valued functions~$A$ and~$B$. In particular~$\xi$ is bounded on~$P$, since~$G^{-1}$ is; also the function~$\xi^{ab}_{,ab}$ is bounded on~$P$, as it is just the expression in symplectic coordinates of~$\mrm{div}\left(\diff^*\alpha\right)$, up to a constant.

\begin{lemma}\label{lemma:xi_complex_mm_intbyparts}
Let~$\xi$ be the representative of a first-order deformation of the complex structure. Then for any~$f\in\m{C}^\infty(P)$ we have
\begin{equation*}
\int_P(\xi^{ij})_{,ij}f\mrm{d}\mu=\int_P\xi^{ij}f_{,ij}\mrm{d}\mu.
\end{equation*}
\end{lemma}
To prove this result it will be notationally convenient to let~$\mrm{d}\sigma$ be the measure on~$\diff P$ that on each side~$S_r$ is given by the Lebesgue measure multiplied by~$\abs{\nabla\ell^r}^{-2}$; this is the same measure~$\mrm{d}\sigma$ considered by Donaldson in~\cite{Donaldson_stability_toric}.
\begin{proof}
For any sufficiently regular function~$f$ on~$P$, integration by parts with respect to the Lebesgue measure~$\mrm{d}\mu$ gives
\begin{equation*}
\int_{P}\xi^{ij}_{,ij}\,f\,\mrm{d}\mu=-\int_{\diff P}\left(\xi\,\nabla\ell\right)^{i}_{,i}f\,\mrm{d}\sigma+\int_{\diff P}\left(\xi\,\nabla\ell\right)^{i}f_{,i}\,\mrm{d}\sigma+\int_P\xi^{ij}\,f_{,ij}\,\mrm{d}\mu
\end{equation*}
since the outer normal to the boundary is given, at the side~$S_r$, by~$-\frac{\nabla\ell^r}{\abs{\nabla\ell^r}^2}$. The proof consists in showing that the boundary terms vanish.

First, since~$\xi$ represents a deformation of the complex structure we have~$\xi=G^{-1}\Phi G^{-1}$ for some symmetric~$\Phi\in\m{C}^\infty(P,\bb{C}^{n\times n})$, so that on the boundary of~$P$ we have~$\xi\nabla\ell=0$. Then the integration term~$\left(\xi\,\nabla\ell\right)^{i}\,f_{,i}$ vanishes, for smooth~$f$.

For any side~$S_r$ of~$\diff P$, let~$V$ be the vector field~$V:=\xi\nabla\ell^r$. We claim that also~$\mrm{div}(V)$ vanishes on~$S_r$; this would imply~$\int_{\diff P}\left(\xi\,\nabla\ell\right)^{i}_{,i}f\,\mrm{d}\sigma=0$, concluding the proof.

We can assume without loss of generality that~$r=1$,~$\ell^1=y^1$ and~$\ell^2=y^2$, so that we are in the same situation considered in our previous description of~$G^{-1}$, c.f. pag. \pageref{boundary_coords}. We have~$\nabla\ell^1=(1,0)$, and we know that~$G^{-1}=\begin{pmatrix}a & b \\ b & c\end{pmatrix}$ for some functions~$a$,~$b$ and~$c$ such that~$a\to 0$ and~$b\to 0$ as~$y^1\to 0$. More precisely, from~\eqref{eq:Hess_inverso} we have
\begin{equation*}
\begin{split}
a=&\frac{1}{\mrm{det}\,\mrm{Hess}(u)}\left(
h_{22}+\frac{1}{p^2}+\sum_{r>2}\frac{\nabla\ell^r_2\,\nabla\ell^r_2}{\ell^r}\right),\\
b=&-\frac{1}{\mrm{det}\,\mrm{Hess}(u)}\left(h_{12}+\sum_{r>2}\frac{\nabla\ell^r_1\,\nabla\ell^r_2}{\ell^r}\right)
\end{split}
\end{equation*}
where~$\mrm{det}\,\mrm{Hess}(u)$ is given by~\eqref{eq:Hess_det}. We can easily compute
\begin{equation}\label{eq:derivate_ab}
\begin{gathered}
\lim_{y^1\to 0}\diff_{y^1}a=1;\quad\lim_{y^1\to 0}\diff_{y^2}a=0;\\
\lim_{y^1\to 0}\diff_{y^1}b=-\frac{y^2\left(h_{12}+\sum_{r>2}\frac{\nabla\ell^r_1\nabla\ell^r_2}{\ell^r}\right)}{y^2h_{22}+1+y^2\sum_{r>2}\frac{\nabla\ell^r_2\nabla\ell^r_2}{\ell^r}};\quad\lim_{p^1\to 0}\diff_{y^2}b=0.
\end{gathered}
\end{equation}
Since~$u_P+h$ satisfies Abreu's conditions on the determinant of~$G=\mrm{Hess}(u_P+h)$ (c.f. Theorem~\ref{thm:complex_str}), the fraction that gives~$\lim_{y^1\to 0}\diff_{y^1}b$ is non-singular on~$\diff P$. Now,~$V$ is
\begin{equation*}
V(\bm{y})=\begin{pmatrix} a & b \\ b & c \end{pmatrix}
\begin{pmatrix} \Phi_{11} & \Phi_{12} \\ \Phi_{12} & \Phi_{22}\end{pmatrix}
\begin{pmatrix} a \\ b \end{pmatrix}=
\begin{pmatrix} \Phi_{11}\,a^2+2\,\Phi_{12}\,a\,b+\Phi_{22}\,b^2\\
\Phi_{11}\,a\,b+\Phi_{12}\left(b^2+a\,c\right)+\Phi_{22}\,b\,c
\end{pmatrix}.
\end{equation*}
Since~$\Phi_{ij}$ are smooth functions on~$P$, it is clear from~\eqref{eq:derivate_ab} that~$\mrm{div}\,V=0$ on~$S_1$, thus proving our claim.
\end{proof}
\begin{cor}\label{cor:toric_integrability}
The matrix~$\xi=G^{-1}\Phi G^{-1}$ is a solution to the complex moment map equation if and only if~$\Phi$ is orthogonal to the space~$\set*{\mrm{Hess}(f)\tc f\in\m{C}^\infty(P)}$ with respect to the~$L^2$ product on~$\m{C}(P)$ defined by~$G$. In particular, if~$G$ and~$\xi$ solve the HcscK system~\eqref{eq:symplectic_HcscK_toric}, then~$\xi$ is integrable if and only if it is identically~$0$.
\end{cor}

\subsection{\textit{K}-stability and the HcscK system on toric surfaces}

We would like to have a variational formulation of the HcscK problem on a toric manifold, similarly to what we saw in Proposition~\ref{HKEulLagr} for abelian varieties. In that situation a simple computation showed that the real moment map equation of the HcscK system is the Euler-Lagrange equation for the periodic HK-functional~$\hat{\m{F}}$, c.f.~\eqref{eq:HK_energy}. However on a toric manifold the computation of Proposition~\ref{HKEulLagr} has to be modified, taking into account some boundary terms arising from the integration by parts that is required.

Fix a symplectic potential~$u$, let~$G:=\mrm{Hess}(u)$ and assume that~$\xi=G^{-1}\Phi G^{-1}$, as in~\eqref{eq:def_cs_sympcoord}. Consider the real moment map equation of~\eqref{eq:symplectic_HcscK_toric}
\begin{equation*}
\left(\left(1-G^{-1}\Phi\,G^{-1}\bar{\Phi}\right)^{\frac{1}{2}}G^{-1}\right)^{ab}_{,ab}=-C.
\end{equation*}
To compute an analogue of the periodic HK-energy in the toric setting, we need an integration by parts formula.
\begin{lemma}\label{lemma:intbyparts_realmm}
For a Delzant polytope~$P\subset\bb{R}^2$, fix a symmetric matrix-valued function~$\Phi\in\m{C}^\infty(P,\bb{C}^{2\times 2})$ and a symplectic potential~$u$. For any function~$f\in\m{C}^0(P)\cap\m{C}^\infty(P^\circ)$ that is either convex or smooth on~$P$ we have 
\begin{equation*}
\begin{gathered}
\int_P\mrm{Tr}\left(\left(1-G^{-1}\Phi\,G^{-1}\bar{\Phi}\right)^{\frac{1}{2}}G^{-1}\,D^2f\right)\mrm{d}\mu=\int_Pf\left(\left(1-G^{-1}\Phi\,G^{-1}\bar{\Phi}\right)^{\frac{1}{2}}G^{-1}\right)^{ab}_{,ab}\mrm{d}\mu+\int_{\diff P}\!f\,\mrm{d}\sigma.
\end{gathered}
\end{equation*}
\end{lemma}
We follow the proof of \cite[Lemma~$3.3.5$]{Donaldson_stability_toric}. 
\begin{proof}
For~$\delta>0$, let~$P_\delta$ be a polytope contained in~$P$, with sides parallel to those of~$P$ separated by a distance~$\delta$. Inside~$P_\delta$ both~$u$ and~$f$ are smooth, so we can integrate by parts and obtain
\begin{equation}\label{eq:intbyparts_realmm}
\begin{gathered}
\int_{P_\delta}\mrm{Tr}\left(\left(1-G^{-1}\Phi\,G^{-1}\bar{\Phi}\right)^{\frac{1}{2}}G^{-1}\,D^2f\right)\mrm{d}\mu=\int_{P_\delta}f\left(\left(1-G^{-1}\Phi\,G^{-1}\bar{\Phi}\right)^{\frac{1}{2}}G^{-1}\right)^{ab}_{,ab}\mrm{d}\mu+\\
+\int_{\diff P_\delta}f\,\mrm{div}\left(\left(1-G^{-1}\Phi\,G^{-1}\bar{\Phi}\right)^{\frac{1}{2}}G^{-1}\nabla\ell\right)\mrm{d}\sigma-\int_{\diff P_\delta}\left(\left(1-G^{-1}\Phi\,G^{-1}\bar{\Phi}\right)^{\frac{1}{2}}G^{-1}\nabla\ell\right)^a\diff_af\mrm{d}\sigma.
\end{gathered}
\end{equation}
We will take the limit for~$\delta\to 0$ of the two boundary terms on the right-hand side. First of all, fix a side~$S_{\delta,r}$ of~$P_\delta$, parallel to the side~$S_r$ of~$P$; as in the proof of Lemma~\ref{lemma:xi_complex_mm_intbyparts} we can assume that~$r=1$,~$\ell^1=y^1$ and~$\ell^2=y^2$,  c.f. the discussion about~$G^{-1}$ at pag. \pageref{boundary_coords}. As the distance between~$S_{\delta,r}$ and~$S_{r}$ is~$\delta$, points of~$S_{\delta,r}$ have coordinates~$(\delta,y^2)$.

Recall from~\eqref{eq:derivate_ab} that, in this system of coordinates, the matrix~$G^{-1}$ is of the shape~$\begin{pmatrix} a & b\\ b &c\end{pmatrix}$ for functions~$a$,~$b$ such that
\begin{equation*}
\begin{gathered}
a(y^1,y^2)=y^1+O\left(\left(y^1\right)^2\right)\\
b(y^1,y^2)=\varphi(y^2)\,y^1+O\left(\left(y^1\right)^2\right)
\end{gathered}
\end{equation*}
and~$\varphi\in O(y^2)$. Then we can expand~$G^{-1}\Phi G^{-1}\bar{\Phi}$ in~$\delta$ to find
\begin{equation*}
\begin{split}
\left(1-G^{-1}\Phi G^{-1}\bar{\Phi}\right)^{\frac{1}{2}}&=\left(\begin{pmatrix}1 &0\\-c^2\bar{\Phi}_{12}\Phi_{22} & 1-\abs{c\Phi_{22}}^2\end{pmatrix}+O(\delta)\right)^{\frac{1}{2}}=\\
&=\begin{pmatrix}1&0\\-\frac{c^2\bar{\Phi}_{12}\Phi_{22}}{1+\sqrt{1-\abs{c\Phi_{22}}^2}} & \sqrt{1-\abs{c\Phi_{22}}^2}\end{pmatrix}+O(\delta).
\end{split}
\end{equation*}
Composing with~$G^{-1}\nabla\ell$ we have
\begin{equation}\label{eq:espansione_radice}
\left(1-G^{-1}\Phi G^{-1}\bar{\Phi}\right)^{\frac{1}{2}}G^{-1}\nabla\ell=\begin{pmatrix}a\\b\sqrt{1-\abs{c\Phi_{22}}^2}-\frac{ac^2\bar{\Phi}_{12}\Phi_{22}}{1+\sqrt{1-\abs{c\Phi_{22}}^2}}\end{pmatrix}+O(\delta^2)
\end{equation}
and the divergence of this vector is
\begin{equation*}
\mrm{div}\left(\left(1-G^{-1}\Phi G^{-1}\bar{\Phi}\right)^{\frac{1}{2}}G^{-1}\nabla\ell\right)=\diff_1a+O(\delta)=1+O(\delta).
\end{equation*}
Hence the first boundary term in~\eqref{eq:intbyparts_realmm} goes to~$\int_{\diff P}f\mrm{d}\sigma$ as~$\delta$ goes to~$0$.

As for the second boundary term, it certainly vanishes if~$f$ is smooth on the whole of~$P$, since we showed in~\eqref{eq:espansione_radice} that~$\left(1-G^{-1}\Phi G^{-1}\bar{\Phi}\right)^{\frac{1}{2}}G^{-1}\nabla\ell\in O(\delta)$. It is slightly more complicated to obtain the same result for a convex function~$f\in\m{C}^0(P)\cap\m{C}^\infty(P^\circ)$, since the gradient of~$f$ might blow up as we go to the boundary of~$P$. The convexity of~$f$ is however sufficient to guarantee that the vanishing of~$G^{-1}\nabla\ell$ will balance out the growth of~$\nabla f$. We follow the approach of \cite{Donaldson_stability_toric}.

Let~$V$ be the real part of~$\left(1-G^{-1}\Phi G^{-1}\bar{\Phi}\right)^{\frac{1}{2}}G^{-1}\nabla\ell$. Since all the terms in~\eqref{eq:intbyparts_realmm} are real, at least in the limit for~$\delta\to0$, we just have to show that
\begin{equation*}
\int_{\diff P_\delta}\nabla_Vf\mrm{d}\mu\to 0\mbox{ as }\delta\to0.
\end{equation*}
For~$p\in S_{\delta,r}$, let~$q=q(p)$ be the closest point to~$p$ at the intersection between~$\diff P$ and the ray~$p-tV$. Notice that, as~$\nabla\ell^\transpose\cdot V\geq 0$, the vector~$V$ points inward from~$p$ to~$P_\delta$. Moreover, the slope of~$V$ is 
\begin{equation*}
\frac{b}{a}\sqrt{1-\abs{c\Phi_{22}}^2}-\frac{c^2\,\mrm{Re}\left(\bar{\Phi}_{12}\Phi_{22}\right)}{1+\sqrt{1-\abs{c\Phi_{22}}^2}}+O(\delta).
\end{equation*}
We can choose~$\delta_0$ such that, for~$\delta<\delta_0$ and every~$p\in S_{\delta,r}$,~$q(p)$ belongs to the side~$S_r$ of~$\diff P$, as the slope of~$V$ is uniformly bounded on~$S_{\delta,r}$.

Let now~$q'=q'(p)=p-(q-p)$, and consider the norm of~$V$ and the distance between~$p$ and~$q$,~$q'$. As~$q$ lies on~$S_r$, we have
\begin{equation*}
\abs{p-q}=\abs{p-q'}=\delta\abs*{\left(1,\ \frac{b}{a}\sqrt{1-\abs{c\Phi_{22}}^2}-\frac{c^2\,\mrm{Re}\left(\bar{\Phi}_{12}\Phi_{22}\right)}{1+\sqrt{1-\abs{c\Phi_{22}}^2}}\right)}
\end{equation*}
and so there is some positive constant~$c$ such that, for~$\delta<\delta_0$
\begin{equation*}
\abs{V}=\abs*{\left(a,b\sqrt{1-\abs{c\Phi_{22}}^2}-\frac{ac^2\,\mrm{Re}\left(\bar{\Phi}_{12}\Phi_{22}\right)}{1+\sqrt{1-\abs{c\Phi_{22}}^2}}\right)+O(\delta^2)}\leq c\abs{p-q}.
\end{equation*}
By the convexity of~$f$, we know that at~$p\in S_{\delta,r}$
\begin{equation*}
\abs*{\nabla_Vf(p)}\leq\frac{\abs{V}}{\abs{q-p}}\mrm{max}\set*{f(q)-f(p),f(q')-f(p)}\leq c\,\mrm{max}\set*{f(q)-f(p),f(q')-f(p)}
\end{equation*}
so we can estimate the second boundary term in~\eqref{eq:intbyparts_realmm} as
\begin{equation*}
\begin{split}
\abs*{\int_{\diff P_\delta}\nabla_Vf\mrm{d}\mu}\leq&\int_{\diff P_\delta}\abs*{\nabla_Vf}\mrm{d}\mu\leq c\int_{p\in\diff P_\delta}\mrm{max}\set*{f(q(p))-f(p),f(q'(p))-f(p)}\mrm{d}\mu\leq\\
\leq&c\,\mrm{Vol}(\diff P_\delta)\max_{p\in\diff P_\delta}\set*{f(q)-f(p),f(q')-f(p)}.
\end{split}
\end{equation*}
As~$f$ is uniformly continuous on~$P$ and~$\abs{p-q}\in O(\delta)$, this inequality shows that, as~$\delta$ goes to~$0$,
\begin{equation*}
\int_{\diff P_\delta}\nabla_Vf\mrm{d}\mu\to 0.\qedhere
\end{equation*}
\end{proof}

\begin{lemma}\label{lemma:squareroot_posdef}
With the previous notation,~$\sqrt{1-G^{-1}\Phi G^{-1}\bar{\Phi}}\,G^{-1}$ is a Hermitian positive-definite matrix.
\end{lemma}
\begin{proof}
We can write~$G^{-1}$ as~$S^\transpose\Lambda S$ for a diagonal positive matrix~$\Lambda$ and an orthogonal matrix~$S$, so that
\begin{equation*}
G^{-1}\Phi G^{-1}\bar{\Phi}=S^\transpose\Lambda^{\frac{1}{2}}\left(\Lambda^{\frac{1}{2}}S\Phi S^\transpose\Lambda^{\frac{1}{2}}\right)\left(\Lambda^{\frac{1}{2}}S\bar{\Phi}S^\transpose\Lambda^{\frac{1}{2}}\right)\Lambda^{-\frac{1}{2}}S.
\end{equation*}
Let~$R:=\Lambda^{\frac{1}{2}}S\Phi S^\transpose\Lambda^{\frac{1}{2}}$, so that
\begin{equation*}
\sqrt{1-G^{-1}\Phi G^{-1}\bar{\Phi}}\,G^{-1}=S^\transpose\Lambda^{\frac{1}{2}}\sqrt{1-R\bar{R}}\Lambda^{-\frac{1}{2}}S\,S^\transpose\Lambda S=S^\transpose\Lambda^{\frac{1}{2}}\sqrt{1-R\bar{R}}\Lambda^{\frac{1}{2}}S
\end{equation*}
and since the eigenvalues of~$R\bar{R}$ are smaller than~$1$,~$\sqrt{1-R\bar{R}}$ is positive-definite.
\end{proof}

\begin{cor}\label{cor:toric_Futaki}
Assume that~$(\xi,u)$ is a solution to the HcscK system~\eqref{eq:symplectic_HcscK_toric}, that~$u$ satisfies Guillemin's boundary condition, and that~$\xi=D^2u^{-1}\Phi\,D^2u^{-1}$ for a symmetric complex matrix~$\Phi\in\m{C}^\infty(P)$. Then for every convex function~$f\in\m{C}^0(P)\cap\m{C}^\infty(P^\circ)$
\begin{equation*}
\int_{\diff P}f\,\mrm{d}\sigma-\int_PC\,f\,\mrm{d}\mu\geq 0.
\end{equation*}
With equality if and only if~$f$ is affine-linear.
\end{cor}

\begin{proof}
From Lemma~\ref{lemma:intbyparts_realmm} we see that if~$(\xi,u)$ is a solution to the real moment map for the toric HcscK system then
\begin{equation*}
\int_P\mrm{Tr}\left(\left(1-G^{-1}\Phi\,G^{-1}\bar{\Phi}\right)^{\frac{1}{2}}G^{-1}\,D^2f\right)\mrm{d}\mu=-\int_Pf\,C\mrm{d}\mu+\int_{\diff P}f\,\mrm{d}\sigma.
\end{equation*}
By the convexity of~$f$ and Lemma~\ref{lemma:squareroot_posdef}, the left hand side is non-negative, and vanishes when~$f$ is affine-linear.
\end{proof}
Linear functions on~$P$ correspond to~$\bb{T}$-invariant holomorphic vector fields on the toric surface via the Hamiltonian construction, and in fact the functional
\begin{equation*}
\m{L}_C(f)=\int_{\diff P}f\,\mrm{d}\sigma-\int_PC\,f\,\mrm{d}\mu%
\end{equation*}
is the Futaki invariant of~$M$ valued on the vector field defined by~$f$, c.f.~\cite[Lemma~$3.2.9$]{Donaldson_stability_toric}. For general convex functions~$f$,~$\m{L}_C(f)$ is the Donaldson-Futaki invariant of a toric degeneration of the surface induced by~$f$. The condition~$\m{L}_C(f)\geq 0$ is equivalent to toric~$K$-stability of the surface, see again \cite{Donaldson_stability_toric}, and implies general~$K$-stability of the surface. In the series of papers culminating in \cite{Donaldson_cscKmetrics_toricsurfaces}, Donaldson showed that toric~$K$-stability implies the existence of a cscK metric.

Corollary~\ref{cor:toric_Futaki} then implies that a necessary condition for the existence of a solution to the toric HcscK system is that there is a symplectic potential~$u_0$ satisfying Abreu's equation~$(u_0^{ij})_{,ij}=-C$. It is not clear, however, if~$K$-stability is sufficient to guarantee the existence of a solution~$(\xi,u)$ for non-zero Higgs terms~$\xi$.

A possible alternative way to study the stability of a polarized toric surface with the addition of a non-integrable Higgs term comes from a variational characterization of the real moment map equation in~\eqref{eq:symplectic_HcscK_toric}. The integration by parts formula for smooth test functions of Lemma~\ref{lemma:intbyparts_realmm} allows us to describe the real moment map equation of~\eqref{eq:symplectic_HcscK_toric} as an Euler-Lagrange equation, for a fixed symmetric complex matrix~$\Phi\in\m{C}^\infty(P,\bb{C}^{2\times 2})$, that induces the Higgs term~$\xi=G^{-1}\Phi G^{-1}$.
\begin{lemma}\label{lemma:toricHKenergy}
Consider the space of symplectic potentials (c.f. Theorem~\ref{thm:complex_str})
\begin{equation*}
\m{A}(\Phi)=\set*{u=u_P+h\tc r(D^2u^{-1}\Phi\,D^2u^{-1}\bar{\Phi})<1}.
\end{equation*}
The real moment map equation in~\eqref{eq:symplectic_HcscK_toric} is the Euler-Lagrange equation of the HK-functional defined on~$\m{A}(\Phi)$ as
\begin{equation*}
\begin{split}
\hat{\m{F}}(u,\Phi)=&\int_{\diff P}u\,\mrm{d}\sigma-\int_PC\,u\,\mrm{d}\mu-\int_P\log\mrm{det}\left(D^2u\right)\mrm{d}\mu-\int_P\rho\left(D^2u^{-1}\Phi\,D^2u^{-1}\bar{\Phi}\right)\mrm{d}\mu.
\end{split}
\end{equation*}
\end{lemma}

As we are not able yet to characterize solutions to the toric HcscK system by an \emph{algebraic} stability notion, a possible alternative approach could be to study the relationship between the toric~$\textit{HK}$-energy of Lemma~\ref{lemma:toricHKenergy} and its linear part~$\m{L}_C$, along the lines of the study of the toric~$K$-energy in~\cite{Donaldson_stability_toric}. This might shed some light into a possible stability condition, generalizing~$K$-stability, to characterize the existence of solutions to the HcscK system on a toric surface at least from an analytic point of view.
	
\backmatter

\addcontentsline{toc}{chapter}{Bibliography}
\bibliographystyle{amsalpha}
\bibliography{../bibliografia_PhD}
\end{document}